\documentclass[a4paper]{article}
\usepackage{arxiv}
\usepackage{comment}
\usepackage[utf8]{inputenc} 
\usepackage[T1]{fontenc}    
\usepackage{hyperref}       
\usepackage{url}            
\usepackage{booktabs}       
\usepackage{amsmath}
\usepackage{amssymb}
\usepackage{amsthm}
\usepackage{amsfonts}       
\usepackage{nicefrac}       
\usepackage{microtype}      
\usepackage{xcolor}
\usepackage{graphicx}
\usepackage{stmaryrd}
\usepackage{bm}
\usepackage[normalem]{ulem}
\usepackage{algorithmicx}
\usepackage{algorithm}
\usepackage{algpseudocode}
\usepackage{caption}
\usepackage{subcaption}
\usepackage{lineno}
\usepackage{orcidlink}
\usepackage[inline,final]{showlabels}

\usepackage{multirow}
\usepackage{comment}
\usepackage{enumitem}

\usepackage{xpatch}
\makeatletter
\xpatchcmd{\algorithmic}
{\ALG@tlm\z@}{\leftmargin\z@\ALG@tlm\z@}
{}{}
\makeatother

\hypersetup{
	colorlinks=true, 
	linkcolor=blue, 
	urlcolor=red, 
	citecolor=magenta,
	linktoc=all 
}

\newlength\myindent 
\newcommand{\Uad}{\mathcal{U}_{\textrm{ad}}}

\newtheorem{lemma}{Lemma}
\newtheorem{theorem}{Theorem}
\newtheorem{remark}{Remark}
\DeclareUnicodeCharacter{0301}{\hspace{-1ex}\'{ }} 

\newcommand{\mb}[1]{\boldsymbol{#1}}
\newcommand{\p}{\partial}

\newcommand{\numflux}[1]{{#1}^{\text{NF}}}

\newcommand{\highorderflux}[1]{{#1}^{\text{H}}}
\newcommand{\iniguessflux}[1]{{#1}^{\text{IG}}}

\newcommand{\new}[1]{{#1}^{\text{new}}}
\newcommand{\old}[1]{{#1}^{\text{old}}}

\usepackage{fancyvrb}
\definecolor{com}{RGB}{0, 100, 100}

\definecolor{spl}{RGB}{0, 140, 10}

\definecolor{mac}{RGB}{150,0,211}

\definecolor{prt}{RGB}{200, 0, 50}

\definecolor{mth}{RGB}{100, 100, 100}

\title{Constraints Preserving Lax-Wendroff Flux Reconstruction for Relativistic Hydrodynamics with General Equations of State}

\author{
	Sujoy Basak \orcidlink{0009-0009-0612-6361}\thanks{Corresponding author}\\
	Department of Mathematics\\
	Indian Institute of Technology Delhi\\
	New Delhi -- 110016, India\\
	Indian Institute of Technology Delhi-Abu Dhabi,\\ Khalifa City B, Abu Dhabi, UAE\\
	\texttt{sujoybasak42@gmail.com} \\
	\And
	Arpit~Babbar \orcidlink{0000-0002-9453-370X} \\
	Institute of Mathematics\\
	Johannes Gutenberg University Mainz\\
	Staudingerweg 9, 55122 Mainz, Germany\\
	\texttt{ababbar@uni-mainz.de} \\
	\And
	Harish Kumar \orcidlink{0000-0003-4746-2336}\\
	Department of Mathematics\\
	Indian Institute of Technology Delhi\\
	New Delhi -- 110016, India\\
	Indian Institute of Technology Delhi-Abu Dhabi,\\ Khalifa City B, Abu Dhabi, UAE\\
	\texttt{hkumar@iitd.ac.in} \\
	\And
	Praveen Chandrashekar \orcidlink{0000-0003-1903-4107}\\
	Centre for Applicable Mathematics\\
	Tata Institute of Fundamental Research\\
	Bangalore -- 560065, India\\
	\texttt{praveen@math.tifrbng.res.in}
}

\begin{document}
	\maketitle
	\begin{abstract}
		In the realm of relativistic astrophysics, the ideal equation of state with a constant adiabatic index provides a poor approximation due to its inconsistency with relativistic kinetic theory. However, it is a common practice to use it for relativistic fluid flow equations due to its simplicity. Here we develop a high-order Lax-Wendroff flux reconstruction method on Cartesian grids for solving relativistic hydrodynamics equations with several general equations of state available in the literature. We also study the conversion from conservative to primitive variables, which depends on the equation of state in use, and provide an alternative method of conversion when the existing approach does not succeed. For the admissibility of the solution, we blend the high-order method with a low-order method on sub-cells and prove its physical admissible property in the case of all the equations of state used here. Lastly, we validate the scheme by several test cases having strong discontinuities, large Lorentz factor, and low density or pressure in one and two dimensions.
	\end{abstract}
	\keywords{Relativistic hydrodynamics \and Equation of state \and Lax-Wendroff flux reconstruction \and Constraints preservation}
	\section{Introduction} 
	The equations of relativistic hydrodynamics (RHD hereafter) play a crucial role in studying various astrophysical phenomena like astrophysical jets, black hole formation, gamma-ray bursts, X-ray binaries~\cite{begelman1984theory,bottcher2012relativistic,mirabel1999sources,zensus1997parsec}, etc. However, an analytical study of these equations is infeasible for many applications due to the presence of strong non-linearities, particularly due to the Lorentz factor arising from relativistic effects. Consequently, numerical simulation becomes the primary way to study these equations. A numerical study of these equations is also not trivial because of the presence of strong shock waves in the solution. In the literature, a wide variety of numerical methods have been developed to solve the RHD equations over the years. Starting with the artificial viscosity technique to capture the shock wave~\cite{wilson1972numerical} given by Wilson, there are other shock capturing methods also present in the literature~\cite{marti1991numerical,marti1994analytical,BALSARA1994,dai1997iterative,ibanez1999riemann}. Various high-order methods like essentially non-oscillatory and weighted essentially non-oscillatory methods~\cite{dolezal1995relativistic,del2002efficient,tchekhovskoy2007wham,chen2022physical}, discontinuous Galerkin method~\cite{radice2011discontinuous}, entropy stable discontinuous Galerkin method~\cite{biswas2022entropy}, approximate piece-wise parabolic reconstruction method~\cite{marti1996extension,aloy1999genesis,mignone2005piecewise}, and other advanced numerical approaches~\cite{wu2014finite,wu2014third,wu2021minimum,xu2024high} have also been developed over the years to solve the RHD equations. For the numerical methods related to other relativistic models, one can also refer to~\cite{bhoriya2023entropy,bhoriya2023high,agnihotri2025second}.
	
	The physically meaningful solution of the RHD equations should satisfy certain {\em constraints} like positivity of density and pressure, and an upper bound on material speeds; this type of solution is also called a {\em physically admissible solution}. However, maintaining the admissibility constraints of the solution is challenging for the numerical schemes, especially at high-orders of accuracy. One remedy would be to use highly diffusive schemes~\cite{zhang2006ram,hughes2002three}, but this approach is time consuming because of the need to use fine spatial grids and small time steps, emphasizing the importance of developing high-order constraints preserving schemes. There has been a significant advancement in this approach in the last decade~\cite{wu2015high,wu2016physical,qin2016bound,wu2017design}. One can also refer to~\cite{wu2023geometric}, where the authors have presented a general framework for transforming the non-linear constraints to linear constraints by introducing some auxiliary variables.
	
	To close the system of RHD equations, an equation of state is needed in addition to the conservation laws. Most of the numerical methods available in the literature use the ideal equation of state for this purpose, but it gives a poor approximation in relativistic cases. This non-realistic equation of state is one of the major concerns in computational astrophysics. Here we will consider general equations of state available in the literature~\cite{mathews1971hydromagnetic,mignone2005piecewise,sokolov2001simple,ryu2006equation} to close the RHD equations. The challenge in developing constraints preserving numerical methods with general equations of state lies in the highly non-linear coupling coming from the effect of general equations of state and the Lorentz factor. More specifically, the difficulty comes from the absence of explicit expressions of primitive variables in terms of laboratory variables, which necessitates solving a non-linear equation for each equation of state.
	
	In~\cite{wu2016physical}, the authors have developed a physical constraints preserving central discontinuous Galerkin method to solve the RHD equations with a general equation of state using Runge-Kutta methods or multi-step methods~\cite{gottlieb2009high} for the time update, which is also a common practice. However, each internal time stage of the Runge-Kutta method involves sharing data between the parallel nodes, and the multi-step methods need information from multiple previous time levels for higher order of accuracy, making these methods inefficient in memory and time constrained environments.
	
	Here, as an alternative, we use the Lax-Wendroff method, which needs only a single step for arbitrary order of accuracy, in contrast to the Runge-Kutta methods. The roots of this method lie in the works of Lax and Wendroff in~\cite{LW1}, where it was introduced as a second-order finite difference method. It was later combined with other finite element frameworks like discontinuous Galerkin~\cite{LWDG2} and flux reconstruction~\cite{lou2020flux}, giving methods with an arbitrary high-order of accuracy. The flux reconstruction method was initially proposed in~\cite{huynh2007flux}, where a continuous approximation of the flux is implemented with the help of a correction function and the numerical fluxes at the cell boundaries. The flux reconstruction method is also suitable for use with the modern vector processors in parallel computing~\cite{vincent2016towards,lopez2014verification,vandenhoeck2019implicit}. Recently, in~\cite{BABBAR2022111423}, the authors have used the approximate Lax-Wendroff procedure~\cite{zorio2017approximate} to give a Jacobian-free Lax-Wendroff flux reconstruction (LWFR hereafter) method that is computationally efficient compared to the earlier version, which uses the chain rule to find the derivative of the fluxes. Here, our primary aim is to design a high-order constraints preserving LWFR method for the RHD equations with various equations of state available in the literature.
	
	The main contributions of this work can be summarized as below:
	\vspace{-0.5em}
	\begin{itemize}
		\item A new method of conversion from conservative to primitive variables is proposed for the case of the equation of state introduced in~\cite{ryu2006equation}. This plays an important role in the successful simulations of several test cases having a high Lorentz factor.
		\item An additional scaling of the non-admissible flux arguments is introduced in Section~\ref{sec:Numerical scheme}. This is essential for the computation of flux functions, as the expressions for the flux functions have primitive variables, and the conversion from conservative to primitive variables needs the quantities to be in the admissible region.
		\item Blending of the high-order scheme with a constraints preserving low-order scheme is used. This assures admissibility of the solution of resulting blended scheme and controls the Gibbs oscillations. We also rigorously prove the constraints preserving nature of the low-order scheme.
		\item Implementation of the proposed scheme in one and two dimensions is carried out for simulating several numerical test cases to validate its accuracy and effectiveness.
	\end{itemize}
	
	The rest of the paper is organized as follows: In Section~\ref{sec:Governing equations}, we discuss some properties of RHD equations along with equations of state and conversion from conservative to primitive variables. In Section~\ref{sec:Numerical scheme}, we discuss the numerical scheme and its constraints preservation property. Then we validate the scheme numerically in one and two dimensions in Section~\ref{sec: Numerical_simulations} for RHD equations with different equations of state using various numerical test cases having large Lorentz factor, low density or pressure, strong discontinuities, etc. Finally, we conclude our work with a summary in Section~\ref{sec: summary}.
	
	\section{Governing equations}\label{sec:Governing equations}
	The equations of  RHD can be written in terms of energy and momentum conservation equations~\cite{sokolov2001simple},
	\begin{align}
		&\frac{1}{c}\frac{\p T^{00}}{\p t}+\sum_{i=1}^d\frac{\p T^{0i}}{\p x_i}=0, \label{eq: energy_momentum con}\\
		&\frac{1}{c}\frac{\p T^{0j}}{\p t}+\sum_{i=1}^d\frac{\p T^{ji}}{\p x_i}=0, \ \ \ \forall\ j=1,\dots,d, \label{eq: momentum con}
	\end{align}
	where,
	\begin{align}
		T^{00} = \frac{w}{1-|\mb{v}|^2/c^2} + p,\quad T^{0j} = \frac{w v_j}{c(1-|\mb{v}|^2/c^2)},\quad T^{ij}=\frac{w v_j v_i}{c^2(1-|\mb{v}|^2/c^2)} + p\delta_{ji},
	\end{align}
	with $w$ being the enthalpy density of the fluid in the local rest frame, $\mb{v}=(v_1, v_2,\dots,v_p)^\top$ is the velocity vector, $p$ is the pressure, and $c$ is the speed of light. One more conservation equation is given by~\cite{sokolov2001simple},
	\begin{align}\label{eq: density_con}
		\frac{\p}{\p t}\frac{\rho}{\sqrt{1-|\mb{v}|^2/c^2}} + \sum_{i=1}^d \frac{\p}{\p x_i}\frac{\rho v_i}{c \sqrt{1-|\mb{v}|^2/c^2}} = 0,
	\end{align}
	where $\rho$ is the fluid density in the local rest frame.  However, in the ultra-relativistic limit, the rest mass density in equation~\eqref{eq: density_con} makes no sense, but this difficulty can be overcome by considering a correct equation of state~\cite{sokolov2001simple}. 
	
	In the rest of this paper, we take the velocity of light to be unity, $c=1$, which is possible by changing the units of time and spatial coordinates.  Now introducing,
	\begin{align*}
		&\textrm{Lorentz factor}, \quad \Gamma = \frac{1}{\sqrt{1-|\mb{v}|^2}},\\ &\textrm{relativistic density}, \quad D = \rho \Gamma,\\ &\textrm{enthalpy}, \quad h=\frac{w}{\rho},
	\end{align*}
	the equations~\eqref{eq: energy_momentum con},~\eqref{eq: momentum con}~and~\eqref{eq: density_con} can be re-written in the following form~\cite{sokolov2001simple},
	\begin{align}\label{eq: RHD_system}
		\frac{\p \mb{u}}{\p t}+\sum_{i=1}^d\frac{\p \mb{f}_i(\mb{u})}{\p x_i}=0,
	\end{align}
	with $\mb{u}$ as the vector of conservative variables, and $\mb{f}_i$ as the flux vector in $x_i$ direction, given by,
	\begin{align}
		\mb{u}&=(D, m_1,m_2,\dots,m_{d}, E)^\top,\\
		\mb{f}_i(\mb{u})&=(Dv_i, m_1v_i+p\delta_{1,i}, m_2v_i+p\delta_{2,i},\dots, m_dv_i+p\delta_{d,i}, m_i)^\top,\mspace{4mu} \forall i=1,\dots, d.\label{eq:RHD.flux}
	\end{align}
	Here,
	\begin{align*}
		&\textrm{energy density}, \quad E={\rho h\Gamma^2}-p, \\
		&\textrm{and the momentum density}, \quad \mb{m}=(m_1,m_2,\dots,m_{d})^\top = \rho h \mb{v}\Gamma^2.
	\end{align*}
	To close the system~\eqref{eq: RHD_system}, we need to express the primitive variables $\rho$, $\mb{v}$, $p$ in terms of the conservative variables $\mb{u}$ with the help of an equation of state.
	
	\subsection{Equation of state}
	Without loss of generality, the equation of state can be given as,
	\begin{align}\label{eq: general_EOS}
		h(p,\rho) = 1 + \epsilon(p,\rho) + \frac{p}{\rho},
	\end{align}
	where $\epsilon$ denotes the specific internal energy. The sound speed $c_s$ and the polytropic index $n$ are given by,
	\begin{align}\label{eq:sound.speed}
		c_s^2 = -\frac{\rho}{n h} \frac{\p h}{\p \rho},\qquad n=\rho \frac{\p h}{\p p}-1.
	\end{align}
	For the hyperbolicity of the system~\eqref{eq: RHD_system}, the sound speed should satisfy $0<c_s<1$ and by the relativistic kinetic theory it can be shown that~\cite{wu2016physical,xu2024high}, 
	\begin{align}\label{eq: weaker_taub}
		h(p,\rho) \geq \sqrt{1+\frac{p^2}{\rho^2}}+\frac{p}{\rho}.
	\end{align}
	The most commonly used equation of state in the literature is the {\em ideal equation of state} given by,
	\begin{align}\label{eq: ID_eos}
		h = 1+\frac{\gamma}{\gamma -1}\frac{p}{\rho},
	\end{align}
	where $\gamma \in (1,2]$ is a constant called the specific heat ratio. Here, $\gamma > 2$ gives sound speed $c_s >1$, violating the principle of special relativity. Usually, one considers $\gamma = \frac{5}{3}$ for sub-relativistic flows and $\gamma = \frac{4}{3}$ for ultra-relativistic flows~\cite{wu2016physical}. But, these two values are the rational upper and lower bounds for it and can not be considered as a constant~\cite{sokolov2001simple}. In fact, the equation of state~\eqref{eq: ID_eos} is derived from the non-relativistic thermodynamics, and as mentioned in~\cite{ryu2006equation,wu2016physical}, it is a poor choice for many relativistic flows, primarily for the case of semi-relativistic fluids or two-component fluids.
	
	In the relativistic framework, the correct equation of state for a single-component perfect gas is given by~\cite{MR0088362},
	\begin{align}\label{eq: synge_eos}
		h = \frac{B_3(\rho/p)}{B_2(\rho/p)}.
	\end{align}
	
	Here, $B_2$ and $B_3$ are the modified Bessel functions of the second kind having order two and three, respectively, which makes the implementation of this equation of state expensive~\cite{falle1996upwind}. Hence, there are several general equations of state studied in the literature, which are more accurate than~\eqref{eq: ID_eos} and less complicated than~\eqref{eq: synge_eos}. Here, we will limit our study to the following three such equations of state along with the ideal equation of state~\eqref{eq: ID_eos}.
	
	The first equation of state that we will consider and which gives a better approximation than the ideal equation of state~\eqref{eq: ID_eos} in the relativistic regime, was derived in~\cite{mathews1971hydromagnetic} and was later used in~\cite{mignone2005piecewise},
	\begin{align}\label{eq: TM_eos}
		h = \frac{5p}{2\rho} + \sqrt{\frac{9p^2}{4\rho^2}+1}.
	\end{align}
	The second equation of state can be found in~\cite{sokolov2001simple}, which is given by,
	\begin{align}\label{eq: IP_eos}
		h = \frac{2p}{\rho}+\sqrt{\frac{4p^2}{\rho^2} + 1}.
	\end{align}
	The third equation of state that we will consider is derived in~\cite{ryu2006equation} and shown as a better approximation of~\eqref{eq: synge_eos} than~\eqref{eq: TM_eos}. This can be expressed as,
	\begin{align}\label{eq: RC_eos}
		h = \frac{2(6p^2 + 4p\rho +\rho^2)}{\rho (3p+ 2\rho)}.
	\end{align}
	All these equations of state,~\eqref{eq: ID_eos}-\eqref{eq: RC_eos} satisfy the equation~\eqref{eq: weaker_taub}~\cite{wu2016physical,xu2024high}. Following~\cite{ryu2006equation,xu2024high}, we will denote these equations of state \eqref{eq: ID_eos}, \eqref{eq: TM_eos}, \eqref{eq: IP_eos}, and \eqref{eq: RC_eos} as ID-EOS, TM-EOS, IP-EOS, and RC-EOS, respectively, for the rest of the paper.
	
	From equation~\eqref{eq:sound.speed}, after some calculations, we get the expression for the sound speed $c_s$ for the TM-EOS as,
	\begin{equation}\label{eq:sound.TMEOS}
		c_s^2 = \frac{5p \sqrt{9p^2 +4\rho^2} + 9p^2}{12p\sqrt{9p^2 +4\rho^2} + 36p^2 +6\rho^2},
	\end{equation}
	for the IP-EOS,
	\begin{equation}\label{eq:sound.IPEOS}
		c_s^2 = \frac{2p \sqrt{4p^2 + \rho^2}}{4p \sqrt{4p^2 + \rho^2} + 4p^2 + \rho^2},
	\end{equation}
	for the RC-EOS,
	\begin{equation}\label{eq:sound.RCEOS}
		c_s^2 = \frac{p (3p +2\rho) (18p^2 + 24 p\rho + 5\rho^2)}{3 (6p^2 + 4p\rho + \rho^2) (9p^2 + 12 p\rho + 2\rho^2)},
	\end{equation}
	and for ID-EOS, it is given by,
	\begin{equation}\label{eq:sound.IDEOS}
		c_s^2 = \frac{\gamma p}{h \rho}.
	\end{equation}
	
	\subsection{Admissible set and extraction of primitive variables}\label{sec: admissible set and extraction of primitive}
	The solutions of the system~\eqref{eq: RHD_system} with~\eqref{eq: general_EOS} should satisfy certain constraints to be physically meaningful and belong to the so-called {\em admissible set},
	\begin{align}\label{eq: ad_region_1}
		\Uad &=\{\mb{u}=(D,\mb{m},E)^\top: \rho (\mb{u}) >0, p(\mb{u})>0, \epsilon(\mb{u})>0, 1 -|\mb{v}(\mb{u})|>0\}.
	\end{align}
	It is also important that the solution of the numerical scheme belongs to this set at every step. If the solution goes out of this set, the hyperbolicity of the system will be violated, resulting in possible failure of the computations. 
	
	The admissible set $\Uad$ has constraints involving primitive variables, and the conversion from conservative to primitive variables involves solving a non-linear equation. Hence, the verification of the admissibility of the solution becomes computationally expensive, which needs to be done in every time step for applying the positivity limiter. Moreover, we need the constraints of the admissible set $\rho$, $p$, $\epsilon$, and $1-|v|$ to be concave functions of the conservative variables, but the pressure $p$ in the RHD equations is not concave; it is shown in~\cite{wu2015high} for the ID-EOS. As a remedy, we consider a different characterization of the admissible set, given by
	\begin{align}\label{eq: ad_region_2}
		\Uad' &= \{\mb{u}=(D,\mb{m},E)^\top: D(\mb{u}) = D >0,\ q(\mb{u}):= E-\sqrt{D^2 + |\mb{m}|^2}>0\}.
	\end{align}
	The equivalence of $\Uad$~\eqref{eq: ad_region_1} and $\Uad'$~\eqref{eq: ad_region_2} is proved in~\cite{wu2016physical} with the equations of state~(\ref{eq: ID_eos}, \ref{eq: TM_eos}, \ref{eq: IP_eos}, \ref{eq: RC_eos}). The representation~$\Uad'$ makes use of constraint functions $D(\mb{u})$, $q(\mb{u})$, which can be directly calculated from the conservative variables. Moreover, these constraint functions $D(\mb{u})$ and $q(\mb{u})$ are concave functions of the conservative variables; for a detailed proof, one can refer to~\cite{my_paper}.
	
	Even though the representation~\eqref{eq: ad_region_2} eliminates the need of converting the conservative variables to the primitive variables for checking the admissibility, the conversion is still needed in other stages of execution of the scheme, as can be seen later in Section~\ref{sec:Numerical scheme}. For ID-EOS~\eqref{eq: ID_eos} we follow the method of conversion from~\cite{cai2024provably}. More specifically, they have proposed three methods of conversion from conservative to primitive, and we follow the second method of conversion (Section~2.3. NR-II method in~\cite{cai2024provably}) because of its high accuracy (see Appendix~\ref{sec: appendix_cons2prim_IDEOS}).
	For TM-EOS~\eqref{eq: TM_eos}, we follow the method of conversion from~\cite{ryu2006equation}, more specifically, Section~3.2 in~\cite{ryu2006equation}. For IP-EOS~\eqref{eq: IP_eos} we follow Section~4 in~\cite{sokolov2001simple}. 
	
	The method of conversion for RC-EOS~\eqref{eq: RC_eos}, given in~\cite{ryu2006equation} by iteratively solving a nonlinear equation for $\Gamma$, does not match expectations. Some illustrative cases for the same can be found in the Jupyter Notebook \texttt{con2prim\_RCEOS.ipynb} in~\cite{RHDTenkai}, where we have shown the failure of the iterative method to converge to the correct root. Hence, we propose a new way of conversion using an iterative method on a different non-linear equation, which behaves nicely as explained below.
	
	Let us first assume that the solution $\mb{u}$ is in the admissible region $\Uad'$~\eqref{eq: ad_region_2}. The RC-EOS~\eqref{eq: RC_eos} can be expressed as,
	\[
	p = \rho \left(\frac{3h-8}{24}\pm \frac{\sqrt{(3h+8)^2-96}}{24}\right),
	\]
	and we take,
	\begin{equation}\label{eq: p_RC}
		p = \rho \left(\frac{3h-8}{24} + \frac{\sqrt{(3h+8)^2-96}}{24}\right),
	\end{equation}
	which gives a positive pressure ($p > 0$), since $\rho > 0$ and $h > 1$; the other choice gives a negative pressure in the admissible region. The equation~\eqref{eq: p_RC} can be expressed as,
	\begin{equation}\label{eq:pf.RC}
		p = \rho T(h),
	\end{equation}
	with $T(h)$ being a non-linear function of $h$. Now, taking $\Pi = E + p$, we have,
	\begin{equation}\label{eq:gamma.in.pi}
		\Pi = \rho h \Gamma^2 = Dh\Gamma \quad\implies\quad \Gamma = \frac{\Pi}{Dh},
	\end{equation}
	and 
	\begin{align}\label{eq:h.L.invariant}
		hD = h \rho \Gamma = h \rho \Gamma^2 \sqrt{1-|\mb{v}|^2} = \sqrt{\Pi^2 - |\mb{m}|^2} \qquad \implies \qquad h = \frac{\sqrt{\Pi^2 - |\mb{m}|^2}}{D}.
	\end{align}
	Since
	\begin{align}
		&\Pi - E = p = \frac{D}{\Gamma} T(h)
		= \frac{D^2 h}{\Pi} T(h), \qquad \text{using~(\ref{eq:pf.RC},\ref{eq:gamma.in.pi})}
	\end{align}
	we get an equation for $\Pi$,
	\begin{equation}\label{eq: pi_eq}
		S(\Pi) := \Pi^2 - \Pi E - D^2 h  T(h) = 0.
	\end{equation}
	
	\begin{lemma}
		The function $S(\Pi)$~\eqref{eq: pi_eq} has a unique real root in $(E, \infty)$.
	\end{lemma}
	\begin{proof}
		Differentiating $S(\Pi)$ with respect to $\Pi$ we get,
		\begin{equation}\label{eq:s'pi}
			S'(\Pi)= \left(2 - \frac{T(h)}{h}-T'(h)\right)\Pi - E.
		\end{equation}
		For $\Pi \geq E$, we have $h>1$ and $S'(\Pi) > 0$, see Appendix~\ref{sec: appendix_s'pi_proof}, hence $S(\Pi)$ is an increasing function in $[E, +\infty)$. At $\Pi = E$ we have,
		\begin{align*}
			S(E) = - D \sqrt{E^2 - |\mb{m}|^2} T\left(\frac{\sqrt{E^2 - |\mb{m}|^2}}{D}\right) < 0,
		\end{align*}
		and for $\Pi \to +\infty$, $S(\Pi)\to +\infty$ for any fixed $(D, \mb{m}, E)$ in the admissible region. Hence, there exists a unique real root of $S(\Pi)$ in the interval $(E, +\infty)$.
		\qed
	\end{proof}
	We apply the Newton-Raphson method
	\begin{equation}\label{eq:NR}
		\Pi_{r+1} = \Pi_r - \frac{S(\Pi_r)}{S'(\Pi_r)}, \qquad r = 0,1,2,\ldots
	\end{equation}
	taking the initial guess $\Pi_0 = E$, to approximate the unique real root $\Pi_*$ of $S(\Pi)$~\eqref{eq: pi_eq} in $(E, \infty)$.
	\begin{lemma}
		For any $\Pi_r\in [E,\Pi_*)$ the Newton-Raphson method~\eqref{eq:NR} generates increasing iterates,  that is $\Pi_{r+1}>\Pi_r$.
	\end{lemma}
	\begin{proof}
		We have, $S(E)<0$ and $S(\Pi)$ is increasing in $[E,\infty)$ with $\Pi_*$ as the unique root in $(E,\infty)$. Hence,
		\begin{equation*}
			S(\Pi_r)<0, \ \forall\ \Pi_r\in [E,\Pi_*).
		\end{equation*}
		Again, $S'(\Pi_r)>0,\ \forall\ \Pi_r\in[E,\infty)$ (Appendix~\ref{sec: appendix_s'pi_proof}). Hence, $\forall\ \Pi_r\in [E,\Pi_*)$,
		\begin{align*}
			&\frac{S(\Pi_r)}{S'(\Pi_r)} < 0 \quad \implies \quad \Pi_{r+1}>\Pi_r
		\end{align*} 
		\qed
	\end{proof}
	\begin{remark}
		The expressions of $S(\Pi),\ S'(\Pi)$ are too complicated to algebraically prove that all the iterations in the Newton-Raphson method~\eqref{eq:NR} satisfy $\Pi_{r+1} > E$, but we always get the approximated root to satisfy this inequality for all the test cases in Section~\ref{sec: Numerical_simulations}. We also find that the average number of iterations that the Newton-Raphson method needs for convergence is between 4 to 5, even for problems with low density ($\rho$), low pressure ($p$), and absolute velocity ($|\mb{v}|$) near unity.
		One can always use a damped Newton-Raphson method
		\begin{equation}\label{eq:damped_NR}
			\Pi_{r+1} = \Pi_r - k_r\frac{S(\Pi_r)}{S'(\Pi_r)}, \qquad r = 0,1,2,\ldots
		\end{equation}
		to find the root of $S(\Pi)$~\eqref{eq: pi_eq}, where $k_r\in(0,1]$ can be used to ensure that the iterates do not cross $\Pi_*$, and thereby $\Pi_{r+1} > E$ is satisfied at all steps. However, we did not need this damped method for any of the test cases in this paper, and the standard Newton-Raphson method itself generated increasing iterates converging to the root.
	\end{remark}
	
	After the Newton-Raphson iterations terminate, we get a value of $\Pi > E > |\mb{m}|$. Then the primitive variables can be found as,
	\begin{equation*}
		p = \Pi - E,\qquad
		\mb{v} = \frac{\mb{m}}{\Pi},\qquad
		\rho = D \sqrt{1-|\mb{v}|^2}.
	\end{equation*}
	Here, since $\Pi > E$ we have $p > 0$ and $\Pi > |\mb{m}|$ implying $|\mb{v}|<1$ and $\rho > 0$.
	A Julia code snippet (Algorithm~\ref{alg:cons2prim}) is added in Appendix~\ref{sec: appendix_cons2prim}, which is used for the extraction of primitive variables in the case of RC-EOS.
	
	\section{Numerical scheme}\label{sec:Numerical scheme}
	For the sake of simplicity, we will discuss the scheme in one dimension, which can be generalized to higher dimensions by applying the same idea in a dimension-by-dimension manner; also refer to~\cite{my_paper}, where the LWFR scheme for the RHD equation with ID-EOS is discussed in two dimensions. Let the computational domain be $[x_a, x_b]$ and we partition it into cells/elements $\Omega_e = [x_{e-\frac{1}{2}}, x_{e+\frac{1}{2}}]$; the length of the cells is $\Delta x_e = x_{e+\frac{1}{2}}- x_{e-\frac{1}{2}}$.
	
	Following the framework of~\cite{BABBAR2022111423,babbar2024admissibility}, we take the reference element to be $[0, 1]$, and the function for mapping the physical elements to the reference element is denoted by,
	\[
	x \mapsto \xi = \frac{x - x_{e-\frac{1}{2}}}{\Delta x_e}.
	\]
	Now, in each element we want to approximate the solution with a polynomial of degree $N \geq 0$ and hence we take the $N+1$ solution points inside the reference element as $0\leq \xi_0 < \xi_1 < \cdots < \xi_N\leq 1$. For this work, we will take these solution points as the Gauss-Legendre nodes, because they evaluate integrals exactly in the quadrature rule for polynomials of degree up to $2N + 1$ with the corresponding Gauss-Legendre weights. Finally, the solution inside each element can be expressed as,
	\[
	\mb{u}_h(\xi, t) = \sum_{i=0}^N \mb{u}_i^e(t) \ell_i(\xi), \qquad x \in \Omega_e,
	\]
	where the Lagrange polynomial $\ell_i$ of degree $N$ is given by,
	\[
	\ell_i(\xi) =\prod_{\substack{j=0 \\ j \neq i}}^N \frac{\xi - \xi_j}{\xi_i - \xi_j}.
	\]
	The coefficients $\mb{u}_i^e$, which represent the solutions at the solution points, also called the degrees of freedom, need to be evolved in time.
	
	Expanding $\mb{u}_i^e(t)$ around time $t_n$ using Taylor's series and using \eqref{eq: RHD_system} with $d=1$, we get the update equation of the one dimensional LWFR scheme~\cite{BABBAR2022111423},
	\begin{align}\label{eq: high order update}
		(\mb{u}_{i}^e)^{n+1}= (\mb{u}_{i}^e)^n - \frac{\Delta t}{\Delta x_e} \frac{d \mb{F}_h}{d\xi}(\xi_i), \qquad 0\leq i\leq N,
	\end{align}
	where the continuous approximation of the time average flux $\mb{F}_h$ is found as,
	\begin{equation}\label{eq:time.avg.flux}
		\mb{F}_h(\xi) = \Big[\mb{F}_{e-\frac{1}{2}} - \mb{F}_h^\delta(0) \Big] h_L(\xi) + \mb{F}_h^\delta(\xi) + \Big[\mb{F}_{e+\frac{1}{2}} - \mb{F}_h^\delta(1) \Big] h_R(\xi).
	\end{equation}
	Here $h_L, h_R$ are the correction functions used in the flux reconstruction method~\cite{huynh2007flux,vincent2011new}, and $\mb{F}_h^\delta$ is the local approximation of the time average flux which is possibly discontinuous across the elements and is given by,
	\begin{equation}\label{eq:time.avg.flux.loc.approx}
		\mb{F}_h^\delta(\xi) = \sum_{i=0}^N \mb{F}_i\ell_i(\xi),
	\end{equation}
	where $\mb{F}_i$ is an approximation to the time average flux at the solution point $\xi_i$,
	\[
	\mb{F}_i \approx \frac{1}{\Delta t}\int_{t_n}^{t_{n+1}} \mb{f}\left(\mb{u}(\xi_i,t)\right)dt.
	\]
	For the scheme~\eqref{eq: high order update} to be accurate to order $N+1$ in the smooth regions, we expand the flux $\mb{f}(\mb{u})$ with the Taylor's series expansion as below to get the time average fluxes at the solution points as,
	\begin{align}\label{eq:time.avg.flux.sol.point.1}
		\mb{F}_i \approx \frac{1}{\Delta t}\int_{t_n}^{t_{n+1}}\left[ \sum_{r=0}^{N}\frac{(t-t_n)^r}{r!}\frac{\p^r \mb{f}(\mb{u_i^n})}{\p t^r}\right]\textrm{d}t
		= \sum_{r=0}^N \frac{\Delta t^r}{(r+1)!}\frac{\p^r \mb{f}(\mb{u}_i^n)}{\p t^r},
	\end{align}
	where $\mb{u}_i^n$ is the approximation of the solution at solution point $\xi_i$ and time $t_n$. Now, one way to proceed would be to use the chain rule to replace the temporal derivatives in the last equation with spatial derivatives, but due to the large computational cost in this approach~\cite{burger2017approximate}, we approximate the temporal derivatives using finite difference formulae~\cite{zorio2017approximate} as will be explained in Section~\ref{sec:time.avg.flux}.
	
	The inter-element fluxes $\mb{F}_{e\pm \frac{1}{2}}$ in~\eqref{eq:time.avg.flux} are found by blending the high-order flux,
	\begin{equation}\label{eq:high.order.numerical.flux}
		\highorderflux{\mb{F}}_{e+\frac{1}{2}} = \frac{1}{2}[\mb{F}_{e+\frac{1}{2}}^- + \mb{F}_{e+\frac{1}{2}}^+] - \frac{1}{2} \lambda_{e+\frac{1}{2}} [\mb{U}_{e+\frac{1}{2}}^+ - \mb{U}_{e+\frac{1}{2}}^-],
	\end{equation}
	with a low-order flux as will be explained in Section~\ref{sec:blending}. In the above expression of high-order flux, the trace values $\mb{F}_{e+\frac{1}{2}}^\pm$ are calculated at the face by the idea of the approximate Lax-Wendroff procedure after extrapolating the required quantities to the faces. This procedure is termed as the EA (Extrapolate and Average) procedure in~\cite{BABBAR2022111423}, which overcomes the sub-optimal convergence rate coming from the direct extrapolation of $\mb{F}_h^\delta(\xi)$~\eqref{eq:time.avg.flux.loc.approx}; for more details, please refer to Section~5.2 in~\cite{BABBAR2022111423}. The coefficient $\lambda_{e+\frac{1}{2}}$ is found as,
	\begin{equation}
		\lambda_{e+\frac{1}{2}} = \max\left\{\lambda_{\max}\left(\mb{f}'(\Bar{\mb{u}}^n_e)\right), \lambda_{\max}\left(\mb{f}'(\Bar{\mb{u}}^n_{e+1})\right)\right\},
	\end{equation}
	with $\lambda_{\max}(\cdot)$ denoting the spectral radius and $\Bar{\mb{u}}^n_e$ denoting the element average of the solution in the element $\Omega_e$,
	\begin{equation}
		\Bar{\mb{u}}^n_e = \sum\limits_{i=0}^N w_i(\mb{u}_i^e)^n.
	\end{equation}
	Here, $w_i$'s are the Gauss-Legendre quadrature weights. In the dissipative part of the numerical flux, $\mb{U}_{e+\frac{1}{2}}^+$, $\mb{U}_{e+\frac{1}{2}}^-$ are the trace values of the time average solution,
	\begin{equation}
		\mb{U} = \sum^N_{r=0}\frac{\Delta t^r}{(r+1)!}\frac{\p^r \mb{u}}{\p t^r},
	\end{equation}
	from the right and left elements of the face $x_{e+\frac{1}{2}}$ respectively.
	Using the time average solution in the dissipative part results in a stable scheme with a higher CFL number compared to using the solution at time $t_n$~\cite{BABBAR2022111423}.
	
	\subsection{Time average flux}\label{sec:time.avg.flux}
	For notational simplicity, we ignore the time index and take,
	\begin{align}
		\mb{f}_i= \mb{f}(\mb{u}_i),\quad \mb{u}^{(r)} = \Delta t^r \frac{\p^r \mb{u}}{\p t^r}, \quad \mb{f}^{(r)} = \Delta t^r \frac{\p^r \mb{f}}{\p t^r}, \qquad \text{for}\ r=0,1,2\dots,
	\end{align}
	and get a simplified form of the expression~\eqref{eq:time.avg.flux.sol.point.1},
	\begin{equation}\label{eq:time.avg.flux.2}
		\mb{F}_i\approx \sum\limits_{r=0}^{N} \frac{1}{(r+1)!}\mb{f}_{i}^{(r)}.
	\end{equation}
	The quantities $\mb{f}_i^{(r)}$ for different degree $N$ as explained in~\cite{BABBAR2022111423,my_paper} are given below.
	
	\paragraph{For $N=1$.}
	\begin{align}\label{eq:ALW.N1}
		\begin{split}
			&\mb{f}_i^{(1)} = \frac{1}{2}\Big[\mb{f}\Big(\mb{u}_{i}+\mb{u}_{i}^{(1)}\Big)-\mb{f}\Big(\mb{u}_{i}-\mb{u}_{i}^{(1)}\Big) \Big],\\
			&\mb{u}_{i}^{(1)} = -\frac{\Delta t}{\Delta x_e} \mb{f}_{i,\xi},
		\end{split}
	\end{align}
	where $\mb{f}_{i,\xi}$ is calculated by taking the derivative with respect to $\xi$ of the polynomial approximation of flux $\mb{f}(\mb{u})$ at the solution point $\xi_i$.
	
	\paragraph{For $N=2$.}
	\begin{align}\label{eq:ALW.N2}
		\begin{split}
			\mb{f}_{i}^{(1)} &= \frac{1}{2}\left[\mb{f}\left(\sum\limits_{k=0}^1 \frac{1}{k!}\mb{u}_{i}^{(k)}\right) - \mb{f}\left(\sum\limits_{k=0}^1 \frac{(-1)^2}{k!}\mb{u}_{i}^{(k)}\right) \right],\\
			\mb{f}_{i}^{(2)} &= \mb{f}\left(\sum\limits_{k=0}^2 \frac{1}{k!}\mb{u}_{i}^{(k)}\right) - 2 \mb{f}(\mb{u}_{i}) + \mb{f}\left(\sum\limits_{k=0}^2 \frac{(-1)^2}{k!}\mb{u}_{i}^{(k)}\right).
		\end{split}
	\end{align}
	
	\paragraph{For $N=3$.}
	\begin{align}\label{eq:ALW.N3}
		\begin{split}
			\mb{f}_{i}^{(1)} &= \frac{1}{12}\Bigg[-\mb{f}\left(\sum\limits_{k=0}^1 \frac{2^k}{k!}\mb{u}_{i}^{(k)}\right) + 8\mb{f}\left(\sum\limits_{k=0}^1 \frac{1}{k!}\mb{u}_{i}^{(k)}\right) - 8\mb{f}\left(\sum\limits_{k=0}^1 \frac{(-1)^k}{k!}\mb{u}_{i}^{(k)}\right)  + \mb{f}\left(\sum\limits_{k=0}^1 \frac{(-2)^k}{k!}\mb{u}_{i}^{(k)}\right)\Bigg],\\
			\mb{f}_{i}^{(2)} &= \mb{f}\left(\sum\limits_{k=0}^2 \frac{1}{k!}\mb{u}_{i}^{(k)}\right) - 2 \mb{f}(\mb{u}_{i}) + \mb{f}\left(\sum\limits_{k=0}^2 \frac{(-1)^2}{k!}\mb{u}_{i}^{(k)}\right),\\
			\mb{f}_{i}^{(3)} &= \frac{1}{2}\Bigg[\mb{f}\left(\sum\limits_{k=0}^3 \frac{2^k}{k!}\mb{u}_{i}^{(k)}\right) - 2\mb{f}\left(\sum\limits_{k=0}^3 \frac{1}{k!}\mb{u}_{i}^{(k)}\right) + 2\mb{f}\left(\sum\limits_{k=0}^3 \frac{(-1)^k}{k!}\mb{u}_{i}^{(k)}\right)  - \mb{f}\left(\sum\limits_{k=0}^3 \frac{(-2)^k}{k!}\mb{u}_{i}^{(k)}\right) \Bigg].
		\end{split}
	\end{align}
	
	\paragraph{For $N=4$.}
	\begin{align}\label{eq:ALW.N4}
		\begin{split}
			\mb{f}_{i}^{(1)} &= \frac{1}{12}\Bigg[-\mb{f}\left(\sum\limits_{k=0}^1 \frac{2^k}{k!}\mb{u}_{i}^{(k)}\right) + 8\mb{f}\left(\sum\limits_{k=0}^1 \frac{1}{k!}\mb{u}_{i}^{(k)}\right) - 8\mb{f}\left(\sum\limits_{k=0}^1 \frac{(-1)^k}{k!}\mb{u}_{i}^{(k)}\right)  + \mb{f}\left(\sum\limits_{k=0}^1 \frac{(-2)^k}{k!}\mb{u}_{i}^{(k)}\right)\Bigg],\\
			\mb{f}_{i}^{(2)} &= \frac{1}{12}\Bigg[-\mb{f}\left(\sum\limits_{k=0}^2 \frac{2^k}{k!}\mb{u}_{i}^{(k)}\right) + 16\mb{f}\left(\sum\limits_{k=0}^2 \frac{1}{k!}\mb{u}_{i}^{(k)}\right) - 30\mb{f}(\mb{u}_{i})\\
			& \mspace{350mu} + 16\mb{f}\left(\sum\limits_{k=0}^2 \frac{(-1)^k}{k!}\mb{u}_{i}^{(k)}\right) - \mb{f}\left(\sum\limits_{k=0}^2 \frac{(-2)^k}{k!}\mb{u}_{i}^{(k)}\right) \Bigg],\\
			\mb{f}_{i}^{(3)} &= \frac{1}{2}\Bigg[\mb{f}\left(\sum\limits_{k=0}^3 \frac{2^k}{k!}\mb{u}_{i}^{(k)}\right) - 2\mb{f}\left(\sum\limits_{k=0}^3 \frac{1}{k!}\mb{u}_{i}^{(k)}\right) + 2\mb{f}\left(\sum\limits_{k=0}^3 \frac{(-1)^k}{k!}\mb{u}_{i}^{(k)}\right) - \mb{f}\left(\sum\limits_{k=0}^3 \frac{(-2)^k}{k!}\mb{u}_{i}^{(k)}\right) \Bigg],\\
			\mb{f}_{i}^{(4)} &= \mb{f}\left(\sum\limits_{k=0}^4 \frac{2^k}{k!}\mb{u}_{i}^{(k)}\right) - 4\mb{f}\left(\sum\limits_{k=0}^4 \frac{1}{k!}\mb{u}_{i}^{(k)}\right) + 6\mb{f}(\mb{u}_{i}) - 4\mb{f}\left(\sum\limits_{k=0}^4 \frac{(-1)^k}{k!}\mb{u}_{i}^{(k)}\right) + \mb{f}\left(\sum\limits_{k=0}^4 \frac{(-2)^k}{k!}\mb{u}_{i}^{(k)}\right).
		\end{split}
	\end{align}
	Here $\mb{u}_i^{(r)}$ for $r=1,2,3,4$ are found as,
	\begin{equation*}
		\mb{u}_{i}^{(r)} = -\frac{\Delta t}{\Delta x_e} \mb{f}_{i,\xi}^{(r-1)}.
	\end{equation*}
	
	The arguments of the flux function in the expressions of $\mb{f}_i^{(r)}$ in equations~(\ref{eq:ALW.N1},~\ref{eq:ALW.N2},~\ref{eq:ALW.N3},~\ref{eq:ALW.N4}) are in general of the form,
	\begin{equation}\label{eq:flux_arg_ALW}
		\mb{u}^{[r]}_i = \sum_{k=0}^r \frac{a_k}{k!} \mb{u}_i^{(k)},
	\end{equation}
	with $a_k$ as an integer. Here, $\mb{u}^{[r]}_i$ does not necessarily belong to the admissible region~$\Uad'$~\eqref{eq: ad_region_2} for any $0\leq i\leq N$, see for example the 1-D Riemann problem 1 in Section~\ref{sec:1DRP1}. However, to find the flux~\eqref{eq:RHD.flux} we need to convert $\mb{u}^{[r]}_i$ to its primitive equivalent and this process needs $\mb{u}^{[r]}_i$ to be in the admissible region~$\Uad'$ (Section~\ref{sec: admissible set and extraction of primitive}). Hence, we scale the quantities $\mb{u}^{[r]}_i$ to the admissible region~$\Uad'$, which is explained below.
	
	We have $\mb{u}^{[0]}_i\in \Uad'$ for all $0\leq i\leq N$ as these are equal to the solution values $\mb{u}_i$. Now, the corresponding cell average,
	\begin{align}
		\Bar{\mb{u}}^{[0]} = \sum\limits_{i=0}^N w_i \mb{u}^{[0]}_i \in \Uad',
	\end{align}
	since $\Uad'$ is a convex set~\cite{my_paper,wu2015high}. Here, $w_i$'s are the corresponding quadrature weights which satisfy $\sum\limits_{i=0}^N w_i = 1$.
	
	Now, we start by checking the admissibility of the quantities $\mb{u}^{[1]}_i$ for $0\leq i\leq N$. If $\mb{u}^{[1]}_i \notin \Uad'$ for some $i$ then we scale it with respect to $\Bar{\mb{u}}^{[0]}$ as explained in the following steps.
	
	\paragraph{Step~1.} If $D\left(\mb{u}^{[1]}_i\right) > 0$ for all $1\leq i\leq N$, go to \textit{Step~2}. But if $D\left(\mb{u}^{[1]}_i\right) \leq 0$ for some $i$, take $\varepsilon_D$ as a small positive real number, and
	\begin{equation*}
		\theta_D = \min \left\{1, \frac{\left|\varepsilon_D - D\left(\Bar{\mb{u}}^{[0]}\right)\right|}{\left|D\left(\mb{u}^{[1]}_{\text{min}}\right) -  D\left(\Bar{\mb{u}}^{[0]}\right)\right|}\right\},
	\end{equation*}
	where $(\cdot)_{\text{min}}$ denotes the minimum over all the solution points in the cell. Then, the scaling is performed as,
	\begin{equation*}
		\Tilde{\mb{u}}^{[1]}_i = \theta_D \mb{u}^{[1]}_i + (1-\theta_D) \Bar{\mb{u}}^{[0]}, \quad \forall\ 0\leq i\leq N.
	\end{equation*}
	
	\paragraph{Step~2.} Repeat \textit{Step~1} for the other admissibility constraint $q$ and take updated quantities $\Tilde{\mb{u}}^{[1]}_i$ in the place of $\mb{u}^{[1]}_i$ for $0\leq i\leq N$.
	
	In this work, we have taken $\varepsilon_D = \min\left\{\frac{1}{10} D\left(\Bar{\mb{u}}^{[0]}\right), 10^{-13}\right\}$ and $ \varepsilon_q = \min\left\{\frac{1}{10} q\left(\Bar{\mb{u}}^{[0]}\right), 10^{-13}\right\}$ in \textit{Step~1} and \textit{Step~2} respectively, which enforces the positivity of the admissibility constraints~\eqref{eq: ad_region_2}. To elaborate, after the first step, using the concavity of the first admissibility constraint $D$,
	\begin{align}
		\begin{split}
			D\left(\Tilde{\mb{u}}^{[1]}_i\right) = D\left(\theta_D \mb{u}^{[1]}_i + (1-\theta_D) \Bar{\mb{u}}^{[0]} \right) 
			&\geq \theta_D D\left( \mb{u}^{[1]}_i\right) + (1-\theta_D) D\left(\Bar{\mb{u}}^{[0]} \right)\\ 
			&> \varepsilon_D > 0.
		\end{split}
	\end{align}
	Similarly, using the concavity of $q$, we can again show,
	\begin{align}
		q\left(\Tilde{\mb{u}}^{[1]}_i\right) > \varepsilon_q > 0.
	\end{align}
	
	Now following the same procedure, we will again scale $\mb{u}_i^{[r]}$ with respect to the cell average of the updated quantity of $\mb{u}_i^{[r-1]}$ for required $r>1$ and $0\leq i\leq N$. Finally, using the final updated quantities in equations~(\ref{eq:ALW.N1},~\ref{eq:ALW.N1},~\ref{eq:ALW.N3},~\ref{eq:ALW.N4}) we calculate the approximation of time average fluxes at solution points, denoted by $\mb{F}_i$ \eqref{eq:time.avg.flux.2}.

	\subsection{Blending of the scheme}\label{sec:blending}
	The next hurdle that needs to be addressed is controlling the spurious Gibbs oscillations, which generally arise when using high-order methods for discontinuous solutions. In literature, one can often find works with TVD limiter and its modifications~\cite{cockburn1991runge,cockburn1989tvb} which can handle this problem, but they have their own disadvantages~\cite{babbar2024admissibility,my_paper}. In this work, we will use the blending limiter~\cite{babbar2024admissibility,my_paper} where the high-order method gets blended with a low-order method which is known for producing minimal oscillation, to control the oscillation in the final result. Specifically, the solution at time $t_{n+1}$ will be,
	\begin{equation}\label{eq:blending.solution}
		(\mb{u}^e)^{n+1} = (1-\alpha_e)(\mb{u}^e)^{H,n+1}+\alpha_e (\mb{u}^e)^{L,n+1},
	\end{equation}
	where the superscript $H$ denotes the high-order update~\eqref{eq: high order update} and $L$ denotes the low-order update as will be explained below. The blending coefficient $\alpha_e$ in each element is found from the discontinuity indicator model designed in~\cite{my_paper} by taking a modal representation of the indicator quantity $K=\rho p \Gamma$, as taken in~\cite{my_paper}. 
	
	Now, for finding the low-order update at time level $n+1$, we divide the element $\Omega_e$ into $N+1$ sub-elements with the length of the sub-elements as,
	\begin{equation}\label{eq:sub.element.length}
		x_{i+\frac{1}{2}} - x_{i-\frac{1}{2}} = w_i \Delta x_e, \qquad \forall\ 0 \leq i \leq N,
	\end{equation}
	where $x_{i\pm\frac{1}{2}}$ denotes the sub-faces of the sub-elements. Here, the sub-faces $x_{-\frac{1}{2}}, x_{N+\frac{1}{2}}$ are the same as the faces of the parent element $x_{e-\frac{1}{2}}$, $x_{e+\frac{1}{2}}$, respectively. The relation~\eqref{eq:sub.element.length} is necessary for maintaining the conservative property of the scheme~\cite{babbar2024admissibility,my_paper}. The solutions inside each sub-element is found as,
	\begin{align}\label{eq: low order update}
		(\mb{u}_{i}^e)^{L,n+1}= (\mb{u}_{i}^e)^n - \frac{\Delta t}{w_i\Delta x_e} [\mb{f}^L_{i+\frac{1}{2}} - \mb{f}^L_{i-\frac{1}{2}}], \qquad 0\leq i\leq N,
	\end{align}
	where $w_i$'s are the corresponding quadrature weights according to the reference coordinates. The fluxes $\mb{f}^L_{i\pm\frac{1}{2}}$ for the interior sub-element faces are given by,
	\begin{align}\label{eq:low order update notations}
		\begin{split}
			&\mb{f}^L_{-\frac{1}{2}} = \mb{F}_{e-\frac{1}{2}},\\
			&\mb{f}^L_{i+\frac{1}{2}} = \numflux{\mb{f}}(\mb{u}_i^e, \mb{u}_{i+1}^e),\quad 0\leq i\leq N-1,\\
			&\mb{f}^L_{N+\frac{1}{2}} = \mb{F}_{e+\frac{1}{2}},
		\end{split}
	\end{align}
	where $\numflux{\mb{f}}$ is the Rusanov flux~\cite{rusanov1962calculation}, since we are blending with first-order finite volume scheme with Rusanov flux following~\cite{my_paper}; and is given by,
	\begin{equation}\label{eq:rusanov flux}
		\mb{f}^{\text{NF}}(\mb{u}^-, \mb{u}^+) = \frac{1}{2}[\mb{f}(\mb{u}^-) + \mb{f}(\mb{u}^+)] - \frac{1}{2} \lambda [\mb{u}^+ - \mb{u}^-],
	\end{equation}
	with
	\[
	\lambda = \max\{\lambda_{\max}\left( \mb{f}'(\mb{u}^-\right), \lambda_{\max}\left( \mb{f}'(\mb{u}^+\right) \},
	\]
	where $\lambda_{\max}$ is the spectral radius of the flux Jacobian matrix. In~\eqref{eq:low order update notations}, the inter-element fluxes $\mb{F}_{e\pm\frac{1}{2}}$ should be the same as the inter-element fluxes used for the high-order scheme for the conservative property of the blended scheme (Section~4.1 in~\cite{my_paper}). But using the high-order fluxes in the low-order scheme generates spurious oscillations and hence we use a blended inter-element flux coming from a convex combination of high and low-order fluxes~\cite{babbar2024admissibility,my_paper},
	\begin{equation}\label{eq:initial.guess.flux}
		\iniguessflux{\mb{F}}_{e\pm\frac{1}{2}} = (1 - \alpha_{e\pm\frac{1}{2}}) \highorderflux{\mb{F}}_{e\pm\frac{1}{2}} + \alpha_{e\pm\frac{1}{2}} \mb{f}^{L}_{e\pm\frac{1}{2}},
	\end{equation}
	where,
	\begin{equation}\label{eq:low.order.num.flux}
		\mb{f}^{L}_{e-\frac{1}{2}} = \mb{f}^{\text{NF}}(\mb{u}^{e-1}_N, \mb{u}^{e}_0), \qquad \mb{f}^{L}_{e+\frac{1}{2}} = \mb{f}^{\text{NF}}(\mb{u}^e_N, \mb{u}^{e+1}_0), \qquad \alpha_{e\pm\frac{1}{2}} = \frac{1}{2}(\alpha_{e\pm 1} + \alpha_{e}),
	\end{equation}
	as the initial guess, which needs to be blended again for admissibility.
	
	For the admissibility of the solution, the key is to make the low-order update~\eqref{eq: low order update} admissible, as the element means of the high and low-order updates are same (Section~4.1 in~\cite{my_paper}). In our work, we take the low-order scheme to be the first-order finite volume scheme with Rusanov flux~\eqref{eq:rusanov flux}, which is admissibility constraints preserving as stated in the following theorem.
	\begin{theorem}\label{theorem: admissibility of first order FVM}
		At any solution point, the solution of the RHD equations~\eqref{eq: RHD_system} at time $t=t_{n+1}$ with equations of state ID-EOS~\eqref{eq: ID_eos}, TM-EOS~\eqref{eq: TM_eos}, IP-EOS~\eqref{eq: IP_eos}, and RC-EOS~\eqref{eq: RC_eos} computed with first-order finite volume method using Rusanov flux~\cite{rusanov1962calculation} is in the admissible set $\Uad'$~\eqref{eq: ad_region_2} under some CFL type restrictions, provided the solution is admissible at the previous time $t = t_n$.
	\end{theorem}
	The proof of Theorem~\ref{theorem: admissibility of first order FVM} can be found in Appendix~\ref{sec: appendix}. 
	
	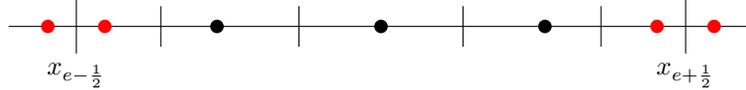
\begin{figure}[htbp]
		\centering
		\begin{tikzpicture}[scale=0.9]
			\draw (-0.5,0)-- (10.5,0); 
			\draw (0.5,0.4)--(0.5,-0.4) node[below]{$x_{e-\frac{1}{2}}$};
			\draw (1.749,0.3)--(1.749,-0.3) node[below]{};
			\draw (3.7885,0.3)--(3.7885,-0.3) node[below]{};
			\draw (6.211,0.3)--(6.211,-0.3) node[below]{};
			\draw (8.25,0.3)--(8.25,-0.3) node[below]{};
			\draw (9.5,0.4)--(9.5,-0.4) node[below]{$x_{e+\frac{1}{2}}$};
			\fill[red] (0.079, 0) circle (1.0mm);
			\fill[red] (0.921, 0) circle (1.0mm);
			\fill[black] (2.577, 0) circle (1.0mm);
			\fill[black] (5.0, 0) circle (1.0mm);
			\fill[black] (7.422, 0) circle (1.0mm);
			\fill[red] (9.078, 0) circle (1.0mm);
			\fill[red] (9.921, 0) circle (1.0mm);
		\end{tikzpicture}
		\caption{Part of the domain with red dots as extremal solution points} ($N=4$).\label{fig:subcell}
	\end{figure}
	
	However, as discussed above, in the element interfaces it is necessary to use a blended inter-element flux~\eqref{eq:initial.guess.flux}; which violates the admissibility nature of the low-order update at the solution points adjacent to the inter-element faces, referred as the \textit{extremal solution points} (denoted with red dots in Figure~\ref{fig:subcell}). Regarding the admissibility of the low-order update at the extremal solution points, we have the following theorem.
	
	\begin{theorem}
		Suppose $c(\mb{u})$ is a concave function of the conservative variables, then the solution at time $t_{n+1}$ at the extremal solution points adjacent to the face $x_{e+\frac{1}{2}}$ with the low-order scheme~\eqref{eq: low order update} satisfy $c(\mb{u}^{L, n+1}_l) > 0,\ c(\mb{u}^{L, n+1}_r) > 0$ if we use the following blended inter-element flux,
		\begin{equation}\label{eq:final.blended.flux}
			\mb{F}_{e+\frac{1}{2}} = (1 - \theta_c) \iniguessflux{\mb{F}}_{e+\frac{1}{2}} + \theta_c \mb{f}^{L}_{e+\frac{1}{2}},
		\end{equation}
		with
		\begin{equation}\label{eq:ad.blend.coeff}
			\theta_c = \min\Bigg\{\Bigg|\frac{\frac{1}{10}c\Big(\hat{\mb{u}}^{L,n+1}_{l}\Big)-c\Big(\hat{\mb{u}_{l}^{L,n+1}} \Big)}{c\Big(\mb{u}^{L,n+1}_{l,\text{old}}\Big)-c\Big(\hat{\mb{u}}_{l}^{L,n+1} \Big)} \Bigg|, 
			\ \Bigg|\frac{\frac{1}{10}c\Big(\hat{\mb{u}}^{L,n+1}_{r}\Big)-c\Big(\hat{\mb{u}_{r}^{L,n+1}} \Big)}{c\Big(\mb{u}^{L,n+1}_{r,\text{old}}\Big)-c\Big(\hat{\mb{u}}_{r}^{L,n+1} \Big)} \Bigg|,\ 1\Bigg\}, 
		\end{equation}
		where $(\cdot)_{\text{old}}$ denotes the update before applying the above blended flux~\eqref{eq:final.blended.flux} and $\hat{(\cdot)}$ denotes the low-order update with the low-order flux~\eqref{eq:rusanov flux} at the inter-element face. Again $(\cdot)_l,\ (\cdot)_r$ denote the solution point $x_N$ in element $\Omega_{e}$ and solution point $x_0$ in element $\Omega_{e+1}$ respectively.
	\end{theorem}
	\begin{proof}
		For the low-order evolution $\mb{u}^{L,n+1}_{l}$,
		\begin{align*}              
			c\Big(\mb{u}^{L,n+1}_{l}\Big)
			&=c\Big(\theta_c \mb{u}_{l, \text{old}}^{L,n+1} + (1-\theta_c) \hat{\mb{u}}_{l}^{L,n+1}\Big)\\
			&\geq \theta_c c\Big(\mb{u}_{l, \text{old}}^{L,n+1}\Big) + (1-\theta_c) c\Big(\hat{\mb{u}}_{l}^{L,n+1}\Big)\qquad \text{since, $c$ is concave}\\
			&> \frac{1}{10} \bigg(c\Big(\hat{\mb{u}}_{l}^{L,n+1}\Big)\bigg).
		\end{align*}
		
		Similarly we can also prove for $\mb{u}^{L,n+1}_{r}$.
	\end{proof}
	
	Here the factor $\frac{1}{10}$ is taken following~\cite{rueda2021subcell}. For our case, since we have two admissibility constraints~\eqref{eq: ad_region_2}, we need to blend the inter-element flux two times as in equation~\eqref{eq:final.blended.flux} for the admissibility of the low-order updates~\eqref{eq: low order update} at each of the extremal solution points. For more details, refer to~\cite{babbar2024admissibility,my_paper}. 
	
	Once we get the element means of the solution of the blended scheme admissible, we use the scaling limiter from~\cite{zhang2010maximum} to scale the final solution in the admissible region.
	
	Before going to the numerical validations, let us present a high-level overview of the scheme in Algorithm~\ref{alg:high.level.algorithm}.
	\begin{algorithm}
		\caption{High-level overview of the scheme}
		\label{alg:high.level.algorithm}
		\begin{algorithmic}
			\While{$t < T$}
			\State Compute blending coefficient $\{\alpha_e\}$ (Section~5 of~\cite{my_paper})
			\For{$e$ in \texttt{eachelement(mesh)}}
			\Comment{Compute time average flux at solution points}
			\For{$r$ in \texttt{1:N}}
			\For{$i$ in \texttt{eachpoint(element)}}
			\State Compute $\mb{u}_i^{[r]}$'s~\eqref{eq:flux_arg_ALW}  and scale it with corresponding $\mb{u}_i^{[r-1]}$'s
			\State Compute $\mb{f}_i^{(r)}$~\eqref{eq:ALW.N1}-\eqref{eq:ALW.N4} and add its contribution for the computation of $\mb{F}_i$~\eqref{eq:time.avg.flux.sol.point.1} 
			\EndFor
			\EndFor
			\EndFor
			\For{$e+\frac{1}{2}$ in 
				\texttt{eachinterface(mesh)}}
			\Comment{Compute interface flux}
			\State Compute the high-order inter-element flux $\highorderflux{\mb{F}}_{e+\frac{1}{2}}$~\eqref{eq:high.order.numerical.flux}
			\State Compute the low-order inter-element flux $\mb{f}_{e+\frac{1}{2}}^L$~\eqref{eq:low.order.num.flux}
			\State Compute the initial guess of the blended inter-element flux $\iniguessflux{\mb{F}}_{e+\frac{1}{2}}$~\eqref{eq:initial.guess.flux}
			\For{$c$ in \texttt{eachconstraint($\Uad'$)}}
			\State Compute $\theta_c$~\eqref{eq:ad.blend.coeff}
			\State Blend the inter-element flux to compute $\mb{F}_{e+\frac{1}{2}}$~\eqref{eq:final.blended.flux}
			\EndFor
			\EndFor
			\For{$e$ in \texttt{eachelement(mesh)}}
			\Comment{Update solution}
			\For{$i$ in \texttt{eachpoint(element)}}
			\State Calculate the high-order update~\eqref{eq: high order update}
			\State Calculate the low-order update~\eqref{eq: low order update}
			\State Blend the high and low-order solutions~\eqref{eq:blending.solution}
			\EndFor
			\EndFor
			\State Apply positivity correction at solution points using \cite{zhang2010maximum}
			\State $t \gets t + \Delta t$
			\EndWhile
		\end{algorithmic}
	\end{algorithm}
	
	\section{Numerical simulations}\label{sec: Numerical_simulations}
	Several numerical simulations are shown in this section to demonstrate the robustness of the scheme and its capability to capture different wave structures. The simulations are done using \texttt{RHDTenkai.jl}~\cite{RHDTenkai}, which is developed using \texttt{Tenkai.jl}~\cite{tenkai} as a library.
	
	\subsection{One dimensional experiments}
	For all the one dimensional test cases here, we take the time step as,
	\begin{equation} \label{eq:time_step.1d}
		\Delta t= l_s \text{CFL}(N)\min_{e}{\left(\frac{\lambda_{\max}\left(\mb{f}'(\Bar{\mb{u}}_{e})\right)}{\Delta x_e}\right)^{-1}},
	\end{equation}
	where $l_s\leq 1$ is a safety factor, taken to be $0.95$ for all the test cases unless mentioned otherwise. Here, $\lambda_{\max}(\cdot)$ denotes the spectral radius and CFL($N$) is the optimal CFL number for the solution polynomial of degree $N$ as obtained in~\cite{BABBAR2022111423} using Fourier stability analysis (Table~1 in~\cite{BABBAR2022111423}).
	
	In this section, we will test the scheme with one dimensional test cases having strong shocks and other discontinuities along with high Lorentz factor, smooth wave adjacent to discontinuity and thin structures in the solution. However, before going to the main problems of interest, we check the accuracy and numerical order of the scheme with a smooth initial data with different resolutions.
	
	\subsubsection{Accuracy test}\label{sec:1Daccuracy}
	We take a smooth flow of a fluid having initial density $\rho(x,0) = 1 + 0.999\sin\left( 2 \pi x\right)$ and velocity $v_1(x,0) = 0.99$. The initial pressure is taken as $p(x,0) = 0.01$ in the computational domain. The density of the fluid will change with time $t$ as $\rho(x,t) = 1 + 0.999\sin\left( 2 \pi (x-0.99 t)\right)$. We present the results of our simulations at time $t=0.2$ in the Tables~\ref{table: N3_ID5/3}-\ref{table: N4_RC} with different equations of state taking periodic boundaries at $x=0,1$.
	
	\begin{table}[!htbp]
		\centering
		\caption{Numerical results for the fluid density ($\rho$) with $N=3$ using ID-EOS~\eqref{eq: ID_eos} with $\gamma = \frac{5}{3}$.}
		\label{table: N3_ID5/3}
		\begin{tabular}{lllllll}
			\hline\noalign{\smallskip}
			Cells & $L^1$ error  & $L^1$ Order     & $L^2$ error   & $L^2$ Order     & $L^{\infty}$ error & $L^{\infty}$ Order\\
			\noalign{\smallskip}\hline\noalign{\smallskip}
			8   &    1.35447e-04     &      -             &          2.74800e-04       &        -           &          9.79990e-04           &          - \\ 
			16   &    1.51842e-05     &      3.15708       &      3.67970e-05      &       2.90072      &       1.69608e-04        &          2.53057\\ 
			32   &    1.56589e-07     &      6.59945       &      1.86208e-07      &       7.62653      &       3.71109e-07        &          8.83614\\ 
			64   &    9.22607e-09     &      4.08512       &      1.08760e-08      &       4.09770      &       2.19672e-08        &          4.07842\\ 
			128   &    6.09975e-10     &      3.91889       &      7.21485e-10      &       3.91403      &       1.44395e-09        &          3.92725\\ 
			256   &    3.84068e-11     &      3.98932       &      4.54083e-11      &       3.98994      &       9.08849e-11        &          3.98984\\ 
			\noalign{\smallskip}\hline
		\end{tabular}
		\vspace{0.5em}
		\centering
		\caption{Numerical results for the fluid density ($\rho$) with $N=4$ using ID-EOS~\eqref{eq: ID_eos} with $\gamma = \frac{5}{3}$.}
		\label{table: N4_ID5/3}
		
		\begin{tabular}{lllllll}
			\hline\noalign{\smallskip}
			Cells & $L^1$ error  & $L^1$ Order     & $L^2$ error   & $L^2$ Order     & $L^{\infty}$ error & $L^{\infty}$ Order\\
			\noalign{\smallskip}\hline\noalign{\smallskip}
			8   &    2.88032e-05     &      -             &          5.80292e-05       &        -           &          2.10470e-04           &          - \\ 
			16   &    4.67826e-08     &      9.26604       &      5.58369e-08      &       10.02134      &       1.03861e-07        &          10.98475\\ 
			32   &    1.68419e-09     &      4.79584       &      2.04167e-09      &       4.77340      &       3.90200e-09        &          4.73429\\ 
			64   &    5.26892e-11     &      4.99841       &      6.06662e-11      &       5.07271      &       1.12791e-10        &          5.11249\\ 
			\noalign{\smallskip}\hline
		\end{tabular}
		\vspace{0.5em}
		\centering
		\caption{Numerical results for the fluid density ($\rho$) with $N=3$ using ID-EOS~\eqref{eq: ID_eos} with $\gamma = \frac{4}{3}$.}
		\label{table: N3_ID4/3}
		
		\begin{tabular}{lllllll}
			\hline\noalign{\smallskip}
			Cells & $L^1$ error  & $L^1$ Order     & $L^2$ error   & $L^2$ Order     & $L^{\infty}$ error & $L^{\infty}$ Order\\
			\noalign{\smallskip}\hline\noalign{\smallskip} 
			8   &    1.36588e-04     &      -             &          2.76937e-04       &        -           &          9.86693e-04           &          - \\ 
			16   &    1.53319e-05     &      3.15522       &      3.72004e-05      &       2.89617      &       1.72322e-04        &          2.51749\\ 
			32   &    1.56452e-07     &      6.61467       &      1.85966e-07      &       7.64413      &       3.70561e-07        &          8.86118\\ 
			64   &    9.02627e-09     &      4.11545       &      1.06441e-08      &       4.12692      &       2.14578e-08        &          4.11014\\ 
			\noalign{\smallskip}\hline
		\end{tabular}
		\vspace{0.5em}
		\centering
		\caption{Numerical results for the fluid density ($\rho$) with $N=4$ using ID-EOS~\eqref{eq: ID_eos} with $\gamma = \frac{4}{3}$.}
		\label{table: N4_ID4/3}
		
		\begin{tabular}{lllllll}
			\hline\noalign{\smallskip}
			Cells & $L^1$ error  & $L^1$ Order     & $L^2$ error   & $L^2$ Order     & $L^{\infty}$ error & $L^{\infty}$ Order\\
			\noalign{\smallskip}\hline\noalign{\smallskip}
			8   &    2.89947e-05     &      -             &          5.83639e-05       &        -           &          2.11448e-04           &          - \\ 
			16   &    4.75493e-08     &      9.25215       &      5.69429e-08      &       10.00134      &       1.05717e-07        &          10.96589\\ 
			32   &    1.73403e-09     &      4.77722       &      2.08869e-09      &       4.76884      &       3.99009e-09        &          4.72764\\ 
			64   &    5.08093e-11     &      5.09289       &      5.88751e-11      &       5.14880      &       1.08820e-10        &          5.19640\\ 
			\noalign{\smallskip}\hline
		\end{tabular}
		\vspace{0.5em}
		\centering
		\caption{Numerical results for the fluid density ($\rho$) with $N=3$ using TM-EOS~\eqref{eq: TM_eos}.}
		\label{table: N3_TM}
		
		\begin{tabular}{lllllll}
			\hline\noalign{\smallskip}
			Cells & $L^1$ error  & $L^1$ Order     & $L^2$ error   & $L^2$ Order     & $L^{\infty}$ error & $L^{\infty}$ Order\\
			\noalign{\smallskip}\hline\noalign{\smallskip}
			8   &    1.77713e-02     &      -             &          2.16153e-02       &        -           &          4.57774e-02           &          - \\ 
			16   &    1.99109e-03     &      3.15792       &      3.92807e-03      &       2.46016      &       1.22346e-02        &          1.90367\\ 
			32   &    8.24227e-06     &      7.91630       &      2.22302e-05      &       7.46515      &       1.12005e-04        &          6.77126\\ 
			64   &    1.82978e-07     &      5.49330       &      6.70969e-07      &       5.05013      &       5.12770e-06        &          4.44910\\ 
			128   &    7.78286e-09     &      4.55522       &      3.45242e-08      &       4.28057      &       2.84067e-07        &          4.17401\\ 
			256   &    4.39237e-10     &      4.14723       &      1.84301e-09      &       4.22747      &       1.65366e-08        &          4.10250\\ 
			\noalign{\smallskip}\hline
		\end{tabular}
	\end{table}
	
	\begin{table}[!htbp]
		\centering
		\caption{Numerical results for the fluid density ($\rho$) with $N=4$ using TM-EOS~\eqref{eq: TM_eos}.}
		\label{table: N4_TM}
		
		\begin{tabular}{lllllll}
			\hline\noalign{\smallskip}
			Cells & $L^1$ error  & $L^1$ Order     & $L^2$ error   & $L^2$ Order     & $L^{\infty}$ error & $L^{\infty}$ Order\\
			\noalign{\smallskip}\hline\noalign{\smallskip}
			8   &    6.85694e-03     &      -             &          1.02719e-02       &        -           &          2.41027e-02           &          - \\ 
			16   &    5.56375e-04     &      3.62344       &      1.21292e-03      &       3.08214      &       3.86563e-03        &          2.64042\\ 
			32   &    6.77174e-07     &      9.68232       &      2.08020e-06      &       9.18755      &       1.35358e-05        &          8.15778\\ 
			64   &    1.50046e-08     &      5.49605       &      7.02886e-08      &       4.88729      &       7.35403e-07        &          4.20210\\ 
			128   &    4.20525e-10     &      5.15707       &      1.98336e-09      &       5.14727      &       2.30749e-08        &          4.99414\\ 
			\noalign{\smallskip}\hline
		\end{tabular}
		\vspace{0.5em}
		\centering
		\caption{Numerical results for the fluid density ($\rho$) with $N=3$ using IP-EOS~\eqref{eq: IP_eos}.}
		\label{table: N3_IP}
		
		\begin{tabular}{lllllll}
			\hline\noalign{\smallskip}
			Cells & $L^1$ error  & $L^1$ Order     & $L^2$ error   & $L^2$ Order     & $L^{\infty}$ error & $L^{\infty}$ Order\\
			\noalign{\smallskip}\hline\noalign{\smallskip}
			8   &    1.81996e-02     &      -             &          2.42778e-02       &        -           &          5.64409e-02           &          - \\ 
			16   &    2.22355e-03     &      3.03296       &      4.34419e-03      &       2.48248      &       1.29502e-02        &          2.12377\\ 
			32   &    3.79531e-06     &      9.19443       &      1.03238e-05      &       8.71696      &       5.16936e-05        &          7.96877\\ 
			64   &    1.53555e-07     &      4.62739       &      6.05417e-07      &       4.09191      &       4.90159e-06        &          3.39866\\ 
			128   &    7.33559e-09     &      4.38770       &      2.99379e-08      &       4.33789      &       2.59443e-07        &          4.23976\\ 
			256   &    4.26056e-10     &      4.10580       &      1.63895e-09      &       4.19113      &       1.33523e-08        &          4.28026\\ 
			\noalign{\smallskip}\hline
		\end{tabular}
		
		\vspace{0.5em}
		\centering
		\caption{Numerical results for the fluid density ($\rho$) with $N=4$ using IP-EOS~\eqref{eq: IP_eos}.}
		\label{table: N4_IP}
		
		\begin{tabular}{lllllll}
			\hline\noalign{\smallskip}
			Cells & $L^1$ error  & $L^1$ Order     & $L^2$ error   & $L^2$ Order     & $L^{\infty}$ error & $L^{\infty}$ Order\\
			\noalign{\smallskip}\hline\noalign{\smallskip}
			8   &    9.96247e-03     &      -             &          1.45910e-02       &        -           &          3.45552e-02           &          - \\ 
			16   &    6.11523e-04     &      4.02603       &      1.33604e-03      &       3.44904      &       4.29264e-03        &          3.00897\\ 
			32   &    5.69459e-07     &      10.06860       &      1.83876e-06      &       9.50502      &       1.12573e-05        &          8.57486\\ 
			64   &    1.17153e-08     &      5.60313       &      5.12770e-08      &       5.16427      &       4.86726e-07        &          4.53161\\ 
			128   &    3.60158e-10     &      5.02362       &      1.57435e-09      &       5.02549      &       1.75990e-08        &          4.78954\\ 
			\noalign{\smallskip}\hline
		\end{tabular}
		
		\vspace{0.5em}
		\centering
		\caption{Numerical results for the fluid density ($\rho$) with $N=3$ using RC-EOS~\eqref{eq: RC_eos}.}
		\label{table: N3_RC}
		
		\begin{tabular}{lllllll}
			\hline\noalign{\smallskip}
			Cells & $L^1$ error  & $L^1$ Order     & $L^2$ error   & $L^2$ Order     & $L^{\infty}$ error & $L^{\infty}$ Order\\
			\noalign{\smallskip}\hline\noalign{\smallskip}
			8   &    1.61621e-02     &      -             &          1.94391e-02       &        -           &          4.08798e-02           &          - \\ 
			16   &    1.68963e-03     &      3.25784       &      3.23141e-03      &       2.58873      &       1.00716e-02        &          2.02110\\ 
			32   &    1.40743e-05     &      6.90750       &      3.98978e-05      &       6.33971      &       1.59256e-04        &          5.98280\\ 
			64   &    9.80485e-08     &      7.16535       &      3.80577e-07      &       6.71198      &       2.97660e-06        &          5.74154\\ 
			128   &    4.23334e-09     &      4.53363       &      1.71584e-08      &       4.47120      &       1.53399e-07        &          4.27830\\ 
			256   &    2.35897e-10     &      4.16557       &      8.59274e-10      &       4.31966      &       8.06991e-09        &          4.24860\\
			\noalign{\smallskip}\hline
		\end{tabular}
		
		\vspace{0.5em}
		\centering
		\caption{Numerical results for the fluid density ($\rho$) with $N=4$ using RC-EOS~\eqref{eq: RC_eos}.}
		\label{table: N4_RC}
		
		\begin{tabular}{lllllll}
			\hline\noalign{\smallskip}
			Cells & $L^1$ error  & $L^1$ Order     & $L^2$ error   & $L^2$ Order     & $L^{\infty}$ error & $L^{\infty}$ Order\\
			\noalign{\smallskip}\hline\noalign{\smallskip}
			8   &    5.71667e-03     &      -             &          8.81007e-03       &        -           &          2.06541e-02           &          - \\ 
			16   &    2.68642e-04     &      4.41142       &      5.69776e-04      &       3.95069      &       1.80839e-03        &          3.51365\\ 
			32   &    5.49479e-07     &      8.93340       &      1.59137e-06      &       8.48398      &       8.41902e-06        &          7.74684\\ 
			64   &    8.21426e-09     &      6.06379       &      3.84648e-08      &       5.37059      &       3.55182e-07        &          4.56702\\ 
			128   &    2.11005e-10     &      5.28278       &      9.38232e-10      &       5.35745      &       9.00065e-09        &          5.30239\\ 
			\noalign{\smallskip}\hline
		\end{tabular}
	\end{table}
	
	We observe that for all the equations of state, the scheme converges with order $O(\Delta x)^{N+1}$ for the degrees of solution polynomial $N=3,4$.
	
	\subsubsection{1-D Riemann problem 1}\label{sec:1DRP1} 
	This problem is used in~\cite{ryu2006equation} for comparing results with three different equations of state. The solution of this problem has a shock wave, a contact discontinuity, and a rarefaction wave, which makes it a suitable test to check the robustness of the scheme and the effect of different equations of state. The computational domain is taken as $[0,1]$ with initial discontinuity at $x=0.5$. Specifically, the initial condition is taken as,
	\[
	(\rho, v_1, p) = \begin{cases}
		(10, 0, 13.3) \quad &\text{if}\ x<0.5\\
		(1, 0, 10^{-6}) \quad &\text{if}\ x>0.5
	\end{cases}
	\]
	and the boundaries with outflow boundary conditions. The problem is solved with solution polynomial degree $N=3,4$ and with $500$ cells till time $t=0.45$. The results of the scheme for ID-EOS are compared with a reference solution obtained using the exact solver in~\cite{marti2003numerical} and shown in Figure~\ref{fig:1DRyuRPP1.eos1.exact}. For the other equations of state, we have compared the results with a reference solution obtained using the Rusanov scheme~\cite{rusanov1962calculation} with a fine mesh of $100000$ uniform cells and show the results in Figure~\ref{fig:1DRyuRPP1.eos234.exact}. We observe that the scheme can capture the shock, rarefaction, and contact discontinuity effectively in the solution. We also observe that the solutions converge to the exact or reference solution for all the equations of state by increasing the degree $N$ from 3 to 4. 
	This behavior of converging to the reference solution is seen for all the test cases in the paper, although we do not show the reference solution in subsequent test cases in order to save space.
	\begin{figure}
		\centering
		\hspace{0.3mm}\begin{subfigure}{0.32\textwidth}
			\includegraphics[width=\linewidth]{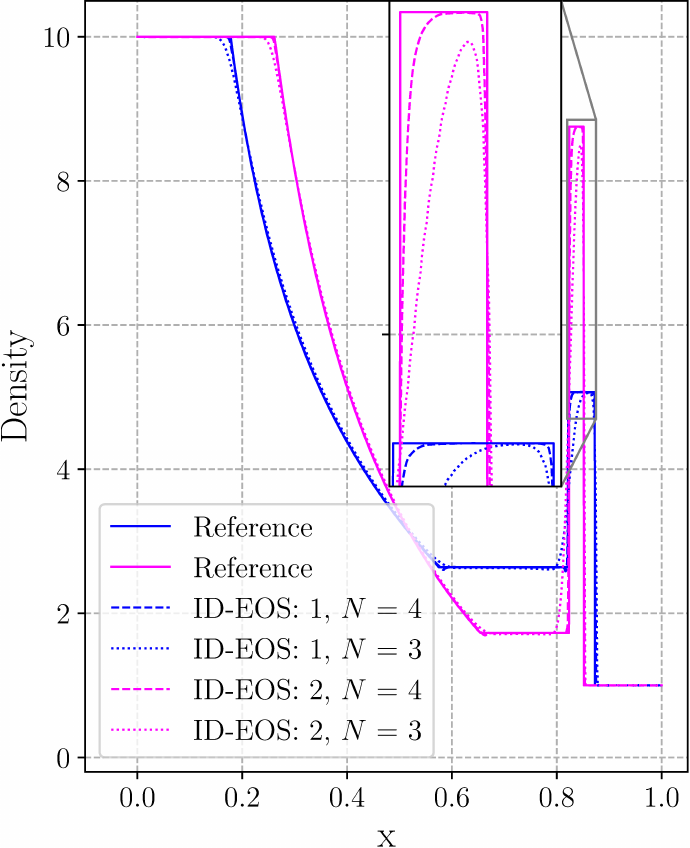}
			\caption{Density ($\rho$)}
		\end{subfigure}
		\begin{subfigure}{0.32\textwidth}
			\includegraphics[width=\linewidth]{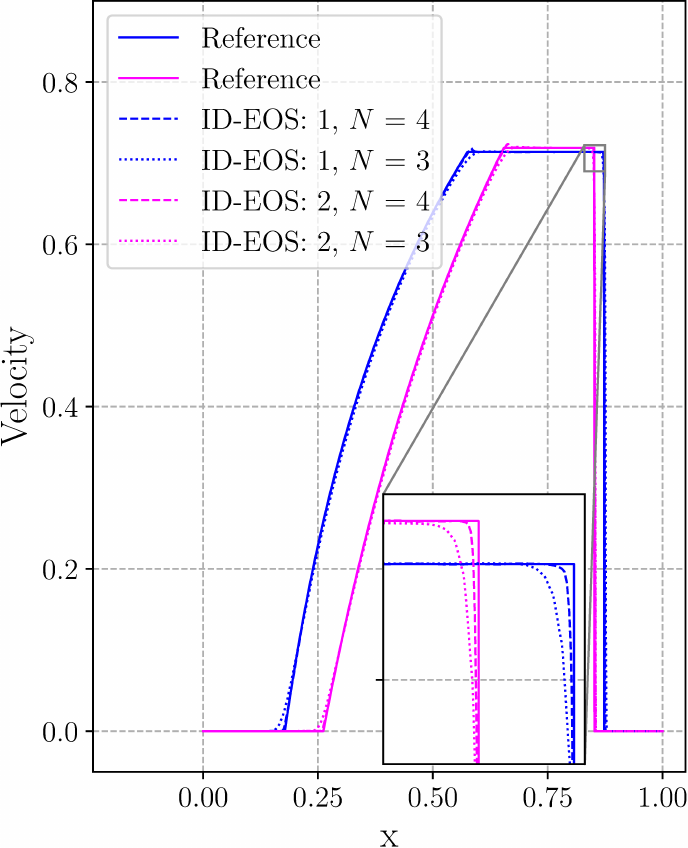}
			\caption{Velocity ($v_1$)}
		\end{subfigure}
		\begin{subfigure}{0.32\textwidth}
			\includegraphics[width=\linewidth]{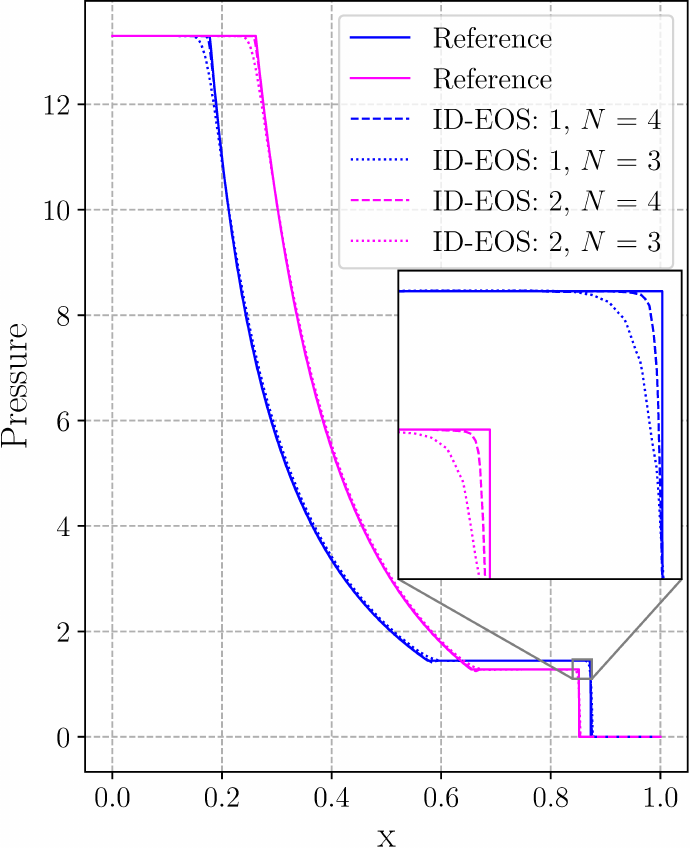}
			\caption{Pressure ($p$)}
		\end{subfigure}
		\caption{1-D Riemann problem 1: Plot with $500$ cells with reference solution. ID-EOS: 1 and ID-EOS: 2 refer to ID-EOS with $\gamma=\frac{5}{3}$ and $\gamma=\frac{4}{3}$ respectively.}
		\label{fig:1DRyuRPP1.eos1.exact}
	\end{figure}
	
	\begin{figure}
		\centering
		\hspace{0.3mm}\begin{subfigure}{0.32\textwidth}
			\includegraphics[width=\linewidth]{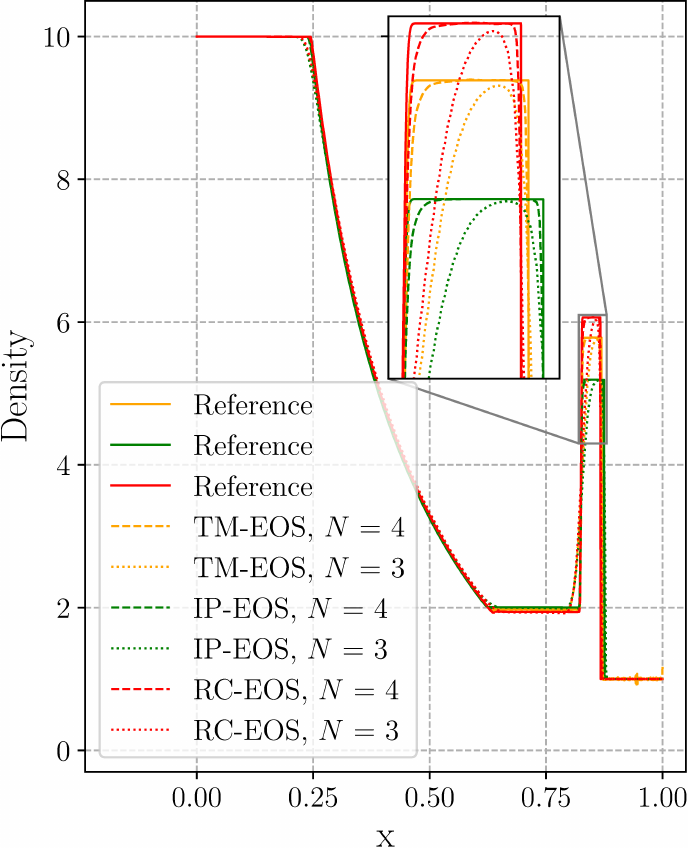}
			\caption{Density ($\rho$)}
		\end{subfigure}
		\begin{subfigure}{0.32\textwidth}
			\includegraphics[width=\linewidth]{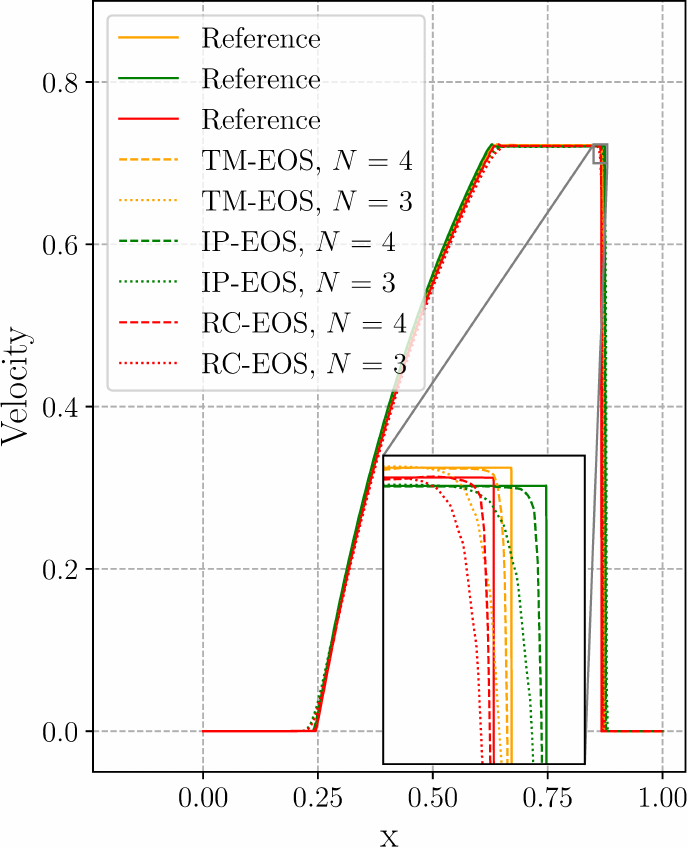}
			\caption{Velocity ($v_1$)}
		\end{subfigure}
		\begin{subfigure}{0.32\textwidth}
			\includegraphics[width=\linewidth]{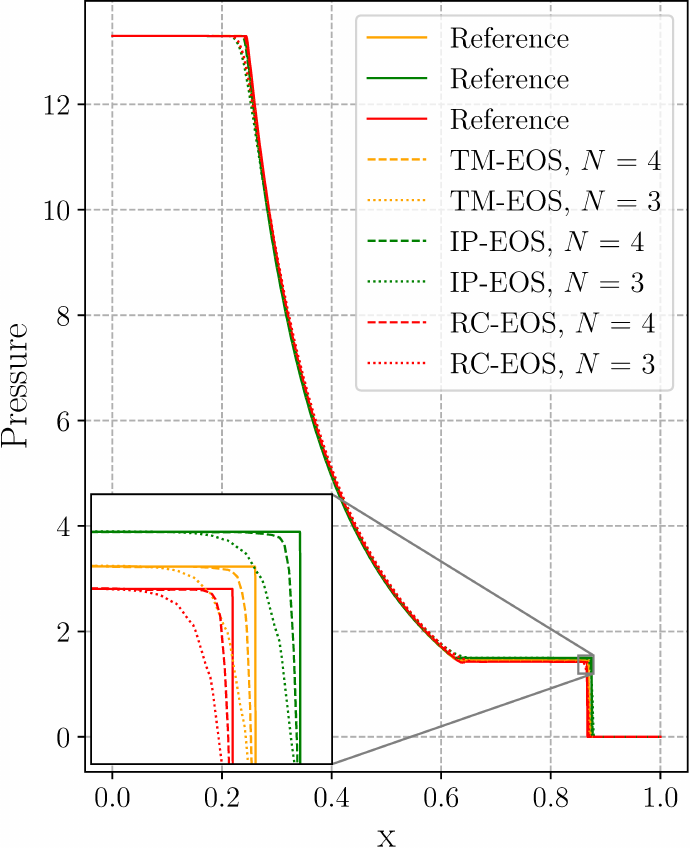}
			\caption{Pressure ($p$)}
		\end{subfigure}
		\caption{1-D Riemann problem 1: Plot with $500$ cells with reference solution.}
		\label{fig:1DRyuRPP1.eos234.exact}
	\end{figure}
	
	\begin{figure}
		\centering
		\hspace{0.3mm}\begin{subfigure}{0.32\textwidth}
			\includegraphics[width=\linewidth]{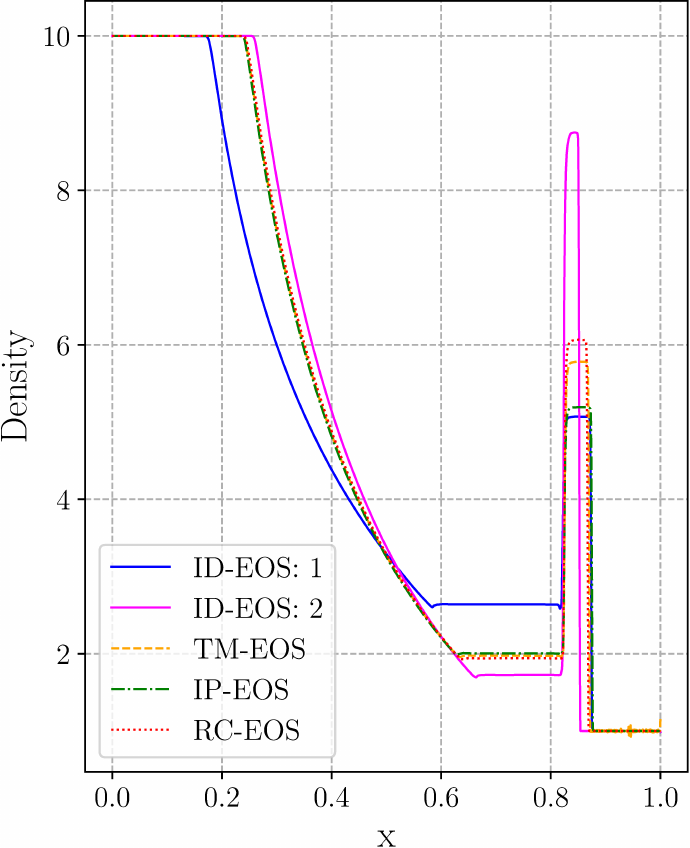}
			\caption{Density ($\rho$)}
		\end{subfigure}
		\begin{subfigure}{0.32\textwidth}
			\includegraphics[width=\linewidth]{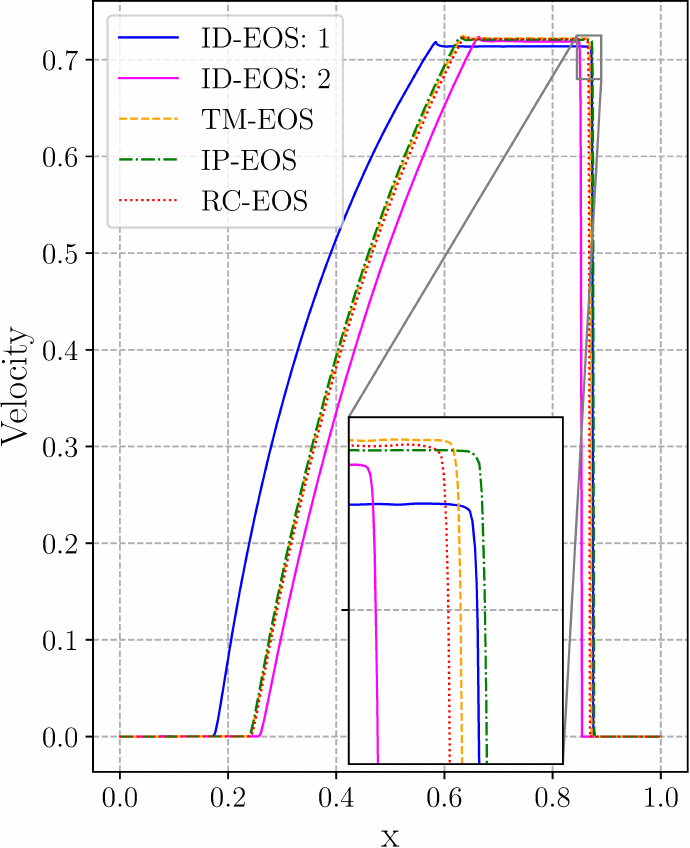}
			\caption{Velocity ($v_1$)}
		\end{subfigure}
		\begin{subfigure}{0.32\textwidth}
			\includegraphics[width=\linewidth]{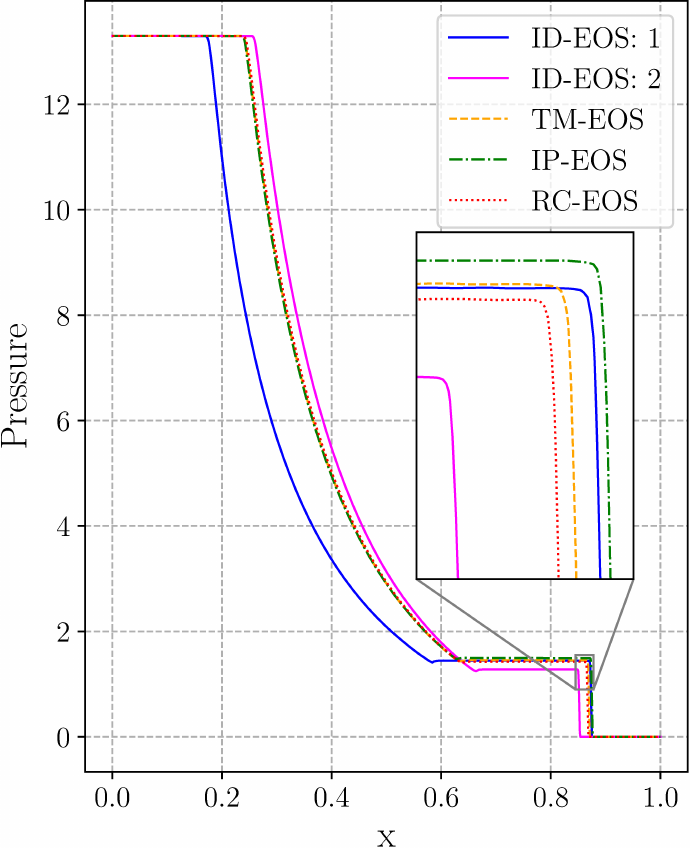}
			\caption{Pressure ($p$)}
		\end{subfigure}
		\caption{1-D Riemann problem 1: Plot with $500$ cells and $N=4$. ID-EOS: 1 and ID-EOS: 2 refer to ID-EOS with $\gamma=\frac{5}{3}$ and $\gamma=\frac{4}{3}$ respectively.}
		\label{fig:1DRyuRPP1}
	\end{figure}
	Now, presenting the results using all the equations of state using degree $N = 4$ in Figure~\ref{fig:1DRyuRPP1}, we see that the solutions with ID-EOS have noticeable differences between $\gamma=\frac{4}{3}$ and $\gamma=\frac{5}{3}$. The solution with ID-EOS with $\gamma=\frac{4}{3}$ has taller and thinner shell-structure between the contact discontinuity and the shock wave, and less elongated rarefaction wave, as also observed in~\cite{ryu2006equation}. Again, similar to the results in~\cite{ryu2006equation}, the solution with TM-EOS, IP-EOS, and RC-EOS is better approximated by ID-EOS with $\gamma = \frac{4}{3}$ compared to $\gamma = \frac{5}{3}$ in the region left of the contact discontinuity. The solutions with TM-EOS, IP-EOS, and RC-EOS have less deviation from each other, showing the similarity in the distribution of specific enthalpy $h$, which is one more property that was observed in~\cite{ryu2006equation}.
	
	\subsubsection{1-D Riemann problem 2}
	This problem is taken from~\cite{marti2003numerical} where the solution has a very thin structure in the post-shock state. It is also used in~\cite{ryu2006equation} for the simulations with three different equations of state. The initial condition is given by,
	\[
	(\rho, v_1, p) = \begin{cases}
		(1, 0, 10^3) \quad &\text{if}\ x<0.5\\
		(1, 0, 10^{-2}) \quad &\text{if}\ x>0.5
	\end{cases}
	\]
	with a jump discontinuity at $x=0.5$ in pressure $p$. The boundaries are taken as outflow boundaries with the computational domain $[0,1]$. The simulations for this problem are performed with $500$ cells and degree $N=4$ with different equations of state, and results at time $t=0.4$ are shown in Figure~\ref{fig:1DRyuRPP2}.
	\begin{figure}
		\centering
		\hspace{0.3mm}\begin{subfigure}{0.32\textwidth}
			\includegraphics[width=\linewidth]{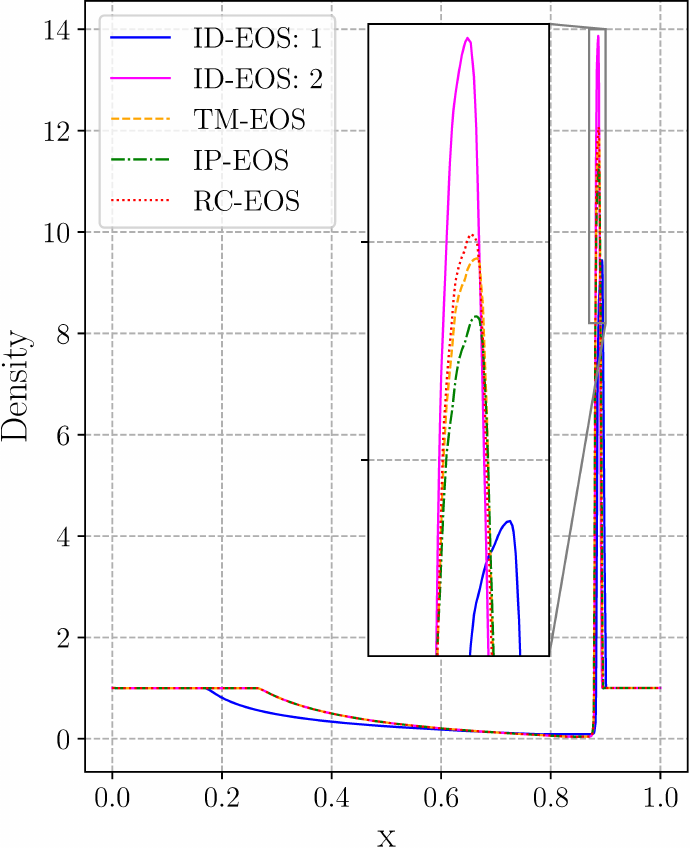}
			\caption{Density ($\rho$)}
		\end{subfigure}
		\begin{subfigure}{0.32\textwidth}
			\includegraphics[width=\linewidth]{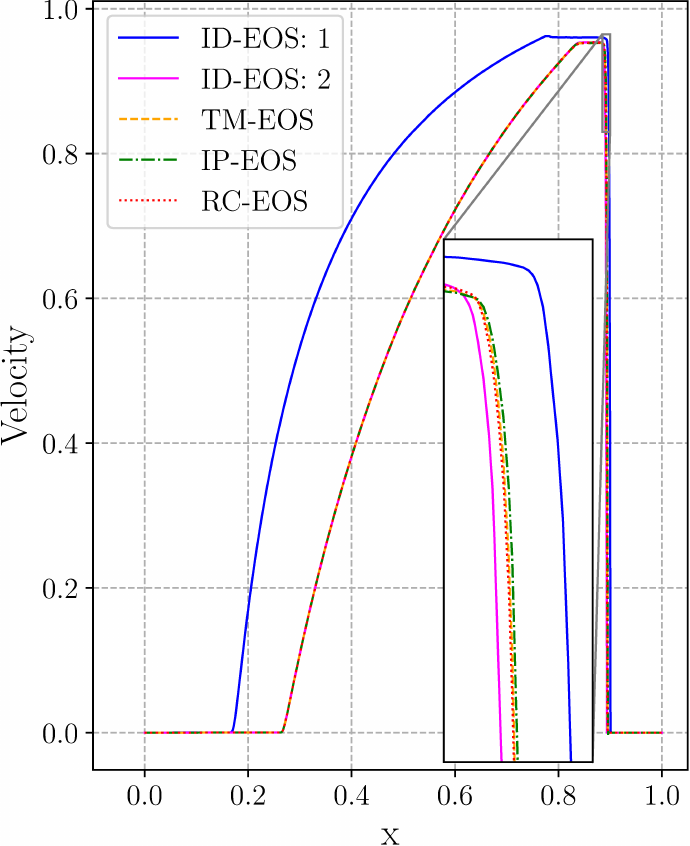}
			\caption{Velocity ($v_1$)}
		\end{subfigure}
		\begin{subfigure}{0.32\textwidth}
			\includegraphics[width=\linewidth]{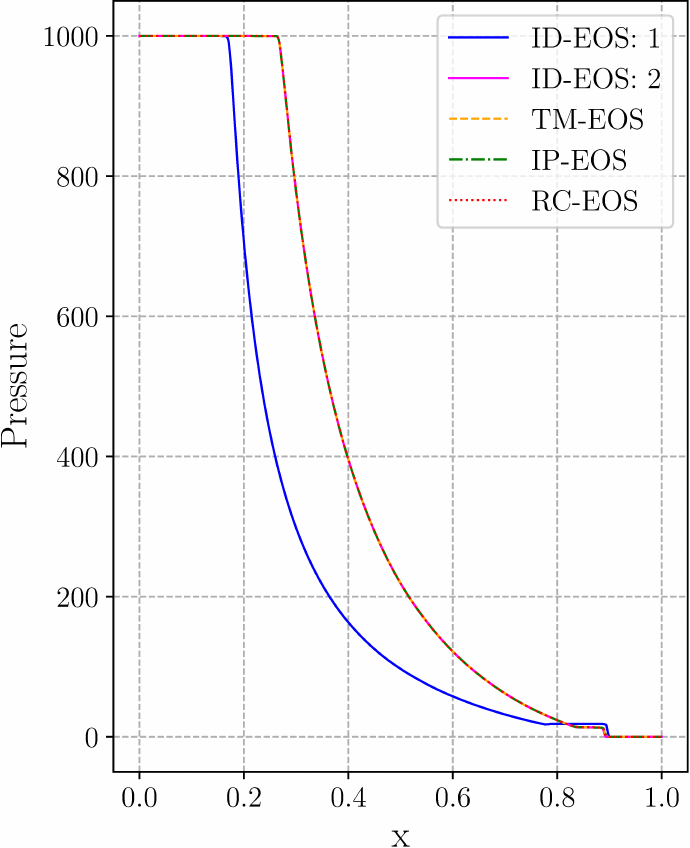}
			\caption{Pressure ($p$)}
		\end{subfigure}
		\caption{1-D Riemann problem 2: Plot with $500$ cells and $N=4$. ID-EOS: 1 and ID-EOS: 2 refer to ID-EOS with $\gamma=\frac{5}{3}$ and $\gamma=\frac{4}{3}$ respectively.}
		\label{fig:1DRyuRPP2}
	\end{figure}
	
	From the figure, we observe that the solution using ID-EOS with $\gamma=\frac{5}{3}$ has a more elongated rarefaction wave than the other equations of state because of the higher sound speed~\cite{ryu2006equation}. We can again observe that the solution with ID-EOS with $\gamma = \frac{4}{3}$ is nearly indistinguishable from the solutions with TM-EOS, IP-EOS, and RC-EOS in the left region to the contact discontinuity, where the domain has ultra-relativistic temperature, $\frac{p}{\rho} \gg 1$. The results with other equations of state are nearly overlapping throughout the domain, which is a behavior similar to~\cite{ryu2006equation}.
	
	\subsubsection{1-D Riemann problem 3}
	This problem is taken from~\cite{marti1994analytical,mignone2005hllc} where two rarefaction waves move away from each other with time and form a contact discontinuity in between. Here, the initial state is taken as,
	\[
	(\rho, v_1, p) = \begin{cases}
		(1, -0.6, 10) \quad &\text{if}\ x<0.5\\
		(10, 0.5, 20) \quad &\text{if}\ x>0.5
	\end{cases}
	\]
	with outflow boundaries at $x=0,1$. We run the simulations using $500$ cells and $N=4$ till time $t=0.4$ and plot the results in Figure~\ref{fig:1DRP1}.
	\begin{figure}
		\centering
		\hspace{0.3mm}\begin{subfigure}{0.32\textwidth}
			\includegraphics[width=\linewidth]{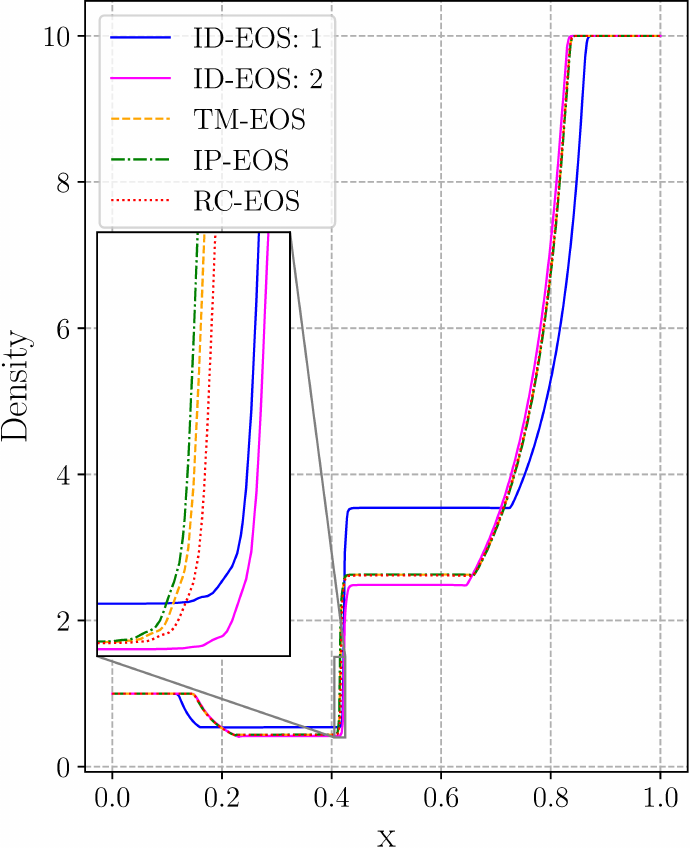}
			\caption{Density ($\rho$)}
		\end{subfigure}
		\begin{subfigure}{0.32\textwidth}
			\includegraphics[width=\linewidth]{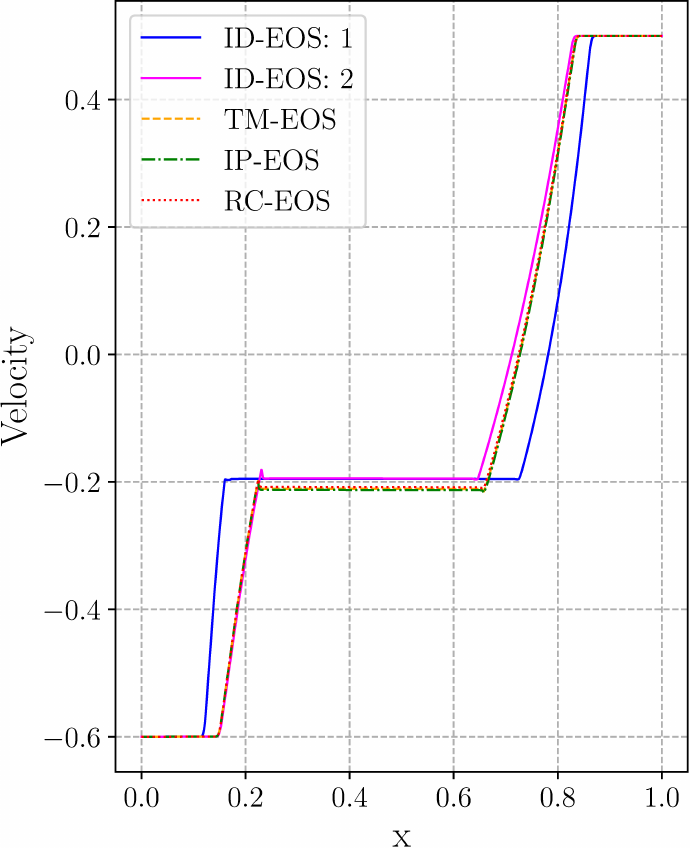}
			\caption{Velocity ($v_1$)}
		\end{subfigure}
		\begin{subfigure}{0.32\textwidth}
			\includegraphics[width=\linewidth]{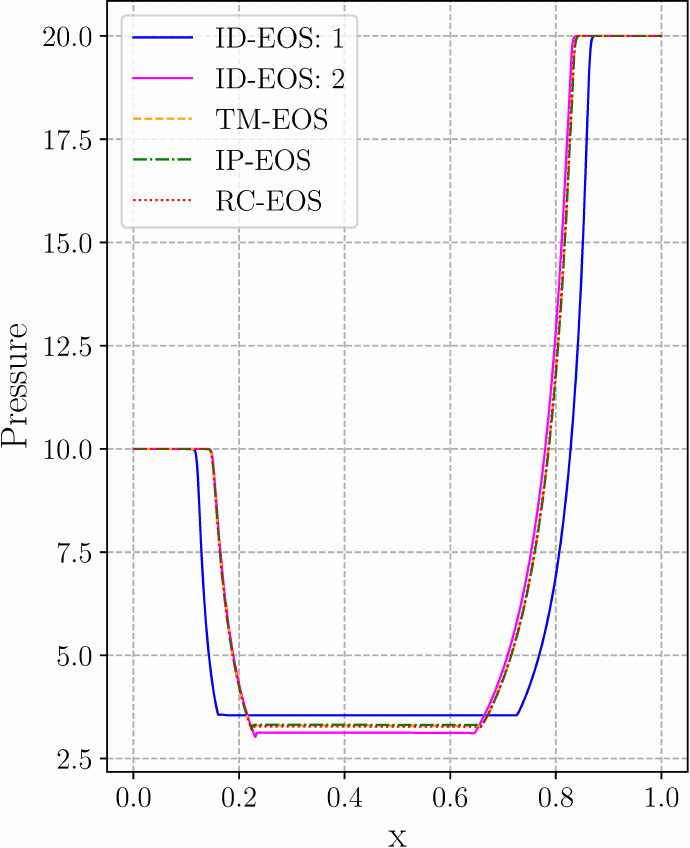}
			\caption{Pressure ($p$)}
		\end{subfigure}
		\caption{1-D Riemann problem 3: Plot with $500$ cells and $N=4$. ID-EOS: 1 and ID-EOS: 2 refer to ID-EOS with $\gamma=\frac{5}{3}$ and $\gamma=\frac{4}{3}$ respectively.}
		\label{fig:1DRP1}
	\end{figure}
	
	We observe from the figure that the solutions with TM-EOS, IP-EOS, and RC-EOS are very close to each other, a characteristic which is observed in previous simulations as well. We can also observe that the ID-EOS approximates the solutions with other equations of state more closely with $\gamma = \frac{4}{3}$ than with $\gamma = \frac{5}{3}$ and the domain has relativistic temperature, $\frac{p}{\rho}>1$, justifying the choice of $\gamma$ as $\frac{4}{3}$ for ultra-relativistic cases~\cite{wu2016physical}.
	
	\subsubsection{1-D density perturbation problem}
	This problem is taken from~\cite{del2002efficient} where a sinusoidal profile is introduced in the fluid density, which makes the problem interesting as the scheme needs to resolve the smooth wave structures while avoiding the generation of spurious oscillations. A similar problem was used in~\cite{dolezal1995relativistic} to show the effectiveness of an ENO-based scheme in treating both discontinuous and smooth features as they were physically close to each other. The initial data for this problem is taken as,
	\[
	(\rho, v_1, p) = \begin{cases}
		(5, 0, 50) \quad &\text{if}\ x<0.5\\
		\big(2 + 0.3\sin(50 x), 0, 5\big) \quad &\text{if}\ x>0.5
	\end{cases}
	\]
	with outflow boundaries at $x=0,1$. The results of the simulations at time $t=0.4$ with $N=4$ and $500$ cells are shown in Figure~\ref{fig:1DDP}.
	\begin{figure}
		\centering
		\hspace{0.3mm}\begin{subfigure}{0.32\textwidth}
			\includegraphics[width=\linewidth]{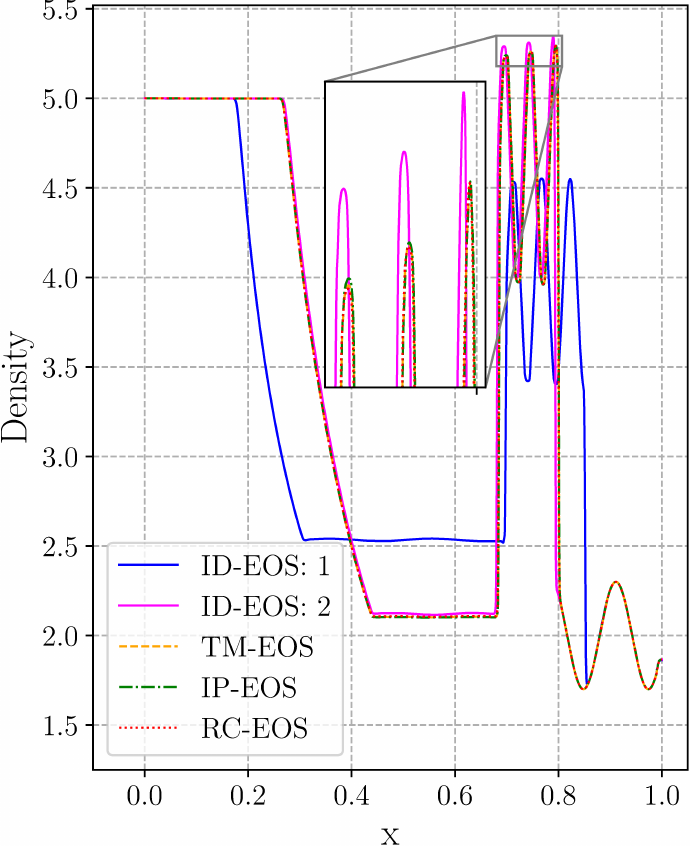}
			\caption{Density ($\rho$)}
		\end{subfigure}
		\begin{subfigure}{0.32\textwidth}
			\includegraphics[width=\linewidth]{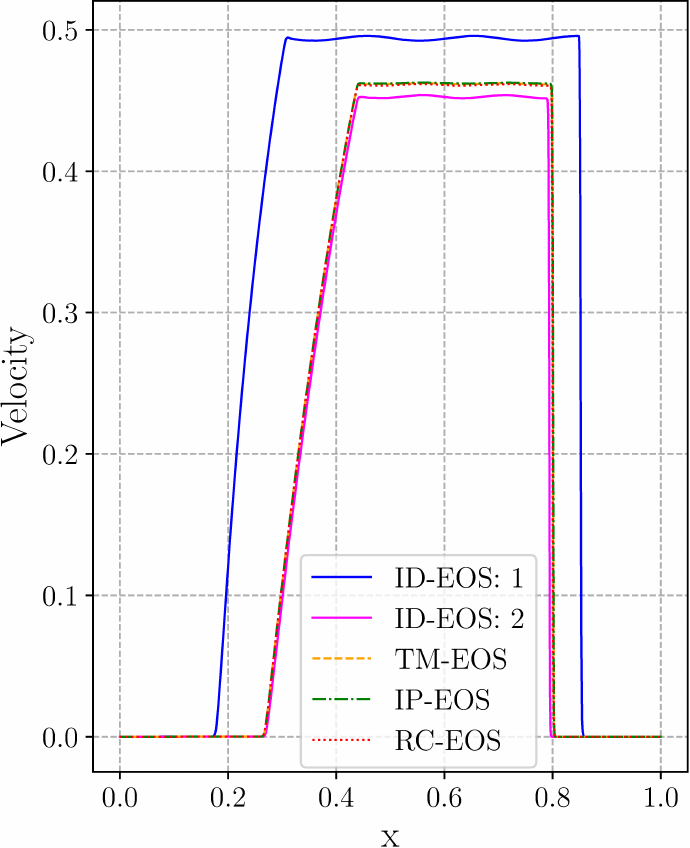}
			\caption{Velocity ($v_1$)}
		\end{subfigure}
		\begin{subfigure}{0.32\textwidth}
			\includegraphics[width=\linewidth]{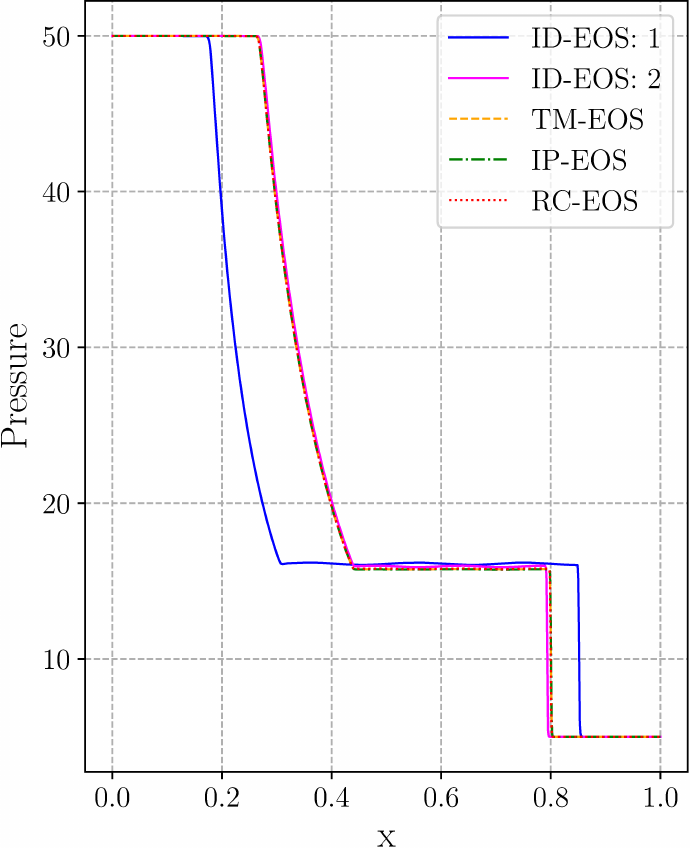}
			\caption{Pressure ($p$)}
		\end{subfigure}
		\caption{1-D density perturbation problem: Plot with $500$ cells and $N=4$. ID-EOS: 1 and ID-EOS: 2 refer to ID-EOS with $\gamma=\frac{5}{3}$ and $\gamma=\frac{4}{3}$ respectively.}
		\label{fig:1DDP}
	\end{figure}
	
	Here we can see that our scheme can capture the smooth sinusoidal profile in the solution effectively, along with the rarefaction, contact discontinuity, and shock waves for all the equations of state. Moreover the solution using ID-EOS with $\gamma = \frac{4}{3}$ approximates the solutions using other equations of state more closely compared to $\gamma = \frac{5}{3}$, because of the relativistic temperature, $\frac{p}{\rho}>1$,  throughout the domain~\cite{ryu2006equation}. Here too, we see the equivalence of the solution graphs using TM-EOS, IP-EOS, and RC-EOS.
	
	\subsubsection{1-D blast wave problem}
	In this problem, the collision of two strong blast waves takes place, and a very thin structure gets formed after a finite time. This test case was used several times in literature \cite{marti1996extension,yang2011direct,wu2015high,wu2016physical} to check the efficiency of the schemes, as it has a very narrow structure to capture and strong shock interaction. The initial condition for this problem is,
	\[
	(\rho, v_1, p) = \begin{cases}
		(1, 0, 10^3) \quad &\text{if}\ x<0.1\\
		(1, 0, 10^{-2}) \quad &\text{if}\ 0.1<x<0.9\\
		(1, 0, 10^2) \quad &\text{if}\ 0.9<x
	\end{cases}
	\]
	where two jump discontinuities are taken at $x=0.1,0.9$ in the pressure profile. We run the simulations with different equations of state in the domain $[0,1]$ with outflow boundaries and taking $5000$ cells, $N=4$, and $l_s = 0.8$. The specific heat ratio $\gamma$ for the ID-EOS is taken as $1.4$ following~\cite{wu2016physical}, and the results are presented in Figure~\ref{fig:1DWublast} at time $t=0.43$.
	\begin{figure}
		\centering
		\hspace{0.3mm}\begin{subfigure}{0.32\textwidth}
			\includegraphics[width=\linewidth]{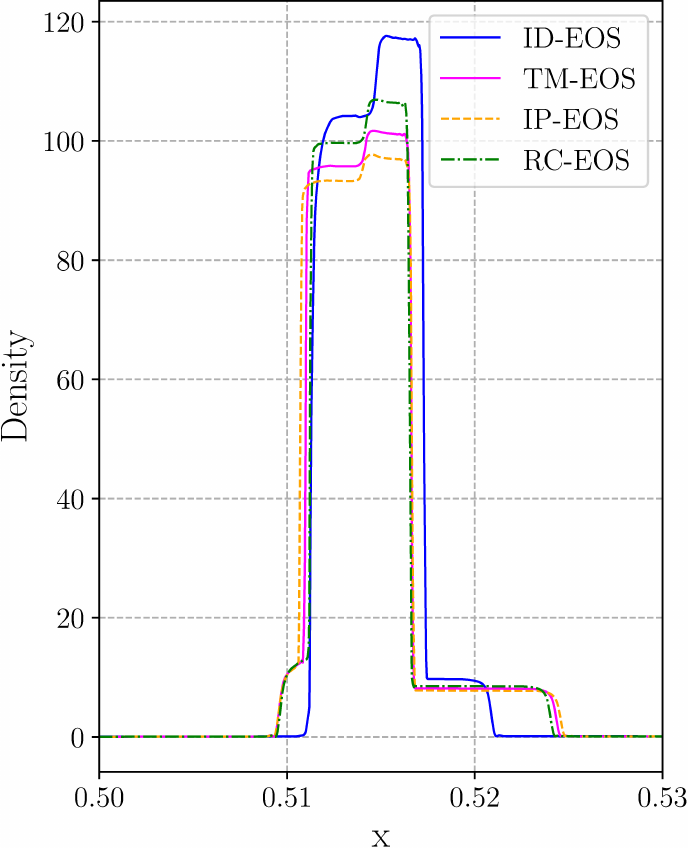}
			\caption{Density ($\rho$) plot zooming in $[0.5, 0.53]$.\\}
		\end{subfigure}
		\begin{subfigure}{0.32\textwidth}
			\includegraphics[width=\linewidth]{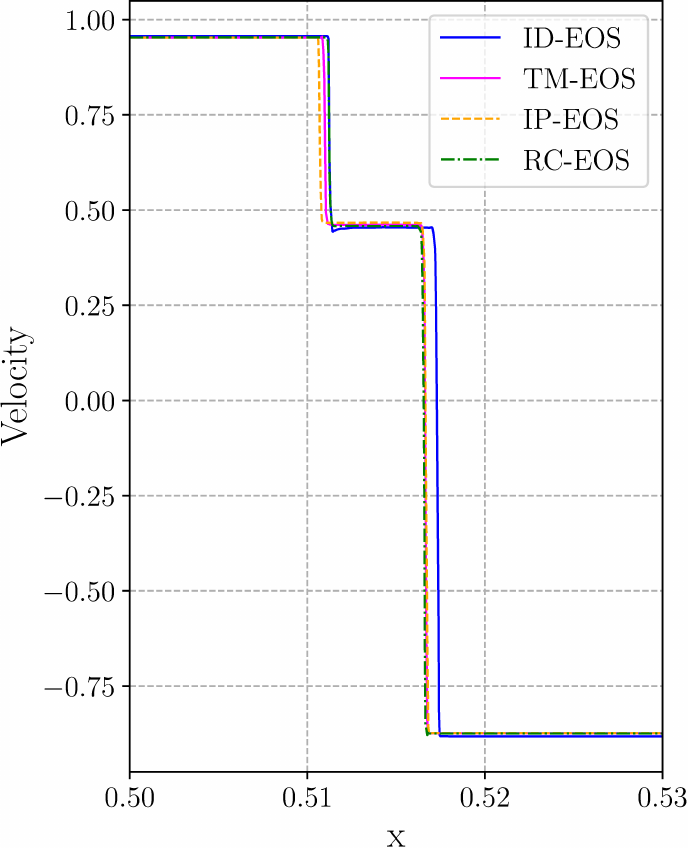}
			\caption{Velocity ($v_1$) plot zooming in $[0.5, 0.53]$.}
		\end{subfigure}
		\begin{subfigure}{0.32\textwidth}
			\includegraphics[width=\linewidth]{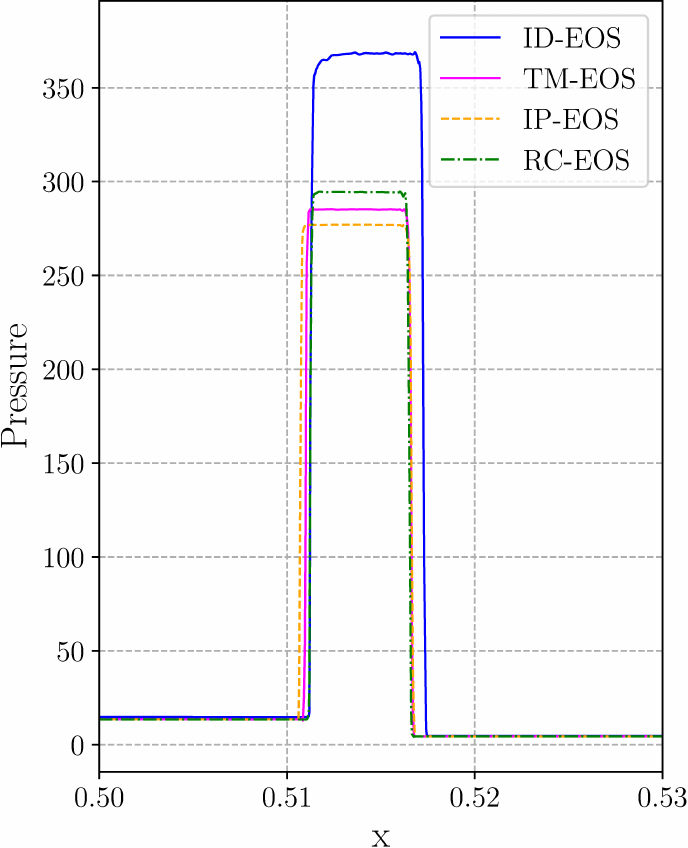}
			\caption{Pressure ($p$) plot zooming in $[0.5, 0.53]$.}
		\end{subfigure}
		\caption{1-D blast wave problem: Plot with $5000$ cells and $N=4$.}
		\label{fig:1DWublast}
	\end{figure}
	
	We observe that our scheme can capture all the structures effectively with all the equations of state. The solutions with TM-EOS, IP-EOS, and RC-EOS are close to each other in terms of the waves that arise, compared to the ID-EOS, but have a noticeable difference in the zoomed views for $x \in [0.5,0.53]$.
	
	\subsection{Two dimensional experiments}
	For the two dimensional test cases, we take the time step as in equation~(42) of~\cite{my_paper} with the safety factor $l_s = 0.95$ for all the test cases, unless mentioned otherwise.
	Like in the one dimensional case, we will perform a grid-convergence study of the scheme with a two dimensional test with known exact solution before attempting more challenging problems.
	
	\subsubsection{Accuracy test}
	Here we have extended the 1-D accuracy test of Section~\ref{sec:1Daccuracy} to two spatial dimensions. The initial condition of the fluid is taken as,
	\[
	(\rho, v_1, v_2, p) = \left(1 + 0.999\sin\left(2 \pi (x+y)\right), \frac{0.99}{\sqrt{2}}, \frac{0.99}{\sqrt{2}}, 0.01\right).
	\]
	in the domain $[0,1]\times [0,1]$ with periodic boundaries. The wave in the initial density profile  propagates diagonally in the domain with time, and the exact solution is given by,
	\[
	(\rho, v_1, v_2, p) = \Bigg(1 + 0.999\sin\left(2 \pi \left(x+y- \frac{0.99t}{\sqrt{2}}\right)\right), \frac{0.99}{\sqrt{2}}, \frac{0.99}{\sqrt{2}}, 0.01\Bigg).
	\]
	We run the simulations till the time $t=0.2$ and present the grid convergence study in Tables~\ref{table:2DN3_ID5/3}-\ref{table:2DN4_RC} with different equations of state.
	\setlength{\tabcolsep}{5pt}
	\begin{table}[!htbp]
		\centering
		\caption{Numerical results for the fluid density ($\rho$) with $N=3$ using ID-EOS~\eqref{eq: ID_eos} with $\gamma = \frac{5}{3}$.}
		\label{table:2DN3_ID5/3}
		
		\begin{tabular}{lllllll}
			\hline\noalign{\smallskip}
			Cells & $L^1$ error  & $L^1$ Order     & $L^2$ error   & $L^2$ Order     & $L^{\infty}$ error & $L^{\infty}$ Order\\
			\noalign{\smallskip}\hline\noalign{\smallskip}
			8 X  8    &    1.20973e-03     &          -             &       1.72898e-03       &        -           &        6.68491e-03       &          -    \\ 
			16 X 16    &    9.38741e-05     &        3.68781         &       1.74901e-04       &     3.30531        &        8.11048e-04       &       3.04305 \\ 
			32 X 32    &    2.14104e-06     &        5.45434         &       2.85220e-06       &     5.93832        &        1.07594e-05       &       6.23612 \\ 
			64 X 64    &    1.32624e-07     &        4.01289         &       1.76838e-07       &     4.01158        &        6.14558e-07       &       4.12990 \\ 
			128 X 128    &    8.27857e-09     &        4.00182         &       1.10521e-08       &     4.00004        &        4.02156e-08       &       3.93372 \\ 
			256 X 256    &    5.17396e-10     &        4.00004         &       6.90960e-10       &     3.99957        &        2.52713e-09       &       3.99218 \\ 
			\noalign{\smallskip}\hline
		\end{tabular}
		\vspace{0.5em}
		\centering
		\caption{Numerical results for the fluid density ($\rho$) with $N=4$ using ID-EOS~\eqref{eq: ID_eos} with $\gamma = \frac{5}{3}$.}
		\label{table:2DN4_ID5/3}
		
		\begin{tabular}{lllllll}
			\hline\noalign{\smallskip}
			Cells & $L^1$ error  & $L^1$ Order     & $L^2$ error   & $L^2$ Order     & $L^{\infty}$ error & $L^{\infty}$ Order\\
			\noalign{\smallskip}\hline\noalign{\smallskip}
			8 X  8    &    3.47437e-04     &          -             &       6.11209e-04       &        -           &        2.56687e-03       &          -    \\ 
			16 X 16    &    6.97557e-07     &        8.96022         &       9.20406e-07       &     9.37518        &        8.73027e-07       &       11.52170 \\ 
			32 X 32    &    2.23360e-08     &        4.96487         &       3.01785e-08       &     4.93068        &        1.10961e-08       &       6.29791 \\ 
			64 X 64    &    6.64866e-10     &        5.07016         &       8.81046e-10       &     5.09816        &        6.60429e-10       &       4.07050 \\ 
			128 X 128    &    2.39666e-11     &        4.79396         &       3.10542e-11       &     4.82636        &        2.29007e-11       &       4.84994 \\  
			\noalign{\smallskip}\hline
		\end{tabular}
		\vspace{0.5em}
		\centering
		\caption{Numerical results for the fluid density ($\rho$) with $N=3$ using ID-EOS~\eqref{eq: ID_eos} with $\gamma = \frac{4}{3}$.}
		\label{table:2DN3_ID4/3}
		
		\begin{tabular}{lllllll}
			\hline\noalign{\smallskip}
			Cells & $L^1$ error  & $L^1$ Order     & $L^2$ error   & $L^2$ Order     & $L^{\infty}$ error & $L^{\infty}$ Order\\
			\noalign{\smallskip}\hline\noalign{\smallskip}
			8 X  8    &    1.20701e-03     &          -             &       1.72004e-03       &        -           &        6.66764e-03       &          -    \\ 
			16 X 16    &    1.07539e-04     &        3.48851         &       2.07389e-04       &     3.05203        &        9.49574e-04       &       2.81183 \\ 
			32 X 32    &    2.13061e-06     &        5.65744         &       2.83217e-06       &     6.19429        &        1.06309e-05       &       6.48095 \\ 
			64 X 64    &    1.33558e-07     &        3.99573         &       1.77836e-07       &     3.99329        &        6.27284e-07       &       4.08300 \\ 
			128 X 128    &    8.24365e-09     &        4.01804         &       1.09660e-08       &     4.01944        &        3.85680e-08       &       4.02364 \\ 
			256 X 256    &    5.22042e-10     &        3.98105         &       6.96164e-10       &     3.97746        &        2.53310e-09       &       3.92843 \\  
			\noalign{\smallskip}\hline
		\end{tabular}
		\vspace{0.5em}
		\centering
		\caption{Numerical results for the fluid density ($\rho$) with $N=4$ using ID-EOS~\eqref{eq: ID_eos} with $\gamma = \frac{4}{3}$.}
		\label{table:2DN4_ID4/3}
		
		\begin{tabular}{lllllll}
			\hline\noalign{\smallskip}
			Cells & $L^1$ error  & $L^1$ Order     & $L^2$ error   & $L^2$ Order     & $L^{\infty}$ error & $L^{\infty}$ Order\\
			\noalign{\smallskip}\hline\noalign{\smallskip}
			8 X  8    &    3.47675e-04     &          -             &       6.10396e-04       &        -           &        2.60346e-03       &          -    \\ 
			16 X 16    &    6.90562e-07     &        8.97575         &       9.10668e-07       &     9.38861        &        8.39637e-07       &       11.59838 \\ 
			32 X 32    &    2.22796e-08     &        4.95398         &       3.01035e-08       &     4.91892        &        1.10837e-08       &       6.24325 \\ 
			64 X 64    &    6.63455e-10     &        5.06958         &       8.77410e-10       &     5.10054        &        6.47817e-10       &       4.09671 \\ 
			128 X 128    &    2.19940e-11     &        4.91482         &       2.87276e-11       &     4.93274        &        2.23230e-11       &       4.85898 \\
			\noalign{\smallskip}\hline
		\end{tabular}
		\vspace{0.5em}
		\centering
		\caption{Numerical results for the fluid density ($\rho$) with $N=3$ using TM-EOS~\eqref{eq: TM_eos}.}
		\label{table:2DN3_TM}
		
		\begin{tabular}{lllllll}
			\hline\noalign{\smallskip}
			Cells & $L^1$ error  & $L^1$ Order     & $L^2$ error   & $L^2$ Order     & $L^{\infty}$ error & $L^{\infty}$ Order\\
			\noalign{\smallskip}\hline\noalign{\smallskip}
			8 X  8    &    1.48610e-03     &          -             &       2.07644e-03       &        -           &        6.48470e-03       &          -    \\ 
			16 X 16    &    1.29184e-04     &        3.52402         &       2.37612e-04       &     3.12744        &        1.07060e-03       &       2.59862 \\ 
			32 X 32    &    1.27860e-05     &        3.33680         &       2.91775e-05       &     3.02567        &        1.33807e-04       &       3.00019 \\ 
			64 X 64    &    8.93192e-07     &        3.83945         &       2.82710e-06       &     3.36747        &        1.95143e-05       &       2.77755 \\ 
			128 X 128    &    4.55783e-08     &        4.29255         &       1.81141e-07       &     3.96414        &        1.80504e-06       &       3.43443 \\ 
			256 X 256    &    2.61695e-09     &        4.12239         &       9.86811e-09       &     4.19819        &        1.01499e-07       &       4.15249 \\ 
			\noalign{\smallskip}\hline
		\end{tabular}
	\end{table}
	\begin{table}[!htbp]
		\centering
		\caption{Numerical results for the fluid density ($\rho$) with $N=4$ using TM-EOS~\eqref{eq: TM_eos}.}
		\label{table:2DN4_TM}
		
		\begin{tabular}{lllllll}
			\hline\noalign{\smallskip}
			Cells & $L^1$ error  & $L^1$ Order     & $L^2$ error   & $L^2$ Order     & $L^{\infty}$ error & $L^{\infty}$ Order\\
			\noalign{\smallskip}\hline\noalign{\smallskip}
			8 X  8    &    4.64078e-04     &          -             &       7.10855e-04       &        -           &        1.47841e-03       &          -    \\ 
			16 X 16    &    4.14113e-05     &        3.48627         &       8.57190e-05       &     3.05187        &        3.73919e-04       &       1.98325 \\ 
			32 X 32    &    2.69768e-06     &        3.94023         &       7.42961e-06       &     3.52826        &        4.67069e-05       &       3.00102 \\ 
			64 X 64    &    8.88519e-08     &        4.92417         &       3.62725e-07       &     4.35634        &        3.18951e-06       &       3.87223 \\ 
			128 X 128    &    2.37090e-09     &        5.22789         &       1.13280e-08       &     5.00092        &        1.08212e-07       &       4.88140 \\ 
			\noalign{\smallskip}\hline
		\end{tabular}
		\vspace{0.5em}
		\centering
		\caption{Numerical results for the fluid density ($\rho$) with $N=3$ using IP-EOS~\eqref{eq: IP_eos}.}
		\label{table:2DN3_IP}
		
		\begin{tabular}{lllllll}
			\hline\noalign{\smallskip}
			Cells & $L^1$ error  & $L^1$ Order     & $L^2$ error   & $L^2$ Order     & $L^{\infty}$ error & $L^{\infty}$ Order\\
			\noalign{\smallskip}\hline\noalign{\smallskip}
			8 X  8    &    1.38954e-03     &          -             &       2.01461e-03       &        -           &        7.25200e-03       &          -    \\ 
			16 X 16    &    1.43020e-04     &        3.28032         &       2.64917e-04       &     2.92689        &        1.25114e-03       &       2.53514 \\ 
			32 X 32    &    1.25677e-05     &        3.50843         &       2.90901e-05       &     3.18694        &        1.40756e-04       &       3.15197 \\ 
			64 X 64    &    7.86051e-07     &        3.99895         &       2.44765e-06       &     3.57106        &        1.22787e-05       &       3.51897 \\ 
			128 X 128    &    4.02499e-08     &        4.28757         &       1.44383e-07       &     4.08343        &        9.00698e-07       &       3.76897 \\ 
			256 X 256    &    2.33386e-09     &        4.10820         &       7.93866e-09       &     4.18486        &        7.09017e-08       &       3.66715 \\ 
			512 X 512    &    1.49692e-10     &        3.96265         &       4.72607e-10       &     4.07018        &        4.30933e-09       &       4.04029 \\
			\noalign{\smallskip}\hline
		\end{tabular}
		\vspace{0.5em}
		\centering
		\caption{Numerical results for the fluid density ($\rho$) with $N=4$ using IP-EOS~\eqref{eq: IP_eos}.}
		\label{table:2DN4_IP}
		
		\begin{tabular}{lllllll}
			\hline\noalign{\smallskip}
			Cells & $L^1$ error  & $L^1$ Order     & $L^2$ error   & $L^2$ Order     & $L^{\infty}$ error & $L^{\infty}$ Order\\
			\noalign{\smallskip}\hline\noalign{\smallskip}
			8 X  8    &    5.50223e-04     &          -             &       8.10582e-04       &        -           &        2.09780e-03       &          -    \\ 
			16 X 16    &    4.67000e-05     &        3.55852         &       9.57590e-05       &     3.08148        &        3.84591e-04       &       2.44748 \\ 
			32 X 32    &    2.43152e-06     &        4.26350         &       6.57403e-06       &     3.86456        &        2.19757e-05       &       4.12934 \\ 
			64 X 64    &    6.37992e-08     &        5.25217         &       2.66344e-07       &     4.62542        &        2.02697e-06       &       3.43851 \\ 
			128 X 128    &    1.87842e-09     &        5.08595         &       8.16970e-09       &     5.02686        &        7.69111e-08       &       4.71999 \\ 
			\noalign{\smallskip}\hline
		\end{tabular}
		\vspace{0.5em}
		\centering
		\caption{Numerical results for the fluid density ($\rho$) with $N=3$ using RC-EOS~\eqref{eq: RC_eos}.}
		\label{table:2DN3_RC}
		
		\begin{tabular}{lllllll}
			\hline\noalign{\smallskip}
			Cells & $L^1$ error  & $L^1$ Order     & $L^2$ error   & $L^2$ Order     & $L^{\infty}$ error & $L^{\infty}$ Order\\
			\noalign{\smallskip}\hline\noalign{\smallskip}
			8 X  8    &    1.20185e-03     &          -             &       1.70112e-03       &        -           &        6.57762e-03       &          -    \\ 
			16 X 16    &    1.14457e-04     &        3.39237         &       2.12693e-04       &     2.99964        &        9.49896e-04       &       2.79173 \\ 
			32 X 32    &    9.59948e-06     &        3.57571         &       2.09787e-05       &     3.34177        &        8.40240e-05       &       3.49890 \\ 
			64 X 64    &    5.26568e-07     &        4.18827         &       1.63985e-06       &     3.67729        &        1.37448e-05       &       2.61192 \\ 
			128 X 128    &    2.42613e-08     &        4.43989         &       8.24832e-08       &     4.31332        &        7.57173e-07       &       4.18212 \\ 
			256 X 256    &    1.33361e-09     &        4.18525         &       4.02644e-09       &     4.35652        &        3.30096e-08       &       4.51966 \\
			\noalign{\smallskip}\hline
		\end{tabular}
		\vspace{0.5em}
		\centering
		\caption{Numerical results for the fluid density ($\rho$) with $N=4$ using RC-EOS~\eqref{eq: RC_eos}.}
		\label{table:2DN4_RC}
		
		\begin{tabular}{lllllll}
			\hline\noalign{\smallskip}
			Cells & $L^1$ error  & $L^1$ Order     & $L^2$ error   & $L^2$ Order     & $L^{\infty}$ error & $L^{\infty}$ Order\\
			\noalign{\smallskip}\hline\noalign{\smallskip}
			8 X  8    &    3.97883e-04     &          -             &       6.57848e-04       &        -           &        1.90436e-03       &          -    \\ 
			16 X 16    &    2.89484e-05     &        3.78079         &       5.90687e-05       &     3.47729        &        2.44736e-04       &       2.96001 \\ 
			32 X 32    &    2.09639e-06     &        3.78750         &       5.72200e-06       &     3.36780        &        3.44920e-05       &       2.82689 \\ 
			64 X 64    &    4.18159e-08     &        5.64771         &       1.89596e-07       &     4.91552        &        1.74939e-06       &       4.30134 \\ 
			128 X 128    &    1.00866e-09     &        5.37354         &       4.62049e-09       &     5.35874        &        4.47439e-08       &       5.28902 \\ 
			\noalign{\smallskip}\hline
		\end{tabular}
	\end{table}
	We can observe that the scheme converges with optimal order of accuracy for all the equations of state. In particular, we observe fourth-order of accuracy for $N=3$ and fifth-order of accuracy for $N=4$.

	\subsubsection{2-D Riemann problem 1}
	This is a Riemann problem in two dimensions, with four constant states in the four quadrants of the square $[0,1]\times [0,1]$ at the initial time. This problem is used in the literature several times~\cite{del2002efficient,lucas2004assessment,wu2016physical} to test the numerical schemes, as the initial condition has two contact waves with very high jumps in the transverse velocity from zero to near the speed of light. Specifically, the initial state of the fluid is given by,
	\[
	(\rho, v_1, v_2, p) = \begin{cases}
		(0.1, 0, 0, 0.01) & \text{if}\ x > 0.5,\ y > 0.5\\
		(0.1, 0.99, 0, 1) & \text{if}\ x < 0.5,\ y>0.5\\
		(0.5, 0, 0, 1) & \text{if}\ x < 0.5,\ y < 0.5\\
		(0.1, 0, 0.99, 1) & \text{if}\ x > 0.5,\ y<0.5.
	\end{cases}
	\]
	We run the simulations for this problem with $400\times 400$ cells and $N=4$ with outflow boundaries using different equations of state and present the results in Figure~\ref{fig:2dwurp1.lnden} and Figure~\ref{fig:2dwurp1.lnpres} at time $t=0.4$. For the ID-EOS~\eqref{eq: ID_eos}, we use $\gamma = \frac{5}{3}$, and $\gamma = \frac{4}{3}$ for the simulations.
	
	\begin{figure}[]
		\centering
		\begin{subfigure}{0.31\textwidth}
			\includegraphics[width=\linewidth]{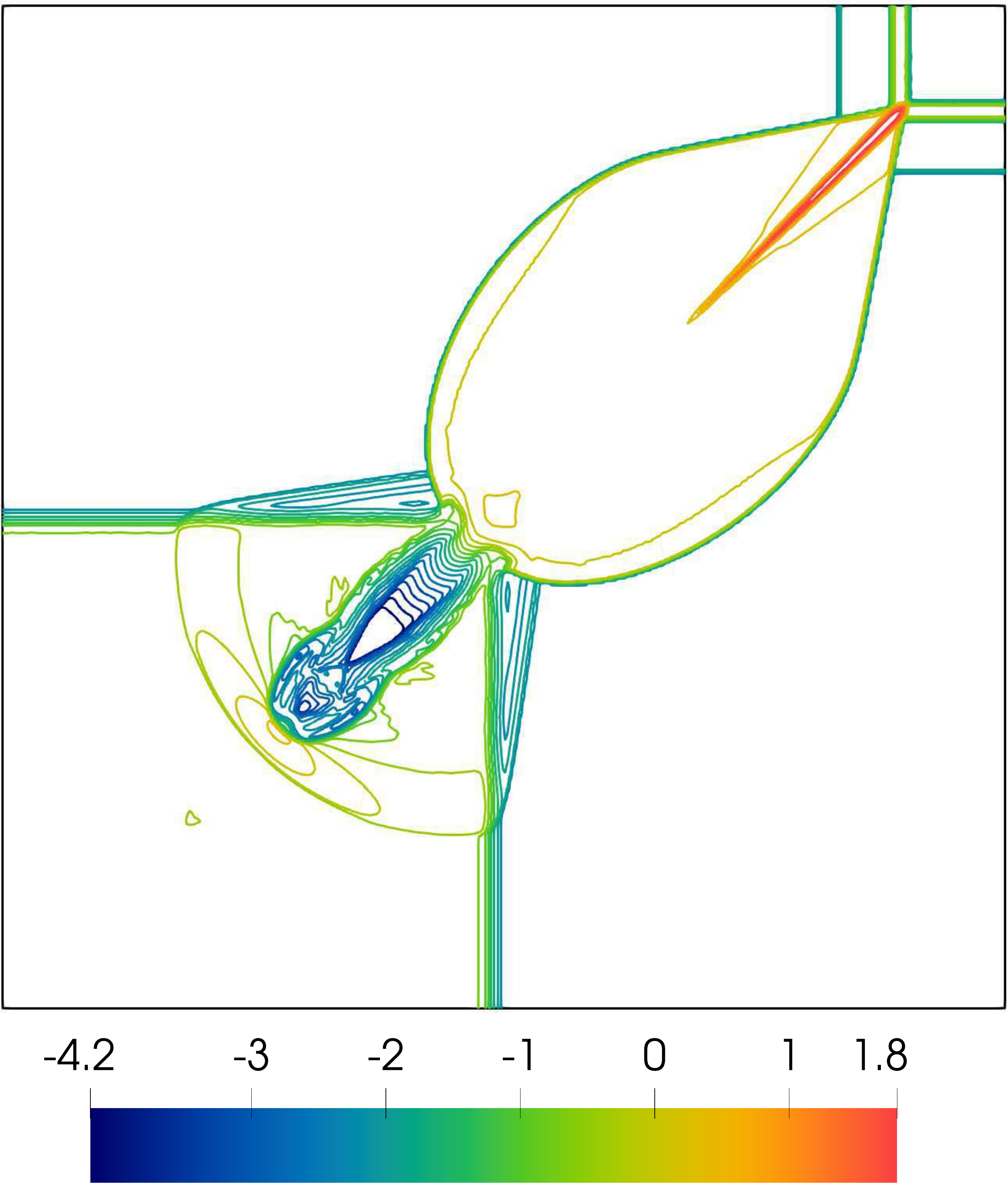}
			\caption{ID-EOS with $\gamma = \frac{5}{3}$: 25 contours in $[-4.2, 1.8]$.}
		\end{subfigure}
		\begin{subfigure}{0.31\textwidth}
			\includegraphics[width=\linewidth]{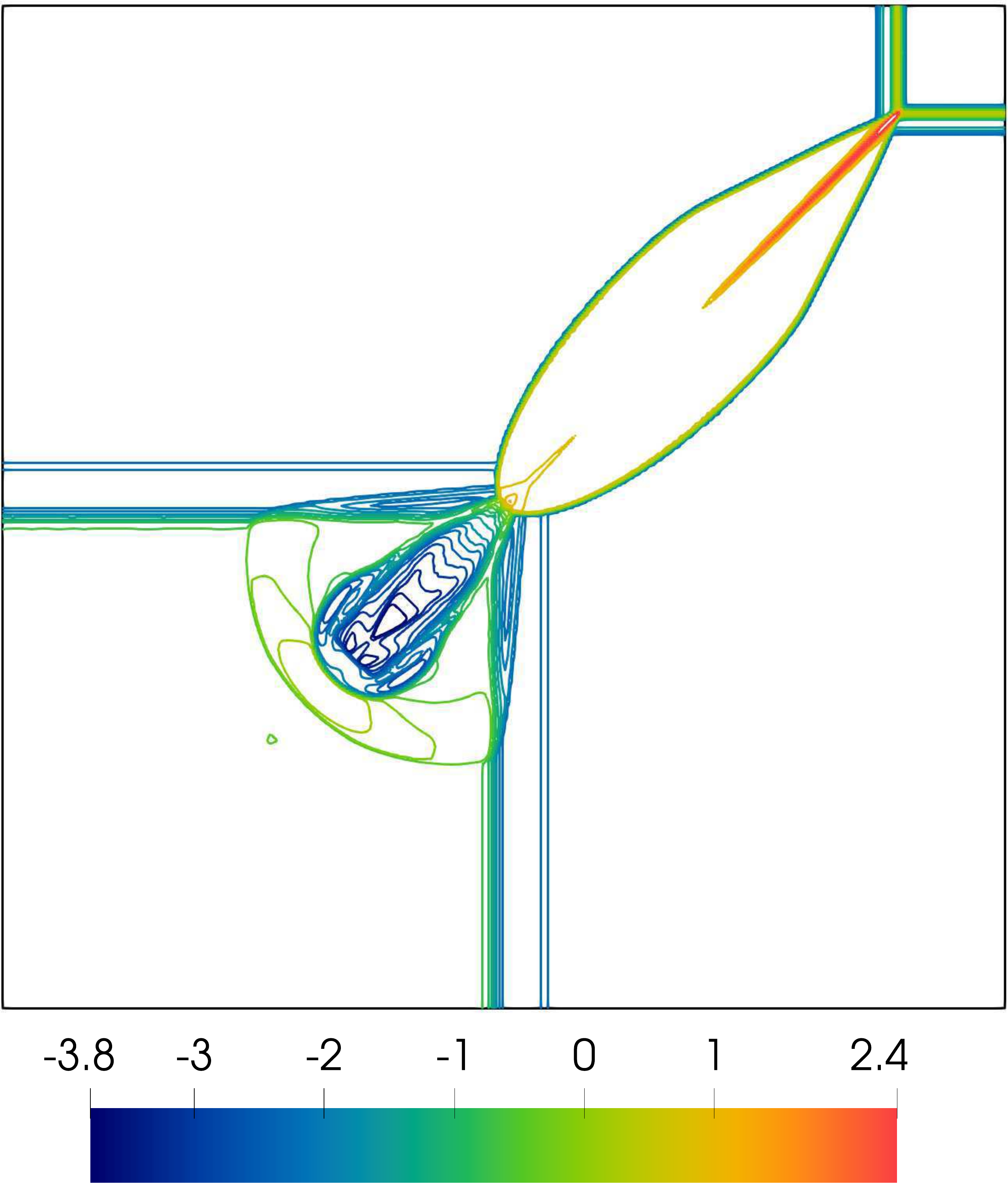}
			\caption{ID-EOS with $\gamma = \frac{4}{3}$: 25 contours in $[-3.8, 2.4]$.}
		\end{subfigure}
		\begin{subfigure}{0.31\textwidth}
			\includegraphics[width=\linewidth]{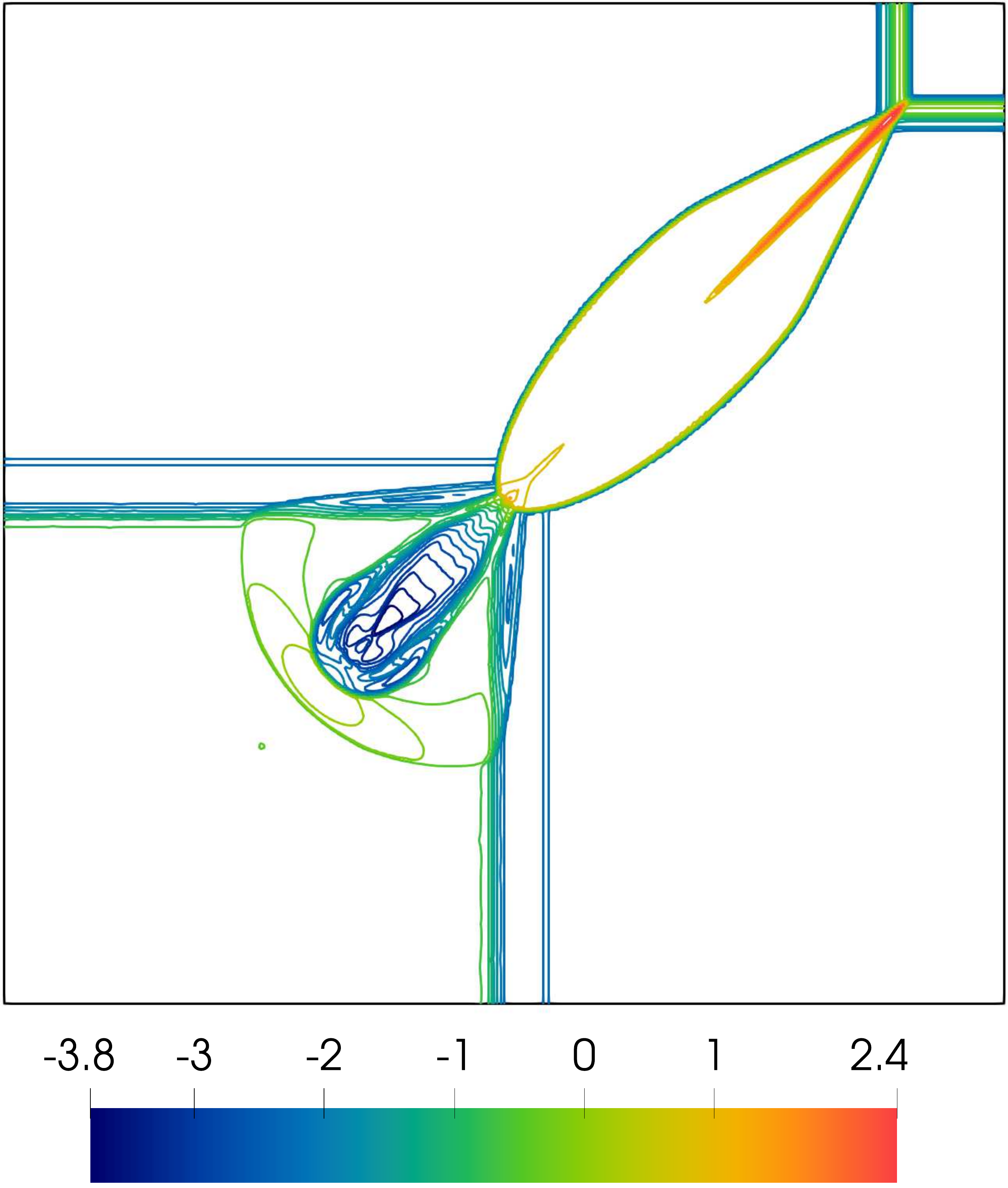}
			\caption{TM-EOS: 25 contours in $[-3.8, 2.4]$.\\}
		\end{subfigure}
		\vspace{0.2cm}
		\begin{subfigure}{0.31\textwidth}
			\includegraphics[width=\linewidth]{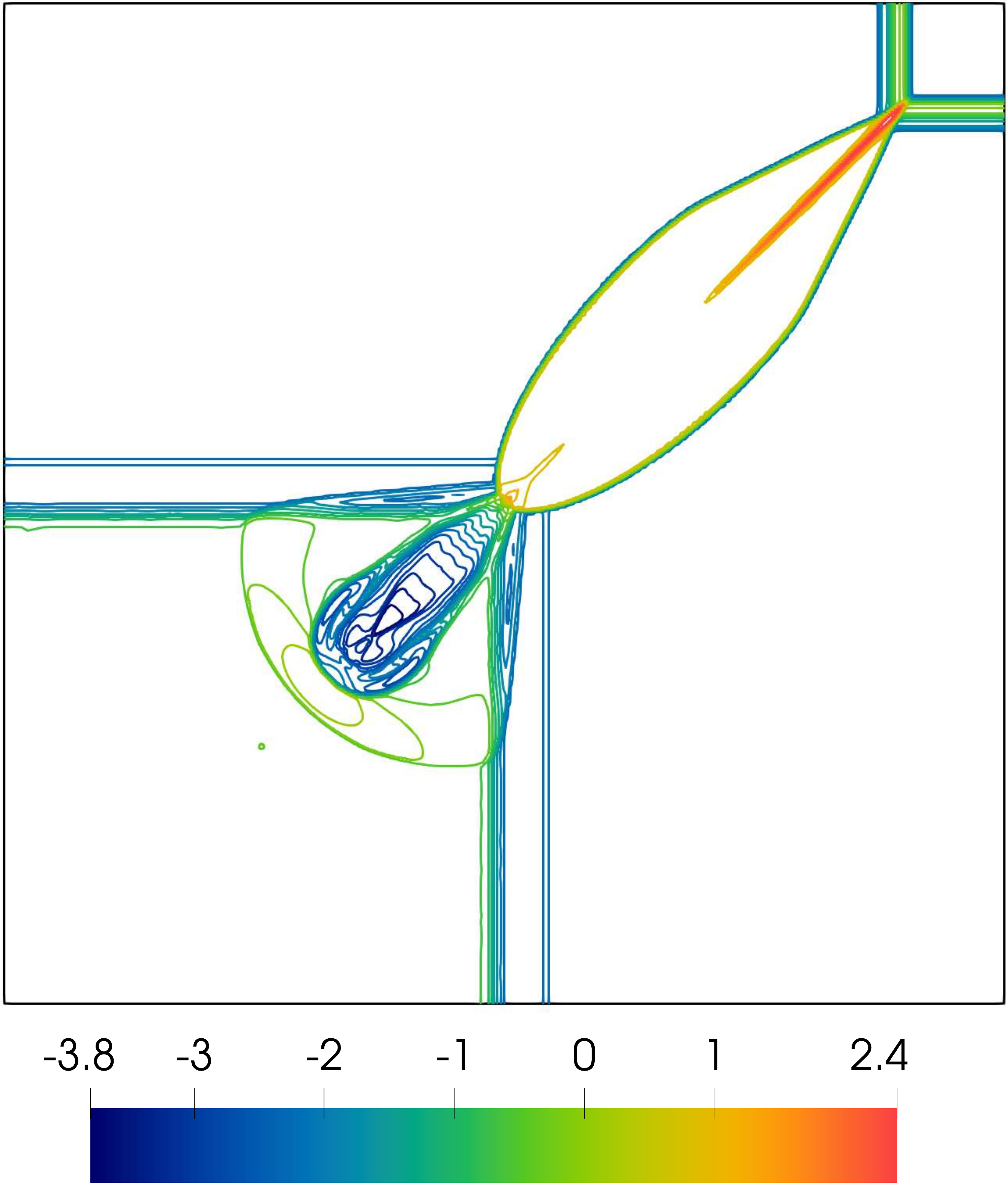}
			\caption{IP-EOS: 25 contours in $[-3.8, 2.4]$.}
		\end{subfigure}
		\begin{subfigure}{0.31\textwidth}
			\includegraphics[width=\linewidth]{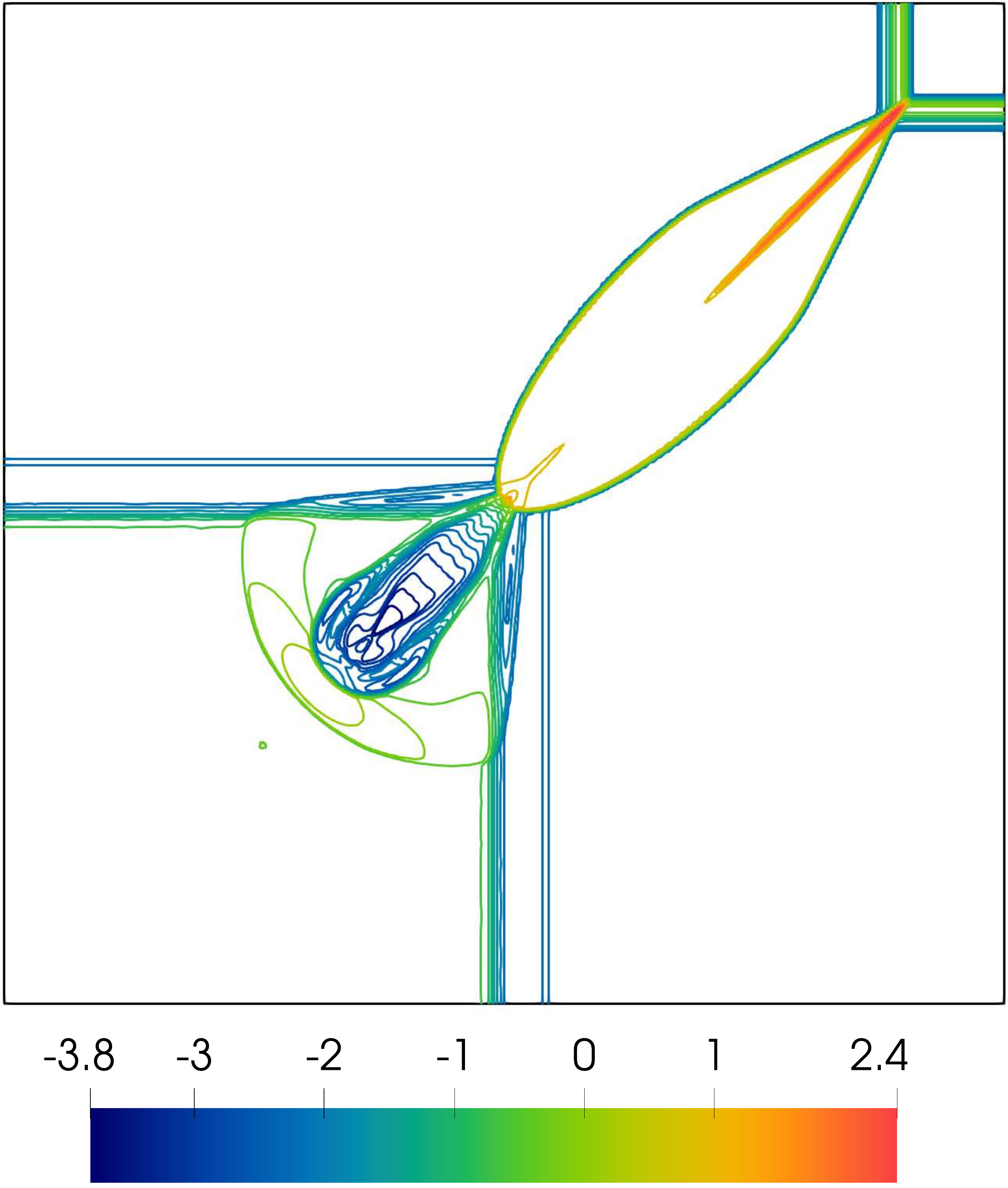}
			\caption{RC-EOS: 25 contours in $[-3.8, 2.4]$.}
		\end{subfigure}
		\begin{subfigure}{0.31\textwidth}
			\includegraphics[width=\linewidth]{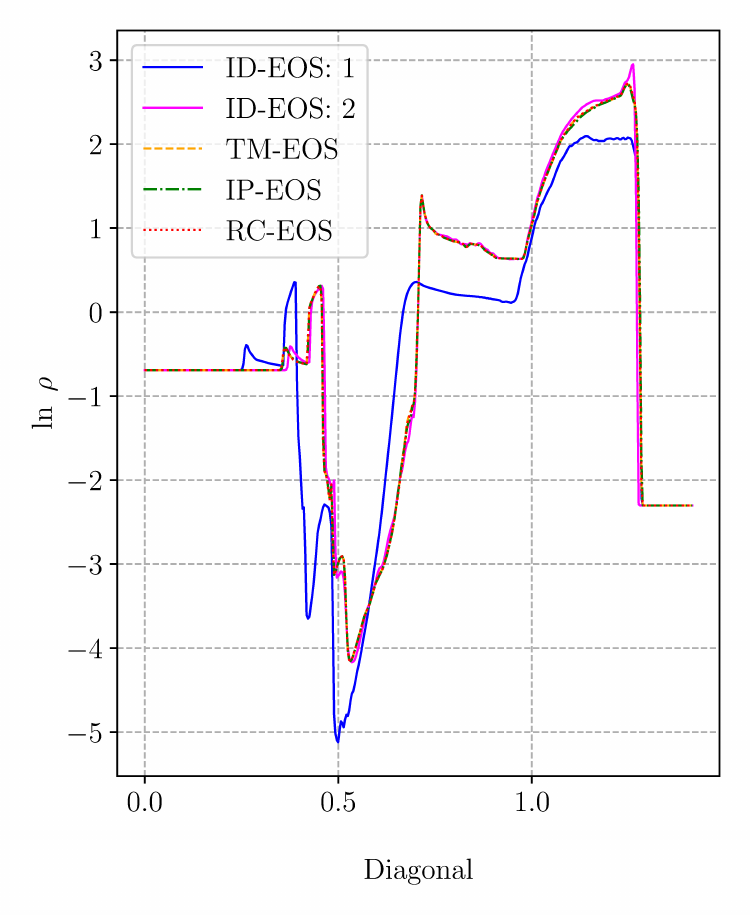}
			\caption{Cut plot from lower-left to upper-right.}
			\label{fig:2dwurp1.cut.lnrho}
		\end{subfigure}
		\vspace{0.2cm}
		\caption{2-D Riemann problem 1: Plot of $\ln \rho$ with $400$ cells and $N=4$.}
		\label{fig:2dwurp1.lnden}
	\end{figure}
	\begin{figure}[]
		\centering
		\begin{subfigure}{0.31\textwidth}
			\includegraphics[width=\linewidth]{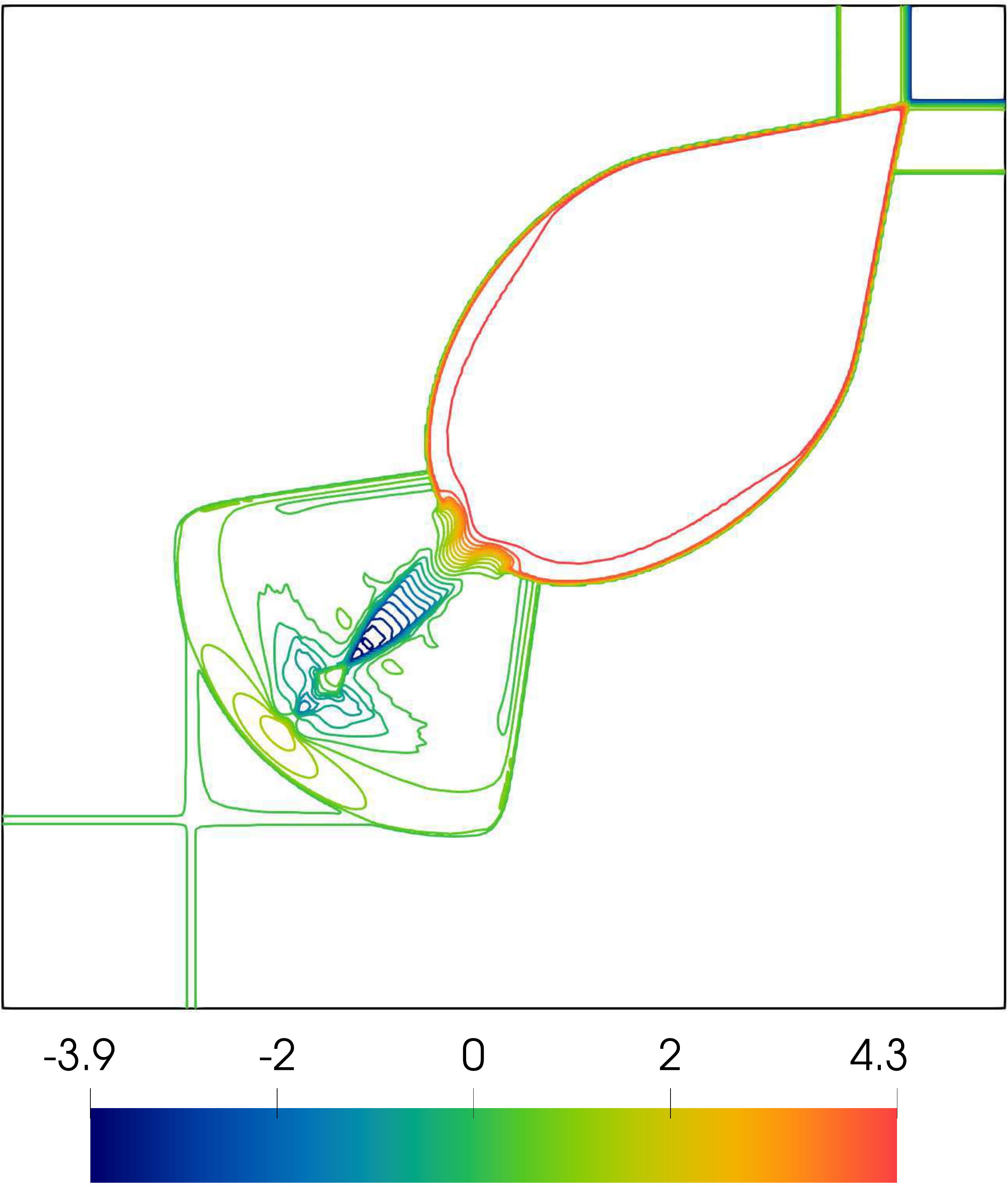}
			\caption{ID-EOS with $\gamma = \frac{5}{3}$: 25 contours in $[-3.9, 4.3]$.}
		\end{subfigure}
		\begin{subfigure}{0.31\textwidth}
			\includegraphics[width=\linewidth]{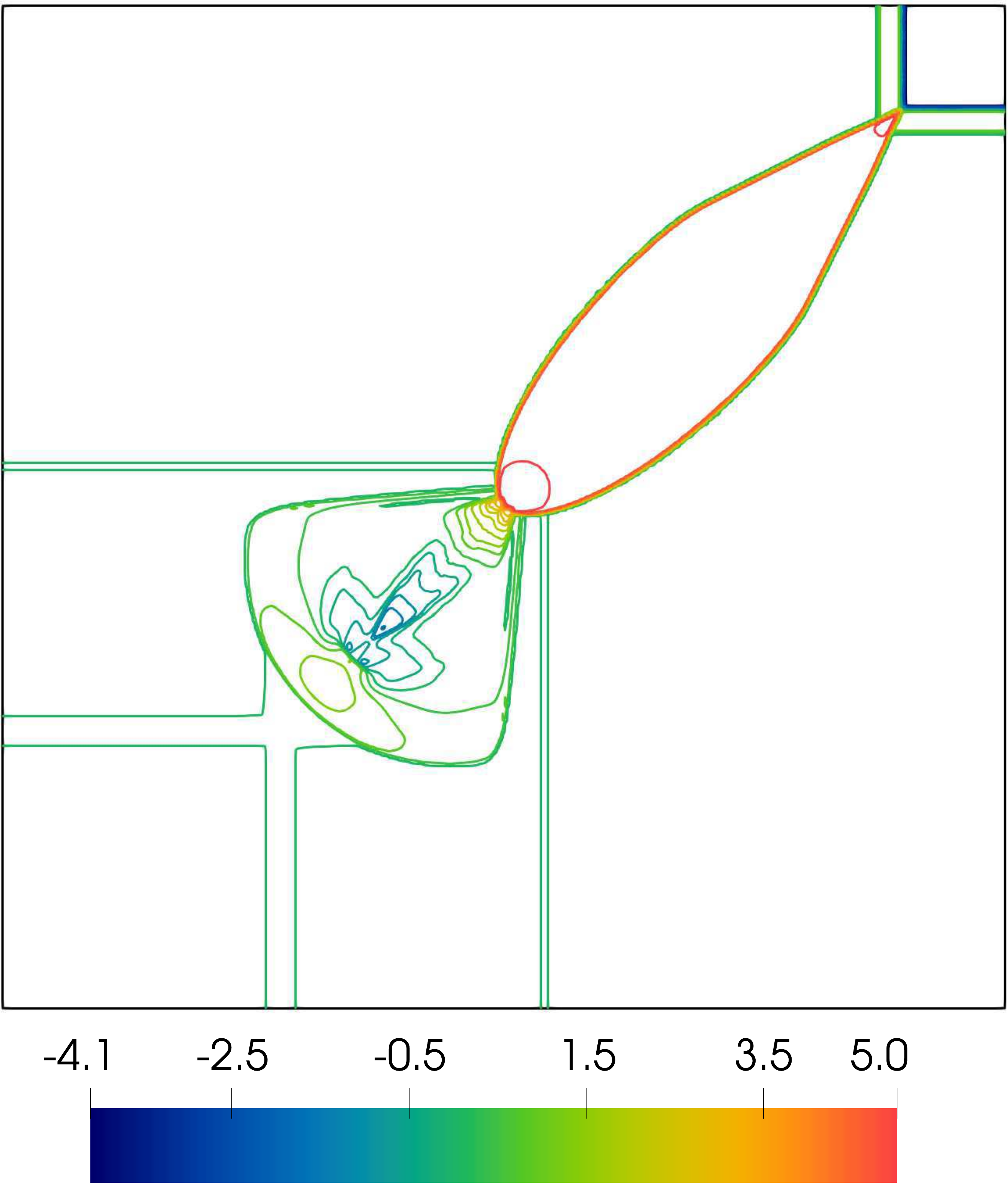}
			\caption{ID-EOS with $\gamma = \frac{4}{3}$: 25 contours in $[-4.1, 5.0]$.}
		\end{subfigure}
		\begin{subfigure}{0.31\textwidth}
			\includegraphics[width=\linewidth]{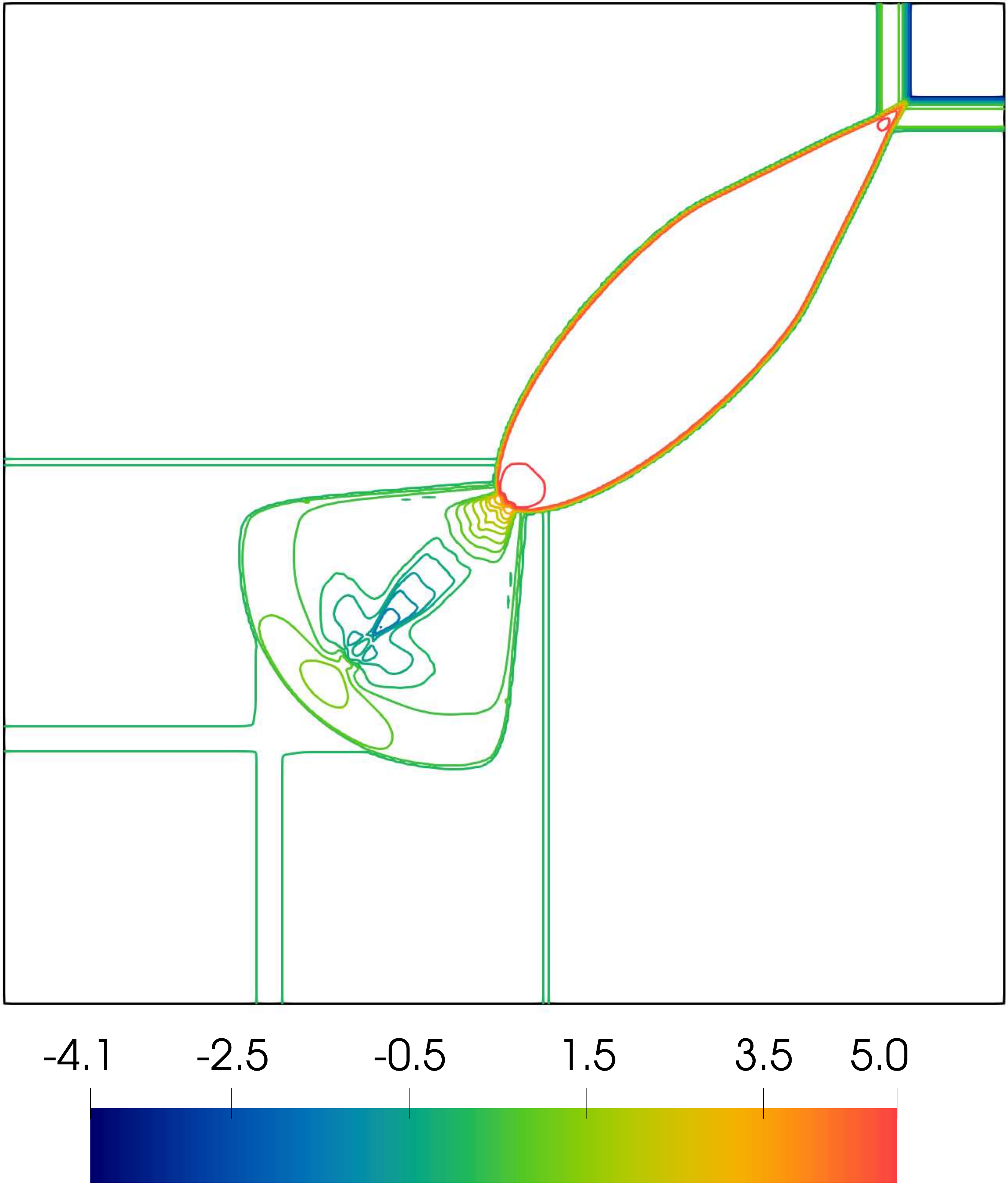}
			\caption{TM-EOS: 25 contours in $[-4.1, 5.0]$.\\}
		\end{subfigure}
		\vspace{0.2cm}
		\begin{subfigure}{0.31\textwidth}
			\includegraphics[width=\linewidth]{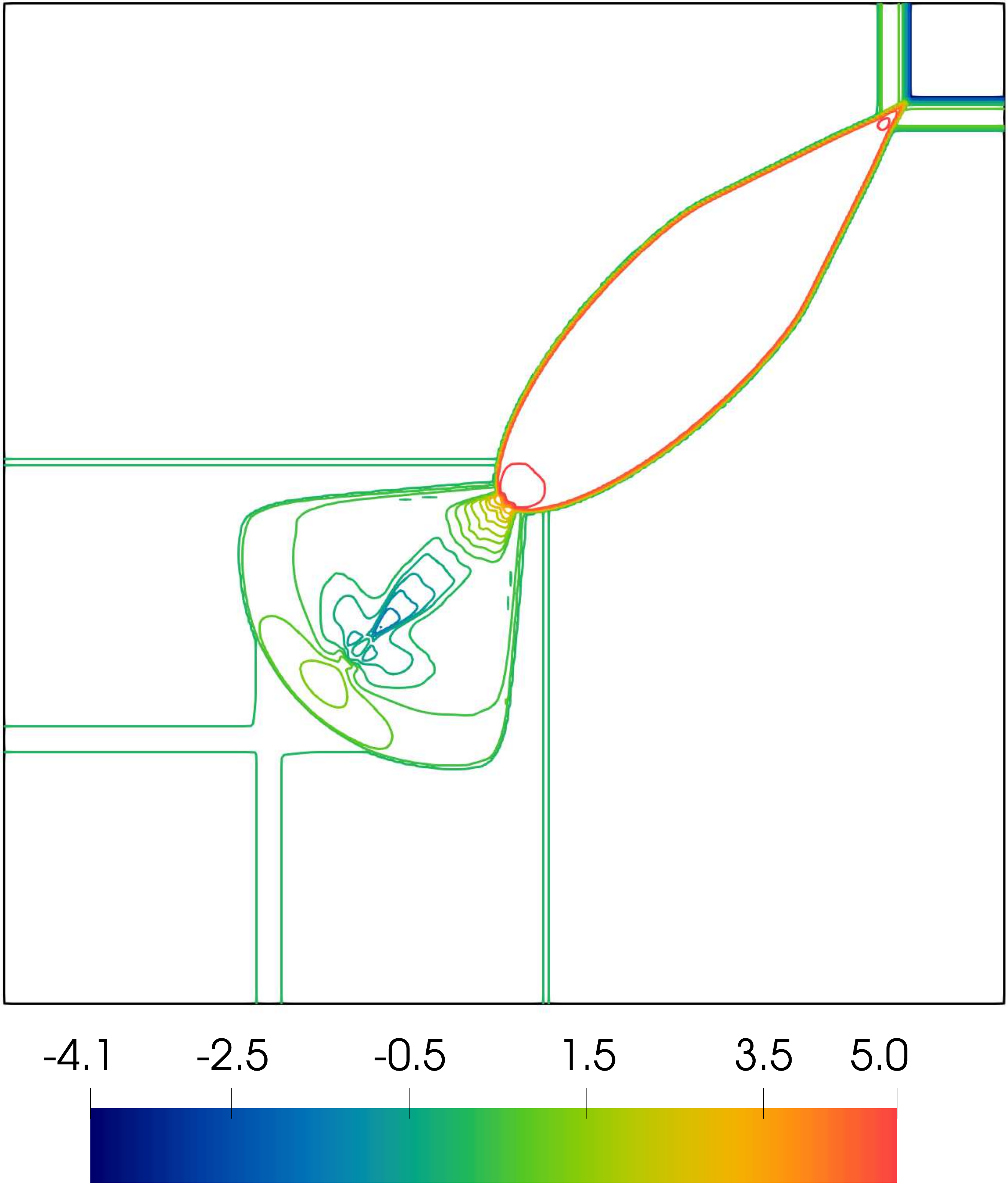}
			\caption{IP-EOS: 25 contours in $[-4.1, 5.0]$.}
		\end{subfigure}
		\begin{subfigure}{0.31\textwidth}
			\includegraphics[width=\linewidth]{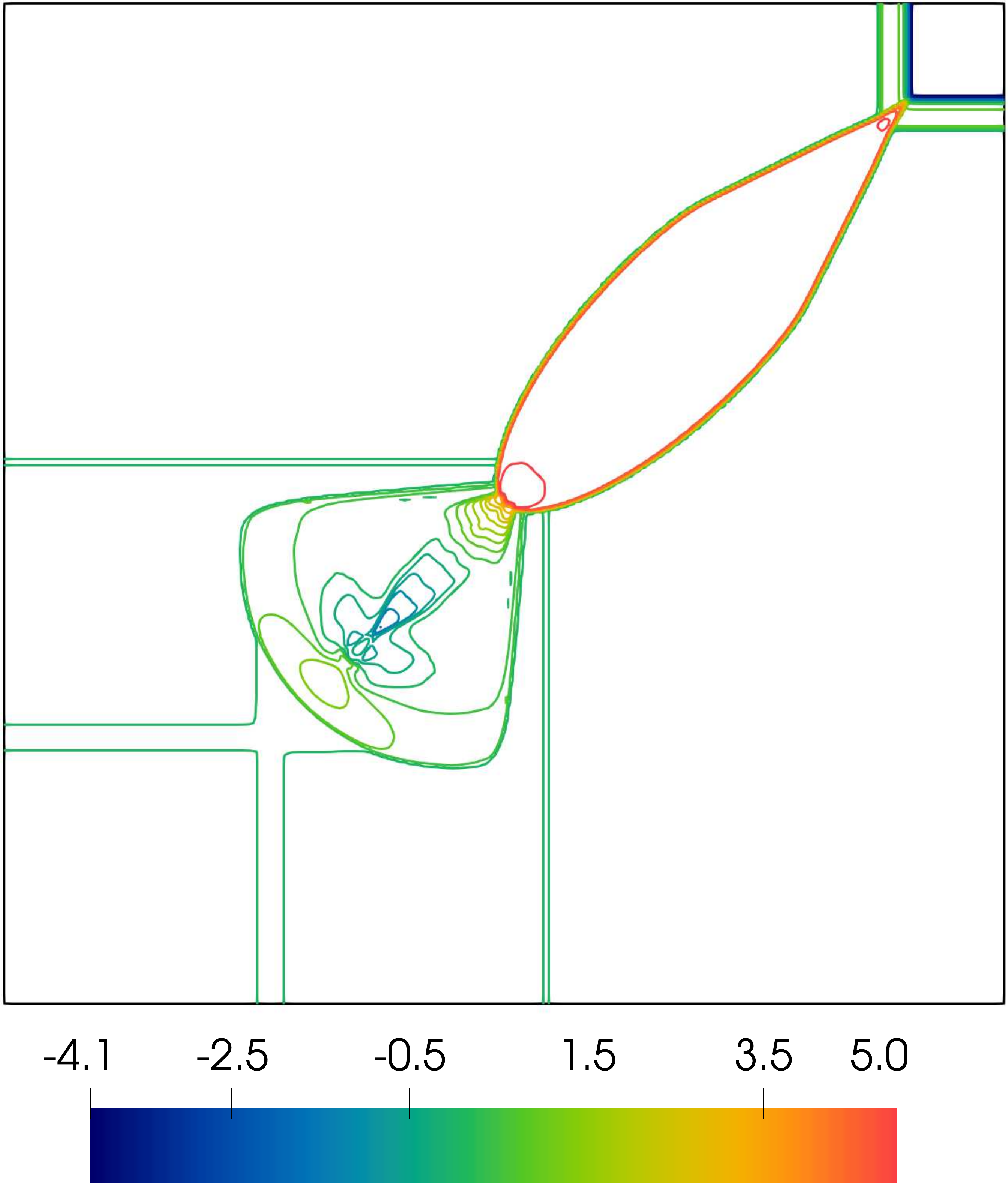}
			\caption{RC-EOS: 25 contours in $[-4.1, 5.0]$.}
		\end{subfigure}
		\begin{subfigure}{0.31\textwidth}
			\includegraphics[width=\linewidth]{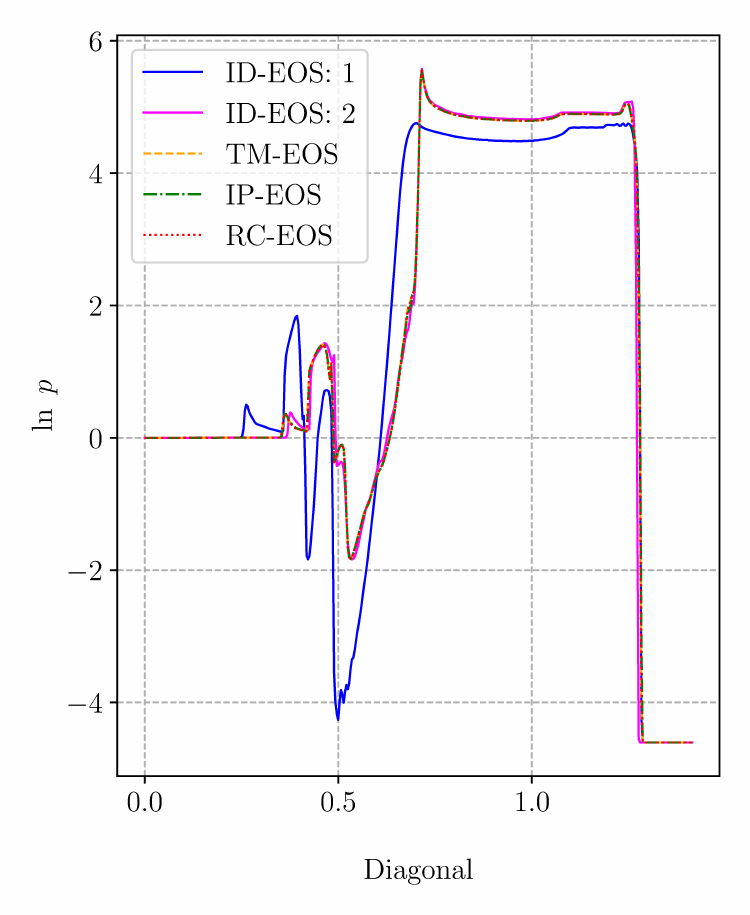}
			\caption{Cut plot from lower-left to upper-right.}
			\label{fig:2dwurp1.cut.lnp}
		\end{subfigure}
		\vspace{0.2cm}
		\caption{2-D Riemann problem 1: Plot of $\ln p$ with $400$ cells and $N=4$.}
		\label{fig:2dwurp1.lnpres}
	\end{figure}
	
	Because of the interactions of the discontinuities, a jet-like structure gets formed in the solution with time, and a mushroom-like structure gets formed in the lower-left quadrant with all the equations of state. The solution also has two curved shock waves, which move with a higher speed when using the ID-EOS with $\gamma =\frac{5}{3}$ compared to the other cases, and the scheme can capture all the waves in the solution effectively. We can also observe from the figure that the solutions with ID-EOS with $\gamma = \frac{4}{3}$, TM-EOS, IP-EOS, and RC-EOS are very similar, hence we have compared the results with cut-plots from the lower-left corner to the upper-right corner of the domain in Figure~\ref{fig:2dwurp1.cut.lnrho} and Figure~\ref{fig:2dwurp1.cut.lnp}. Here and in all the cut-plots hereafter, ID-EOS with $\gamma=\frac{5}{3}$ and $\gamma=\frac{4}{3}$ are denoted as ID-EOS: 1 and ID-EOS: 2, respectively.
	\subsubsection{2-D Riemann problem 2}
	This problem is considered from~\cite{wu2015high}, which also has four constant states in four quadrants of the domain $[0,1]\times[0,1]$ at initial time given by,
	\begin{align*}
		(\rho&, v_1, v_2, p)\\
		&= \begin{cases}
			(0.1, 0, 0, 20) & \text{if}\ x > 0.5,\ y > 0.5\\
			(0.00414329639576, 0.9946418833556542, 0, 0.05) & \text{if}\ x < 0.5,\ y>0.5\\
			(0.01, 0, 0, 0.05) & \text{if}\ x < 0.5,\ y < 0.5\\
			(0.00414329639576, 0,0.9946418833556542, 0.05) & \text{if}\ x > 0.5,\ y<0.5.\\
		\end{cases}
	\end{align*}
	The initial state has two contact discontinuities and two shock waves in it, which interact with each other, forming a mushroom-cloud in the lower-left quadrant. We run the simulations with outflow boundaries and using $400\times 400$ cells and $N=4$ till time $t=0.4$. The results of our simulations are shown in Figure~\ref{fig:2dwurp2.lnden} and Figure~\ref{fig:2dwurp2.lnpres}.
	\begin{figure}[]
		\centering
		\begin{subfigure}{0.31\textwidth}
			\includegraphics[width=\linewidth]{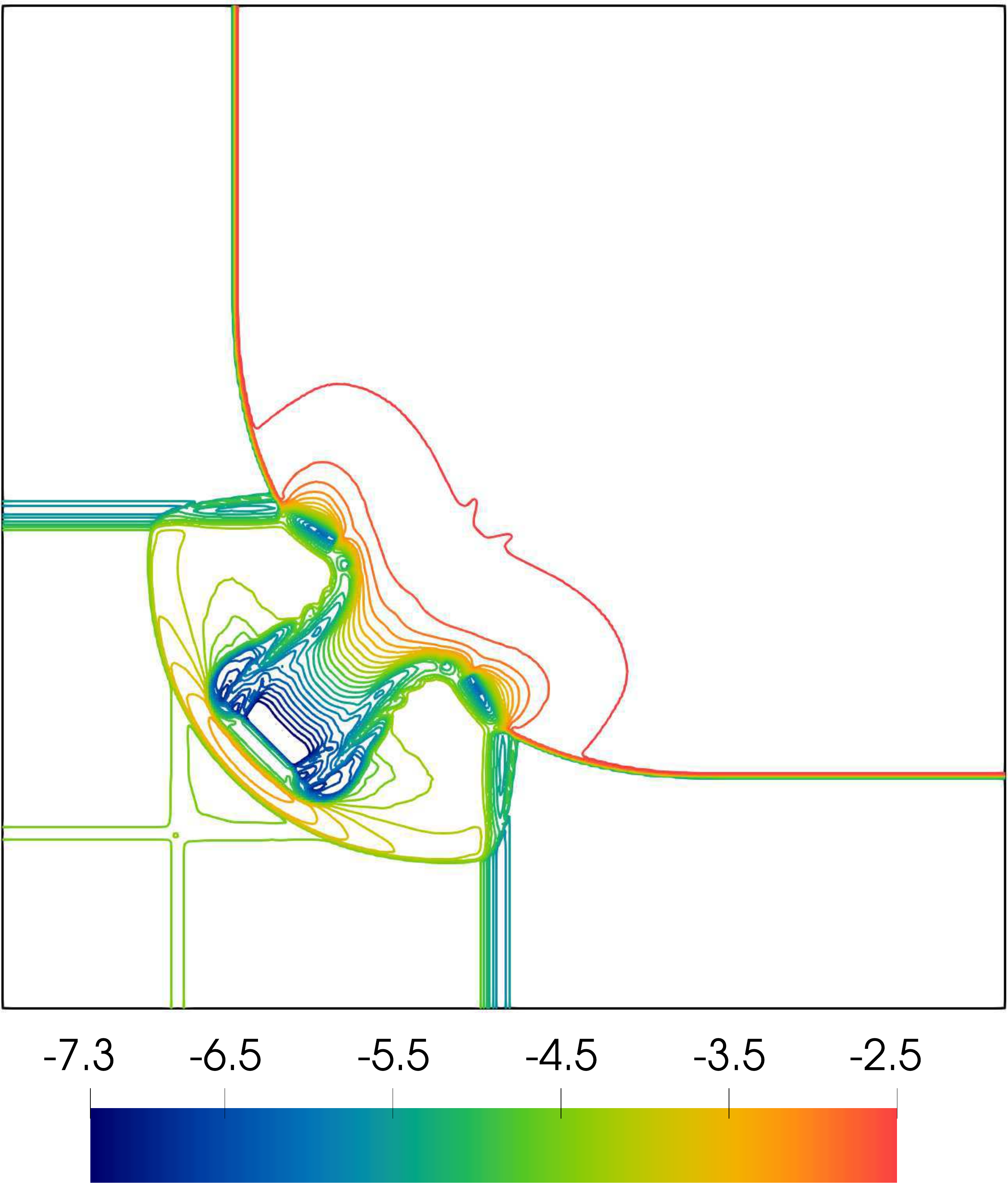}
			\caption{ID-EOS with $\gamma = \frac{5}{3}$: 25 contours in $[-7.3, -2.5]$.}
		\end{subfigure}
		\begin{subfigure}{0.31\textwidth}
			\includegraphics[width=\linewidth]{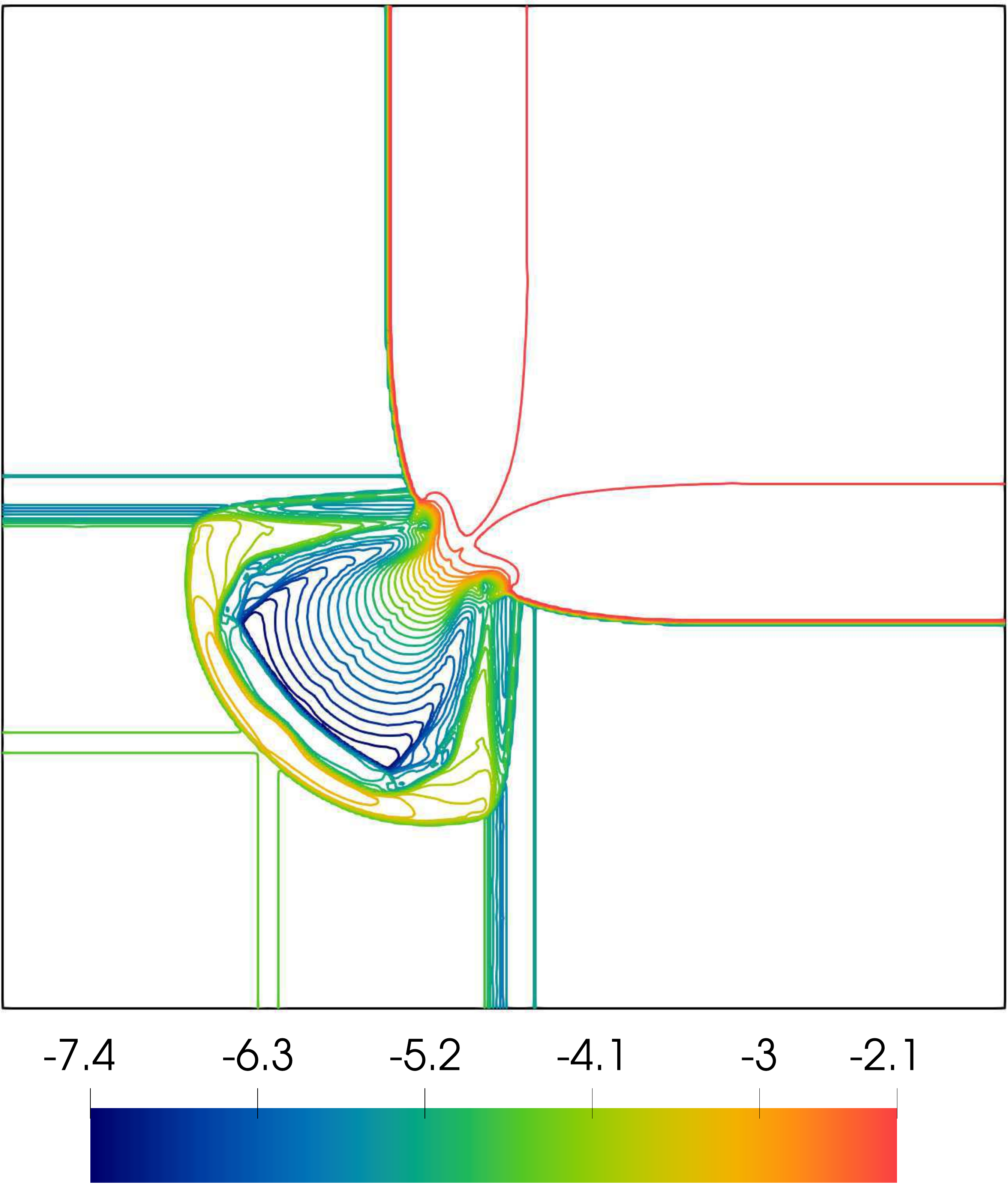}
			\caption{ID-EOS with $\gamma = \frac{4}{3}$: 25 contours in $[-7.4, -2.1]$.}
		\end{subfigure}
		\begin{subfigure}{0.31\textwidth}
			\includegraphics[width=\linewidth]{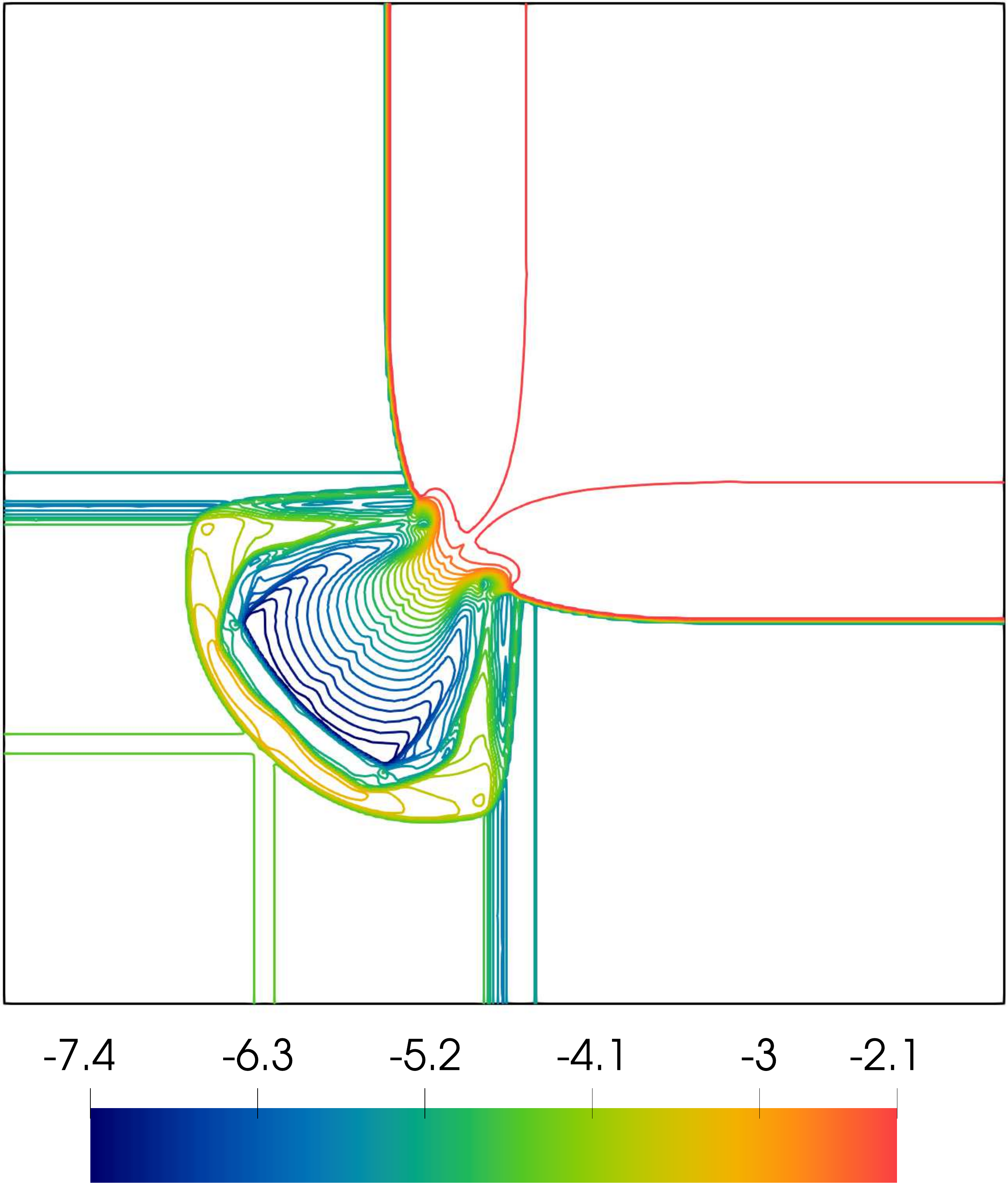}
			\caption{TM-EOS: 25 contours in $[-7.4, -2.1]$.}
		\end{subfigure}
		\begin{subfigure}{0.31\textwidth}
			\includegraphics[width=\linewidth]{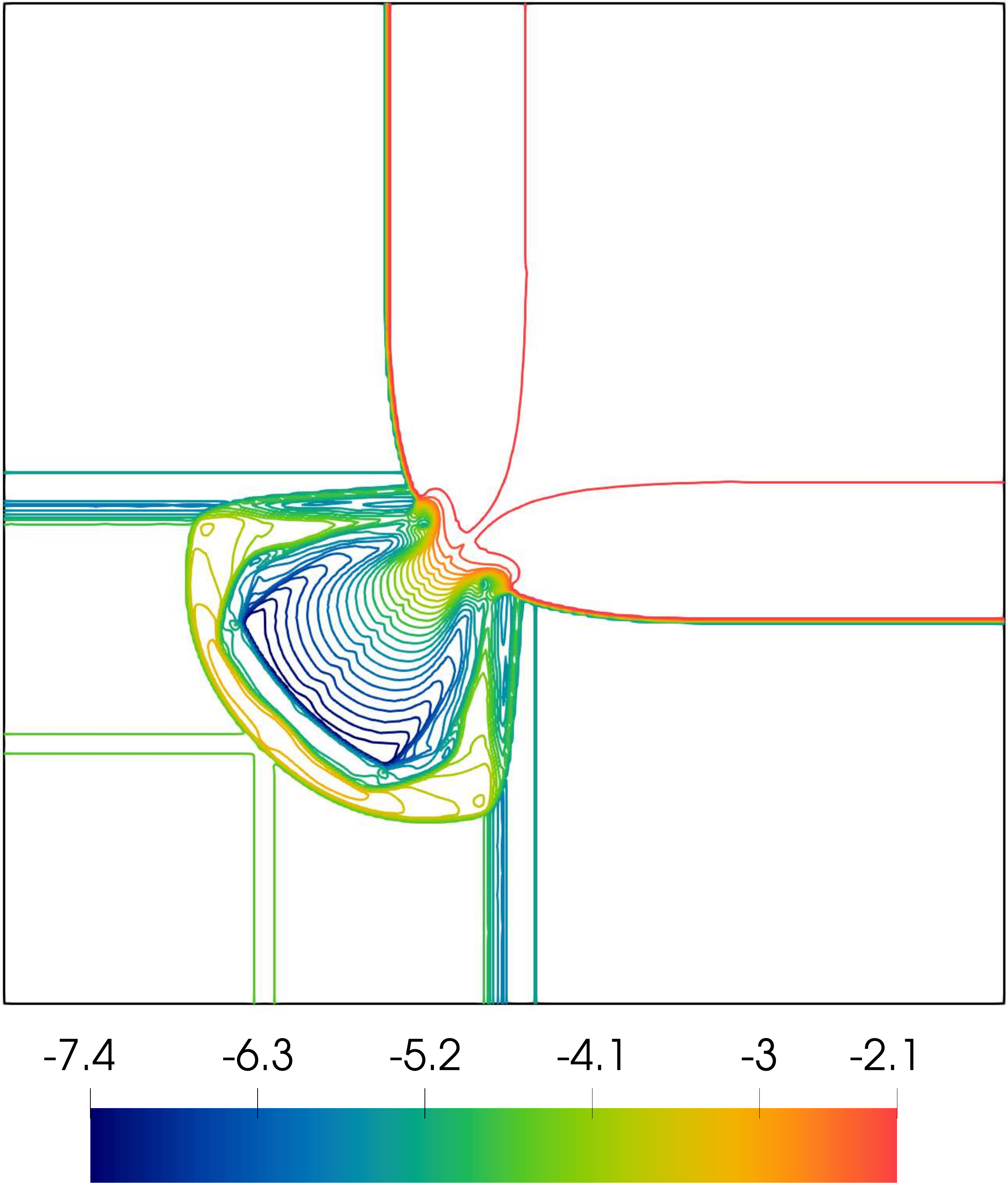}
			\caption{IP-EOS: 25 contours in $[-7.4, -2.1]$.\\}
		\end{subfigure}
		\begin{subfigure}{0.31\textwidth}
			\includegraphics[width=\linewidth]{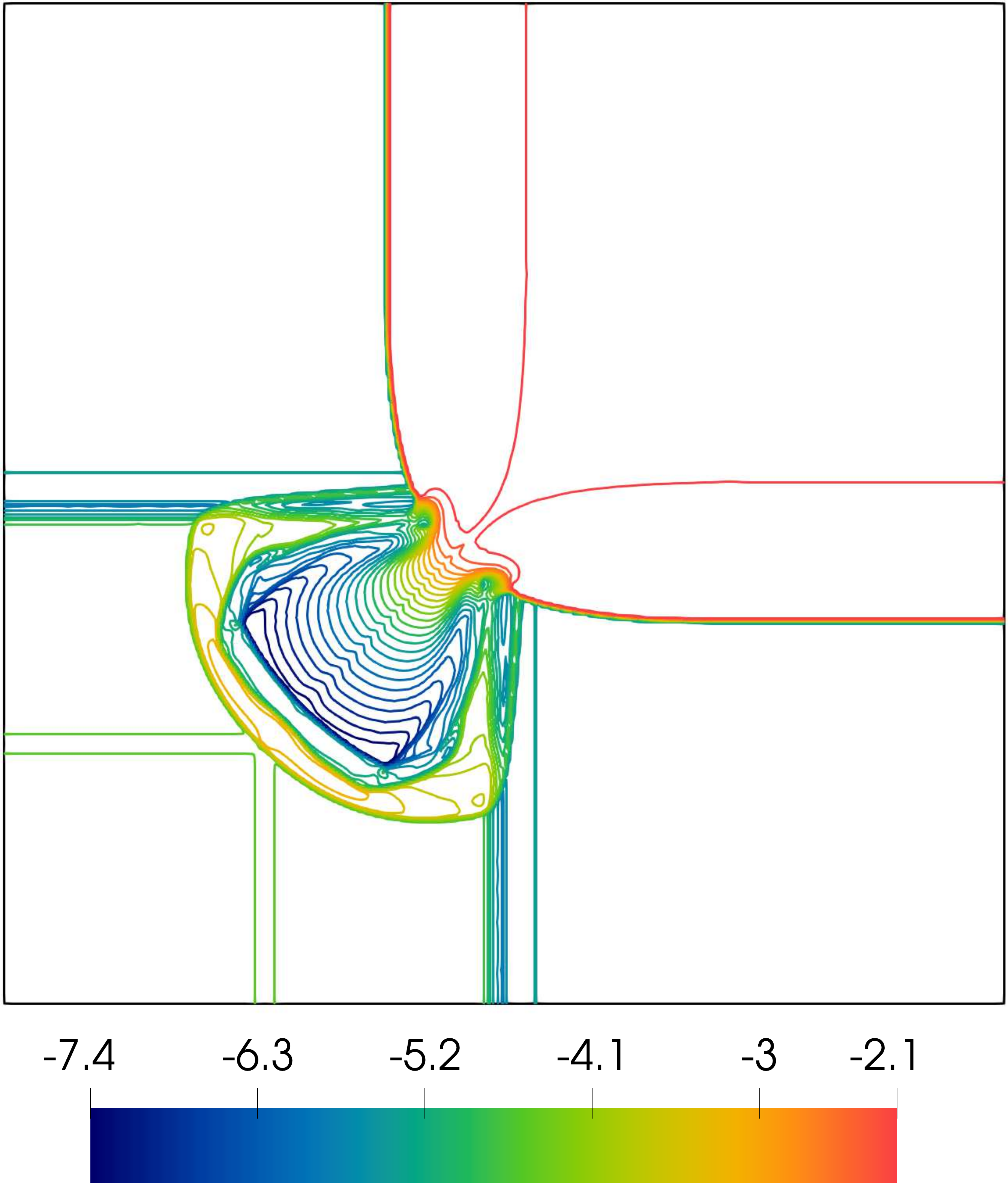}
			\caption{RC-EOS: 25 contours in $[-7.4, -2.1]$.}
		\end{subfigure}
		\begin{subfigure}{0.31\textwidth}
			\includegraphics[width=\linewidth]{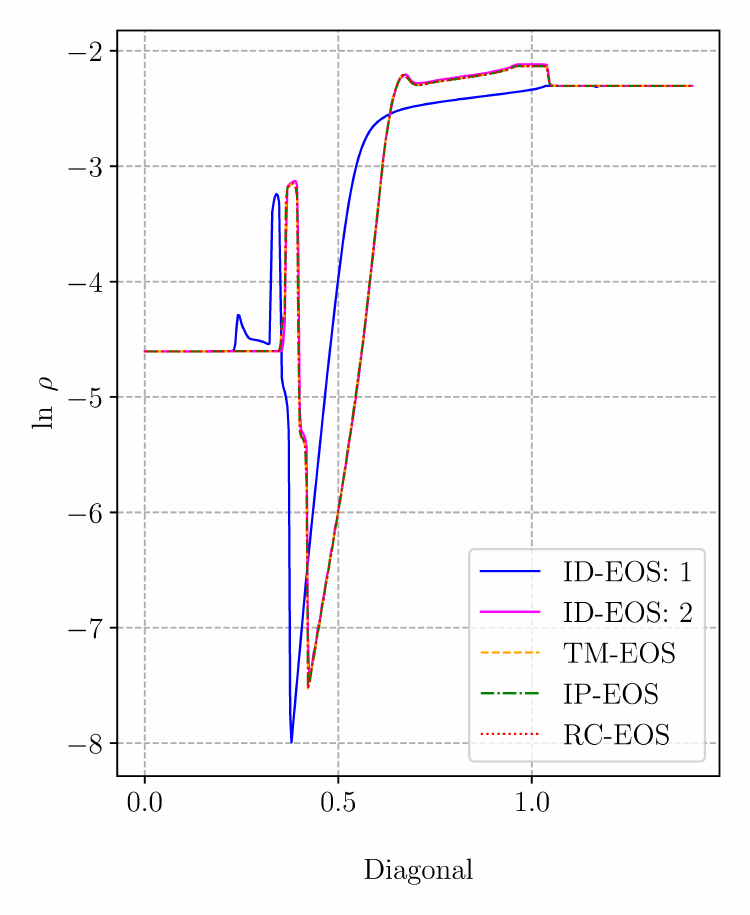}
			\caption{Cut plot from lower-left to upper-right.\\}
		\end{subfigure}
		\vspace{0.2cm}
		\caption{2-D Riemann problem 2: Plot of $\ln \rho$ with $400$ cells and $N=4$.}
		\label{fig:2dwurp2.lnden}
	\end{figure}
	\begin{figure}[]
		\centering
		\begin{subfigure}{0.31\textwidth}
			\includegraphics[width=\linewidth]{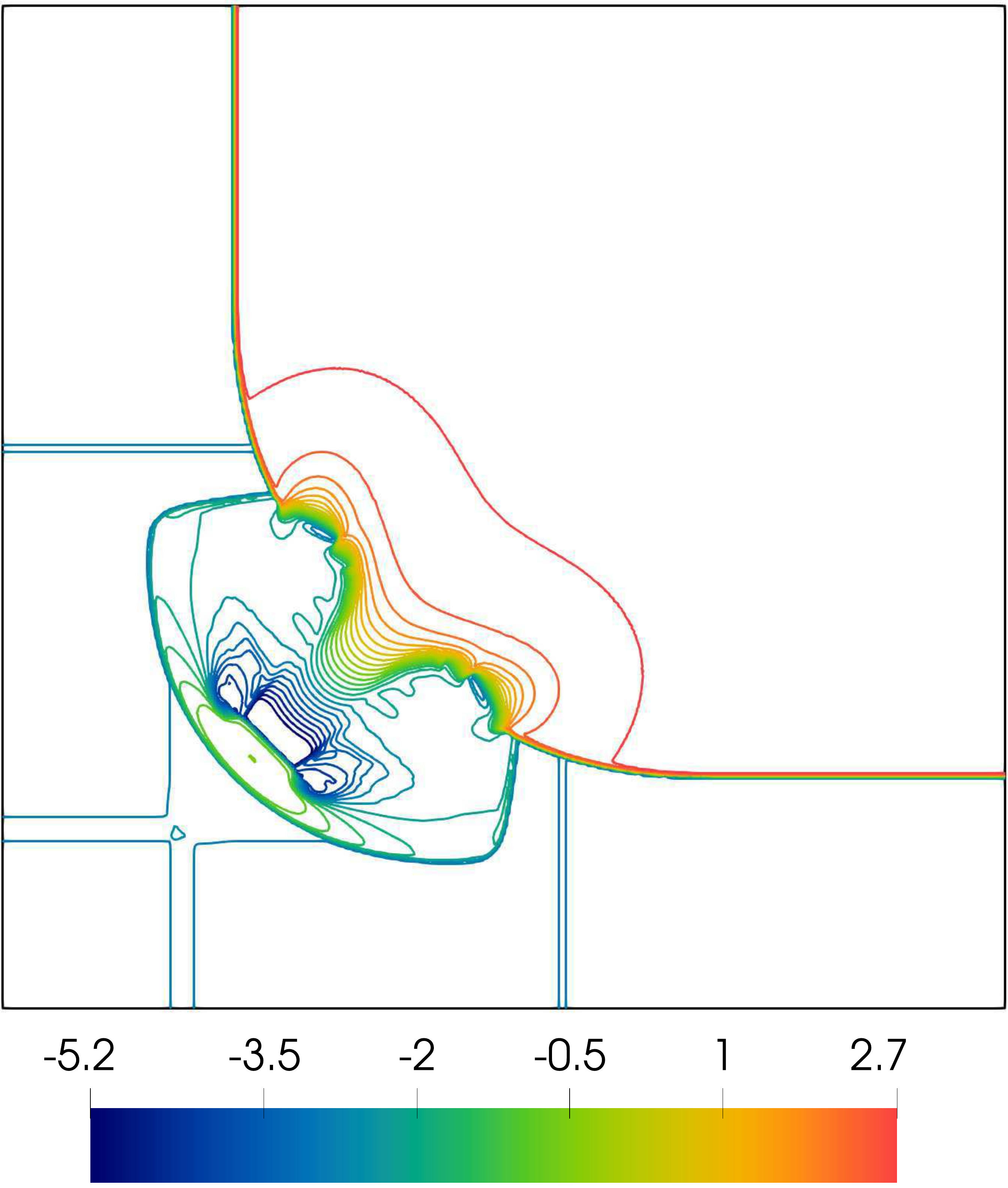}
			\caption{ID-EOS with $\gamma = \frac{5}{3}$: 25 contours in $[-5.2, 2.7]$.}
		\end{subfigure}
		\begin{subfigure}{0.31\textwidth}
			\includegraphics[width=\linewidth]{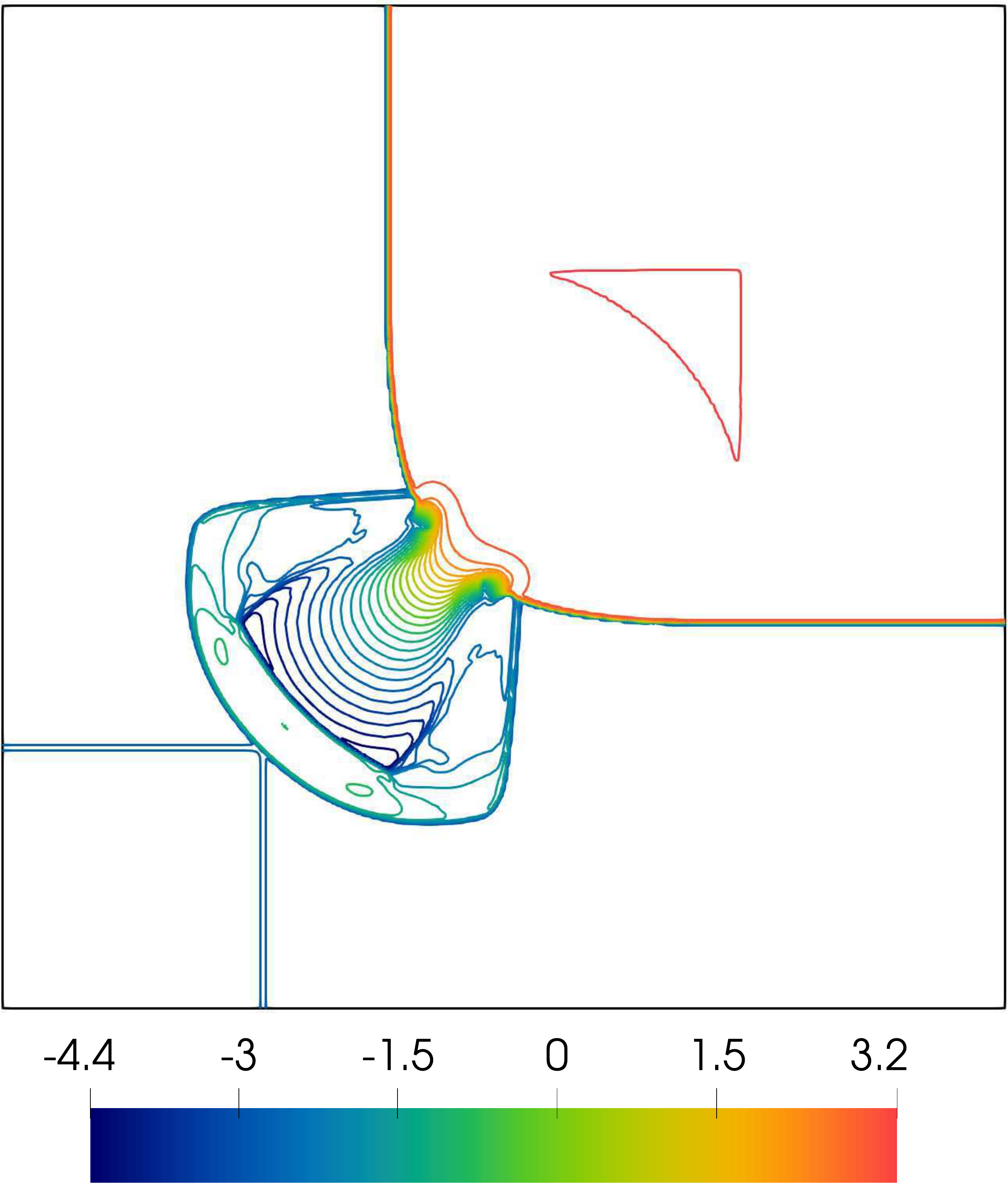}
			\caption{ID-EOS with $\gamma = \frac{4}{3}$: 25 contours in $[-4.4, 3.2]$.}
		\end{subfigure}
		\begin{subfigure}{0.31\textwidth}
			\includegraphics[width=\linewidth]{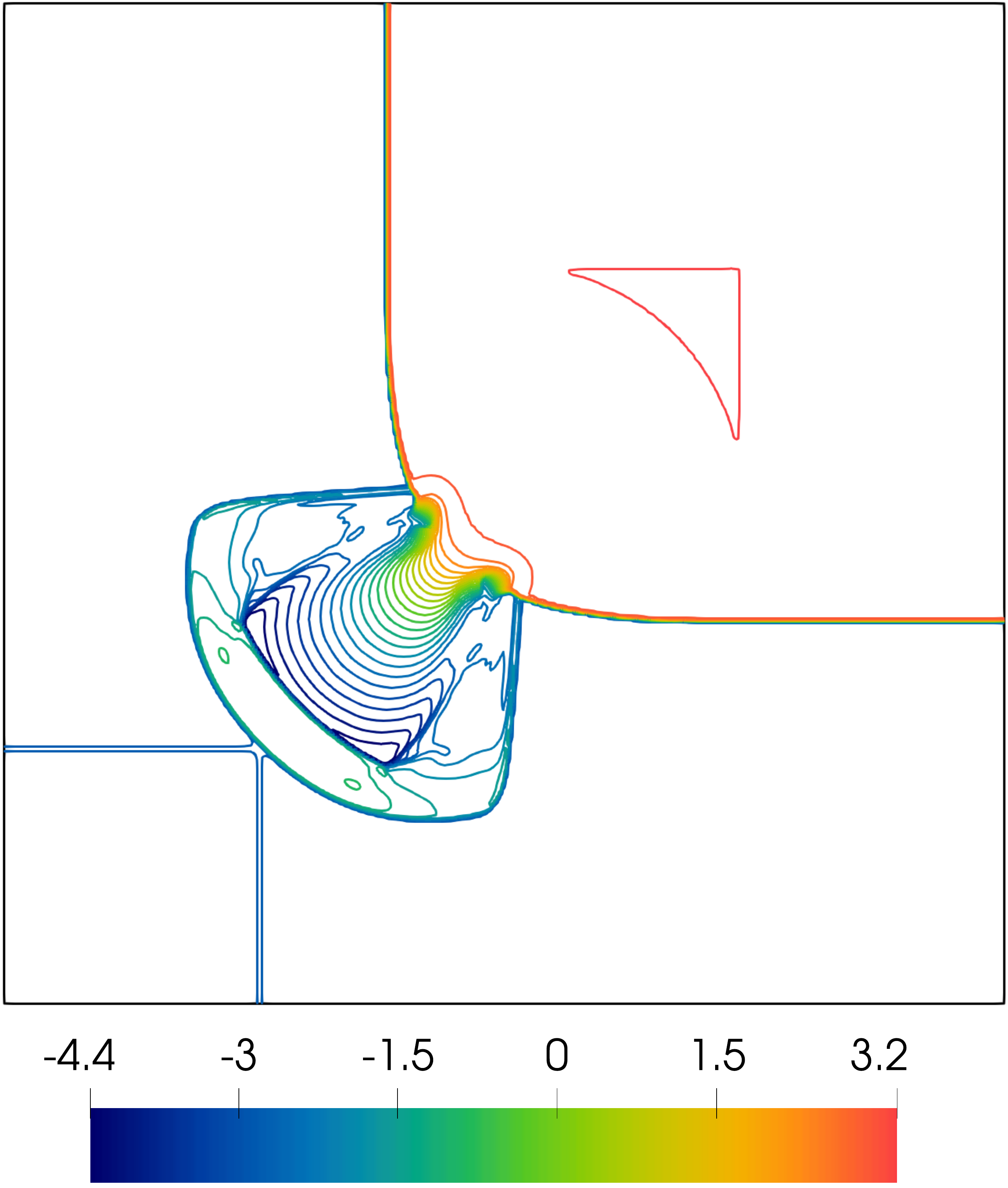}
			\caption{TM-EOS: 25 contours in $[-4.4, 3.2]$.\\}
		\end{subfigure}
		\begin{subfigure}{0.31\textwidth}
			\includegraphics[width=\linewidth]{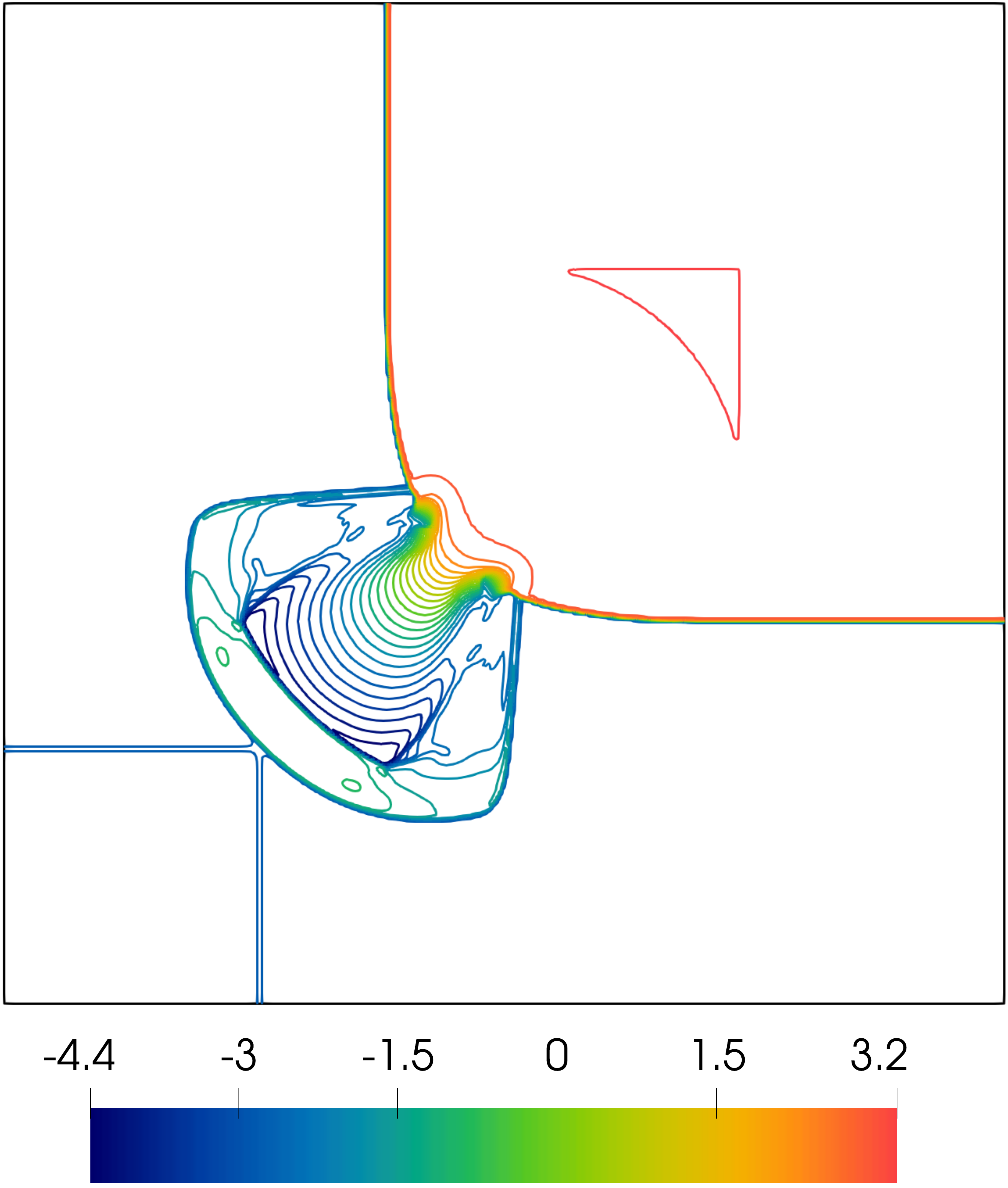}
			\caption{IP-EOS: 25 contours in $[-4.4, 3.2]$.}
		\end{subfigure}
		\begin{subfigure}{0.31\textwidth}
			\includegraphics[width=\linewidth]{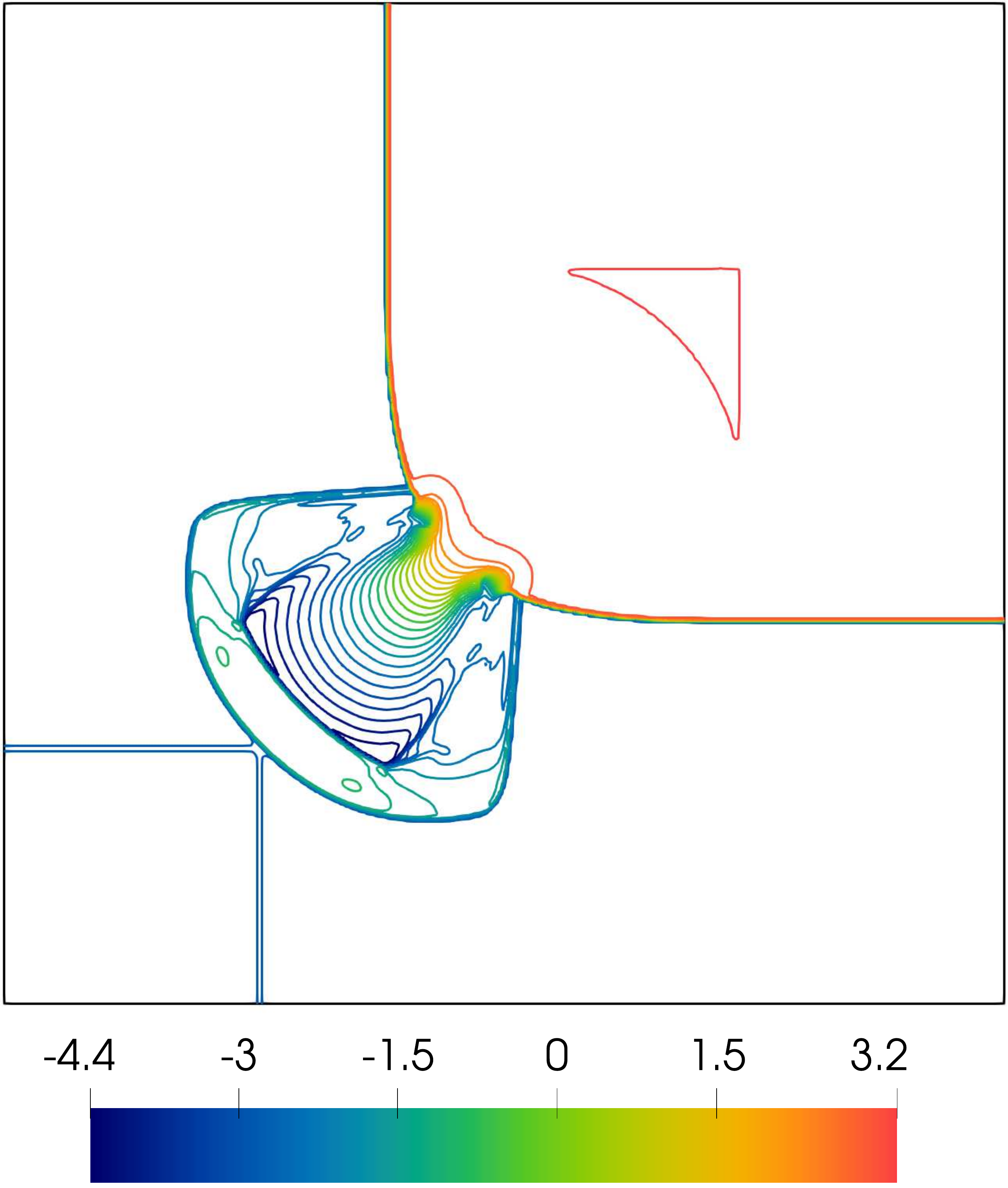}
			\caption{RC-EOS: 25 contours in $[-4.4, 3.2]$.}
		\end{subfigure}
		\begin{subfigure}{0.31\textwidth}
			\includegraphics[width=\linewidth]{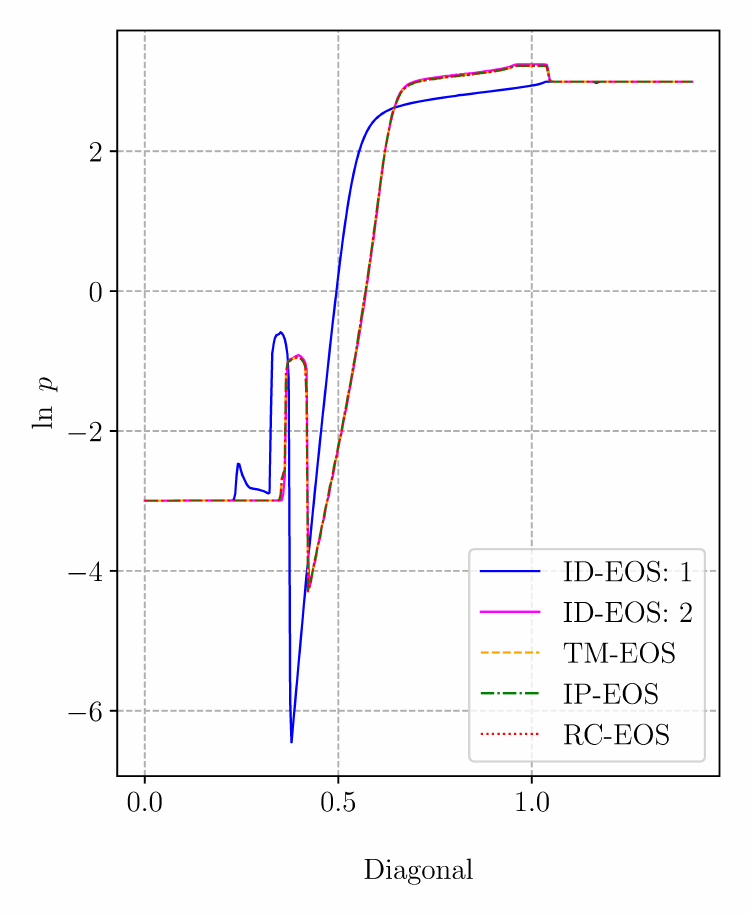}
			\caption{Cut plot from lower-left to upper-right.}
		\end{subfigure}
		\vspace{0.2cm}
		\caption{2-D Riemann problem 2: Plot of $\ln p$ with $400$ cells and $N=4$.}
		\label{fig:2dwurp2.lnpres}
	\end{figure}
	
	We can observe from the figures that our scheme can capture all the structures in the solution effectively, with an obvious difference in ID-EOS with $\gamma=\frac{5}{3}$ compared to the other cases. The solution using ID-EOS approximates the solutions with other equations of state more closely with $\gamma=\frac{4}{3}$. Here as well, we can observe that the results obtained using ID-EOS with $\gamma=\frac{4}{3}$, TM-EOS, IP-EOS, and RC-EOS are very similar.
	
	\subsubsection{2-D Riemann problem 3}
	This Riemann problem is taken from~\cite{he2012adaptive} which is also used in~\cite{nunez2016xtroem}. It has four contact discontinuities in the initial condition, given by,
	\[
	(\rho, v_1, v_2, p) = \begin{cases}
		(0.5, 0.5, -0.5, 5) & \text{if}\ x > 0.5,\ y > 0.5\\
		(1, 0.5, 0.5, 5) & \text{if}\ x < 0.5,\ y>0.5\\
		(3, -0.5, 0.5, 5) & \text{if}\ x < 0.5,\ y < 0.5\\
		(1.5, -0.5, -0.5, 5) & \text{if}\ x > 0.5,\ y<0.5.
	\end{cases}
	\]
	We simulate this problem in domain $[0,1]\times [0,1]$ with outflow boundaries using $400\times 400$ cells and $N=4$ up to time $t=0.4$, and present the results in Figure~\ref{fig:2dwu2rp1.lnden} and Figure~\ref{fig:2dwu2rp1.lnpres}.
	\begin{figure}[]
		\centering
		\begin{subfigure}{0.31\textwidth}
			\includegraphics[width=\linewidth]{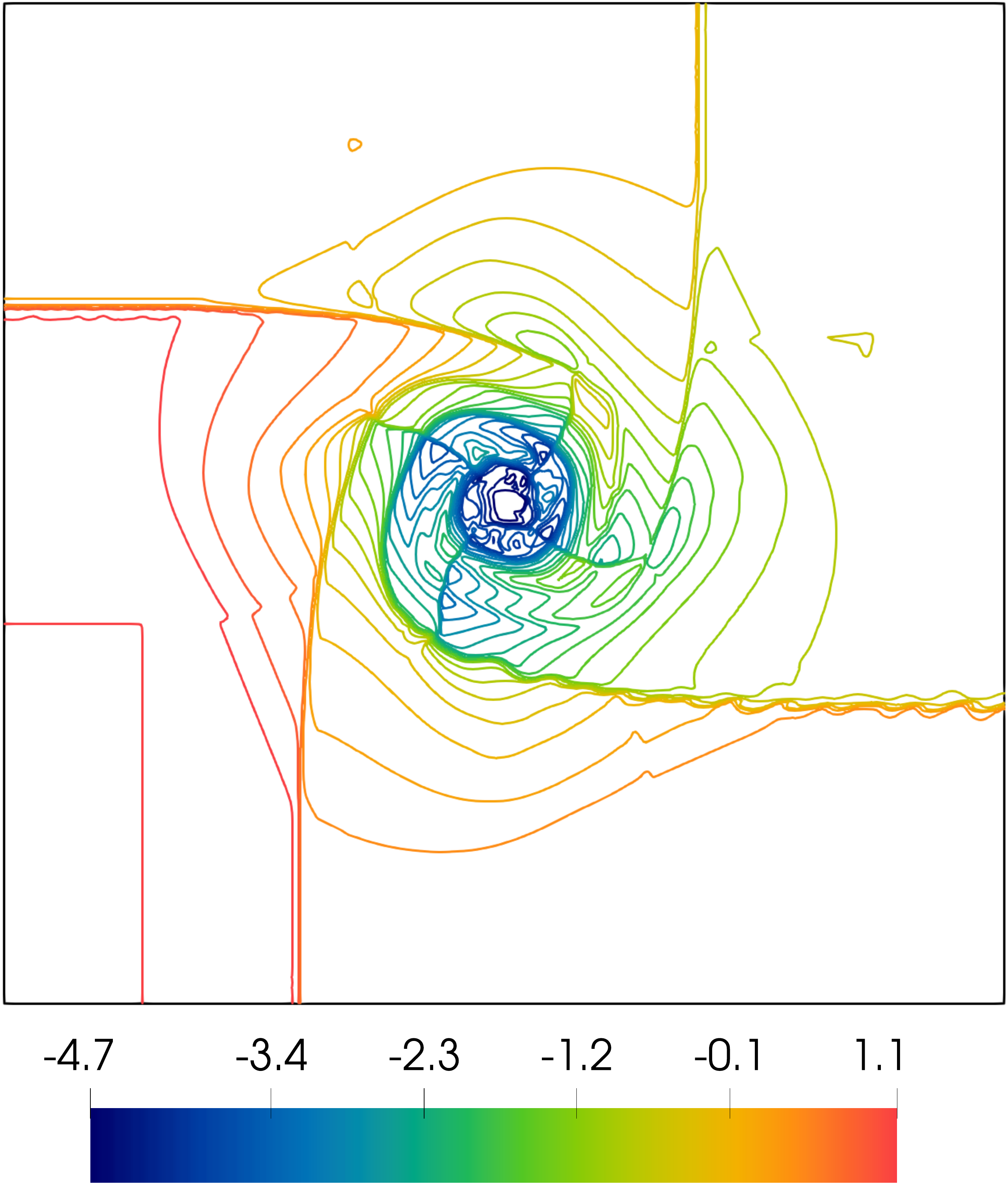}
			\caption{ID-EOS with $\gamma = \frac{5}{3}$: 25 contours in $[-4.7, 1.1]$.}
		\end{subfigure}
		\begin{subfigure}{0.31\textwidth}
			\includegraphics[width=\linewidth]{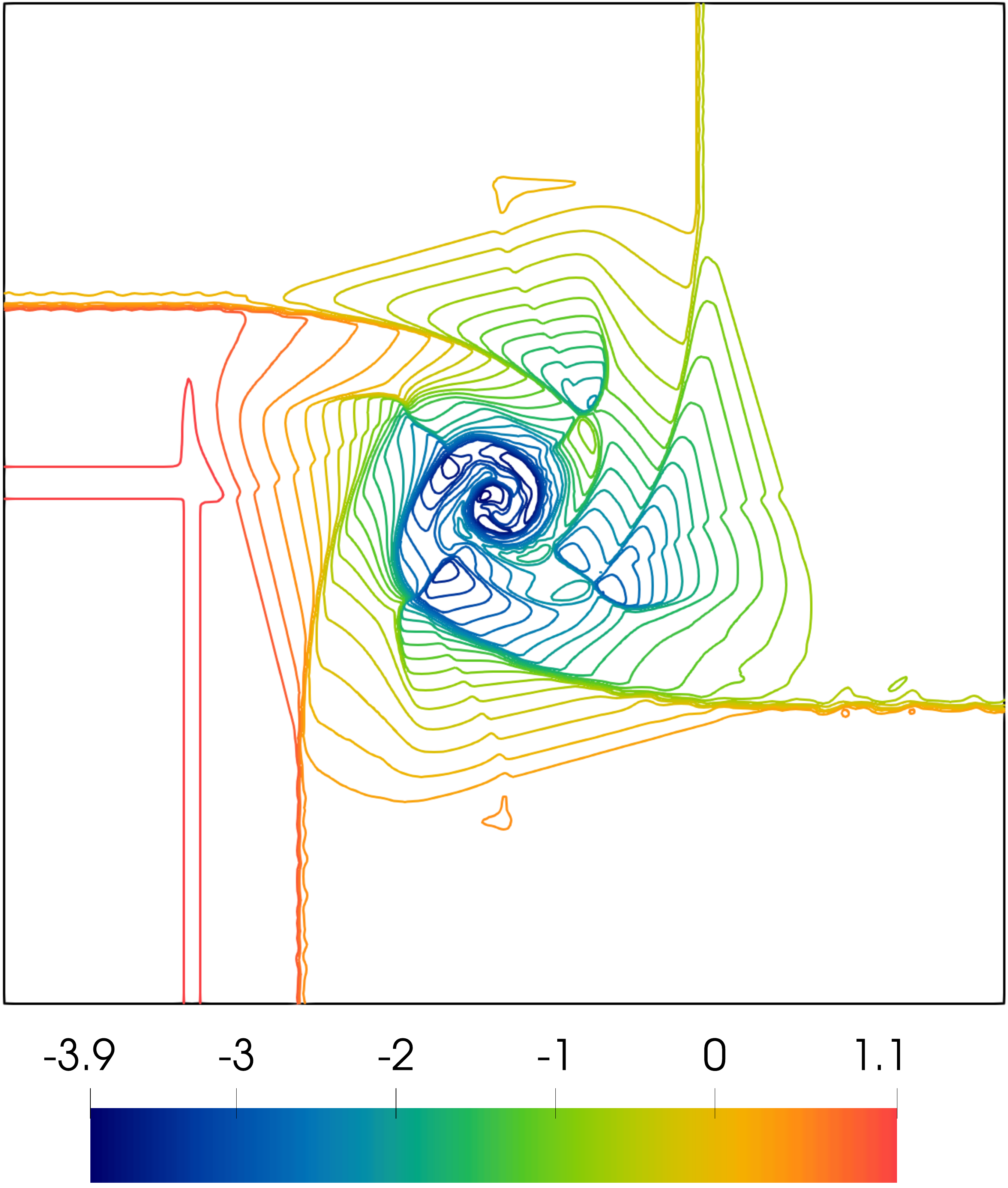}
			\caption{ID-EOS with $\gamma = \frac{4}{3}$: 25 contours in $[-3.9, 1.1]$.}
		\end{subfigure}
		\begin{subfigure}{0.31\textwidth}
			\includegraphics[width=\linewidth]{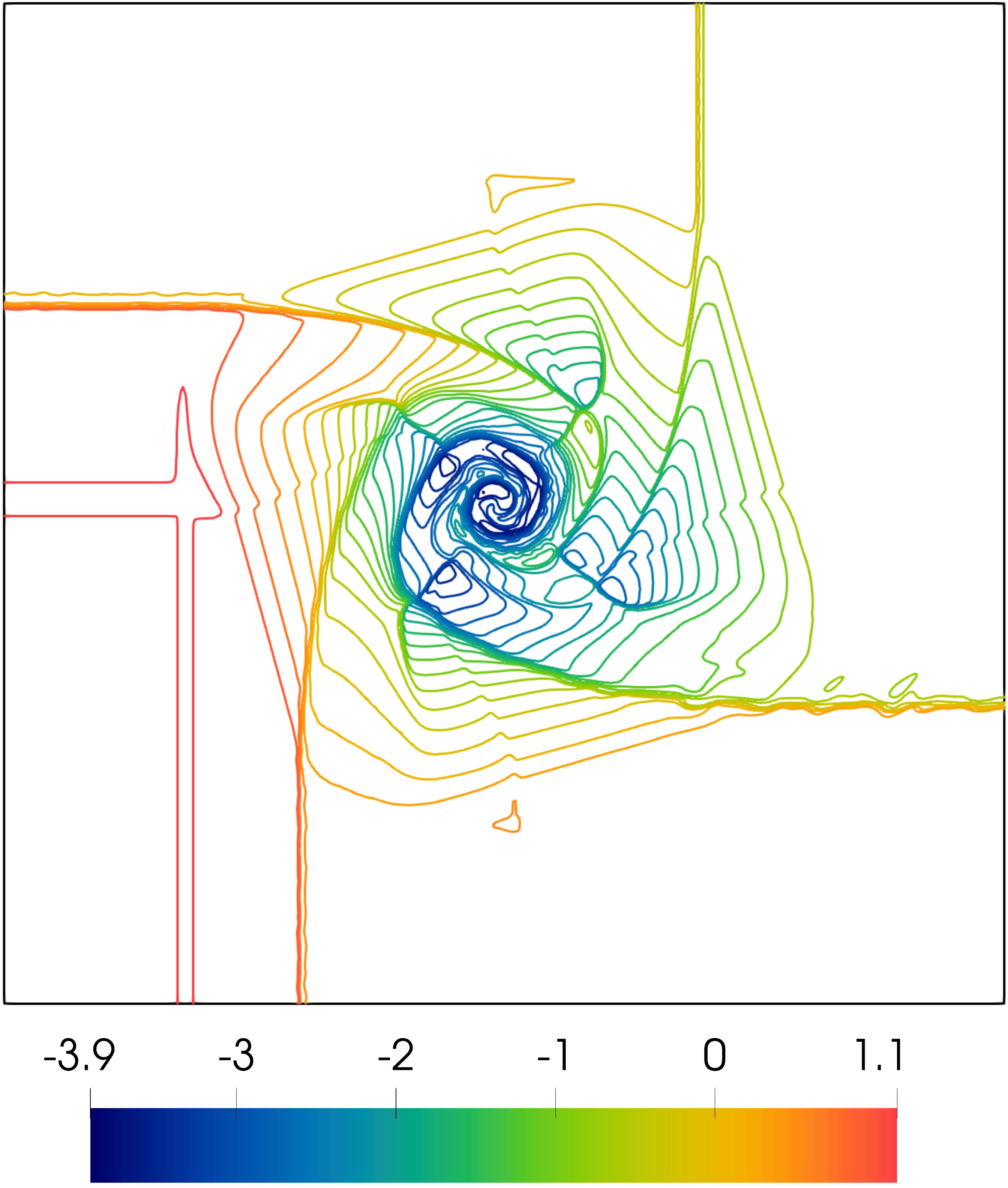}
			\caption{TM-EOS: 25 contours in $[-3.9, 1.1]$.\\}
		\end{subfigure}
		\begin{subfigure}{0.31\textwidth}
			\includegraphics[width=\linewidth]{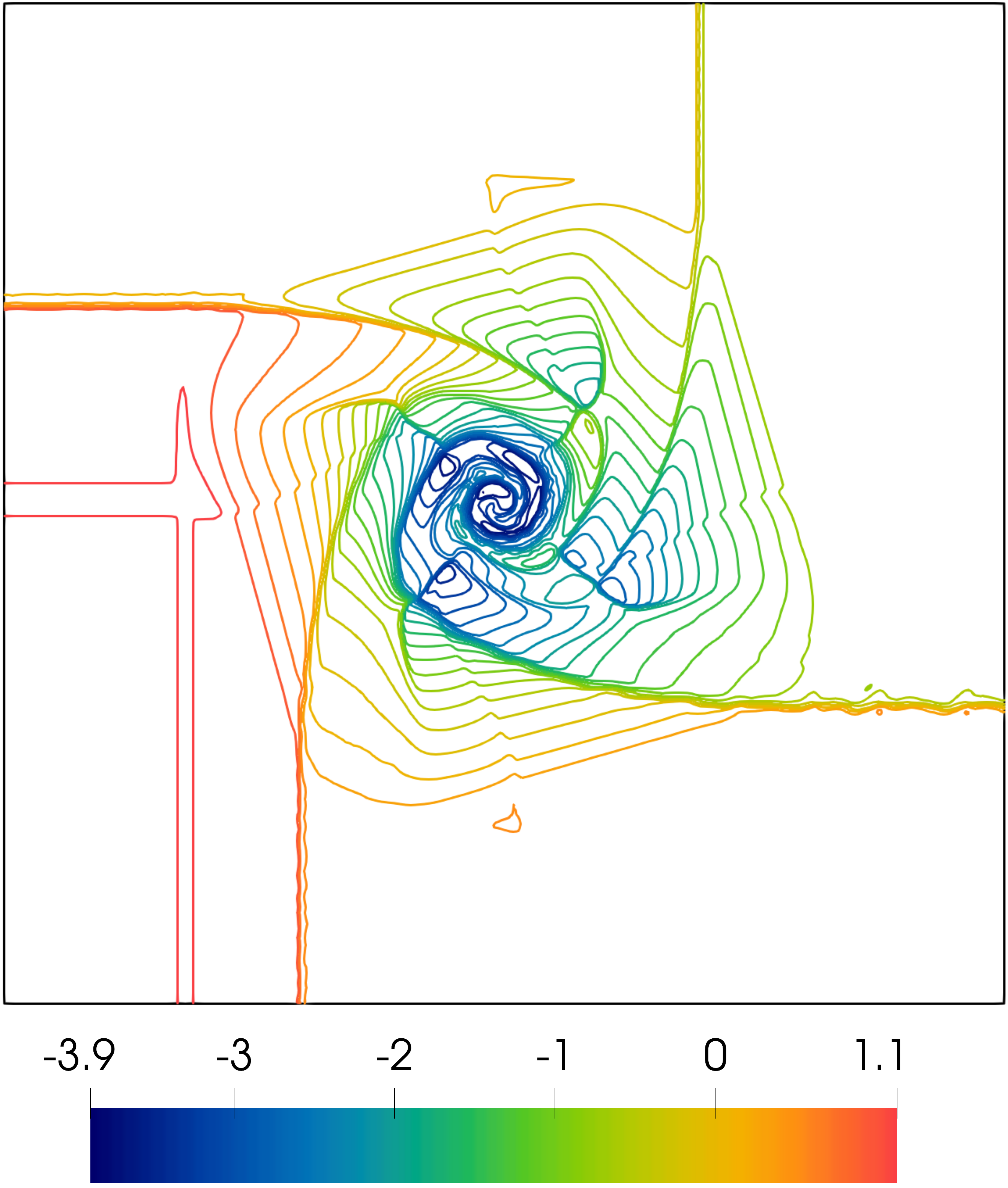}
			\caption{IP-EOS: 25 contours in $[-3.9, 1.1]$.}
		\end{subfigure}
		\begin{subfigure}{0.31\textwidth}
			\includegraphics[width=\linewidth]{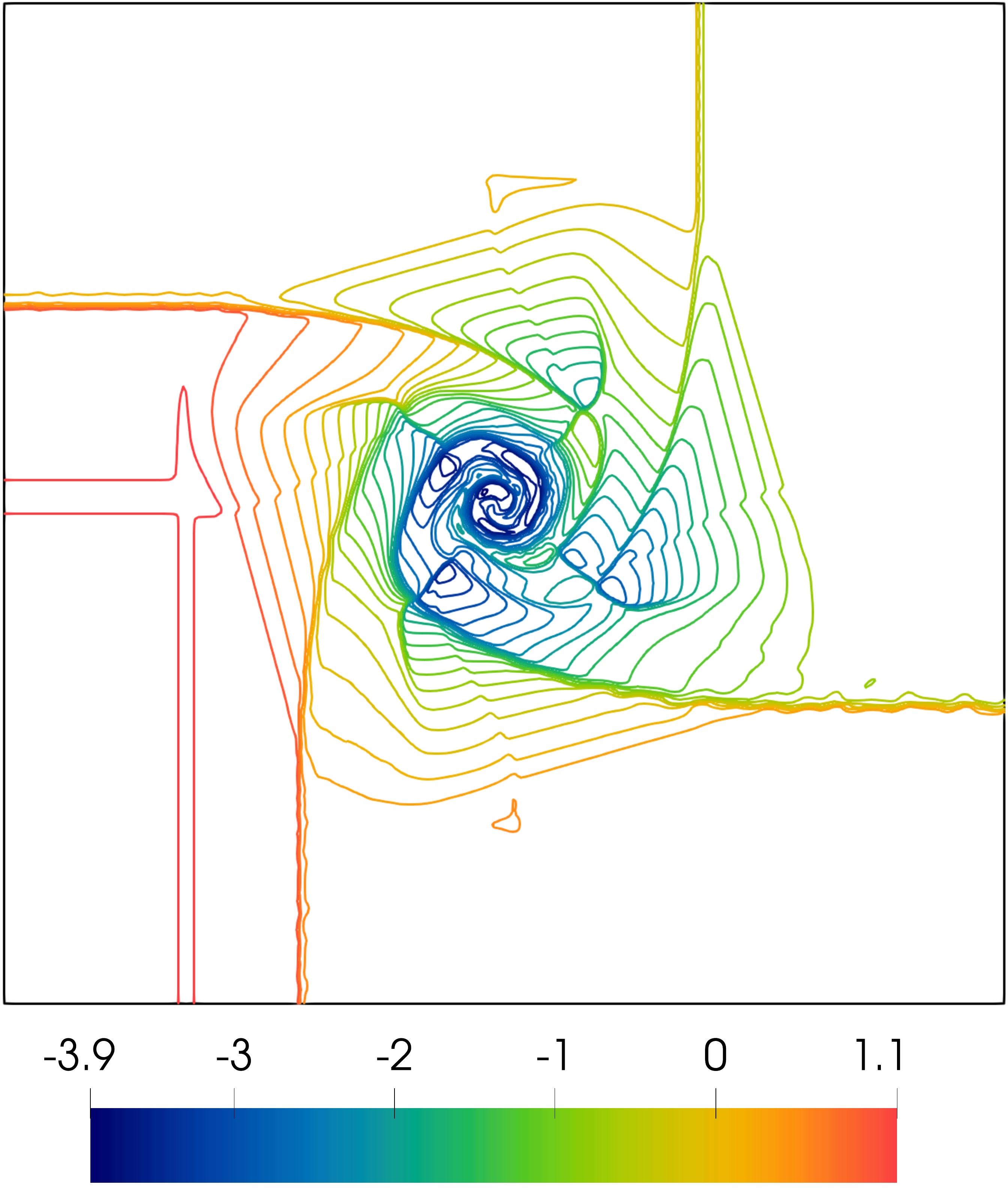}
			\caption{RC-EOS: 25 contours in $[-3.9, 1.1]$.}
		\end{subfigure}
		\begin{subfigure}{0.31\textwidth}
			\includegraphics[width=\linewidth]{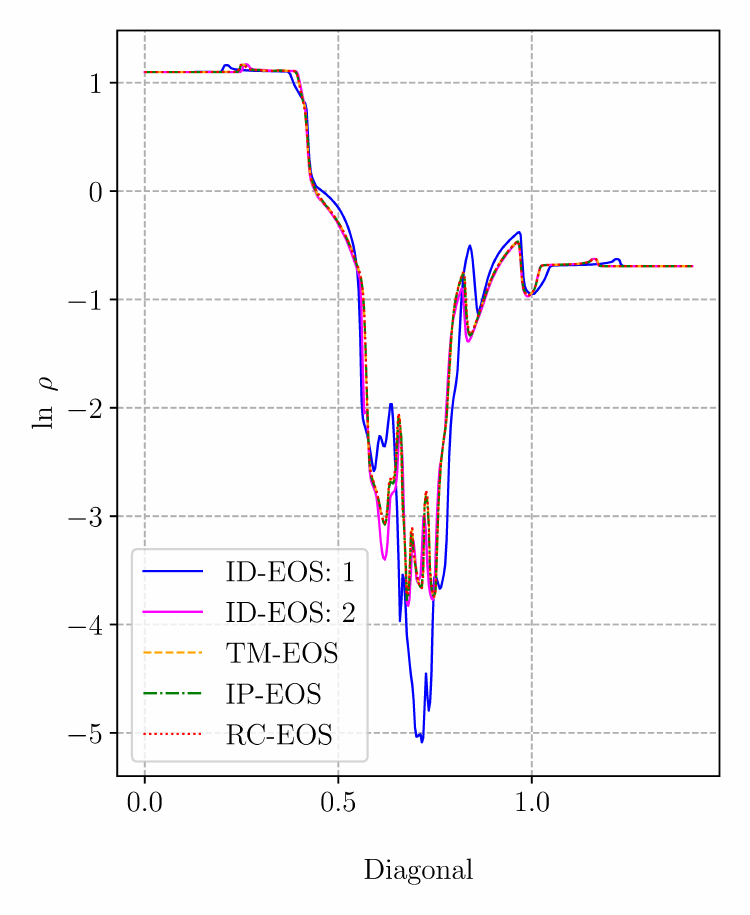}
			\caption{Cut plot from lower-left to upper-right.}
		\end{subfigure}
		\vspace{0.2cm}
		\caption{2-D Riemann problem 3: Plot of $\ln \rho$ with $400$ cells and $N=4$.}
		\label{fig:2dwu2rp1.lnden}
	\end{figure}
	\begin{figure}[]
		\centering
		\begin{subfigure}{0.31\textwidth}
			\includegraphics[width=\linewidth]{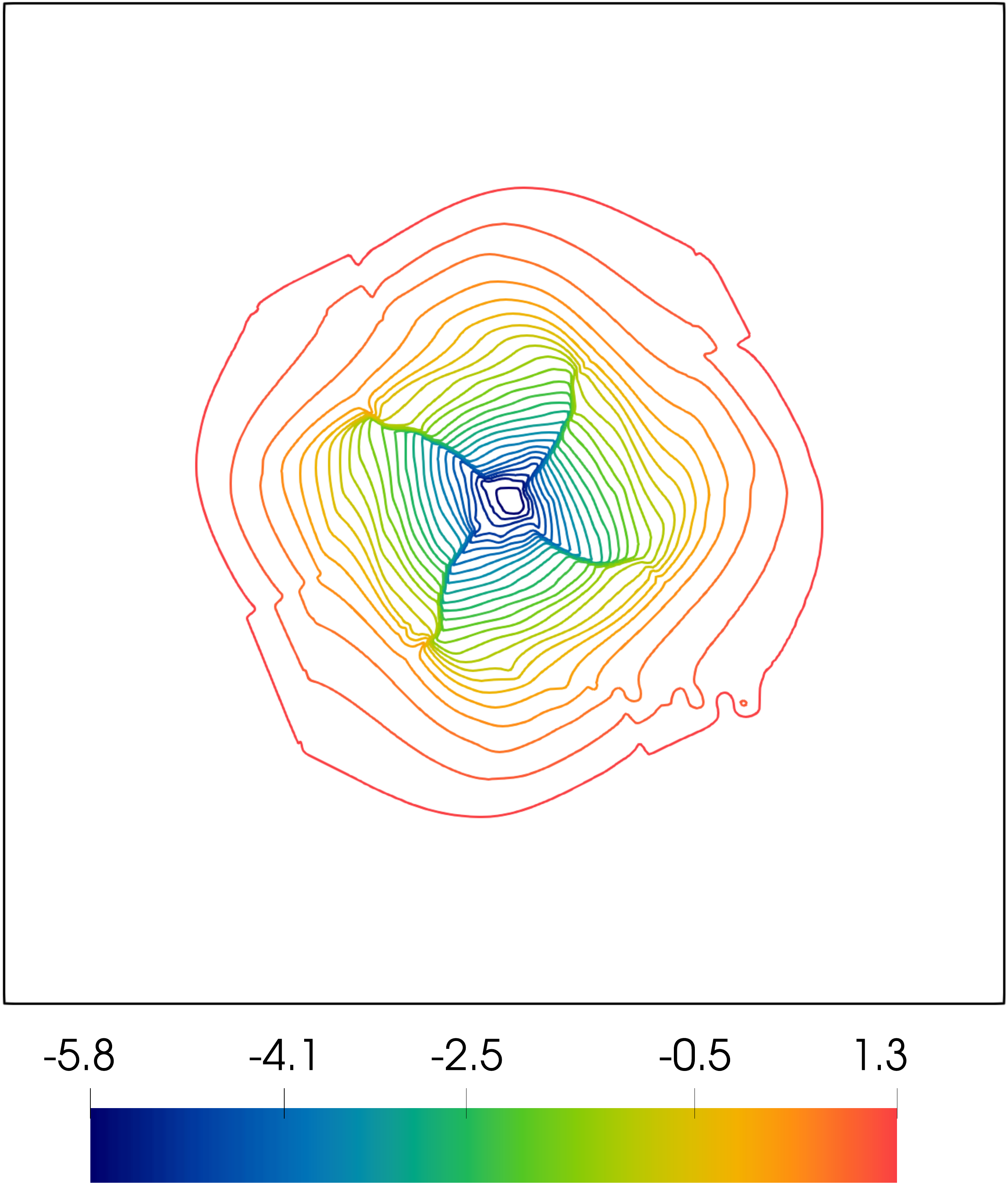}
			\caption{ID-EOS with $\gamma = \frac{5}{3}$: 25 contours in $[-5.8, 1.3]$.}
		\end{subfigure}
		\begin{subfigure}{0.31\textwidth}
			\includegraphics[width=\linewidth]{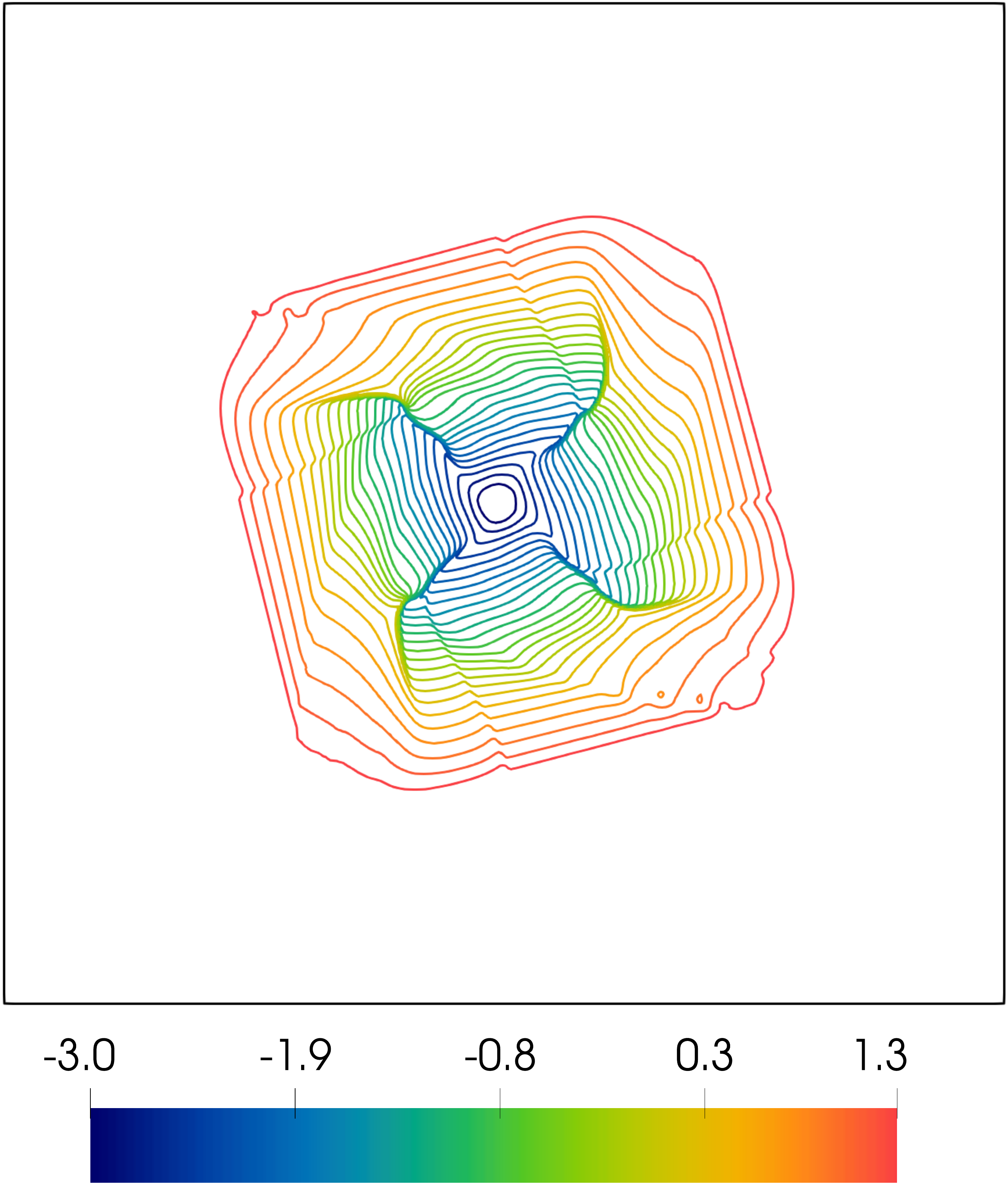}
			\caption{ID-EOS with $\gamma = \frac{4}{3}$: 25 contours in $[-3.0, 1.3]$.}
		\end{subfigure}
		\begin{subfigure}{0.31\textwidth}
			\includegraphics[width=\linewidth]{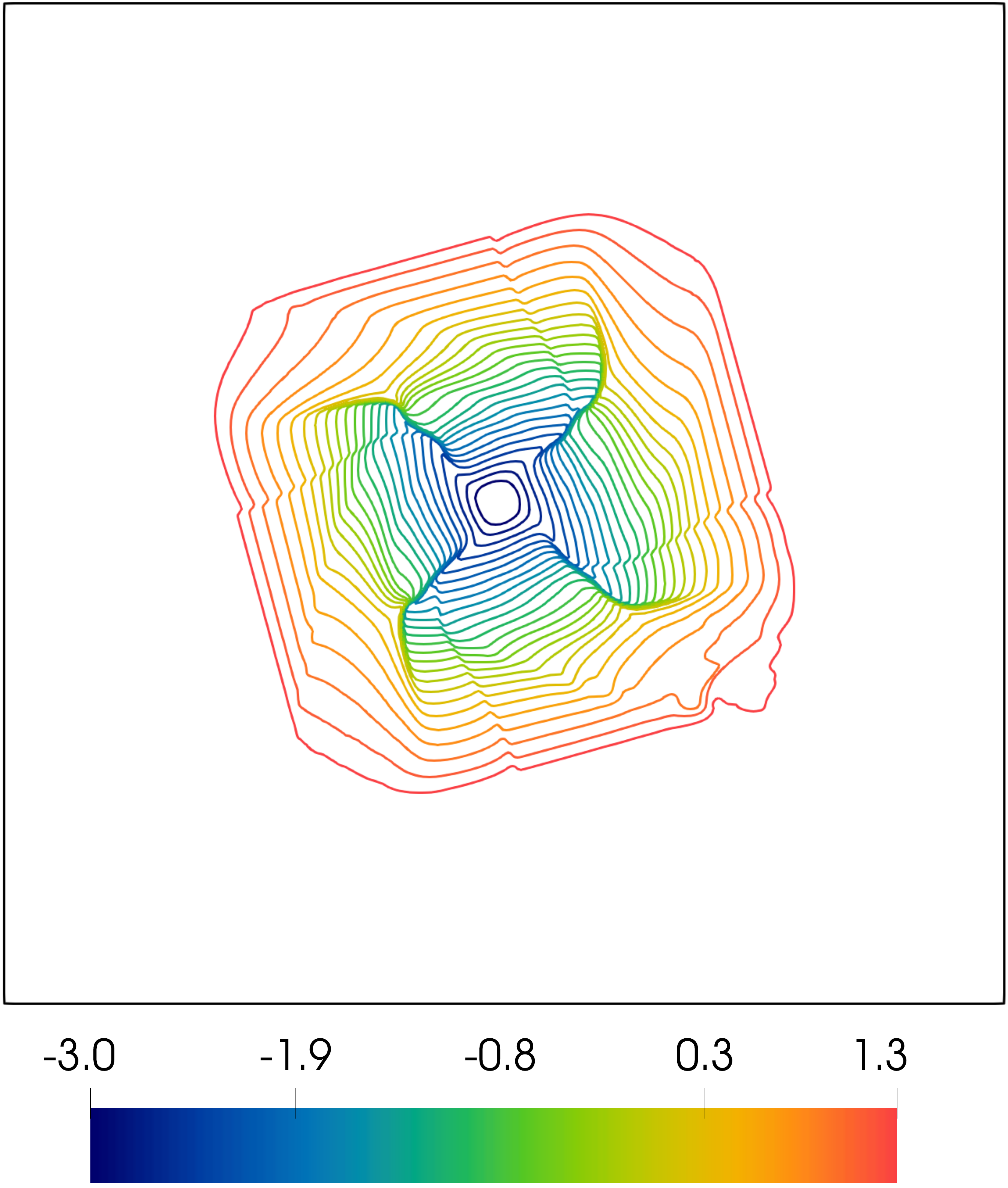}
			\caption{TM-EOS: 25 contours in $[-3.0, 1.3]$.\\}
		\end{subfigure}
		\begin{subfigure}{0.31\textwidth}
			\includegraphics[width=\linewidth]{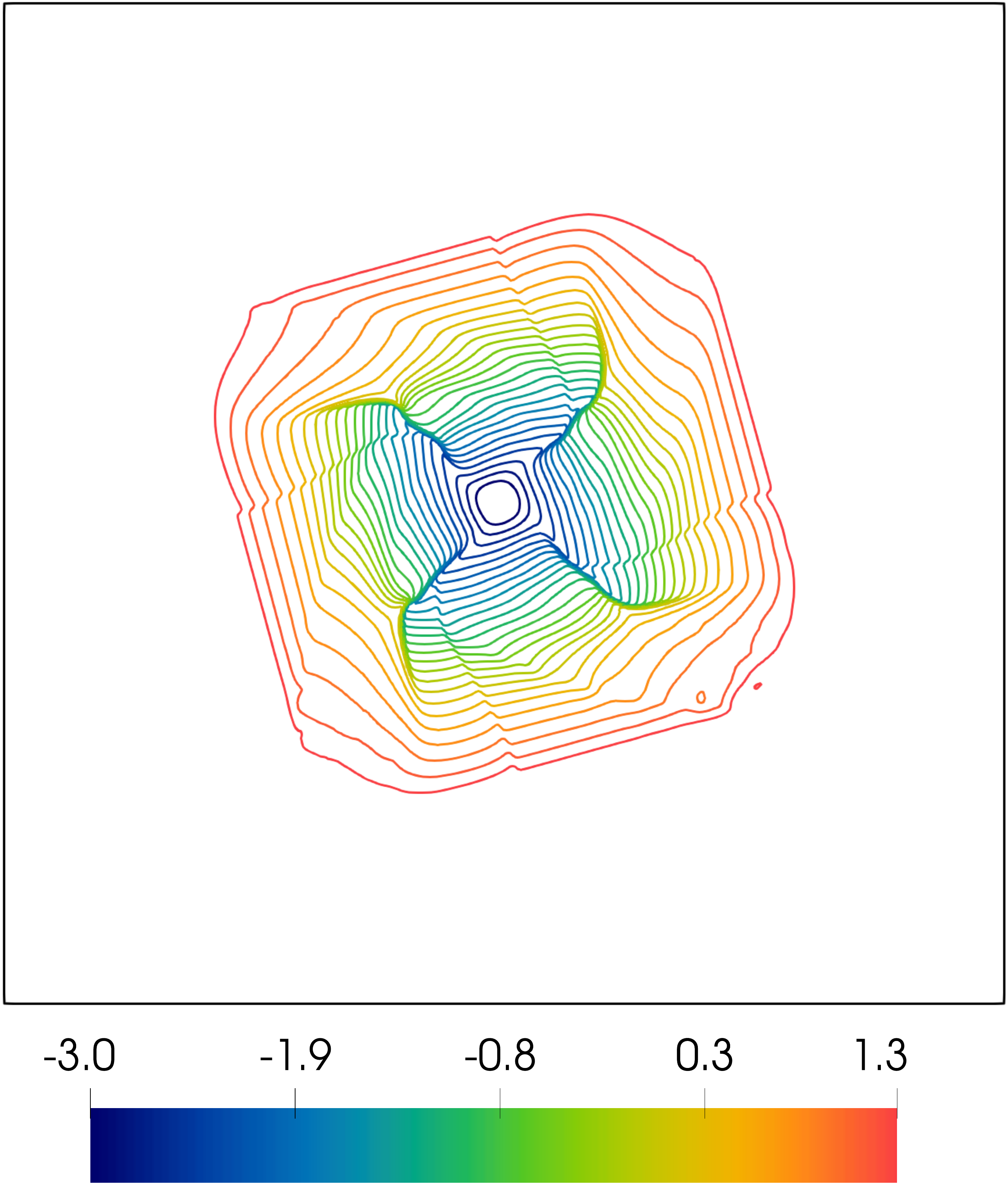}
			\caption{IP-EOS: 25 contours in $[-3.0, 1.3]$.}
		\end{subfigure}
		\begin{subfigure}{0.31\textwidth}
			\includegraphics[width=\linewidth]{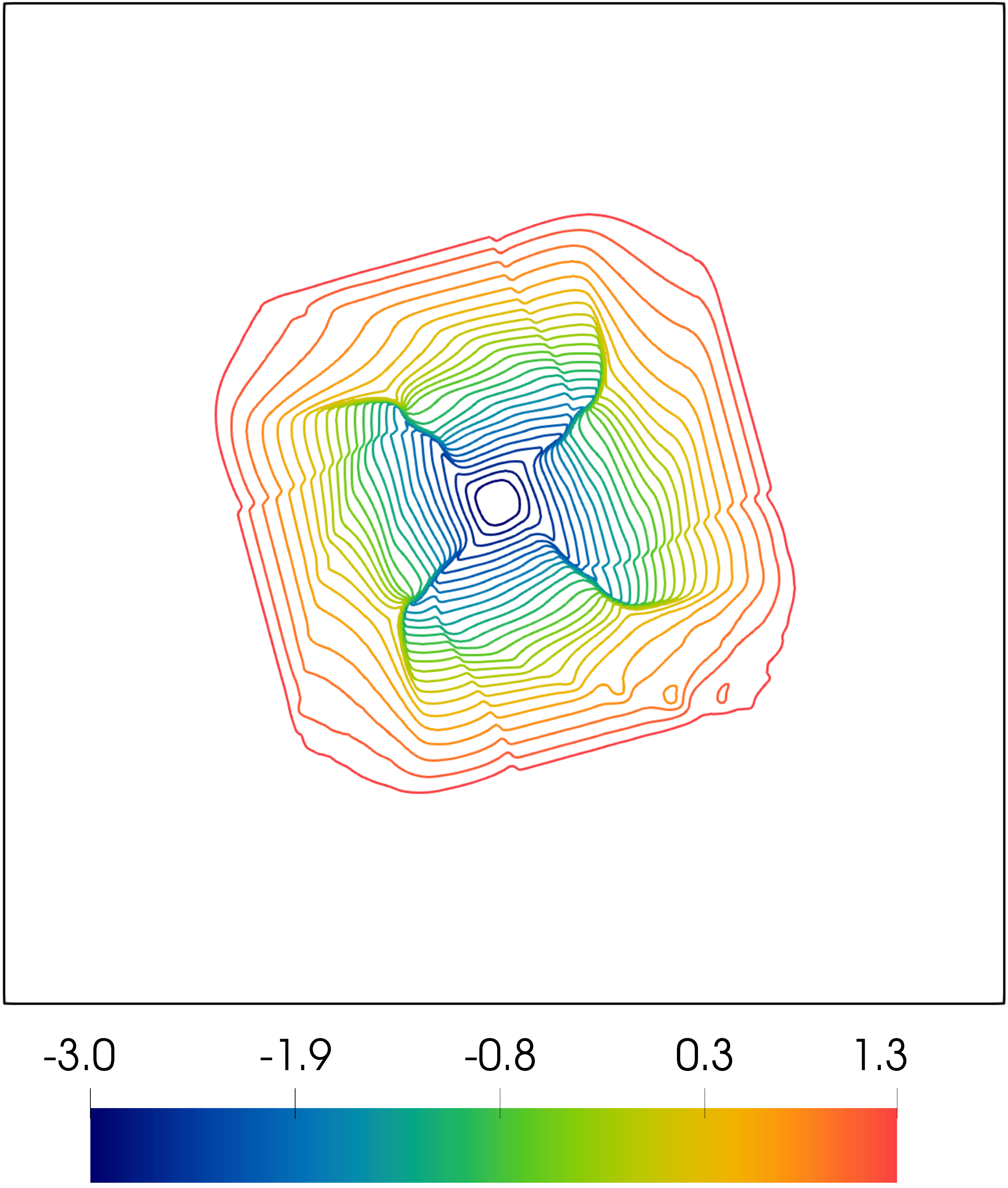}
			\caption{RC-EOS: 25 contours in $[-3.0, 1.3]$.}
		\end{subfigure}
		\begin{subfigure}{0.31\textwidth}
			\includegraphics[width=\linewidth]{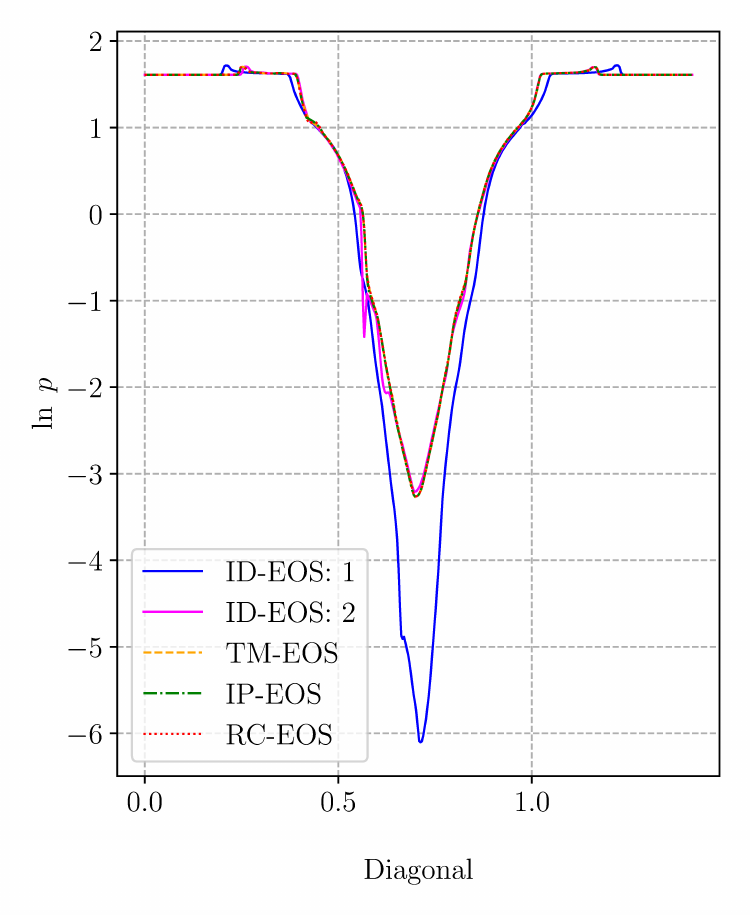}
			\caption{Cut plot from lower-left to upper-right.}
		\end{subfigure}
		\vspace{0.2cm}
		\caption{2-D Riemann problem 3: Plot of $\ln p$ with $400$ cells and $N=4$.}
		\label{fig:2dwu2rp1.lnpres}
	\end{figure}
	The interaction of the discontinuities results in the formation of a spiral structure in the solutions. The scheme captures this structure for all the equations of state with very similar solutions using ID-EOS having $\gamma = \frac{4}{3}$, TM-EOS, IP-EOS, and RC-EOS.
	Using the ID-EOS with $\gamma=\frac{5}{3}$ results in a lower fluid density and pressure in the central region of the spiral compared to the other equation of state. 
	
	\subsubsection{2-D Riemann problem 4}
	For this problem, as studied in~\cite{he2012adaptive}, we take the initial state of the fluid as,
	\[
	(\rho, v_1, v_2, p) = \begin{cases}
		(1, 0, 0, 1) & \text{if}\ x > 0.5,\ y > 0.5\\
		(0.5771, -0.3529, 0, 0.4) & \text{if}\ x < 0.5,\ y>0.5\\
		(1, -0.3529, -0.3529, 1) & \text{if}\ x < 0.5,\ y < 0.5\\
		(0.5771, 0, -0.3529, 0.4) & \text{if}\ x > 0.5,\ y<0.5.
	\end{cases}
	\]
	We run the simulations with different equations of state, taking the computational domain as $[0,1]\times [0,1]$ with outflow boundaries, and using $400\times 400$ cells with $N=4$. We present the outputs in Figure~\ref{fig:2dwu2rp2.lnden} and Figure~\ref{fig:2dwu2rp2.lnpres} at time $t=0.4$.
	\begin{figure}[]
		\centering
		\begin{subfigure}{0.31\textwidth}
			\includegraphics[width=\linewidth]{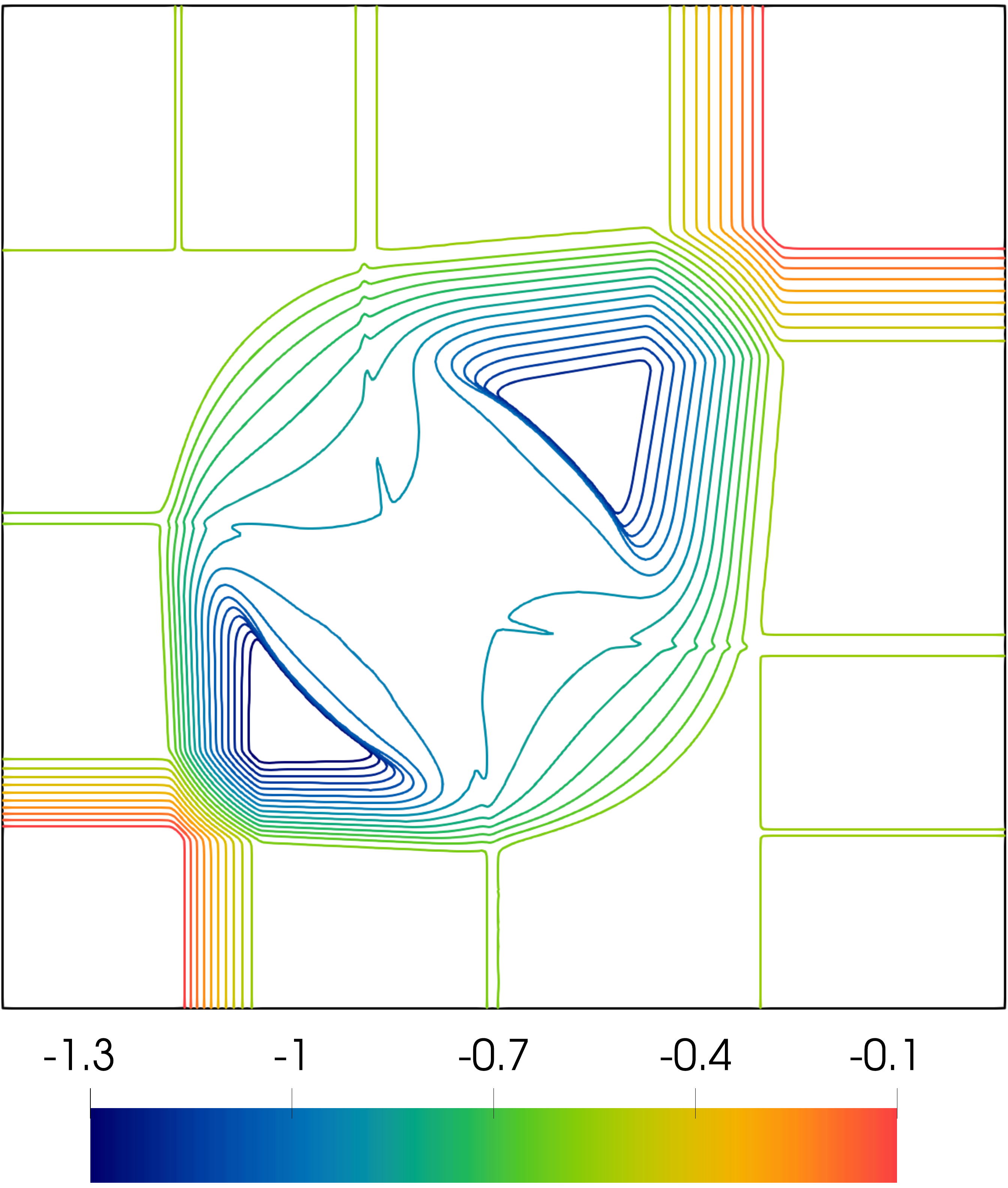}
			\caption{ID-EOS with $\gamma = \frac{5}{3}$: 25 contours in $[-1.3, -0.1]$.}
		\end{subfigure}
		\begin{subfigure}{0.31\textwidth}
			\includegraphics[width=\linewidth]{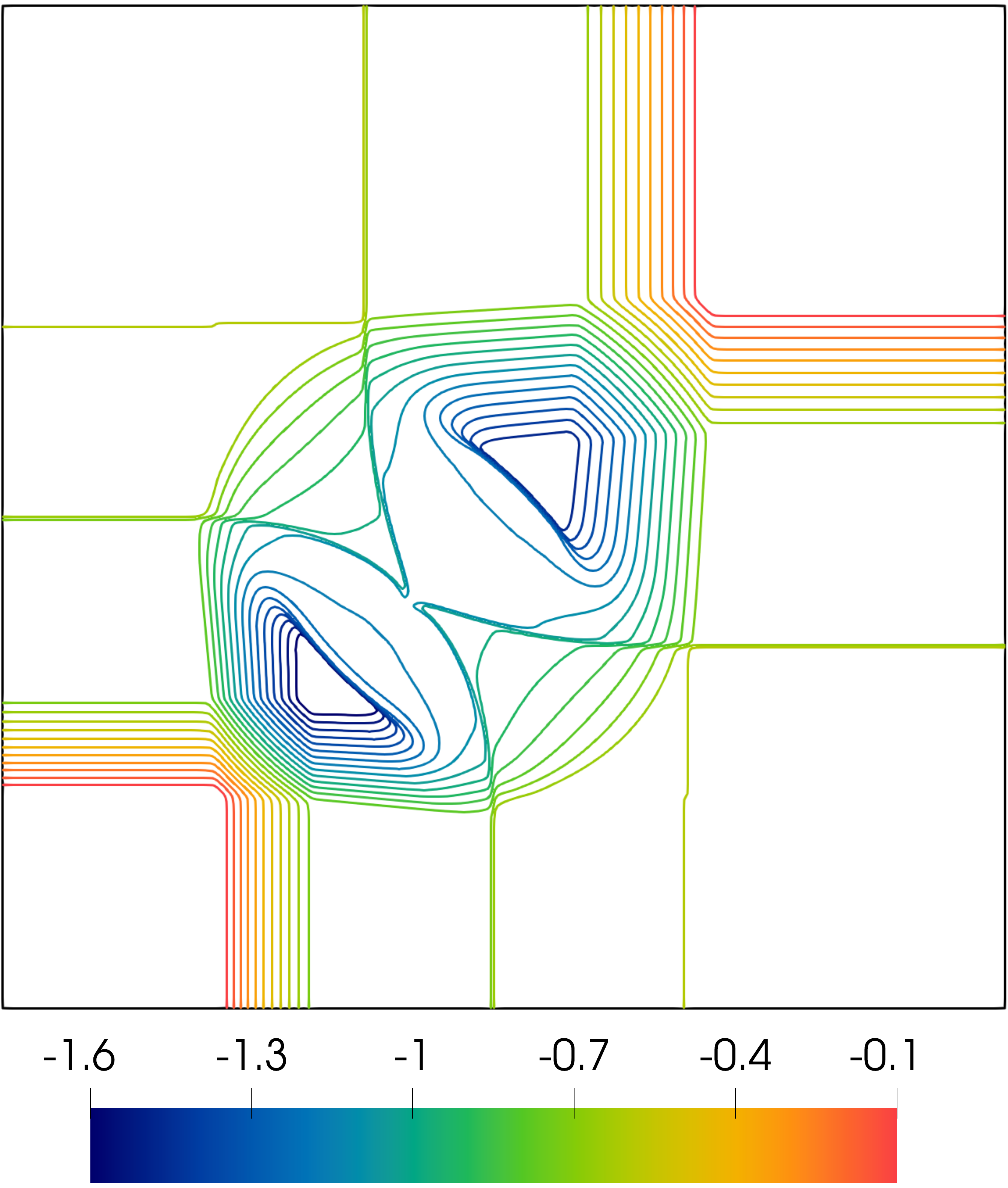}
			\caption{ID-EOS with $\gamma = \frac{4}{3}$: 25 contours in $[-1.6, -0.1]$.}
		\end{subfigure}
		\begin{subfigure}{0.31\textwidth}
			\includegraphics[width=\linewidth]{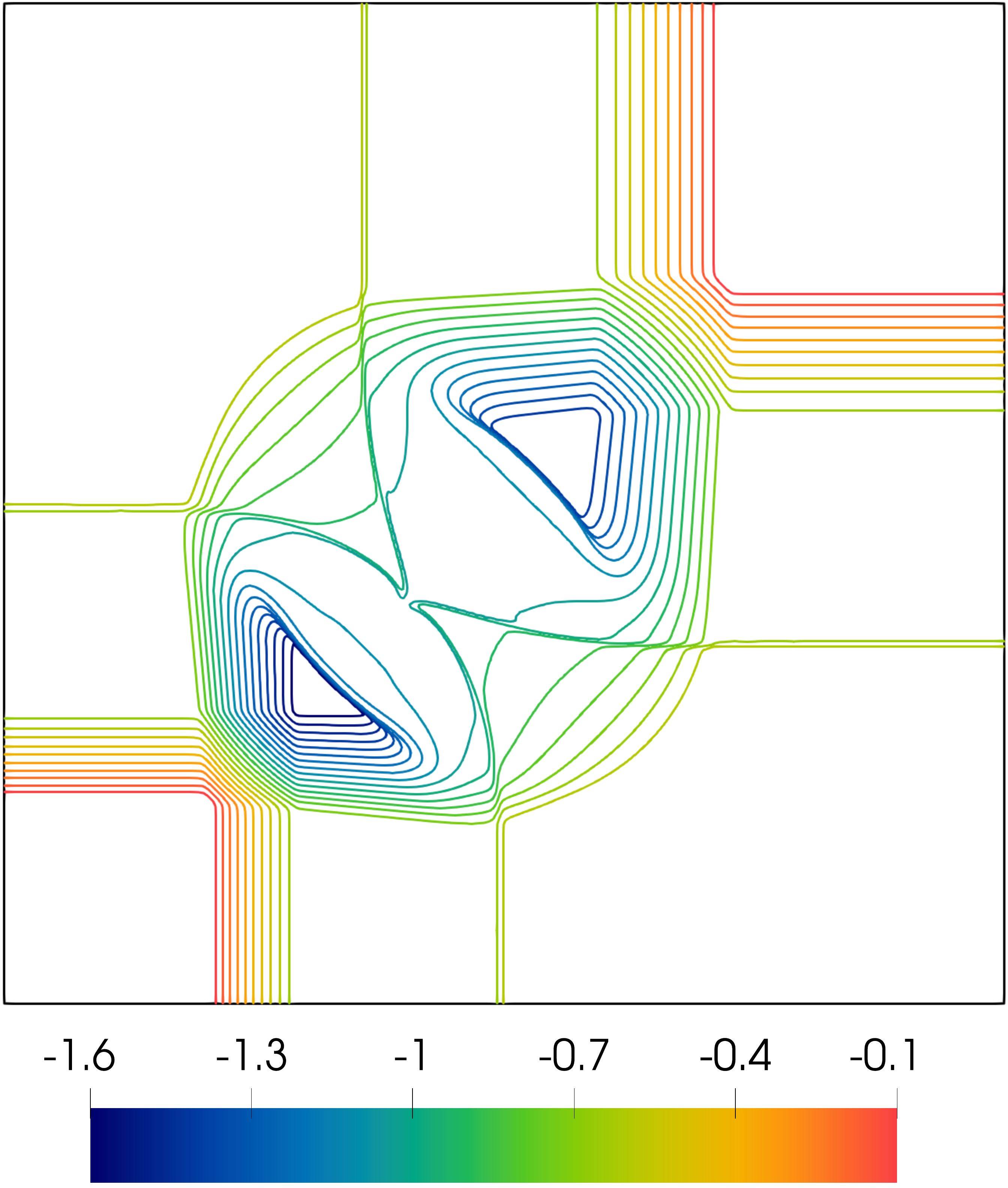}
			\caption{TM-EOS: 25 contours in $[-1.6, -0.1]$.}
		\end{subfigure}
		\begin{subfigure}{0.31\textwidth}
			\includegraphics[width=\linewidth]{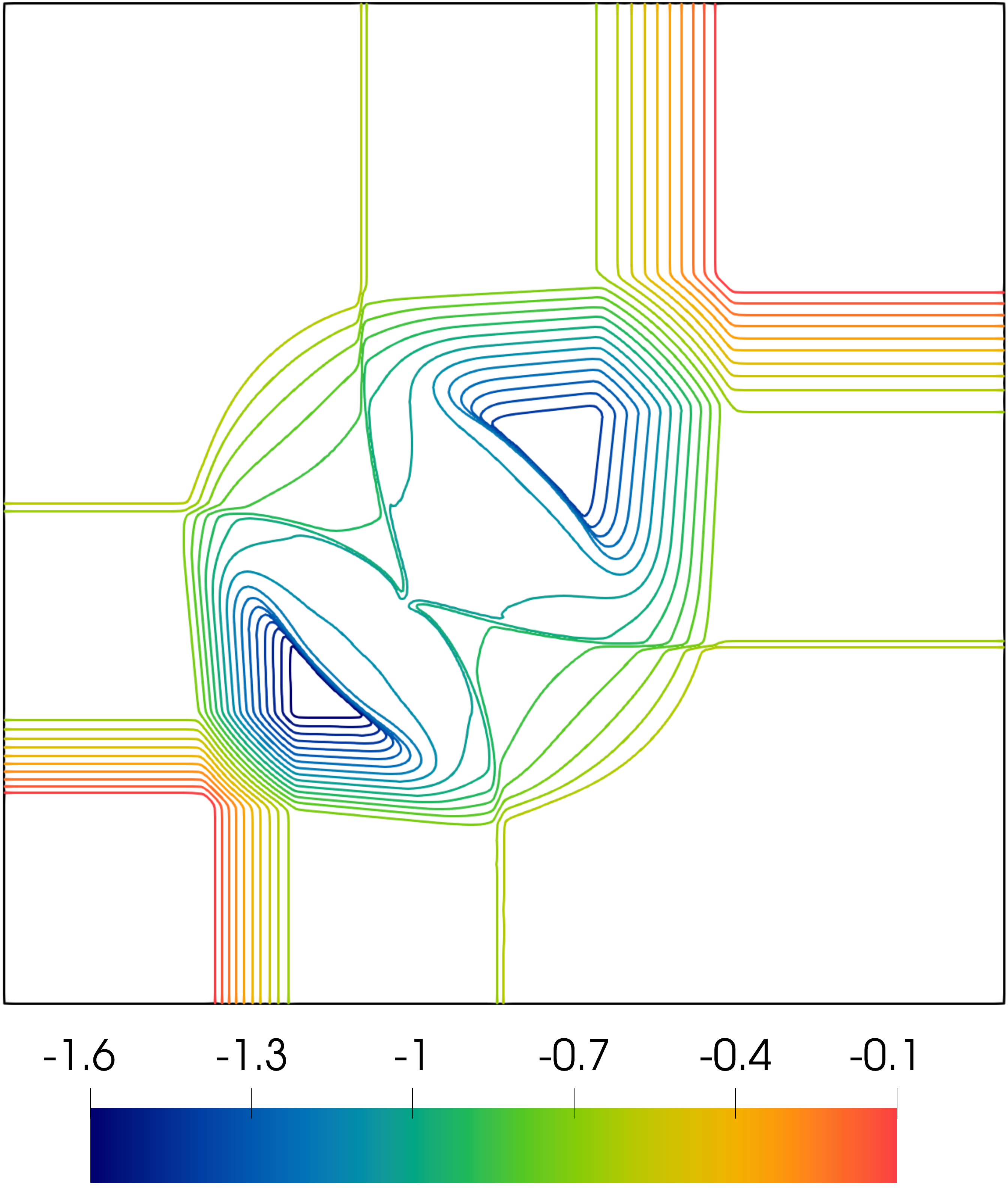}
			\caption{IP-EOS: 25 contours in $[-1.6, -0.1]$.\\}
		\end{subfigure}
		\begin{subfigure}{0.31\textwidth}
			\includegraphics[width=\linewidth]{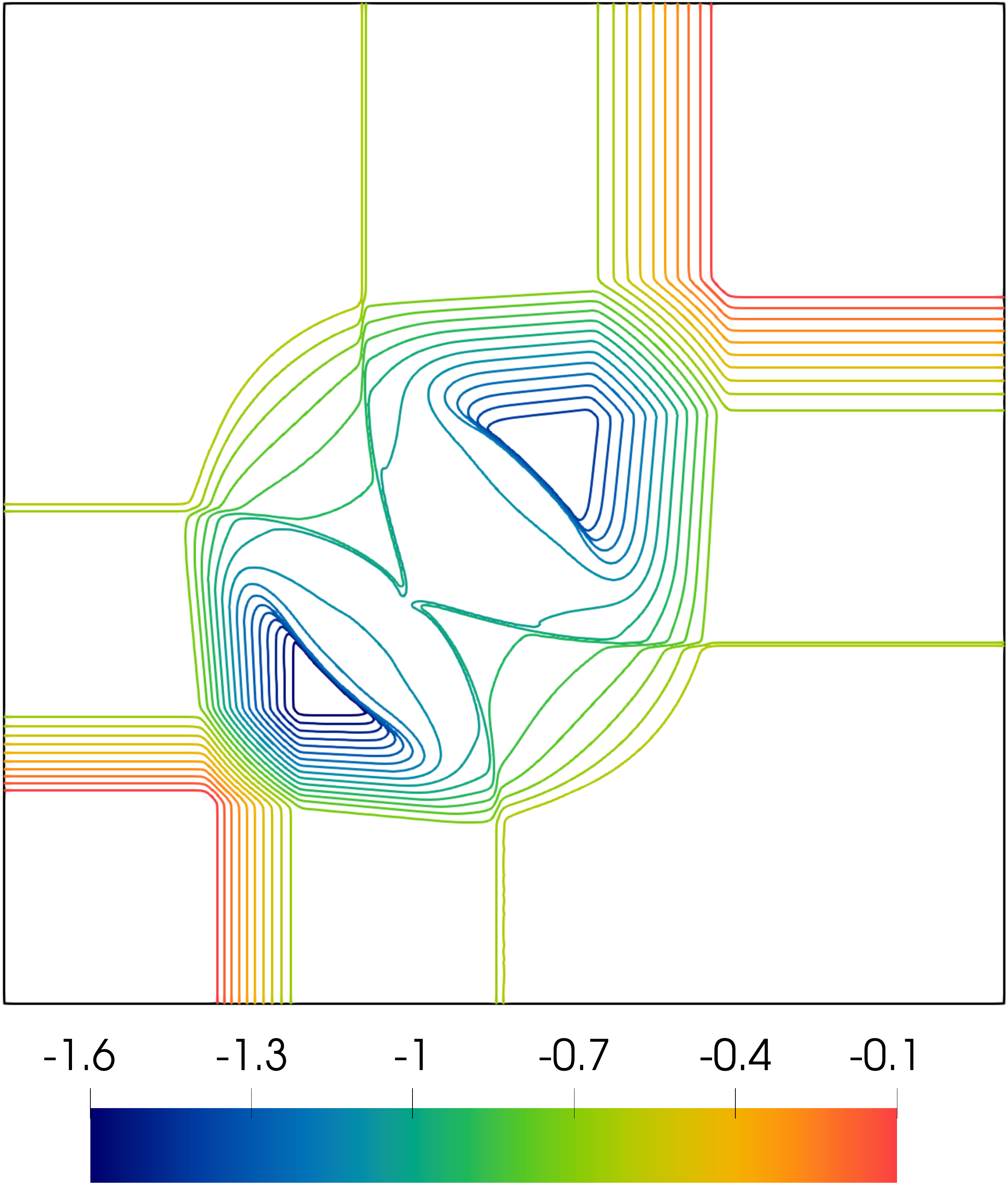}
			\caption{RC-EOS: 25 contours in $[-1.6, -0.1]$.}
		\end{subfigure}
		\begin{subfigure}{0.31\textwidth}
			\includegraphics[width=\linewidth]{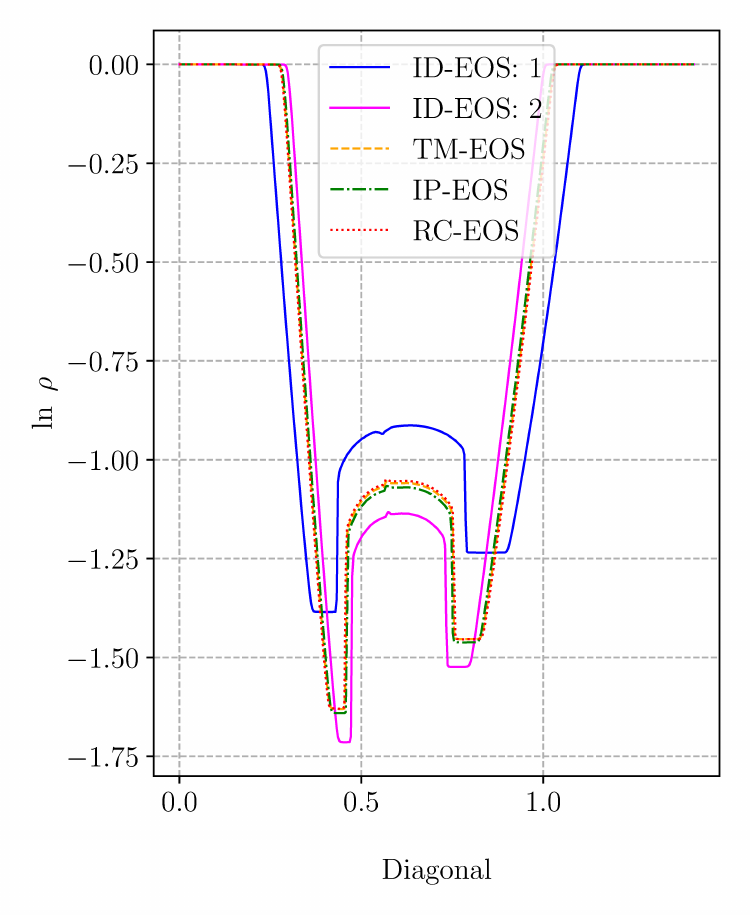}
			\caption{Cut plot from lower-left to upper-right.\\}
		\end{subfigure}
		\vspace{0.2cm}
		\caption{2-D Riemann problem 4: Plot of $\ln \rho$ with $400$ cells and $N=4$.}
		\label{fig:2dwu2rp2.lnden}
	\end{figure}
	\begin{figure}[]
		\centering
		\begin{subfigure}{0.31\textwidth}
			\includegraphics[width=\linewidth]{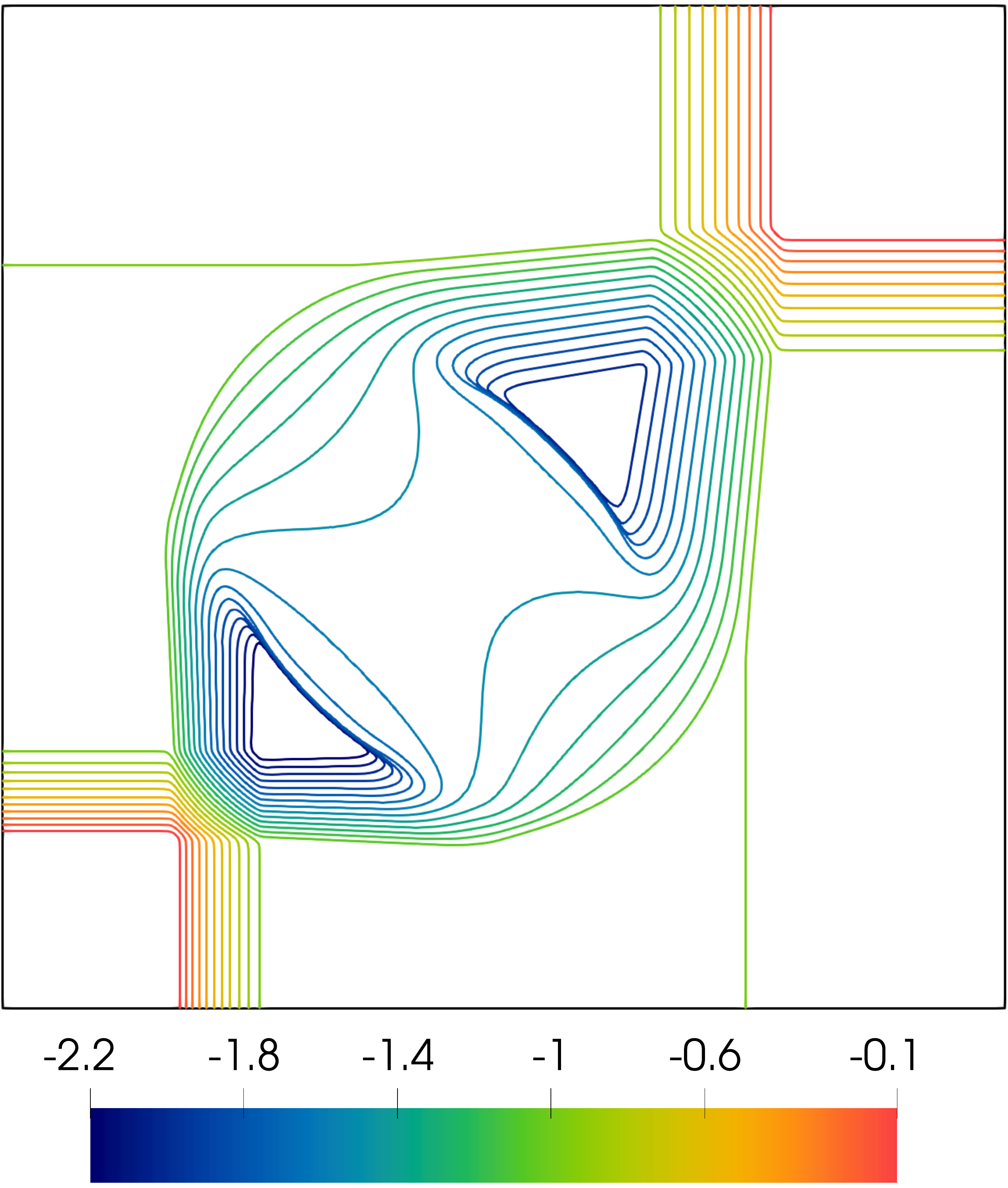}
			\caption{ID-EOS with $\gamma = \frac{5}{3}$: 25 contours in $[-2.2, -0.1]$.}
		\end{subfigure}
		\begin{subfigure}{0.31\textwidth}
			\includegraphics[width=\linewidth]{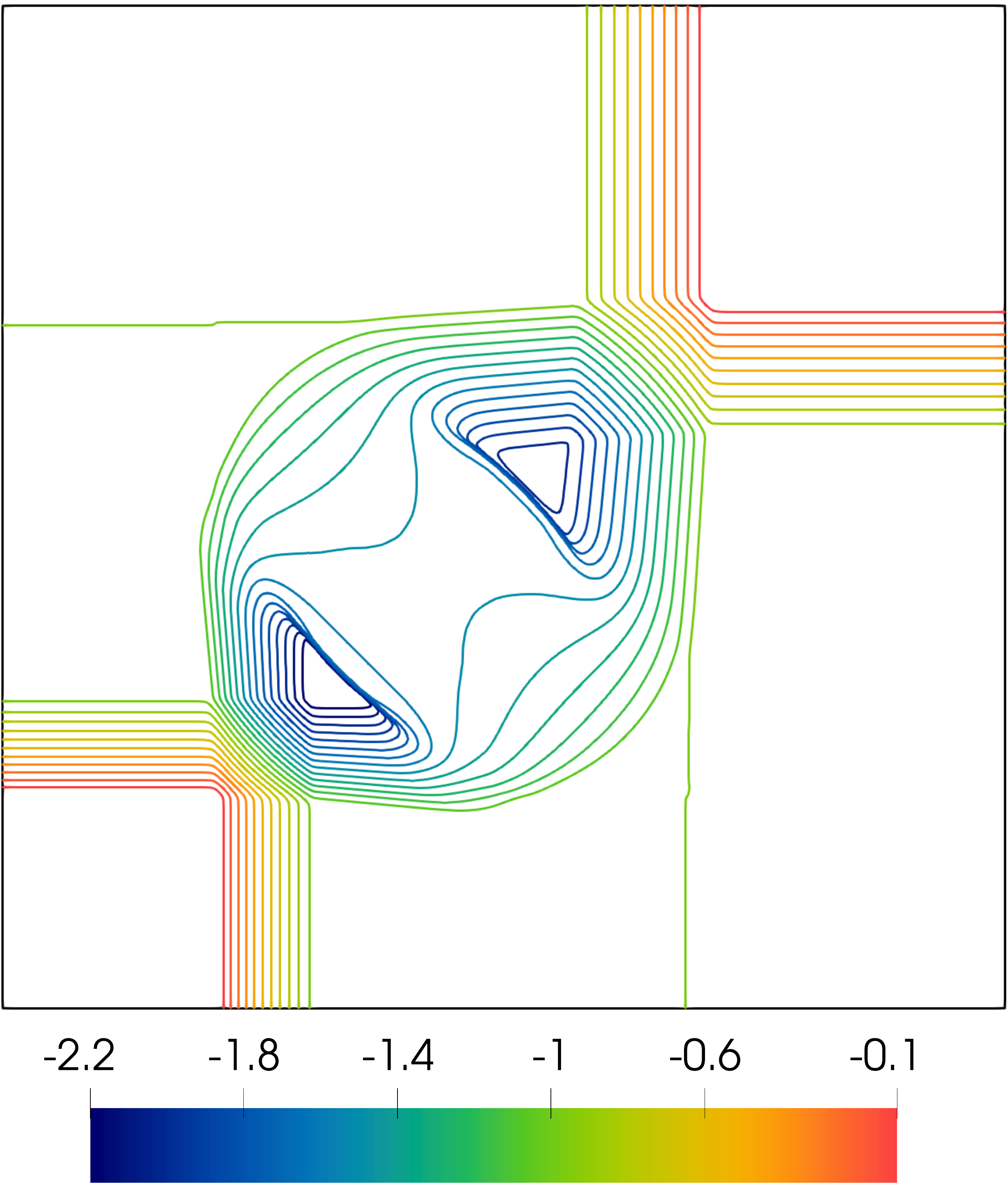}
			\caption{ID-EOS with $\gamma = \frac{4}{3}$: 25 contours in $[-2.2, -0.1]$.}
		\end{subfigure}
		\begin{subfigure}{0.31\textwidth}
			\includegraphics[width=\linewidth]{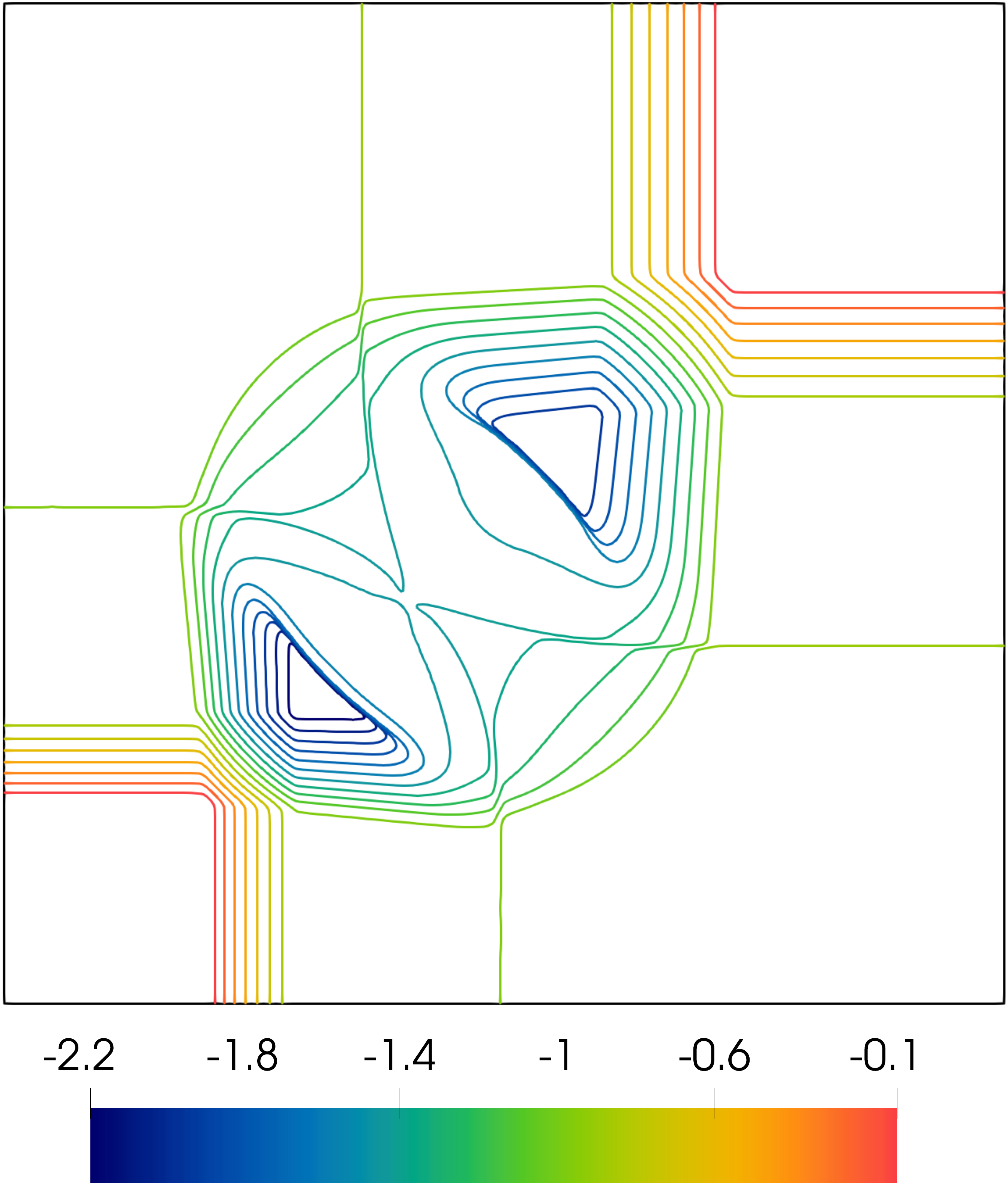}
			\caption{TM-EOS: 25 contours in $[-2.2, -0.1]$.}
		\end{subfigure}
		\begin{subfigure}{0.31\textwidth}
			\includegraphics[width=\linewidth]{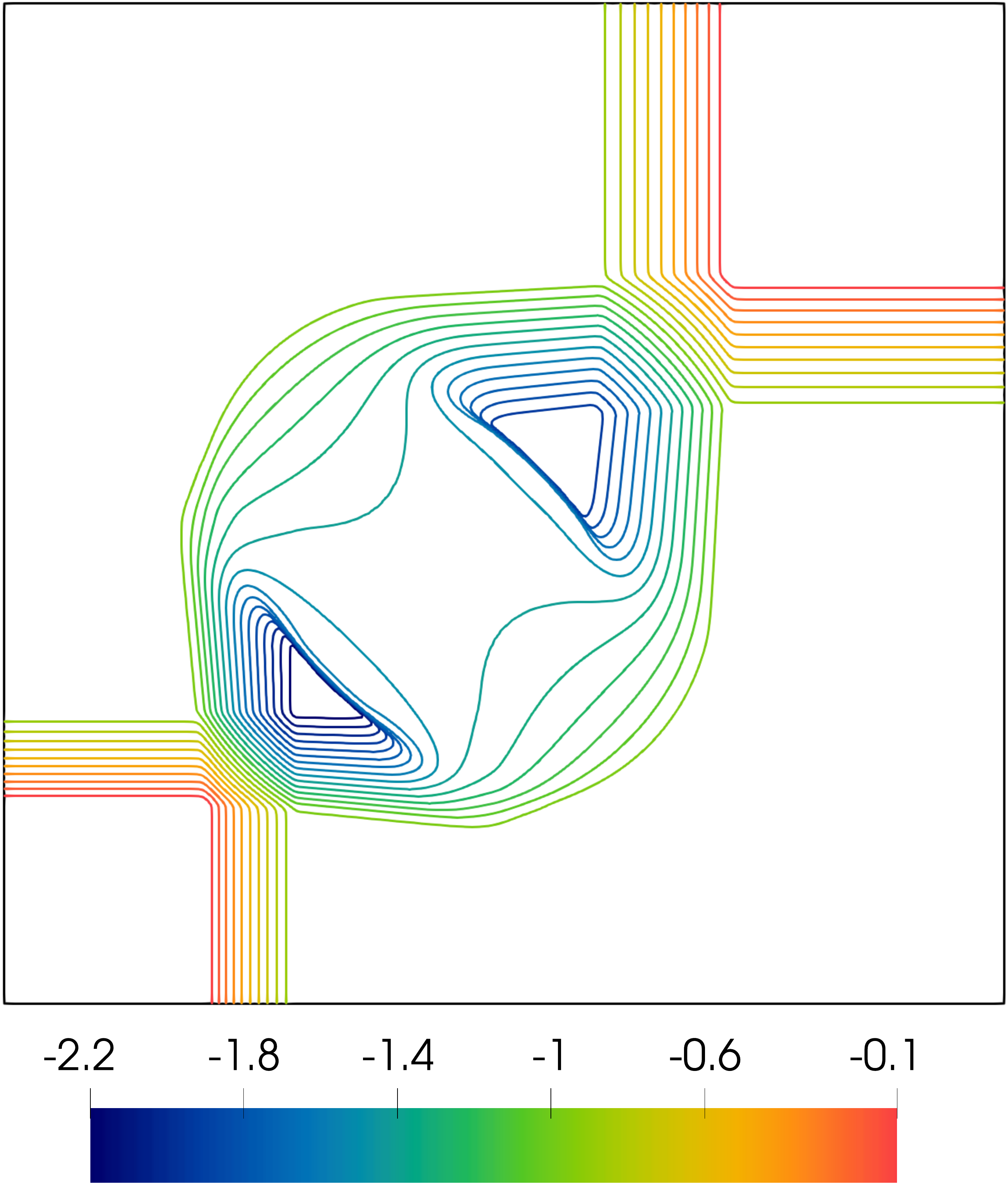}
			\caption{IP-EOS: 25 contours in $[-2.2, -0.1]$.\\}
		\end{subfigure}
		\begin{subfigure}{0.31\textwidth}
			\includegraphics[width=\linewidth]{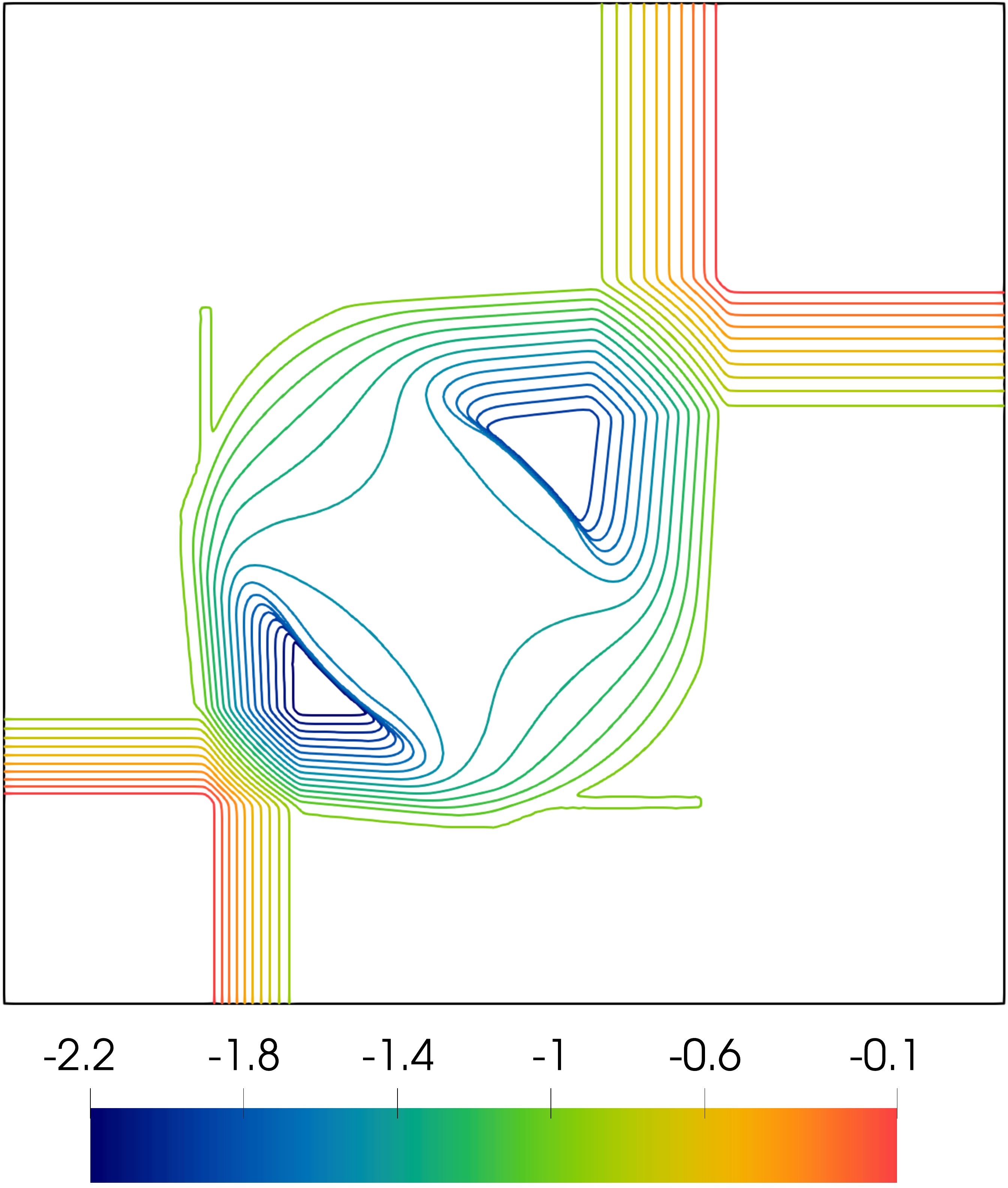}
			\caption{RC-EOS: 25 contours in $[-2.2, -0.1]$.}
		\end{subfigure}
		\begin{subfigure}{0.31\textwidth}
			\includegraphics[width=\linewidth]{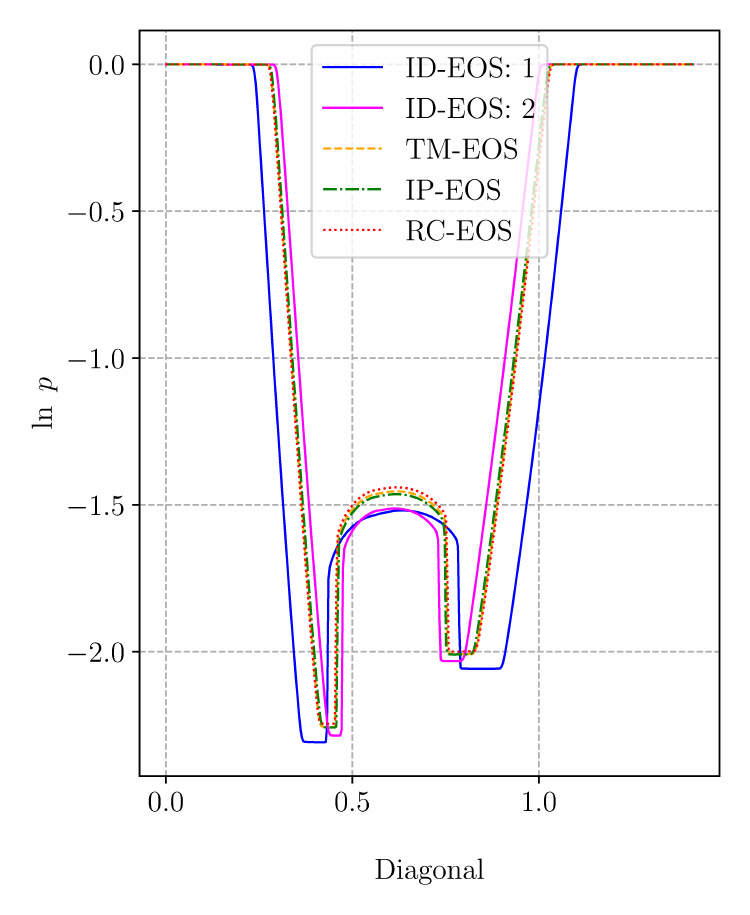}
			\caption{Cut plot from lower-left to upper-right.\\}
		\end{subfigure}
		\caption{2-D Riemann problem 4: Plot of $\ln p$ with $400$ cells and $N=4$.}
		\label{fig:2dwu2rp2.lnpres}
	\end{figure}
	The initial discontinuities evolve to form four rarefaction waves, which later interact and form two symmetric shock waves. We observe that the scheme can capture the shock waves for all the equations of state effectively. We have also compared the results along a cut from the lower-left to the upper-right corner of the domain.
	
	\subsubsection{2-D Riemann problem 5}
	This problem is also taken from~\cite{he2012adaptive}, where the authors have simulated it with ID-EOS. The initial state for this Riemann problem also has four constant states in the four quadrants of the domain $[0,1]\times [0,1]$ as below,
	\begin{align*}
		(\rho&, v_1, v_2, p)\\
		&= \begin{cases}
			(0.035145216124503, 0, 0, 0.162931056509027) & \text{if}\ x > 0.5,\ y > 0.5\\
			(0.1, 0.7, 0, 1) & \text{if}\ x < 0.5,\ y>0.5\\
			(0.5, 0, 0, 1) & \text{if}\ x < 0.5,\ y < 0.5\\
			(0.1, 0, 0.7, 1) & \text{if}\ x > 0.5,\ y<0.5.
		\end{cases}
	\end{align*}
	This problem was also used in~\cite{nunez2016xtroem,xu2024high} to verify numerical schemes. We run the simulations for this problem till time $t=0.4$ with $400\times 400$ cells and $N=4$, taking the boundaries of the domain as outflow boundaries. The results of the simulations are presented in Figure~\ref{fig:2dwu2rp3.lnden} and Figure~\ref{fig:2dwu2rp3.lnpres}.
	\begin{figure}[]
		\centering
		\begin{subfigure}{0.31\textwidth}
			\includegraphics[width=\linewidth]{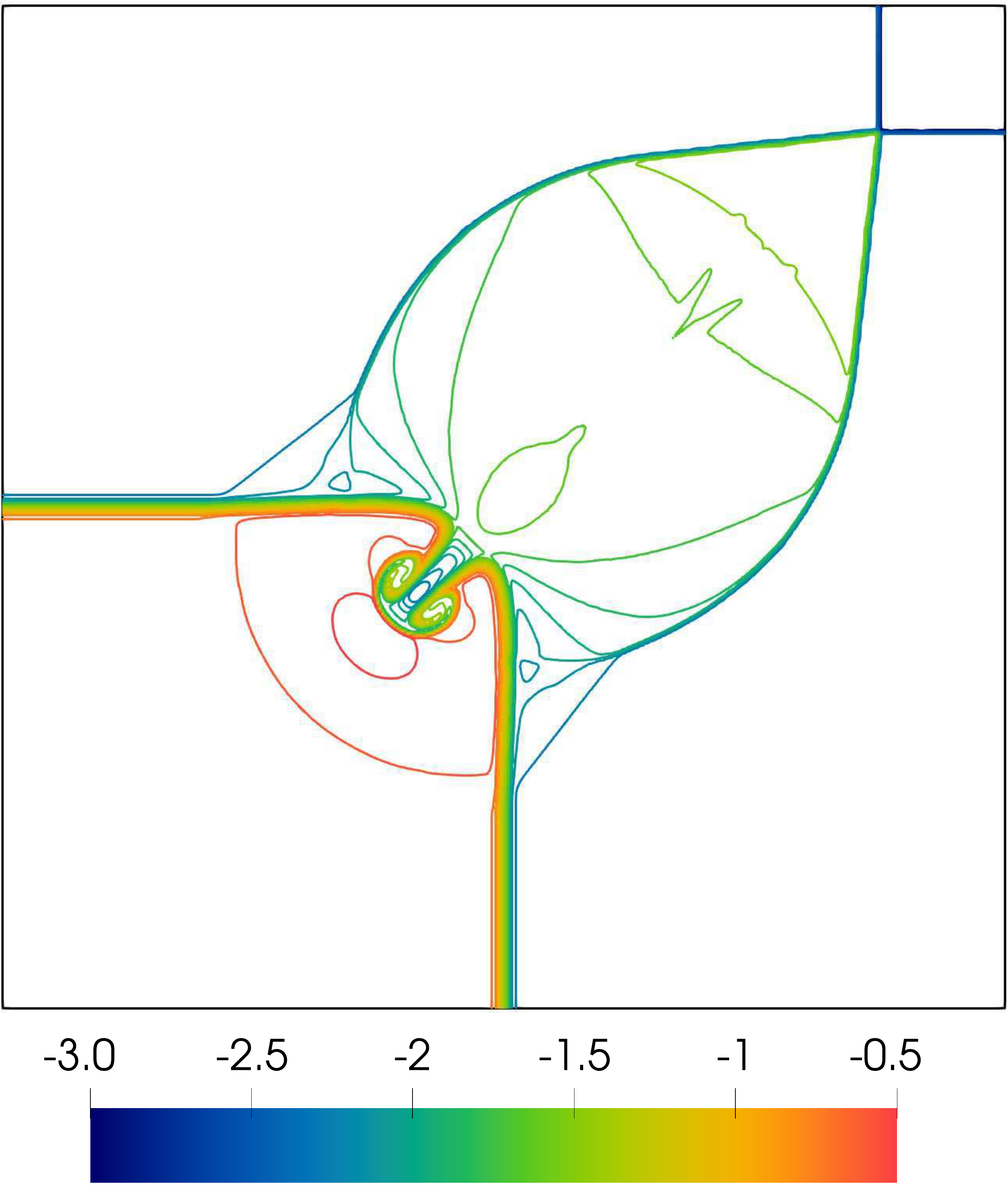}
			\caption{ID-EOS with $\gamma = \frac{5}{3}$: 25 contours in $[-3.0, -0.5]$.}
		\end{subfigure}
		\begin{subfigure}{0.31\textwidth}
			\includegraphics[width=\linewidth]{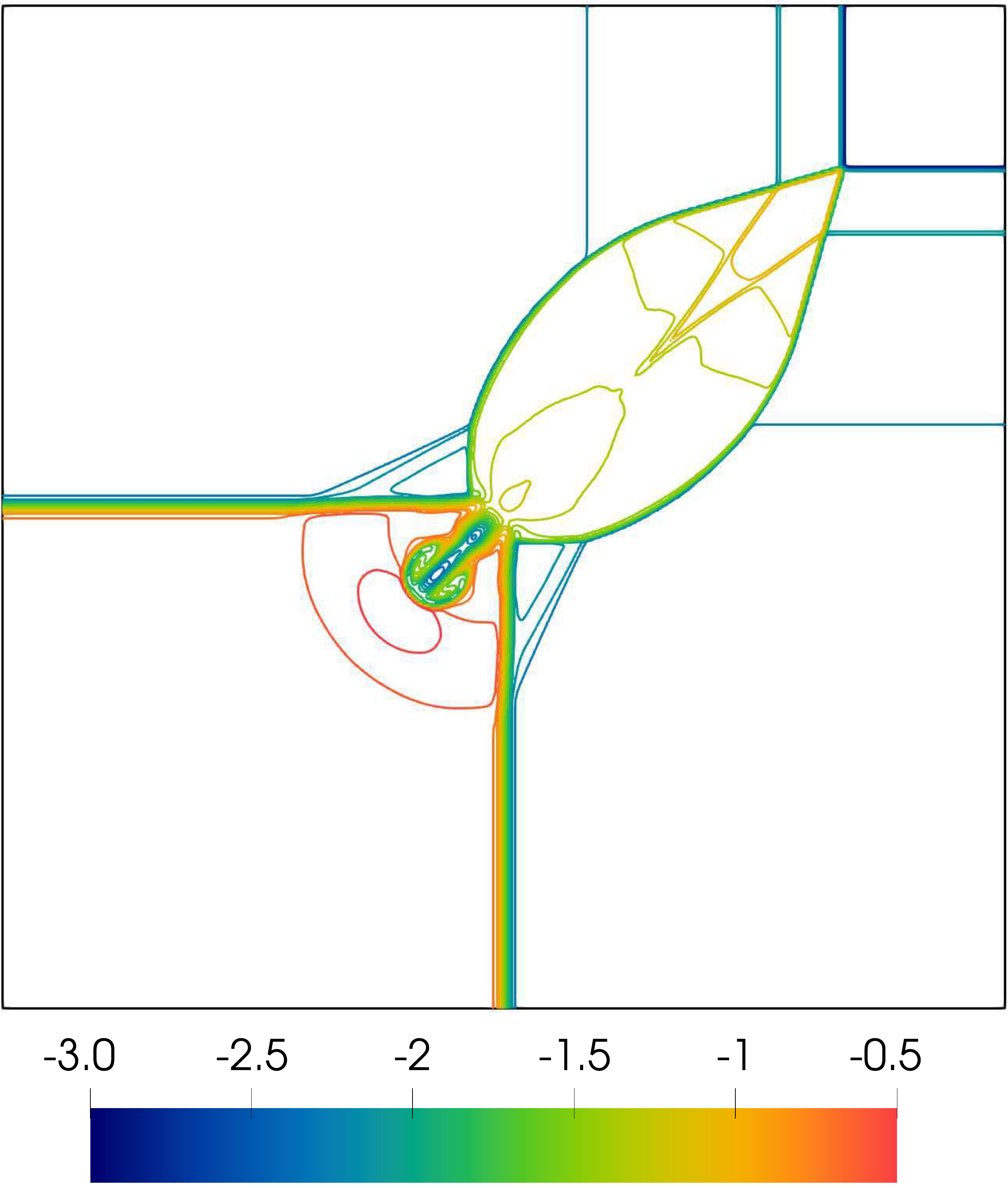}
			\caption{ID-EOS with $\gamma = \frac{4}{3}$: 25 contours in $[-3.0, -0.5]$.}
		\end{subfigure}
		\begin{subfigure}{0.31\textwidth}
			\includegraphics[width=\linewidth]{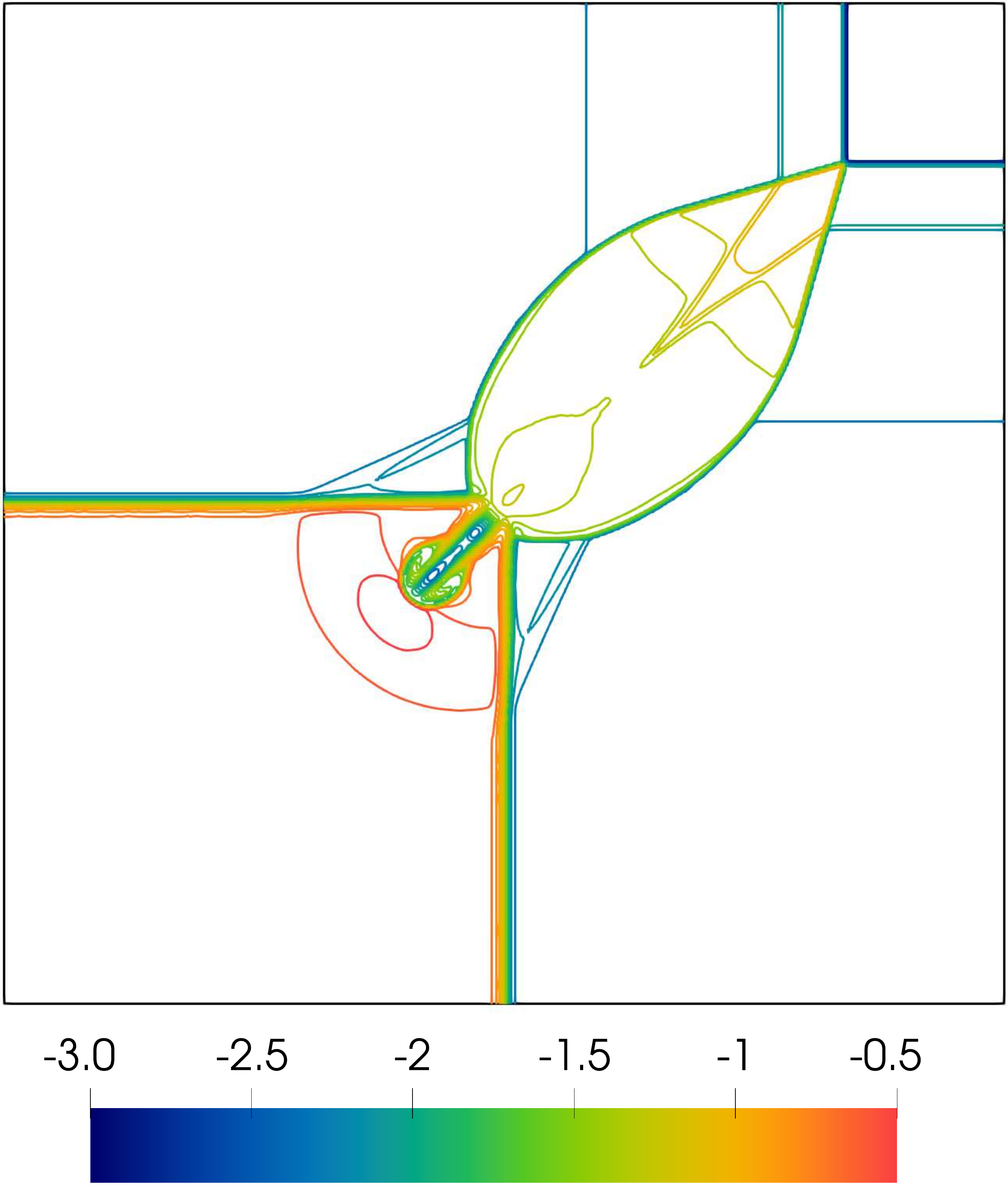}
			\caption{TM-EOS: 25 contours in $[-3.0, -0.5]$.}
		\end{subfigure}    
		\begin{subfigure}{0.31\textwidth}
			\includegraphics[width=\linewidth]{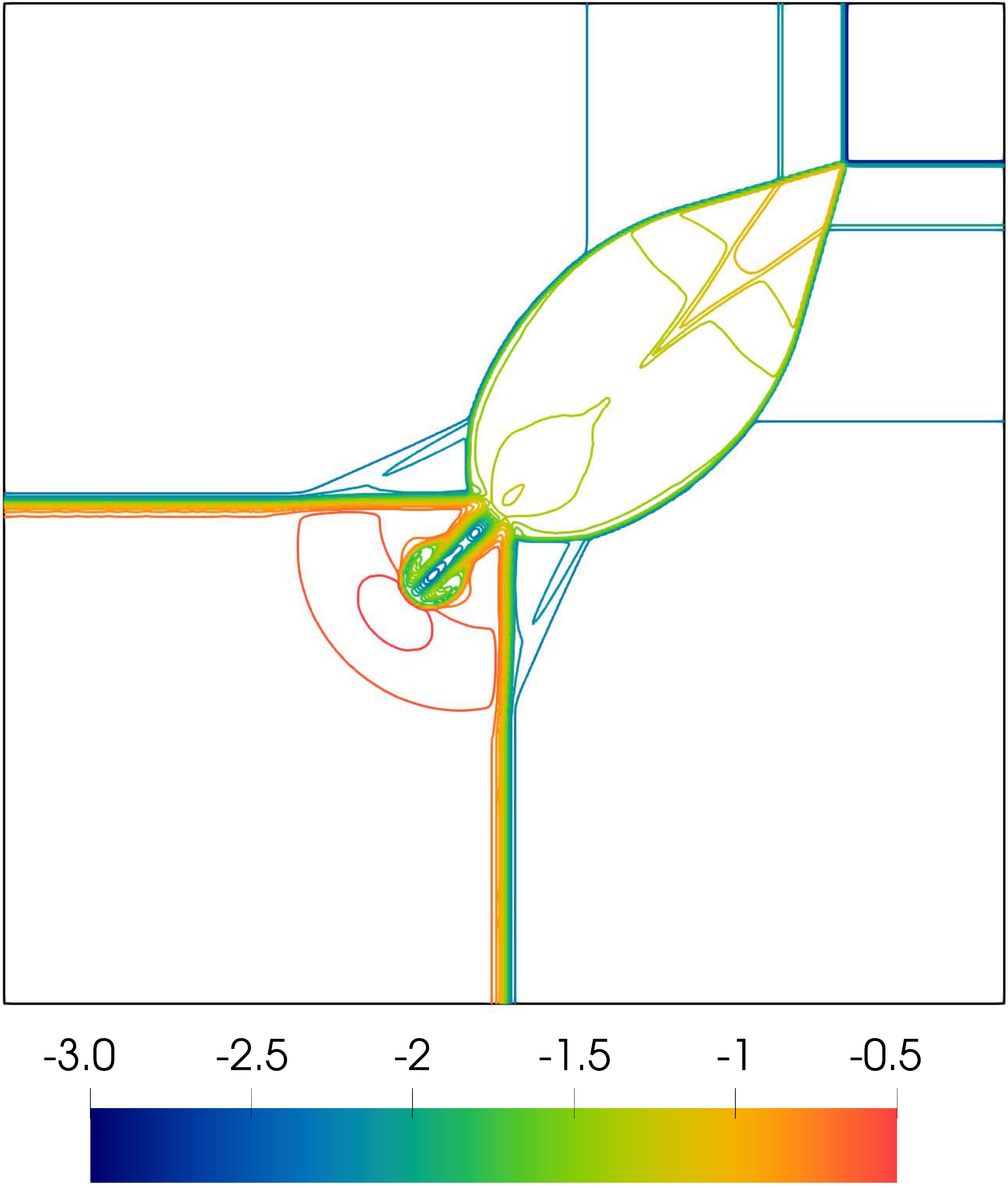}
			\caption{IP-EOS: 25 contours in $[-3.0, -0.5]$.\\}
		\end{subfigure}
		\begin{subfigure}{0.31\textwidth}
			\includegraphics[width=\linewidth]{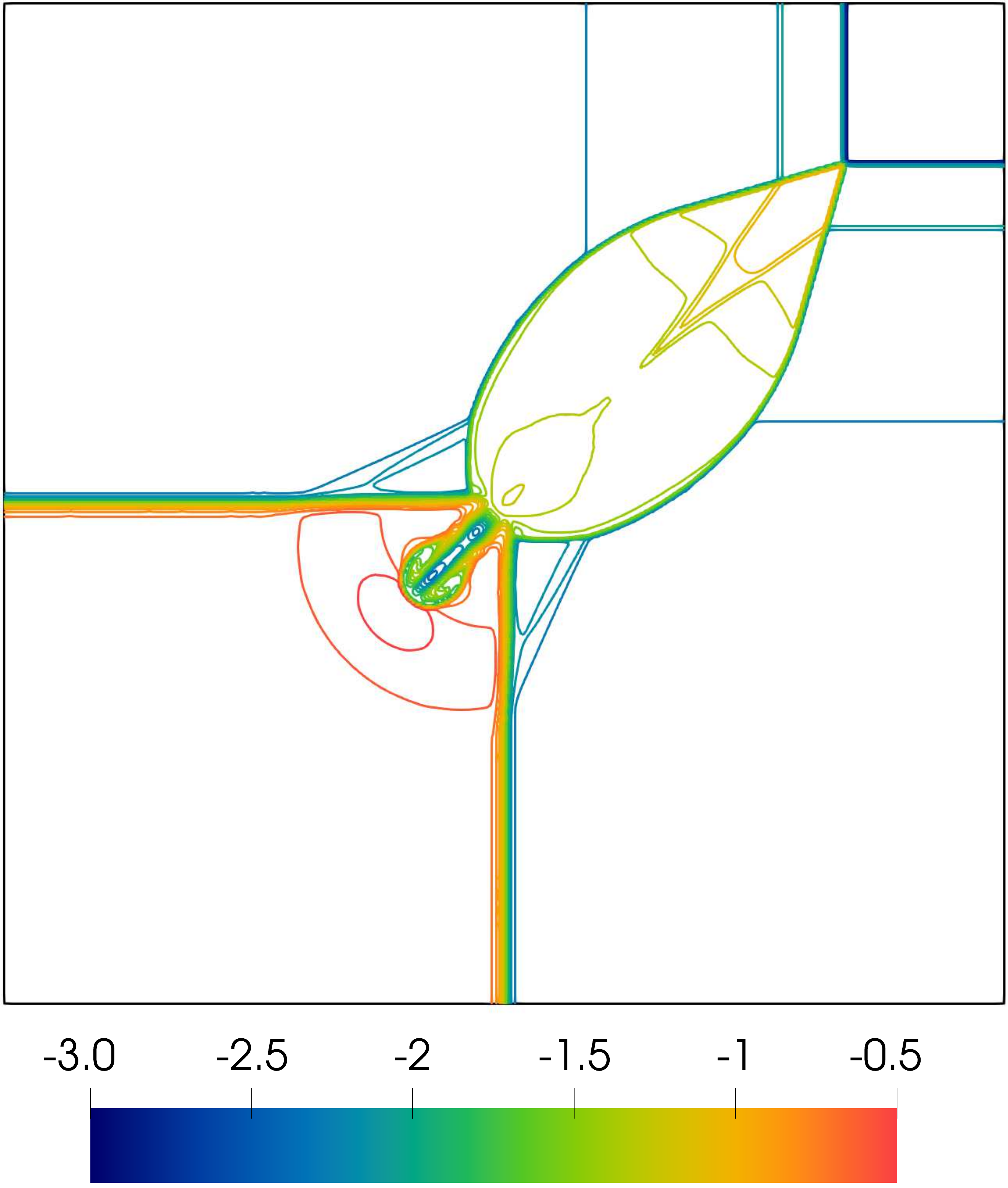}
			\caption{RC-EOS: 25 contours in $[-3.0, -0.5]$.}
		\end{subfigure}
		\begin{subfigure}{0.31\textwidth}
			\includegraphics[width=\linewidth]{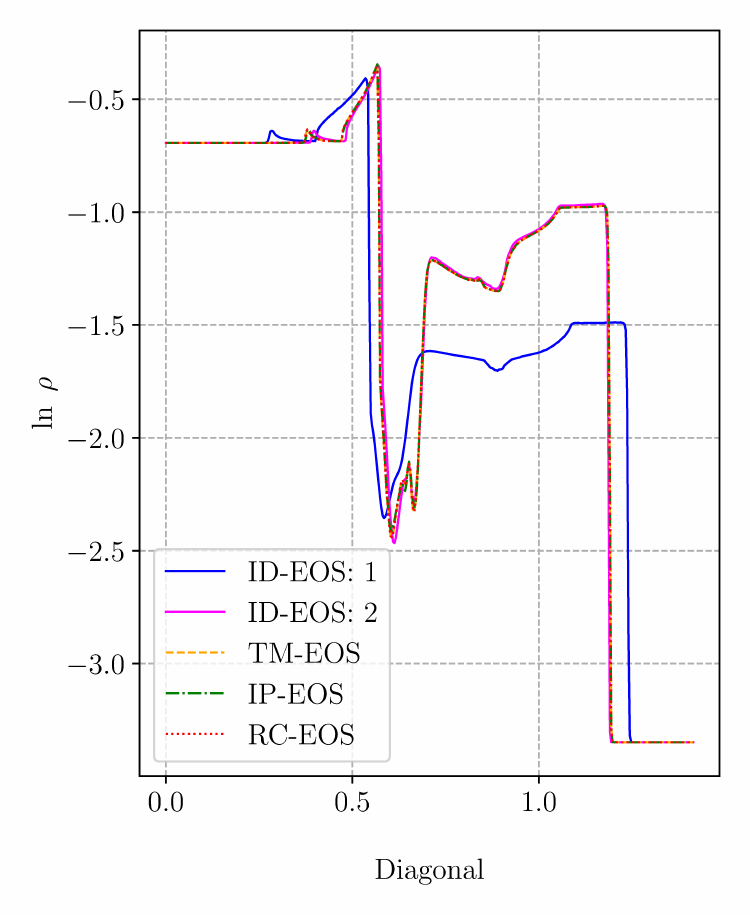}
			\caption{Cut plot from lower-left to upper-right.\\}
		\end{subfigure}
		\vspace{0.2cm}
		\caption{2-D Riemann problem 5: Plot of $\ln \rho$ with $400$ cells and $N=4$.}
		\label{fig:2dwu2rp3.lnden}
	\end{figure}
	\begin{figure}[]
		\centering
		\begin{subfigure}{0.31\textwidth}
			\includegraphics[width=\linewidth]{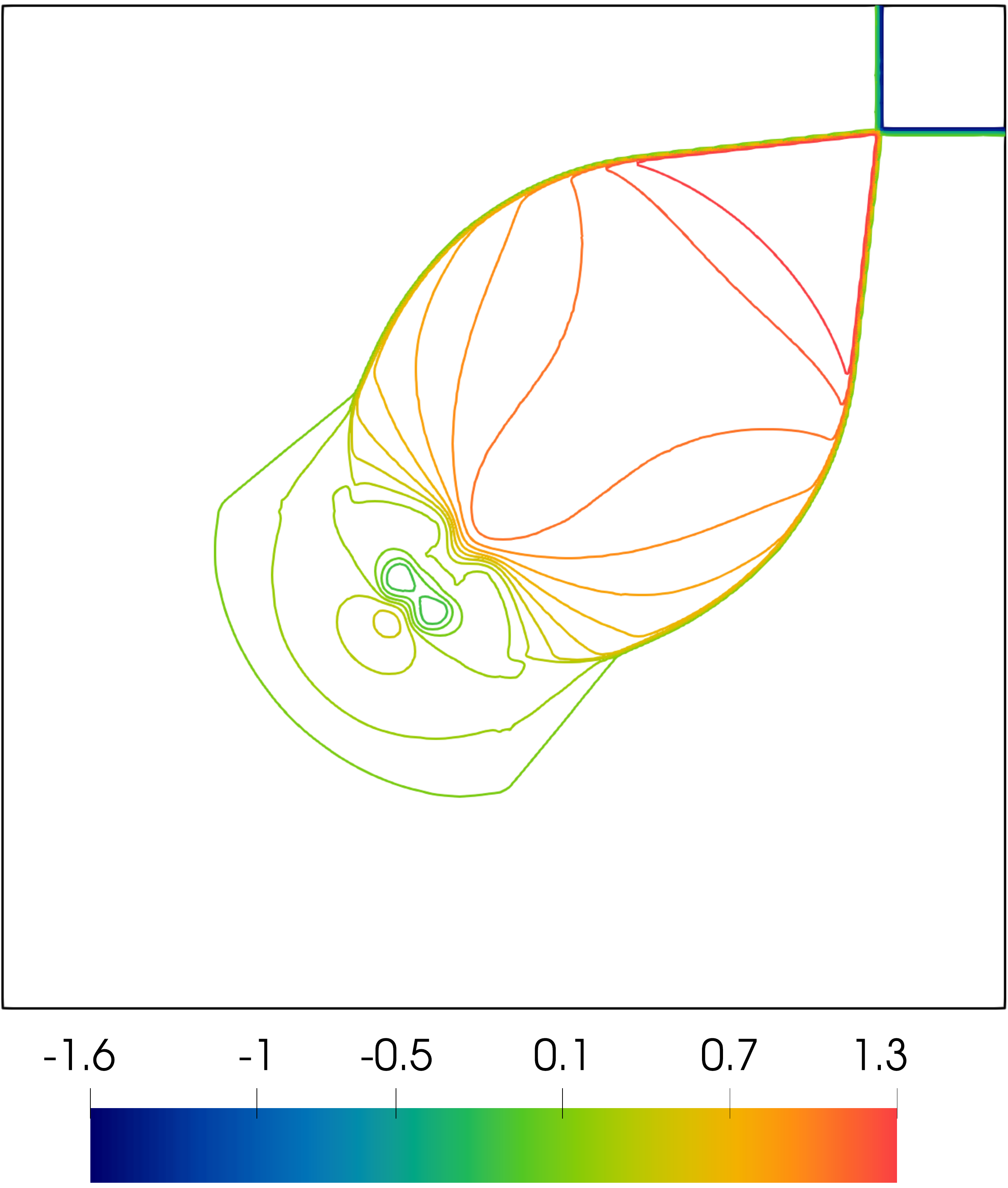}
			\caption{ID-EOS with $\gamma = \frac{5}{3}$: 25 contours in $[-1.6, 1.3]$.}
		\end{subfigure}
		\begin{subfigure}{0.31\textwidth}
			\includegraphics[width=\linewidth]{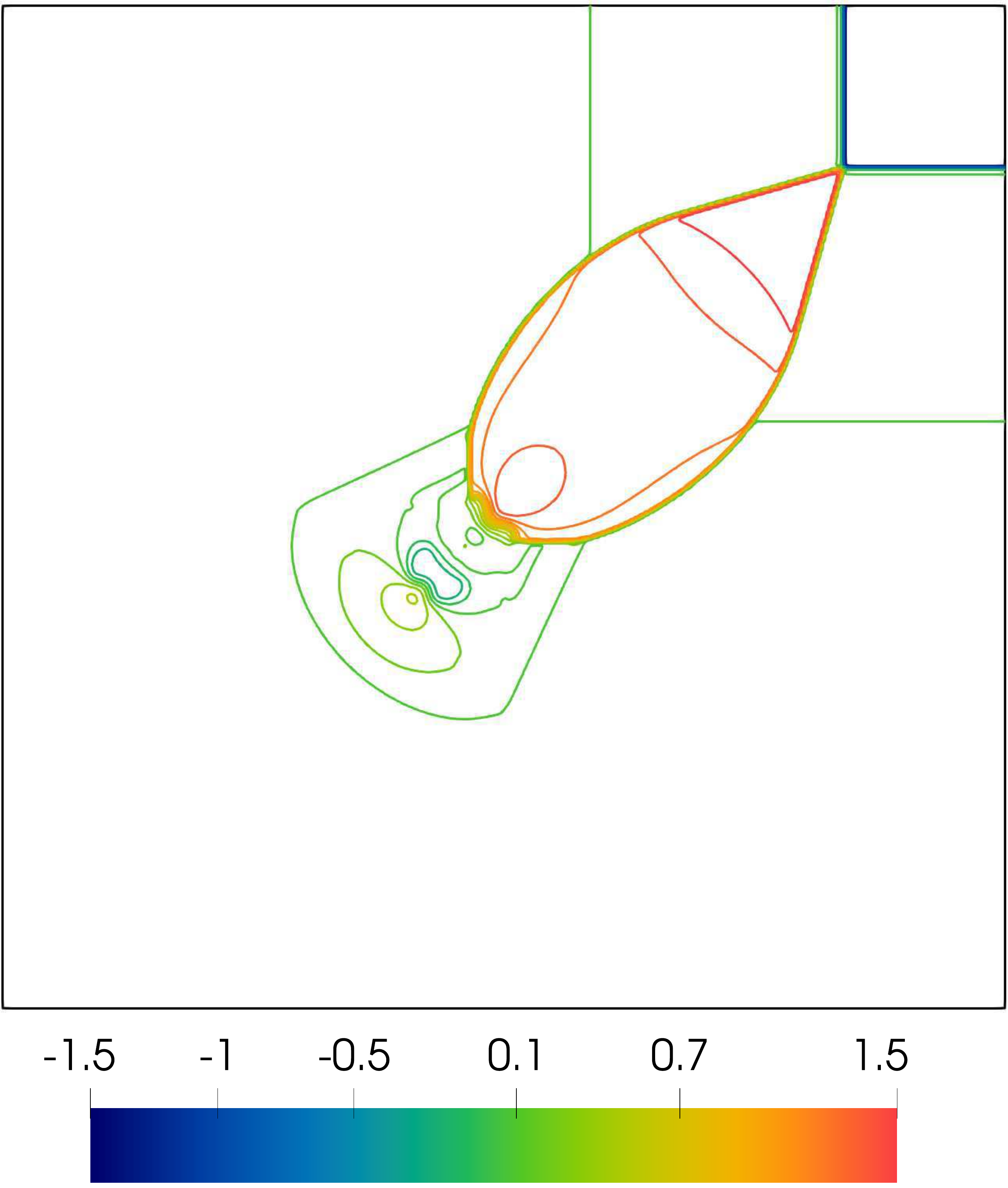}
			\caption{ID-EOS with $\gamma = \frac{4}{3}$: 25 contours in $[-1.5, 1.5]$.}
		\end{subfigure}
		\begin{subfigure}{0.31\textwidth}
			\includegraphics[width=\linewidth]{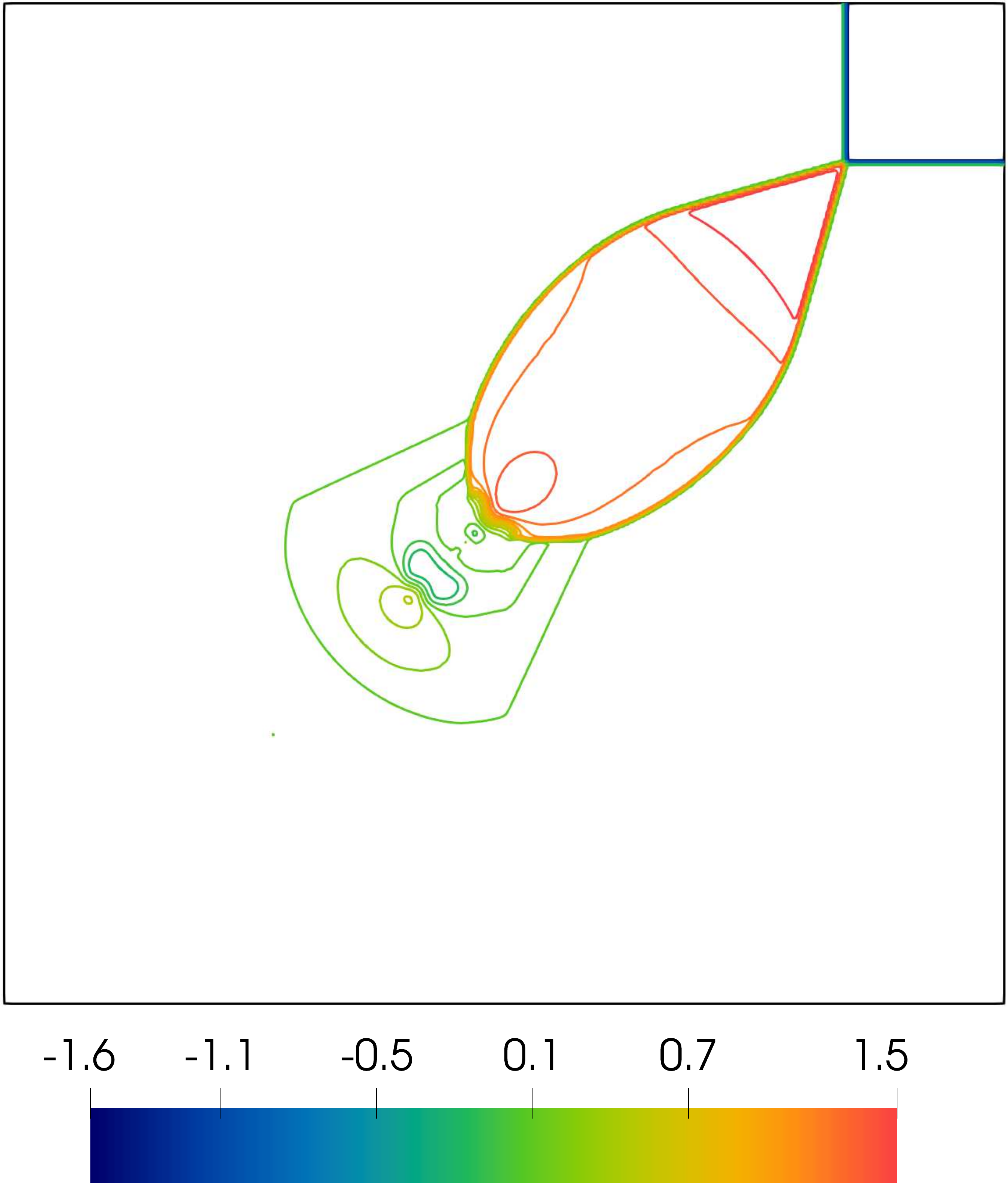}
			\caption{TM-EOS: 25 contours in $[-1.6, 1.5]$.\\}
		\end{subfigure}
		\begin{subfigure}{0.31\textwidth}
			\includegraphics[width=\linewidth]{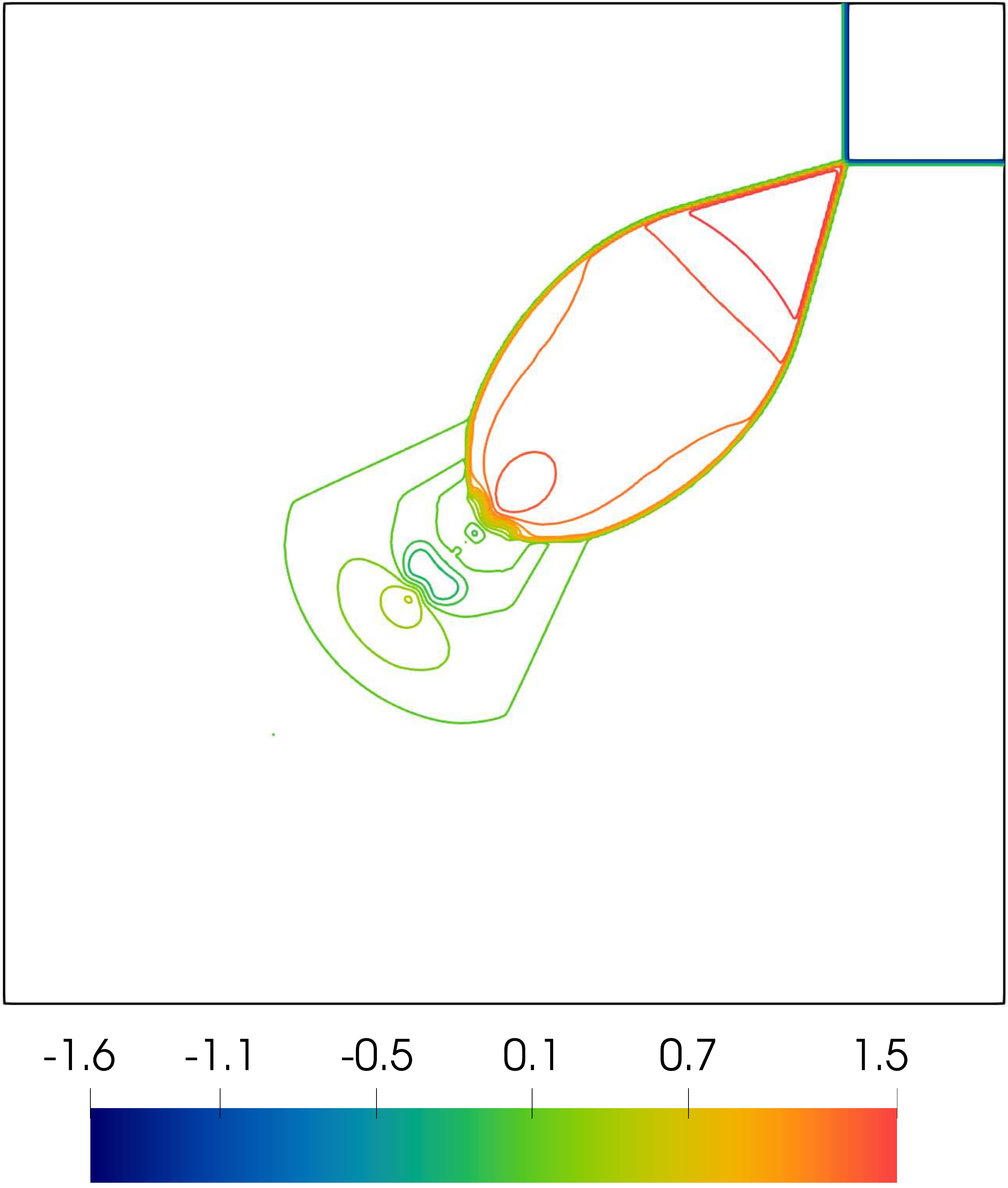}
			\caption{IP-EOS: 25 contours in $[-1.6, 1.5]$.}
		\end{subfigure}
		\begin{subfigure}{0.31\textwidth}
			\includegraphics[width=\linewidth]{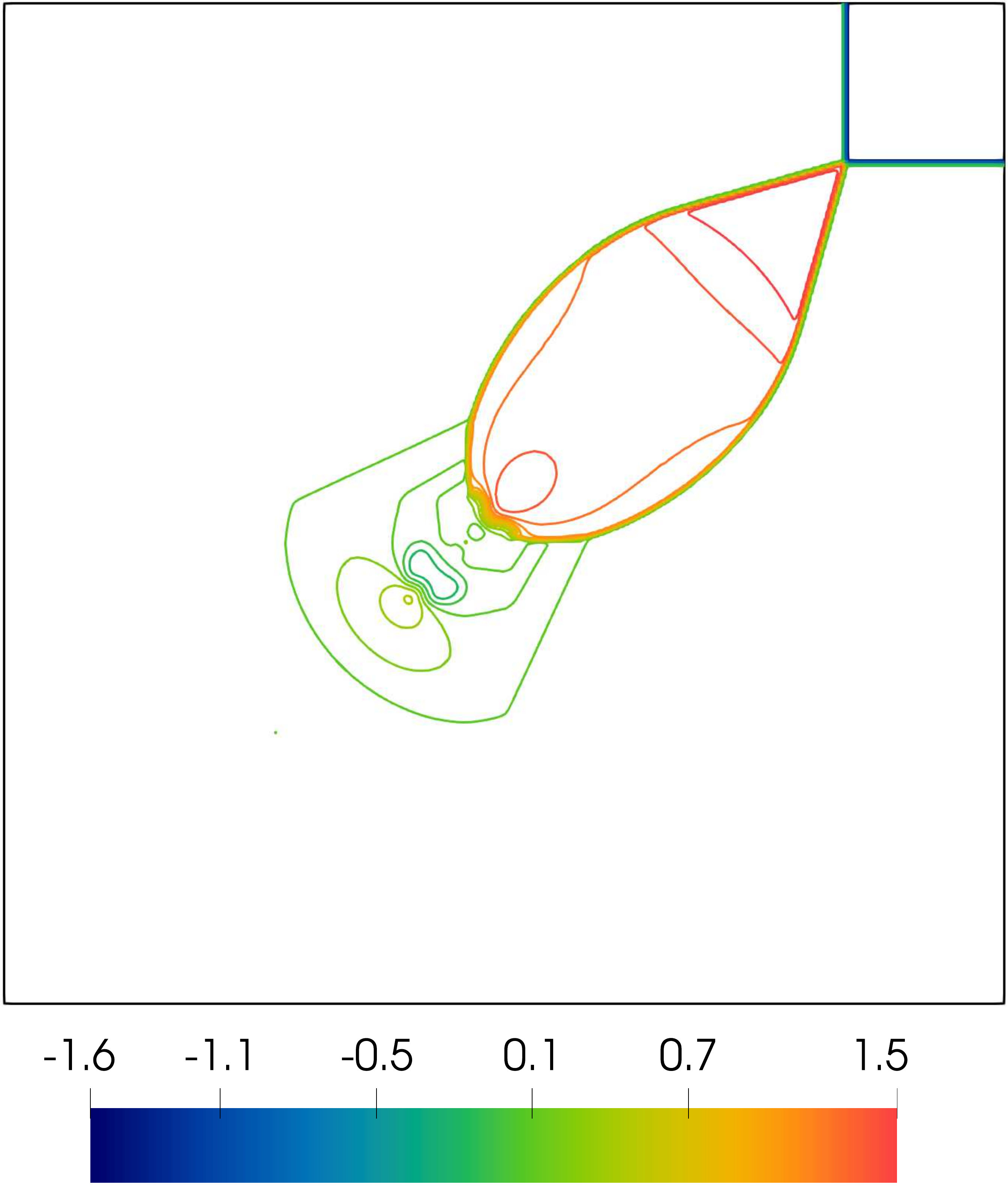}
			\caption{RC-EOS: 25 contours in $[-1.6, 1.5]$.}
		\end{subfigure}
		\begin{subfigure}{0.31\textwidth}
			\includegraphics[width=\linewidth]{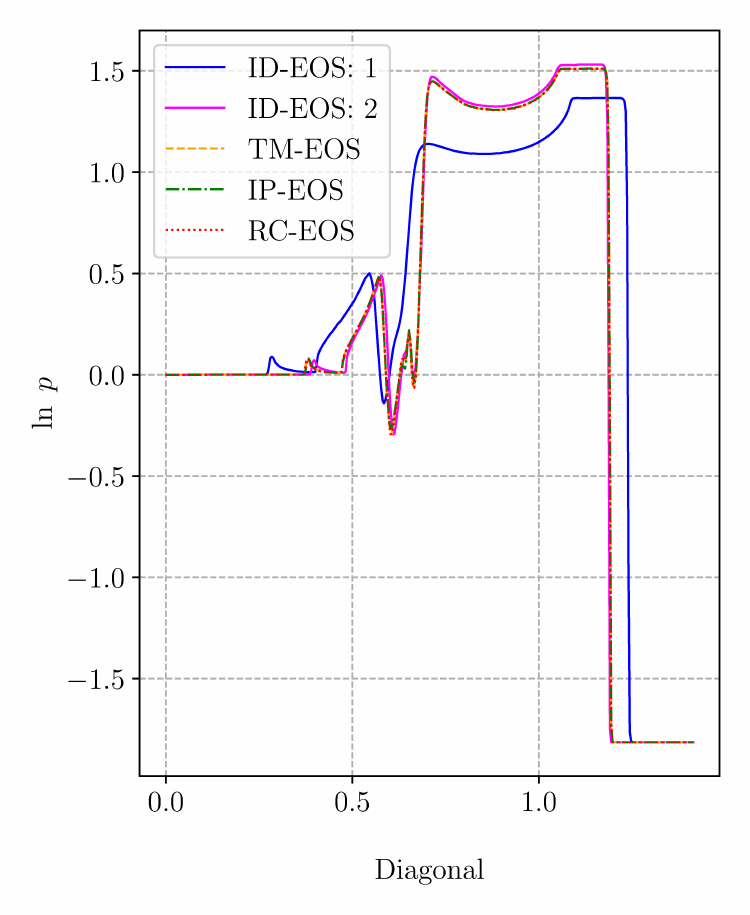}
			\caption{Cut plot from lower-left to upper-right.}
		\end{subfigure}
		\vspace{0.2cm}
		\caption{2-D Riemann problem 5: Plot of $\ln p$ with $400$ cells and $N=4$.}
		\label{fig:2dwu2rp3.lnpres}
	\end{figure}
	
	As time progresses, all the discontinuities interact with each other and a mushroom-like structure gets formed in the solution, which is captured by the scheme effectively. We can also observe from the figures that the scheme can capture the contact discontinuities and the curved shock waves in the solution. Similar to some of the other Riemann problems, here also we can observe the similarity among the results using ID-EOS with $\gamma = \frac{4}{3}$, TM-EOS, IP-EOS, and RC-EOS.
	
	\subsubsection{2-D relativistic jet}
	We now consider a test case from~\cite{wu2016physical}, which has a very high-speed jet with velocity near the speed of light. This is a good test to check the robustness of the scheme as it has strong relativistic shock wave, shear wave, interface instabilities, and ultra-relativistic region in the solution. The simulations are run with different equations of state using the scheme with $480\times 500$ cells and $N=4$ in the domain $[-12, 12]\times [0, 30]$ with outflow boundaries except the part $\{(x,y): |x|<0.5, y=0\}$, where we have used an inflow boundary condition with fluid density $\rho = 0.01$ and velocity in $y$-direction as $v_y = 0.9999$. The pressure of the inflow beam is calculated from the classical Mach number $1.74$, and the same pressure is taken in the rest of the domain initially, where the fluid is at rest with unit density. The safety factor $l_s = 0.7$ is taken for the case of ID-EOS with $\gamma = \frac{4}{3}$.
	
	The results of the simulations are presented in Figure~\ref{fig:wureljet.lnden} and Figure~\ref{fig:wureljet.lnpres} for all the equations of state at time $t=30$. We observe that the scheme can capture the Mach shock wave at the beam head effectively for all the cases. The scheme also effectively captures all the other waves formed in the domain because of the high-speed inflow beam.
	
	\begin{figure}[!htbp]
		\centering
		\begin{subfigure}{0.31\textwidth}    \includegraphics[width=\linewidth]{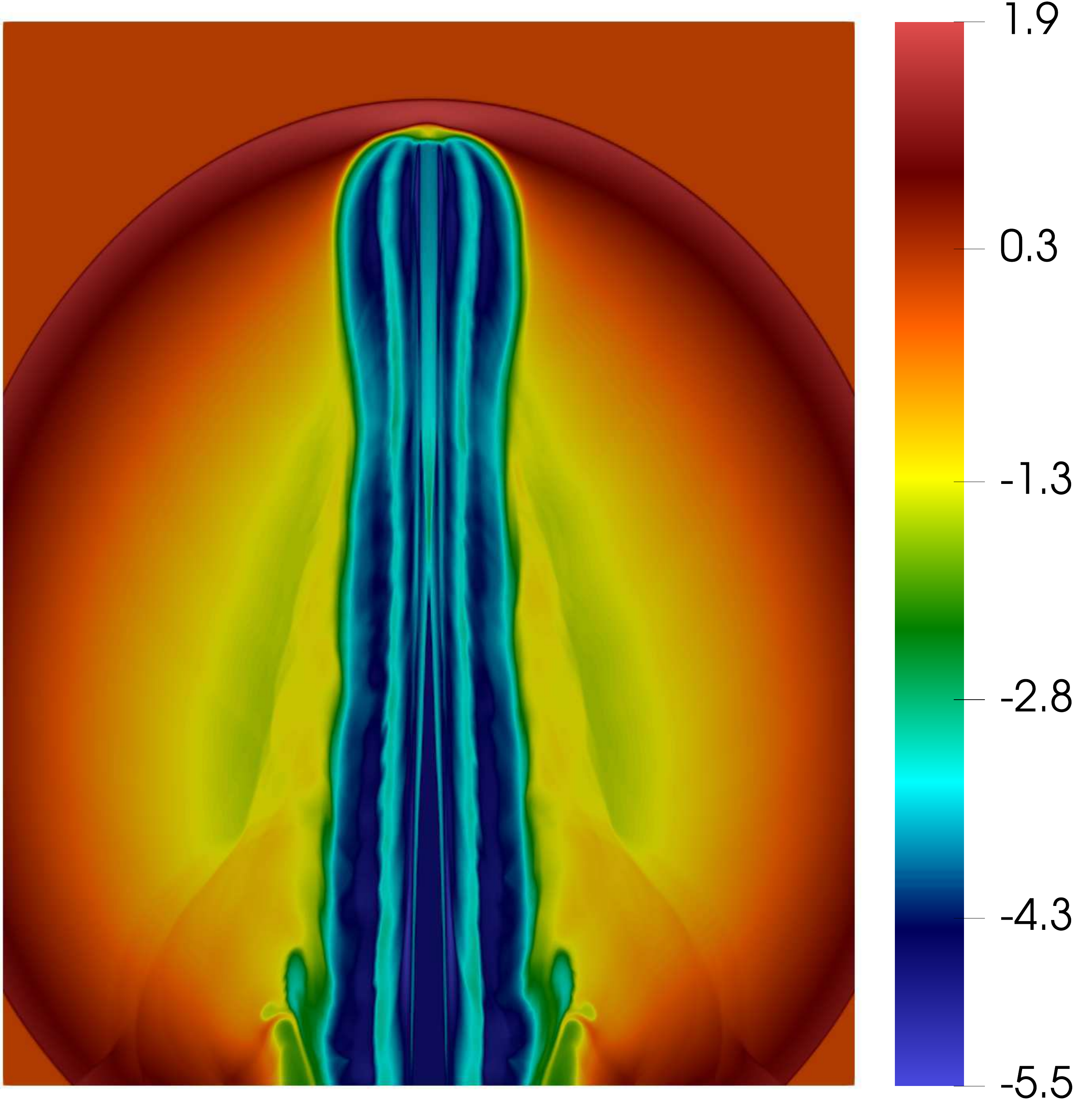}
			\caption{ID-EOS with $\gamma = \frac{5}{3}$.}
		\end{subfigure}
		\begin{subfigure}{0.31\textwidth}
			\includegraphics[width=\linewidth]{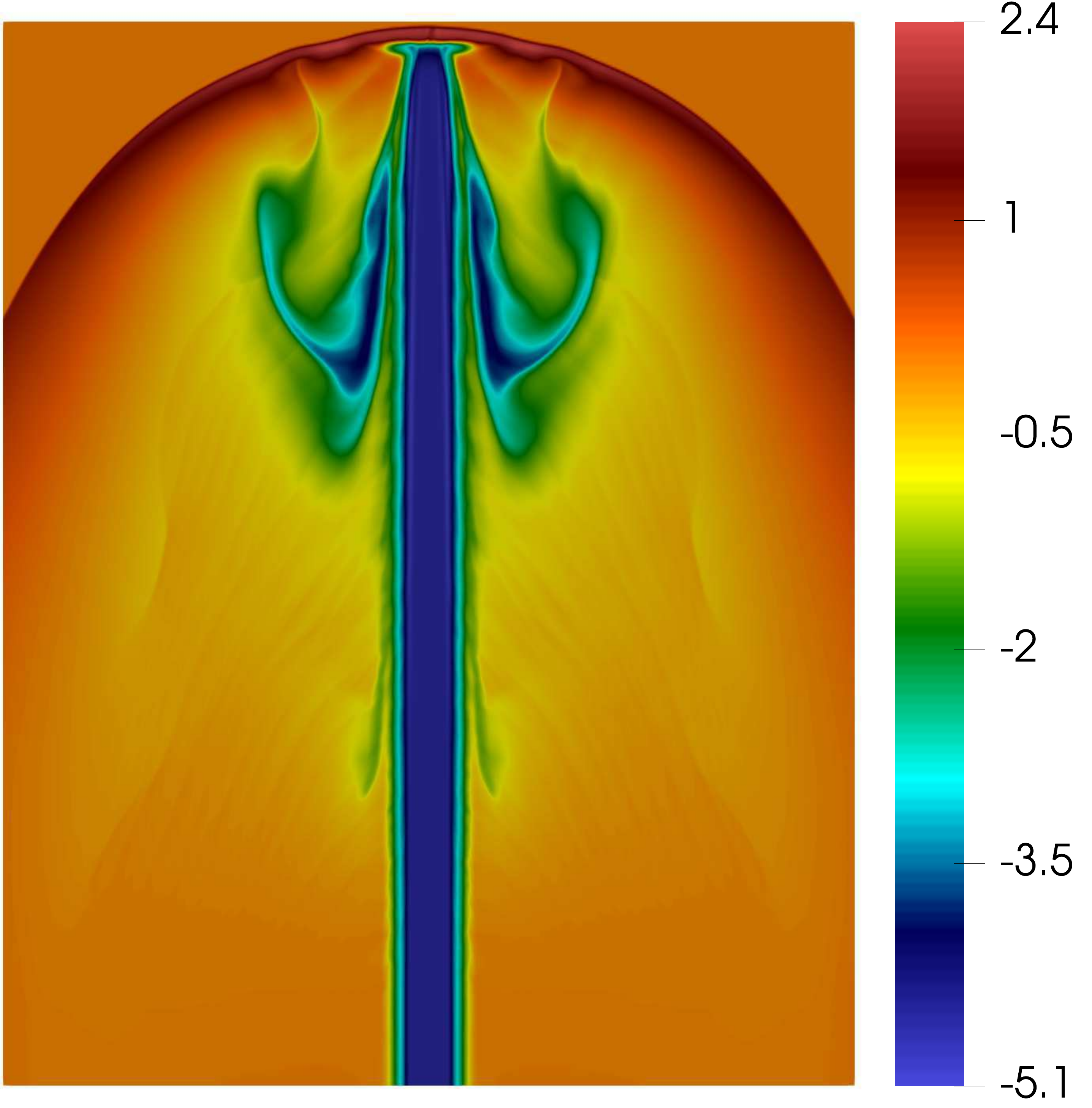}
			\caption{ID-EOS with $\gamma = \frac{4}{3}$.}
		\end{subfigure}
		\begin{subfigure}{0.31\textwidth}
			\includegraphics[width=\linewidth]{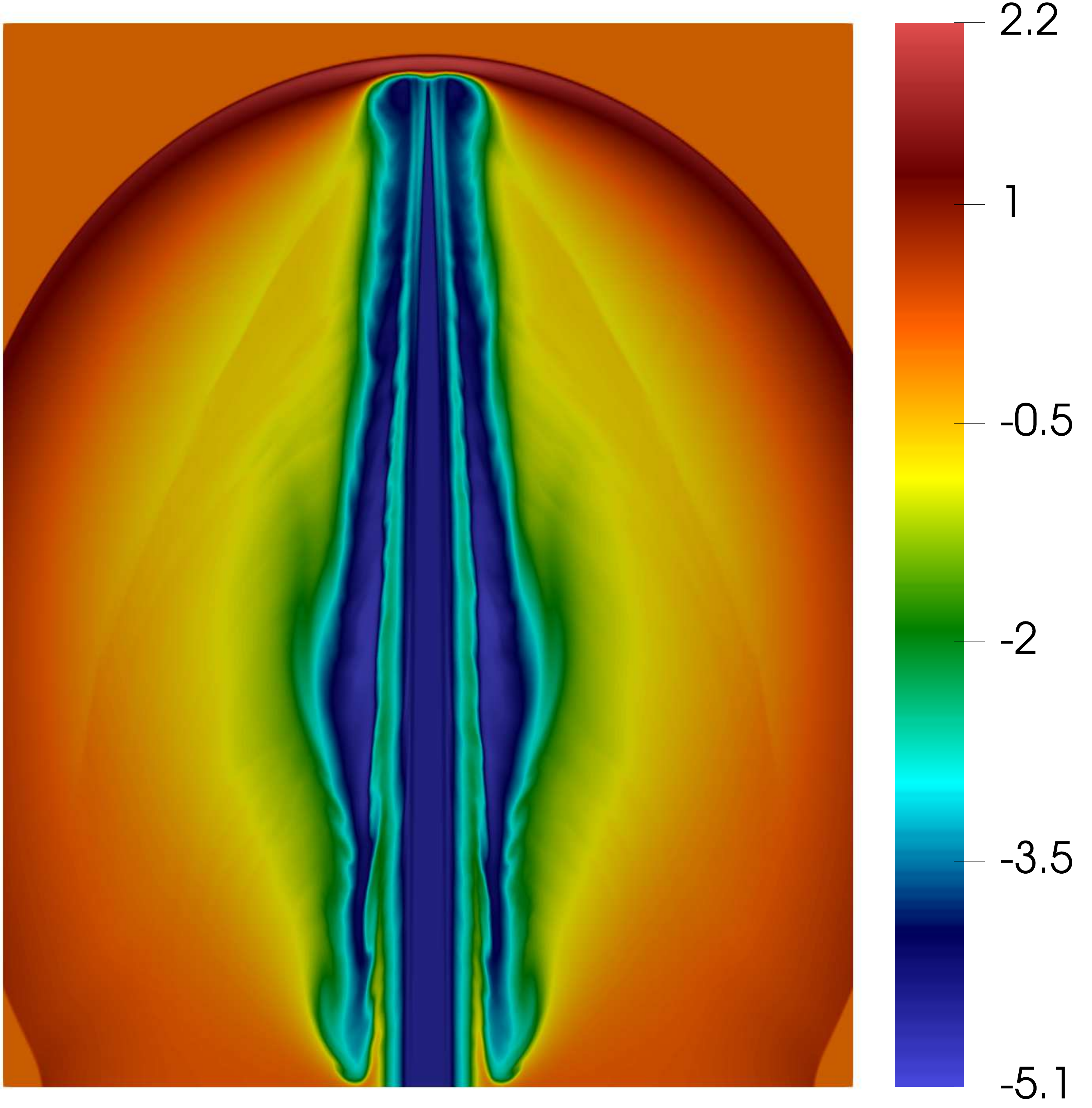}
			\caption{TM-EOS.}
		\end{subfigure}
		\begin{subfigure}{0.31\textwidth}
			\includegraphics[width=\linewidth]{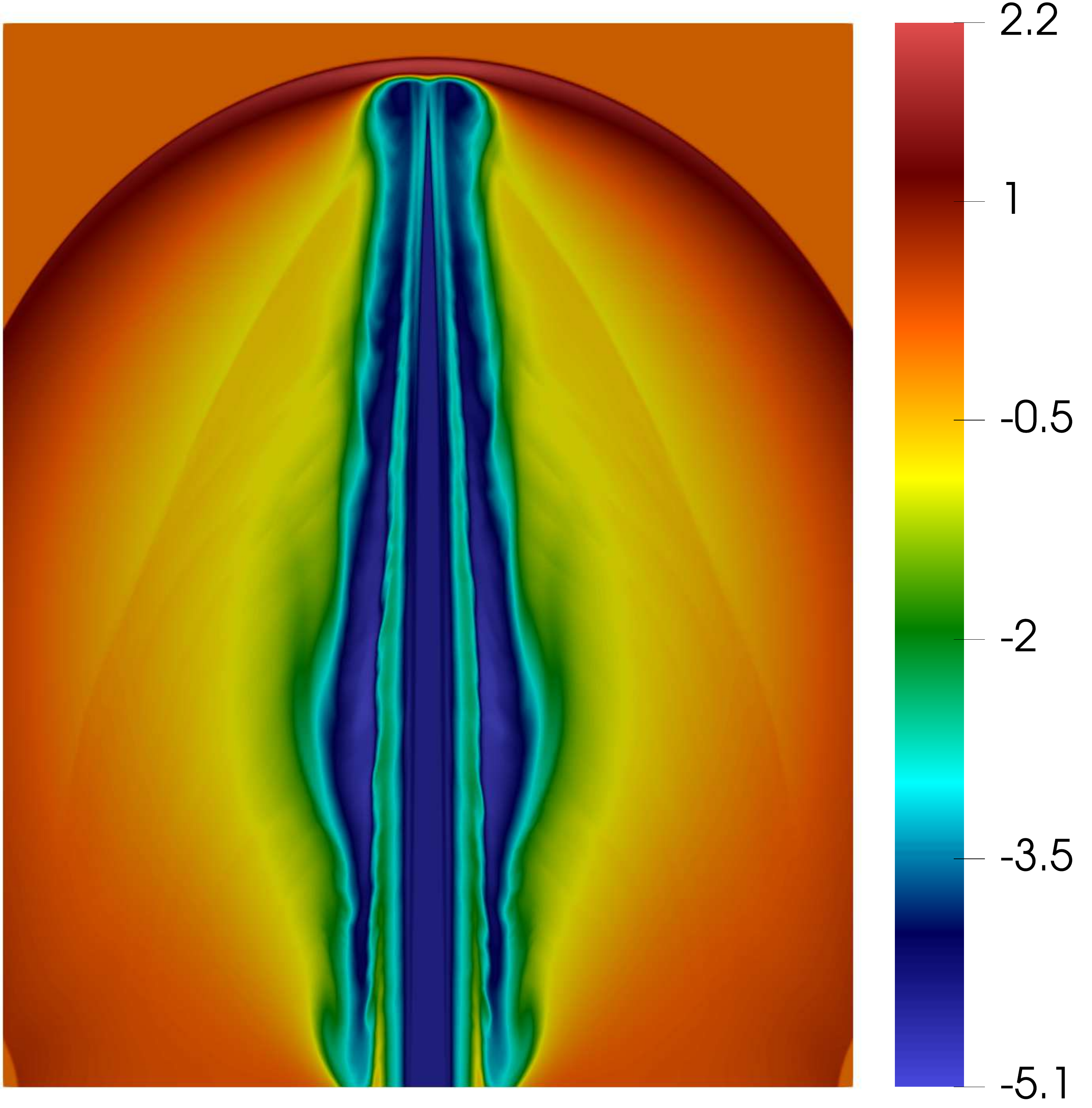}
			\caption{IP-EOS.}
		\end{subfigure}
		\begin{subfigure}{0.31\textwidth}
			\includegraphics[width=\linewidth]{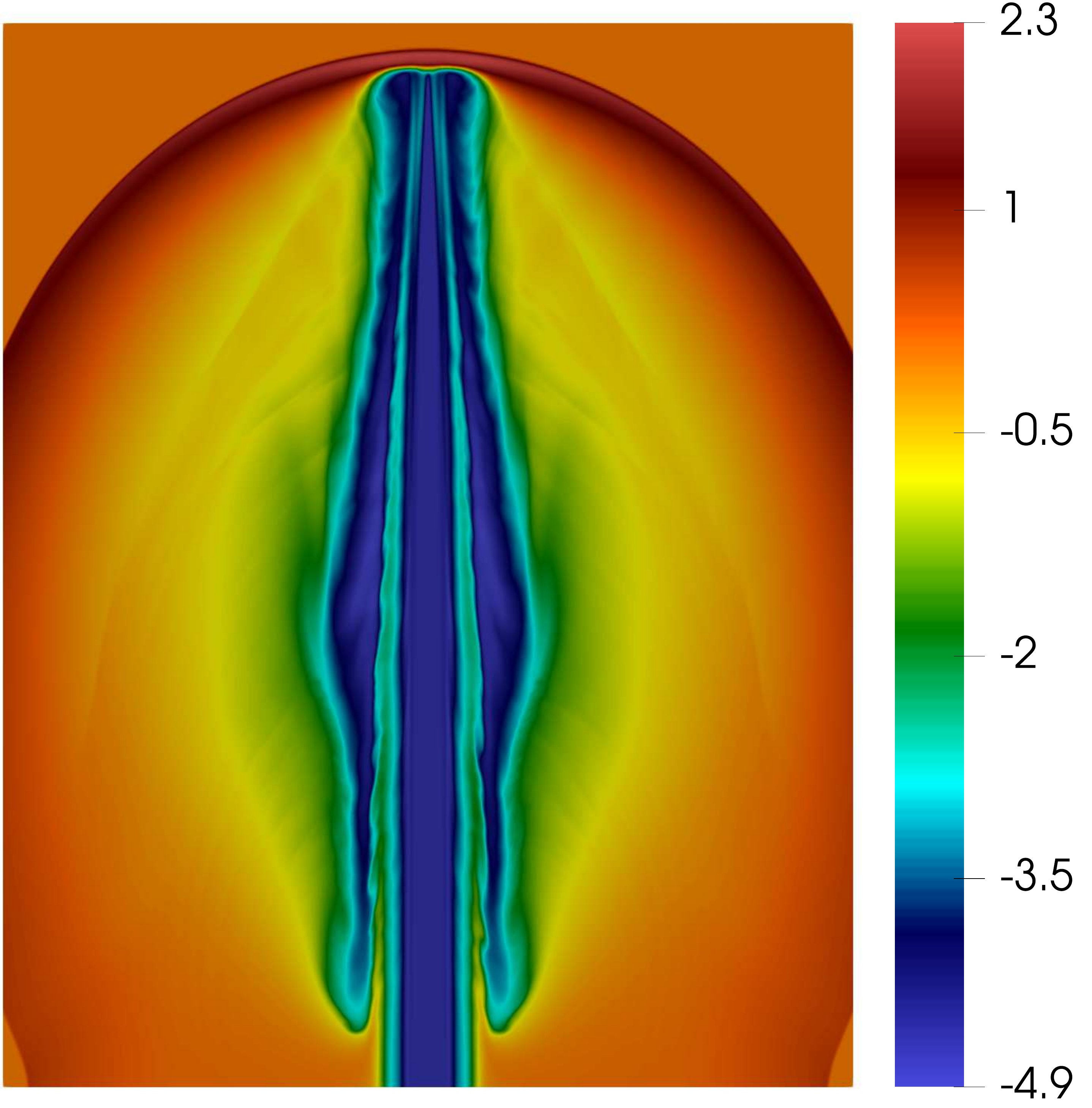}
			\caption{RC-EOS.}
		\end{subfigure}
		\caption{2-D relativistic jet: Plot of $\ln \rho$ using $480\times 500$ cells and $N=4$.}
		\label{fig:wureljet.lnden}
	\end{figure}
	
	\begin{figure}[!htbp]
		\centering
		\begin{subfigure}{0.31\textwidth}    \includegraphics[width=\linewidth]{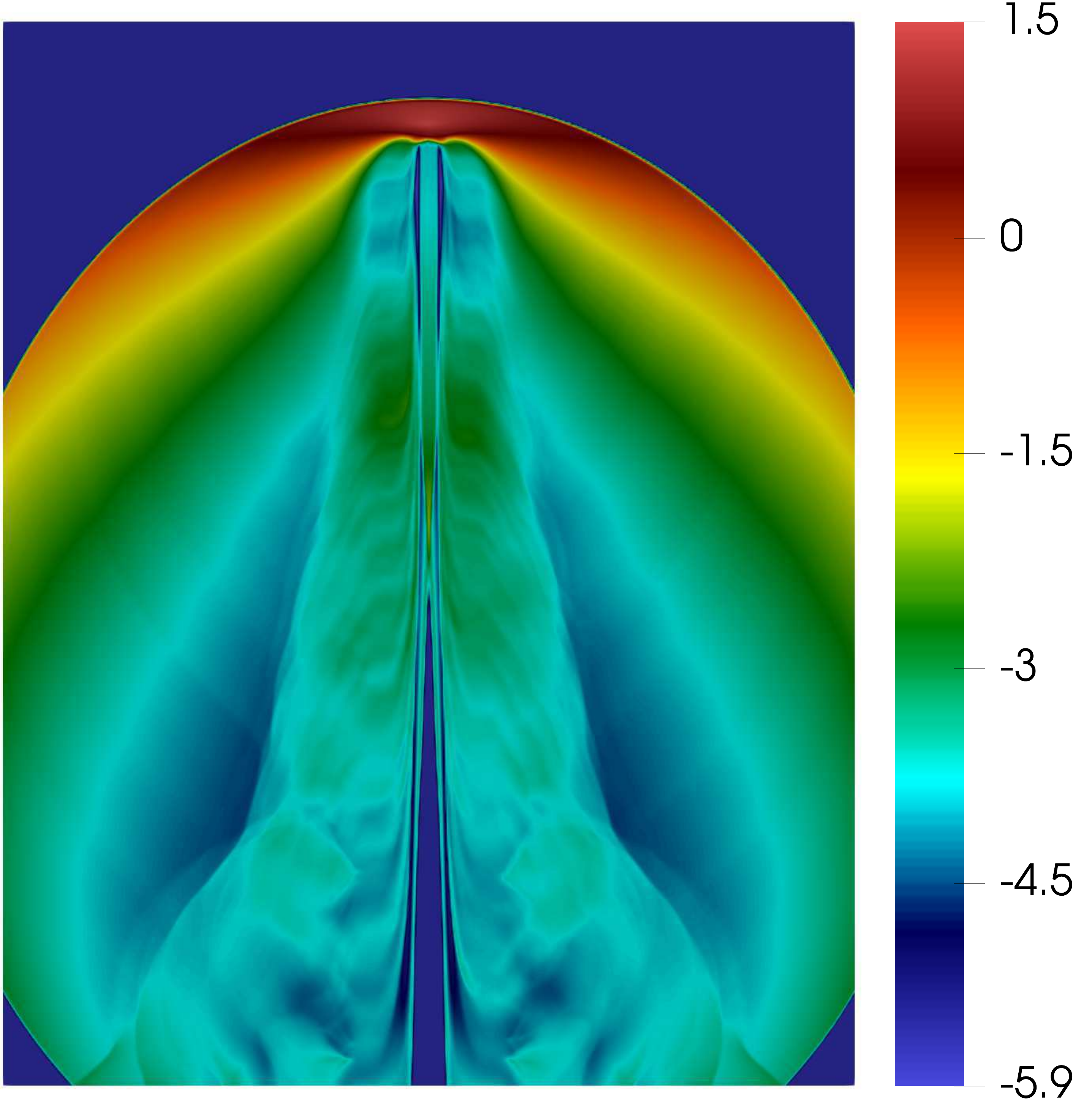}
			\caption{ID-EOS with $\gamma = \frac{5}{3}$.}
		\end{subfigure}
		\begin{subfigure}{0.31\textwidth}
			\includegraphics[width=\linewidth]{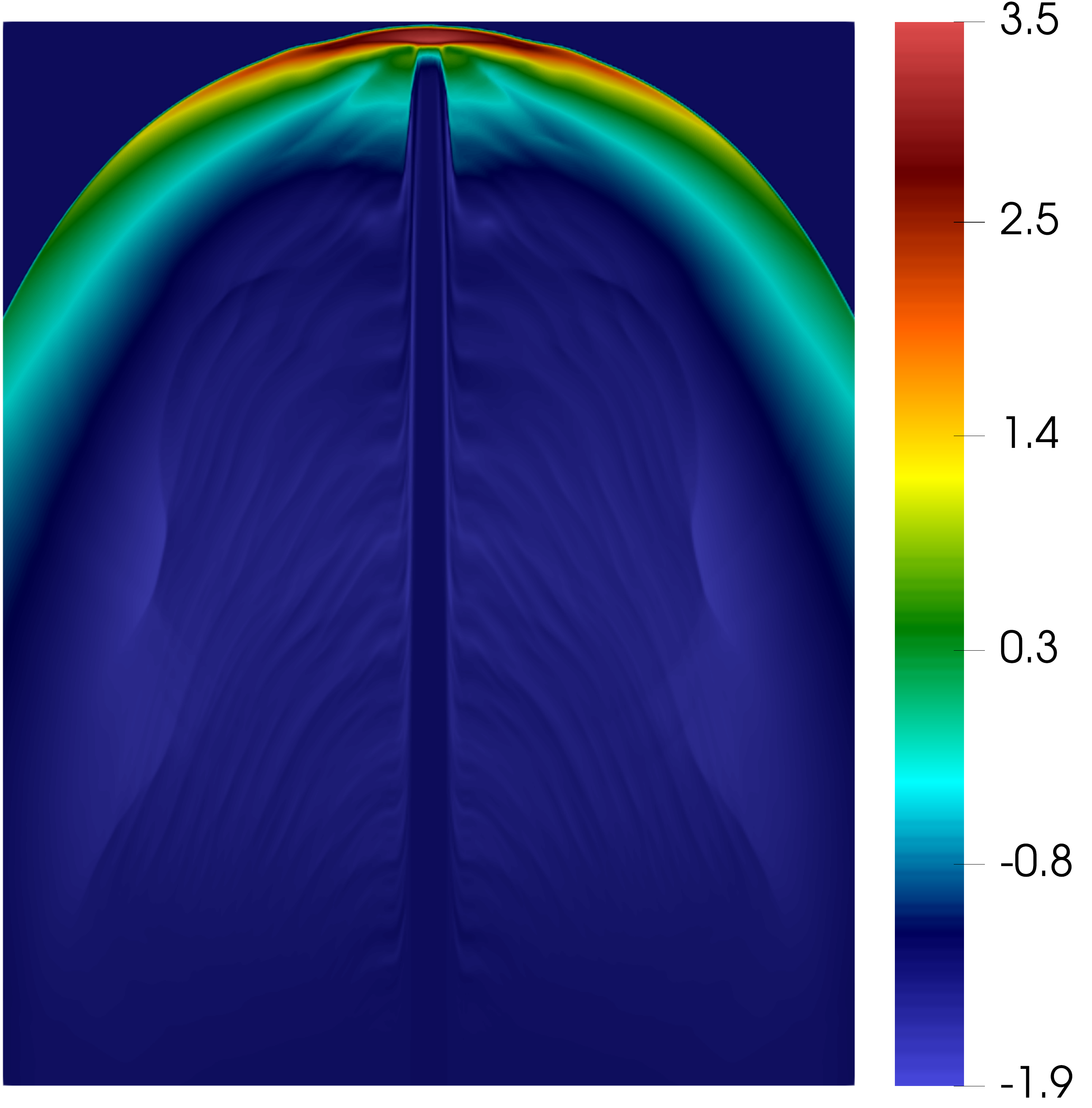}
			\caption{ID-EOS with $\gamma = \frac{4}{3}$.}
		\end{subfigure}
		\begin{subfigure}{0.31\textwidth}
			\includegraphics[width=\linewidth]{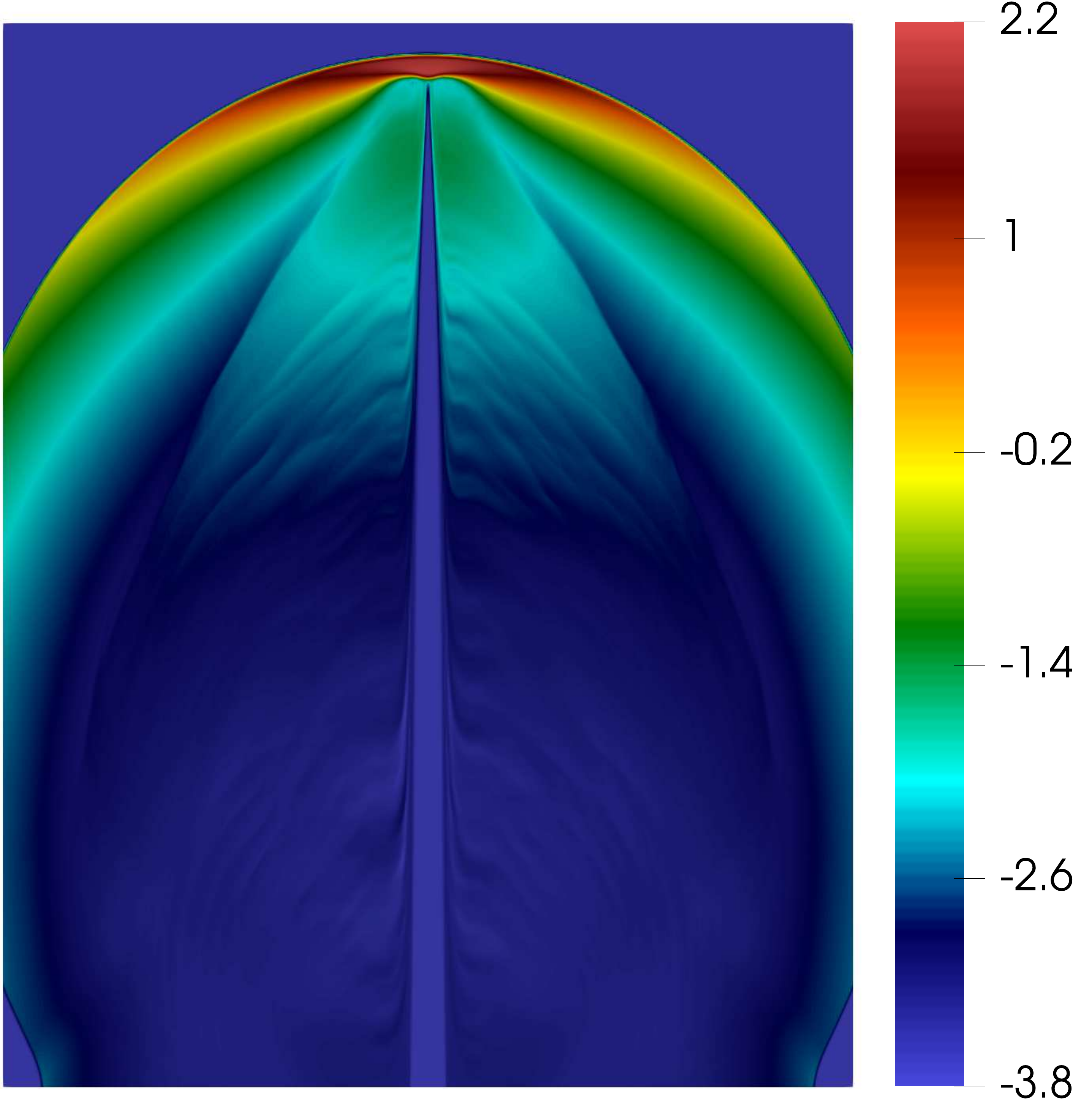}
			\caption{TM-EOS.}
		\end{subfigure}
		\begin{subfigure}{0.31\textwidth}
			\includegraphics[width=\linewidth]{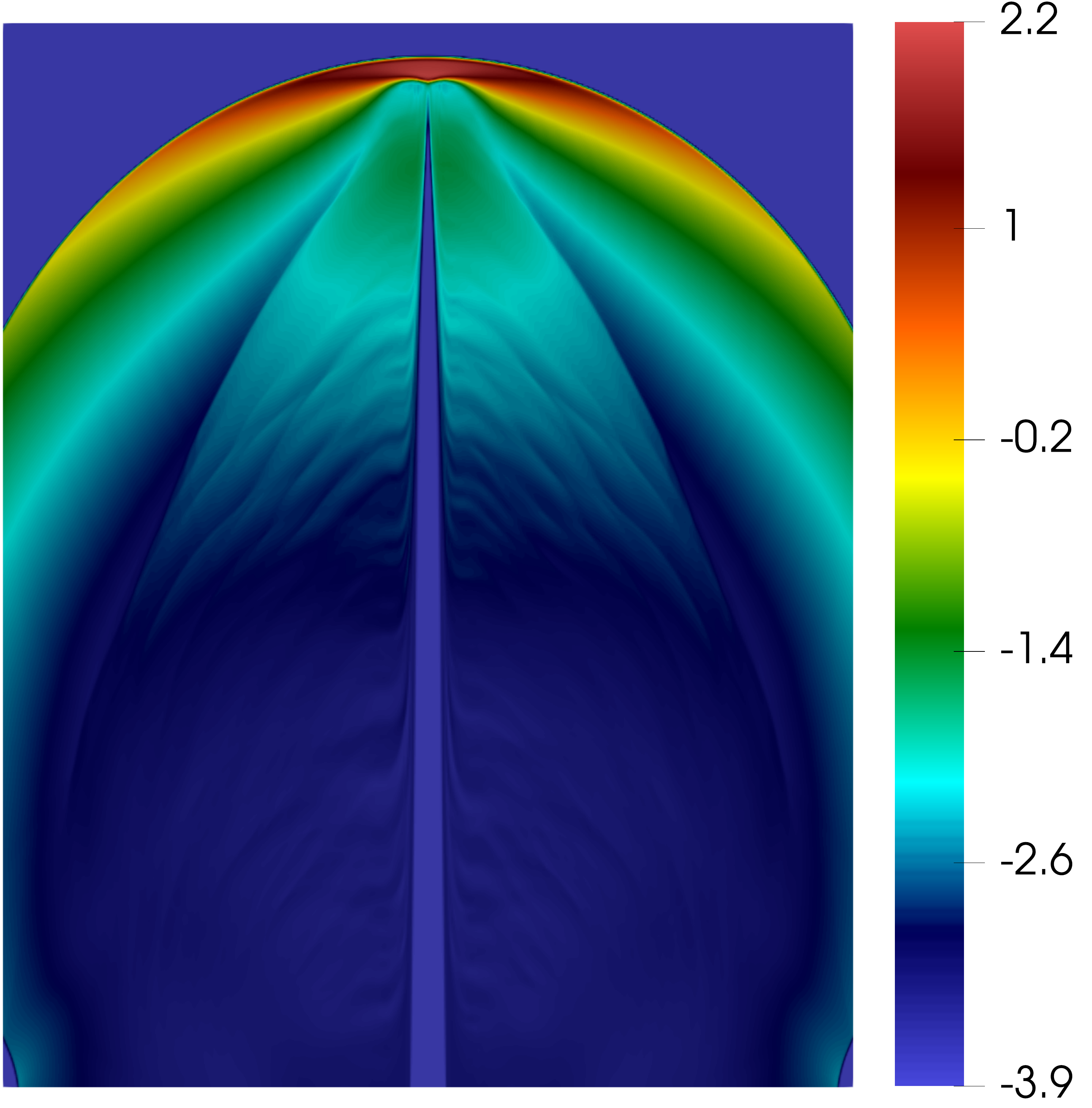}
			\caption{IP-EOS.}
		\end{subfigure}
		\begin{subfigure}{0.31\textwidth}
			\includegraphics[width=\linewidth]{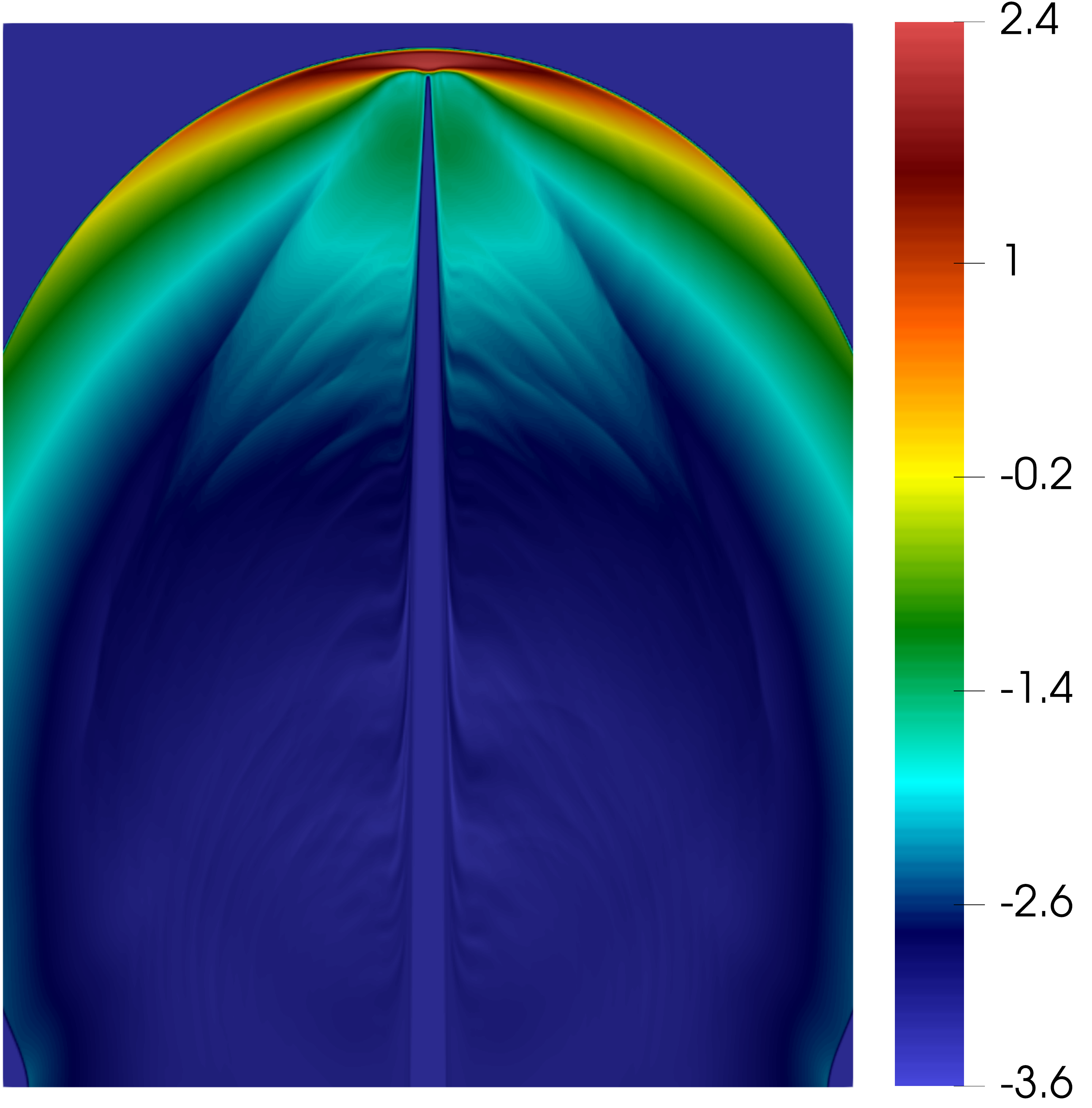}
			\caption{RC-EOS.}
		\end{subfigure}
		\caption{2-D relativistic jet: Plot of $\ln p$ using $480\times 500$ cells and $N=4$.}
		\label{fig:wureljet.lnpres}
	\end{figure}
	
	\subsubsection{2-D bubble shock interaction}
	Here, we test our scheme with a test case where a moving shock wave interacts with a bubble of lighter and higher density and forms different wave structures around it. This test case is taken from~\cite{xu2024high} and is successfully simulated with different equations of state in the domain $[0, 325]\times [0, 90]$ using our scheme with $650\times 180$ cells, $N=4$ with reflective boundaries at $y=0,\ 90$ and constant left and right shock states at the boundaries $x=0,\ 325$. Initially, a bubble of radius $25$ is placed with center at $(215, 45)$ having density $0.1358$ and $3.1538$ for Case~I and Case~II, respectively. The pressure inside the bubble is the same as the ambient pressure. A shock is placed outside the bubble at $x=265$ at time $t=0$ with,
	\begin{align*}
		(\rho&, v_1, v_2, p)\\ 
		&= \begin{cases}
			(1, 0, 0, 0.05) & \text{if}\ x < 265\\
			(1.941272902134272, -0.200661045980881, 0,  0.15) & \text{if}\ x > 265
		\end{cases}
	\end{align*}
	for both cases.
	
	We show here the results with different equations of state in Figure~\ref{fig:wu2sb1.eos1_53}- Figure~\ref{fig:wu2sb2.eos3} at different times for Case~I and Case~II. We can observe that after the interaction, the structure of the bubble gets changed, and the waves created because of the collision are striking the reflective boundaries, coming back, and heating the bubble again, forming a number of waves in the domain. Our scheme can capture all the waves during and after the interaction effectively, along with the deformed structure of the bubble.
	\begin{figure}[]
		\centering
		\begin{subfigure}{0.85\textwidth}
			\centering
			\includegraphics[width=0.7\linewidth, height = 0.8cm]{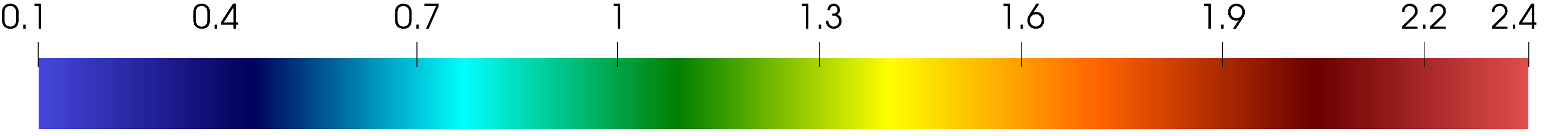}
		\end{subfigure}
		\begin{subfigure}{0.49\textwidth}
			\includegraphics[width=\linewidth]{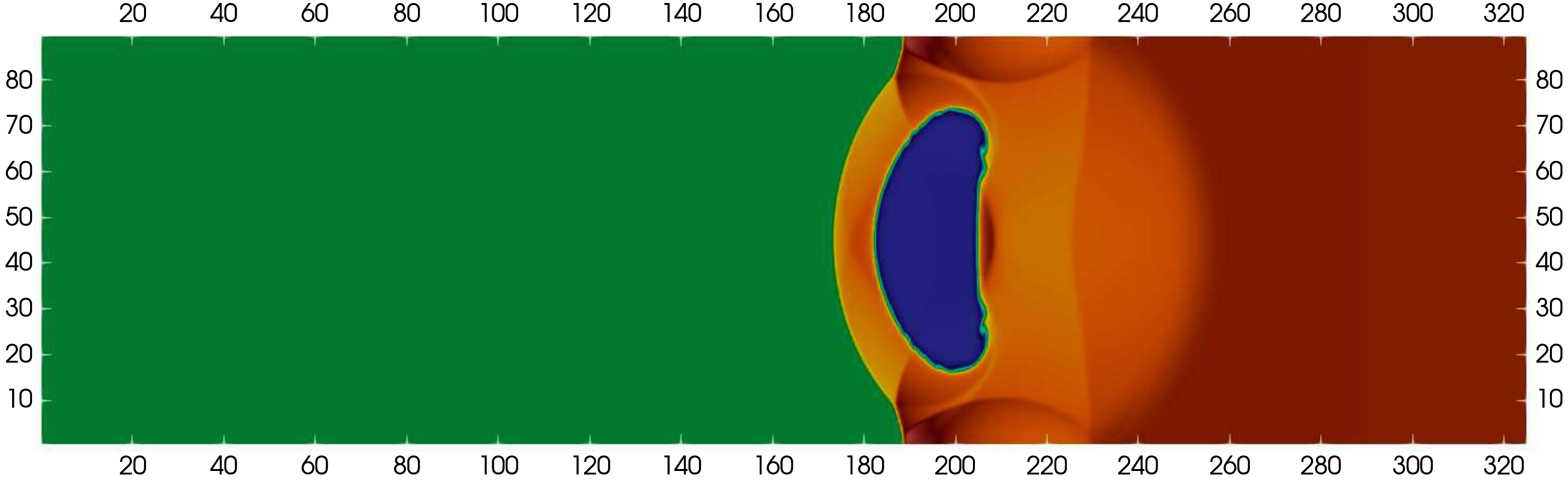}
			\caption{At time $t=180$.}
		\end{subfigure}
		\begin{subfigure}{0.49\textwidth}
			\includegraphics[width=\linewidth]{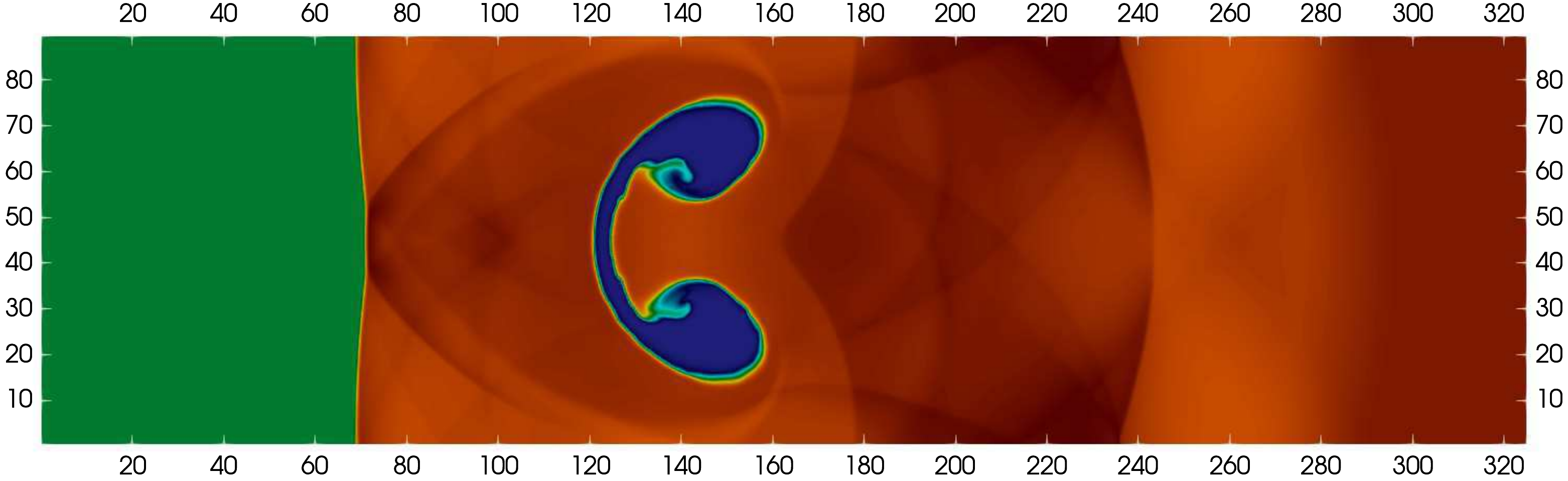}
			\caption{At time $t=450$.}
		\end{subfigure}
		\caption{2-D bubble shock interaction: Plot of density ($\rho$) with $650\times 180$ cells and $N=4$ and using ID-EOS with $\gamma = \frac{5}{3}$ for case~I.}
		\label{fig:wu2sb1.eos1_53}
	\end{figure}
	\begin{figure}[]
		\centering
		\begin{subfigure}{0.85\textwidth}
			\centering
			\includegraphics[width=0.7\linewidth, height = 0.8cm]{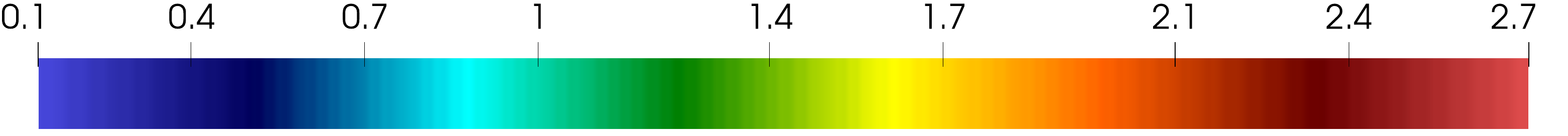}
		\end{subfigure}
		\begin{subfigure}{0.49\textwidth}
			\includegraphics[width=\linewidth]{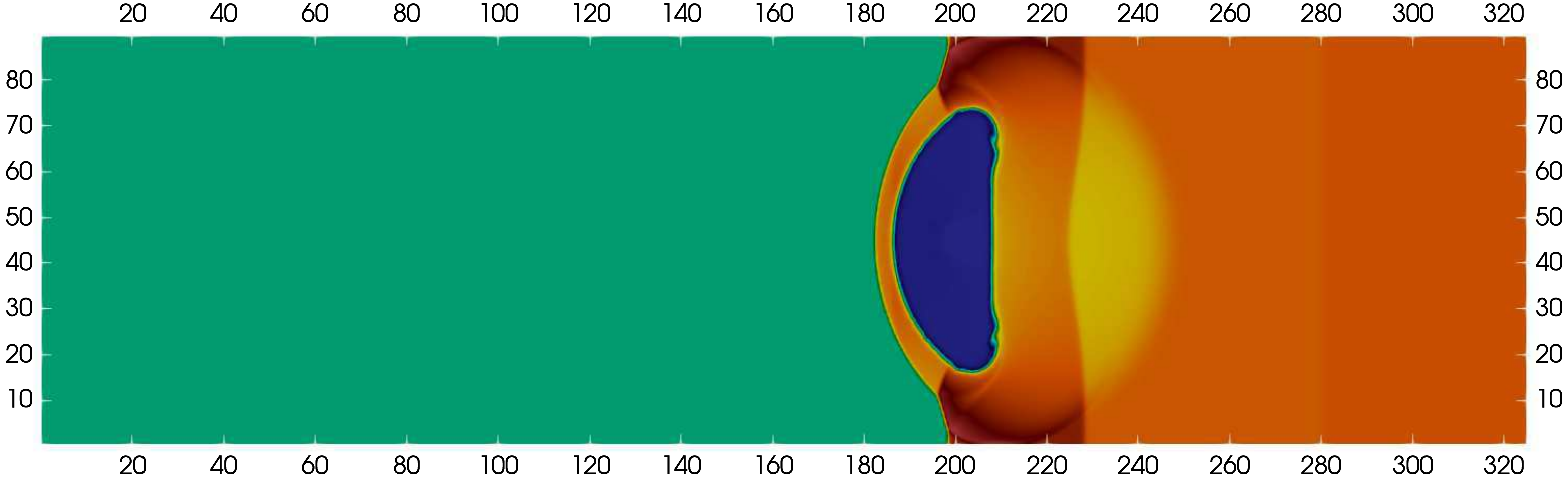}
			\caption{At time $t=180$.}
		\end{subfigure}
		\begin{subfigure}{0.49\textwidth}
			\includegraphics[width=\linewidth]{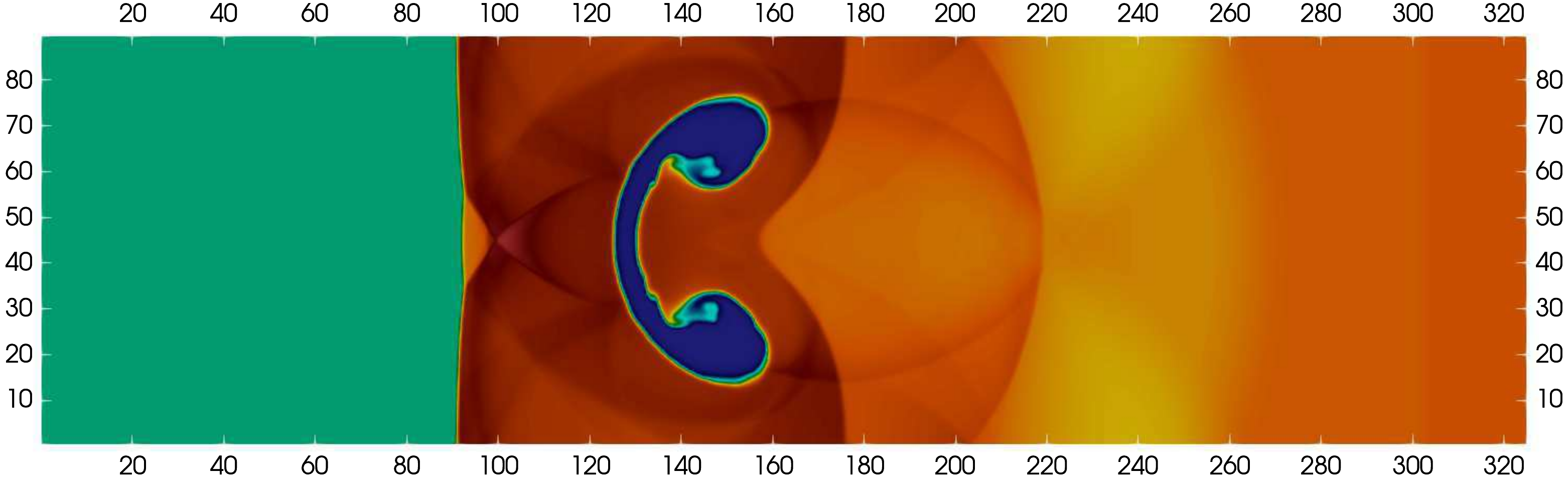}
			\caption{At time $t=450$.}
		\end{subfigure}
		\caption{2-D bubble shock interaction: Plot of density ($\rho$) with $650\times 180$ cells and $N=4$ and using ID-EOS with $\gamma = \frac{4}{3}$ for case~I.}
		\label{fig:wu2sb1.eos1_43}
	\end{figure}
	\begin{figure}[]
		\centering
		\begin{subfigure}{0.85\textwidth}
			\centering
			\includegraphics[width=0.7\linewidth, height = 0.8cm]{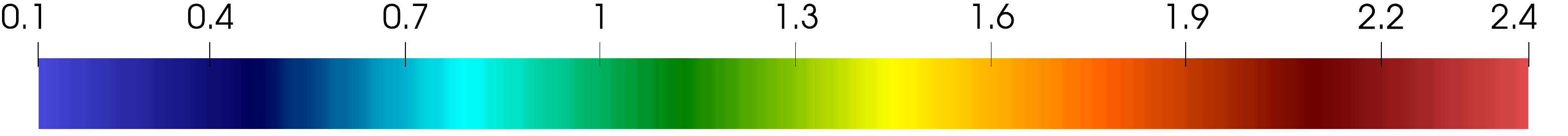}
		\end{subfigure}
		\begin{subfigure}{0.49\textwidth}
			\includegraphics[width=\linewidth]{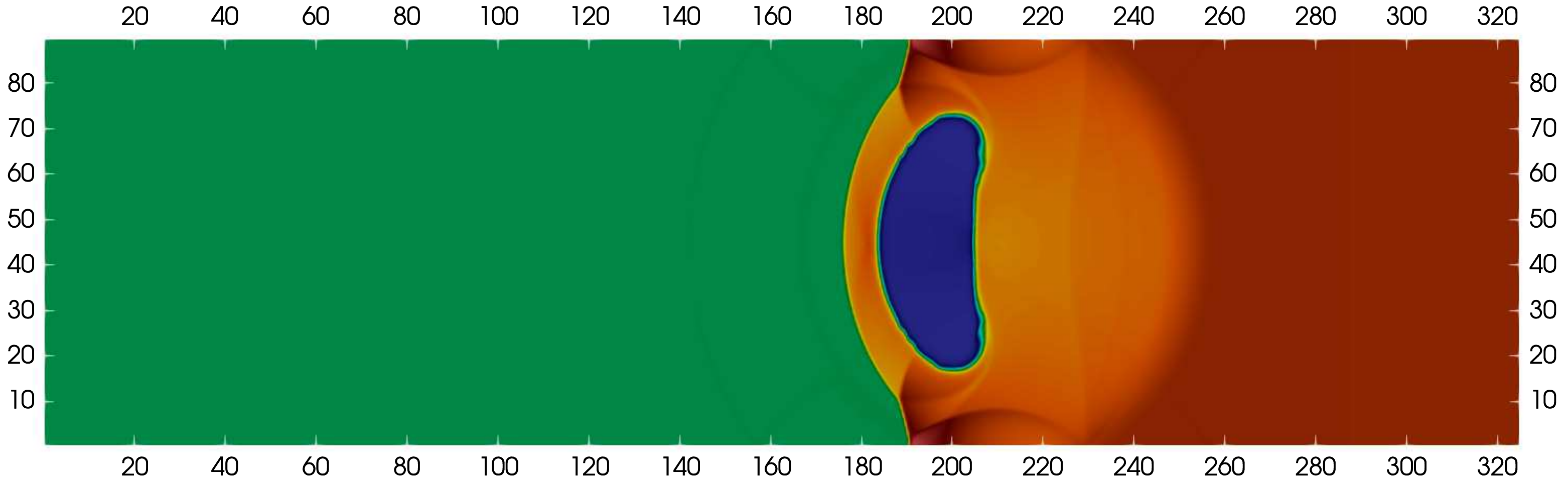}
			\caption{At time $t=180$.}
		\end{subfigure}
		\begin{subfigure}{0.49\textwidth}
			\includegraphics[width=\linewidth]{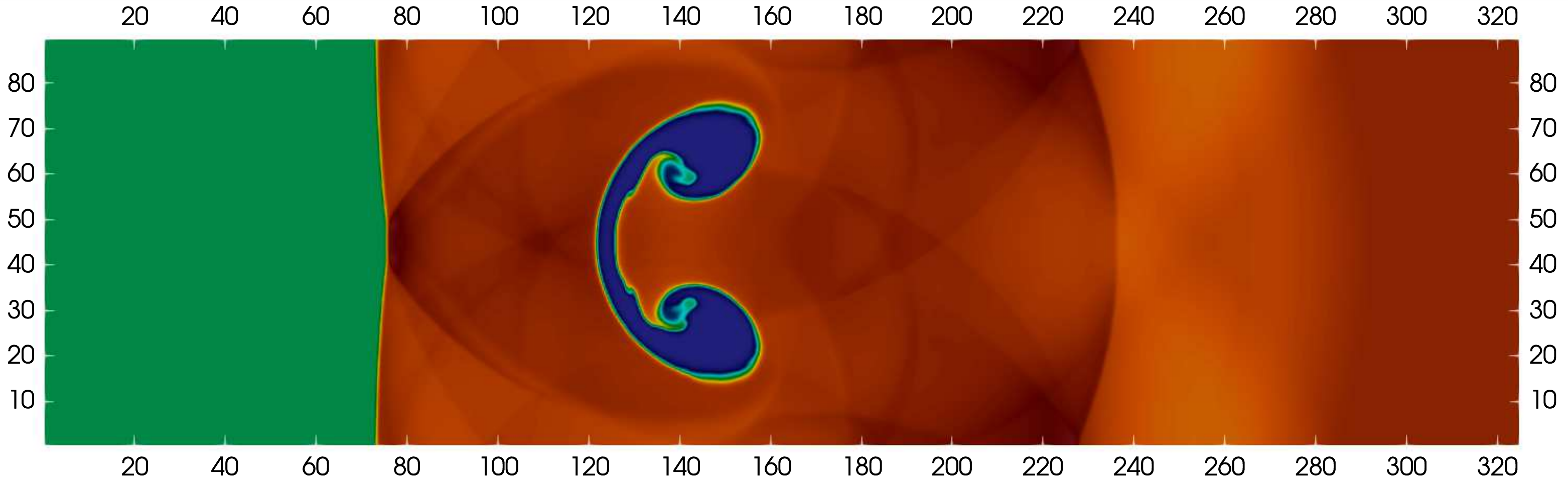}
			\caption{At time $t=450$.}
		\end{subfigure}
		\caption{2-D bubble shock interaction: Plot of density ($\rho$) with $650\times 180$ cells and $N=4$ and using TM-EOS for case~I.}
		\label{fig:wu2sb1.eos2}
	\end{figure}
	\begin{figure}[]
		\centering
		\begin{subfigure}{0.85\textwidth}
			\centering
			\includegraphics[width=0.7\linewidth, height = 0.8cm]{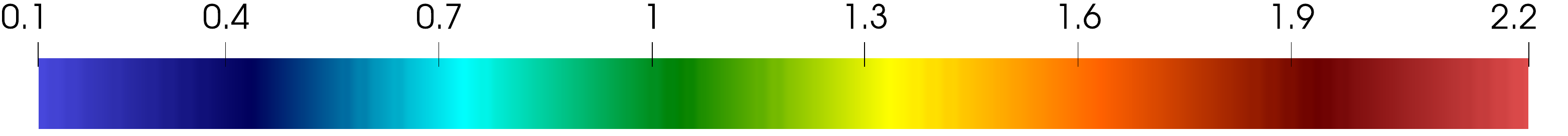}
		\end{subfigure}
		\begin{subfigure}{0.49\textwidth}
			\includegraphics[width=\linewidth]{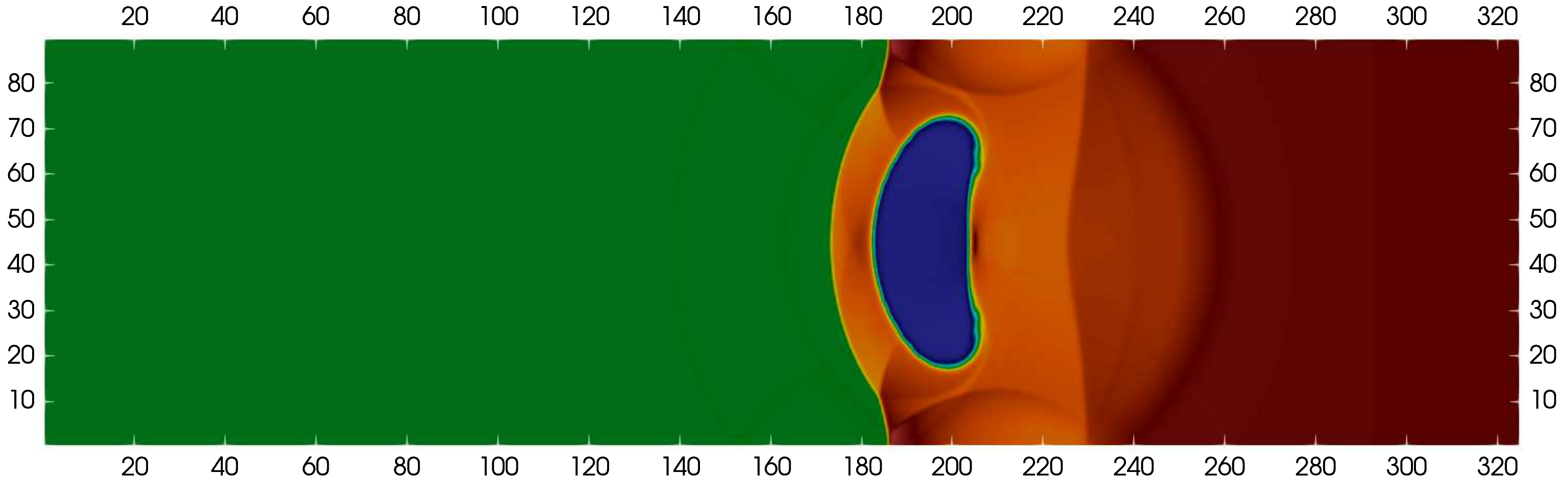}
			\caption{At time $t=180$.}
		\end{subfigure}
		\begin{subfigure}{0.49\textwidth}
			\includegraphics[width=\linewidth]{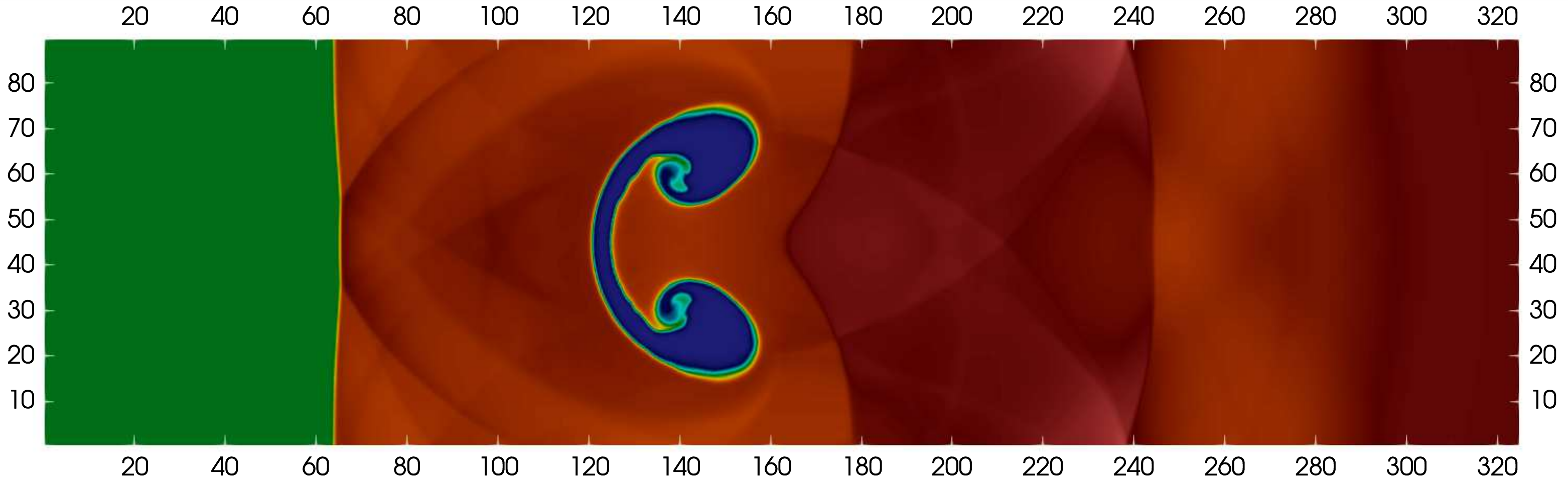}
			\caption{At time $t=450$.}
		\end{subfigure}
		\caption{2-D bubble shock interaction: Plot of density ($\rho$) with $650\times 180$ cells and $N=4$ and using IP-EOS for case~I.}
		\label{fig:wu2sb1.eos4}
	\end{figure}
	\begin{figure}[]
		\centering
		\begin{subfigure}{0.85\textwidth}
			\centering
			\includegraphics[width=0.7\linewidth, height = 0.8cm]{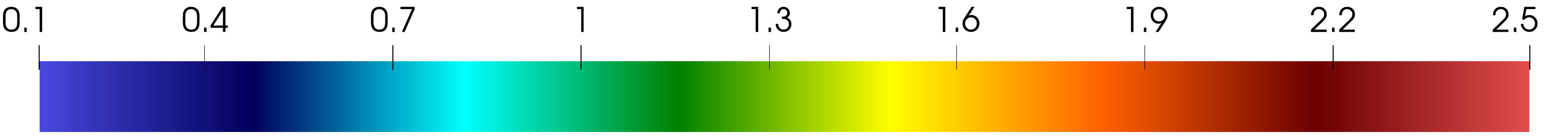}
		\end{subfigure}
		\begin{subfigure}{0.49\textwidth}
			\includegraphics[width=\linewidth]{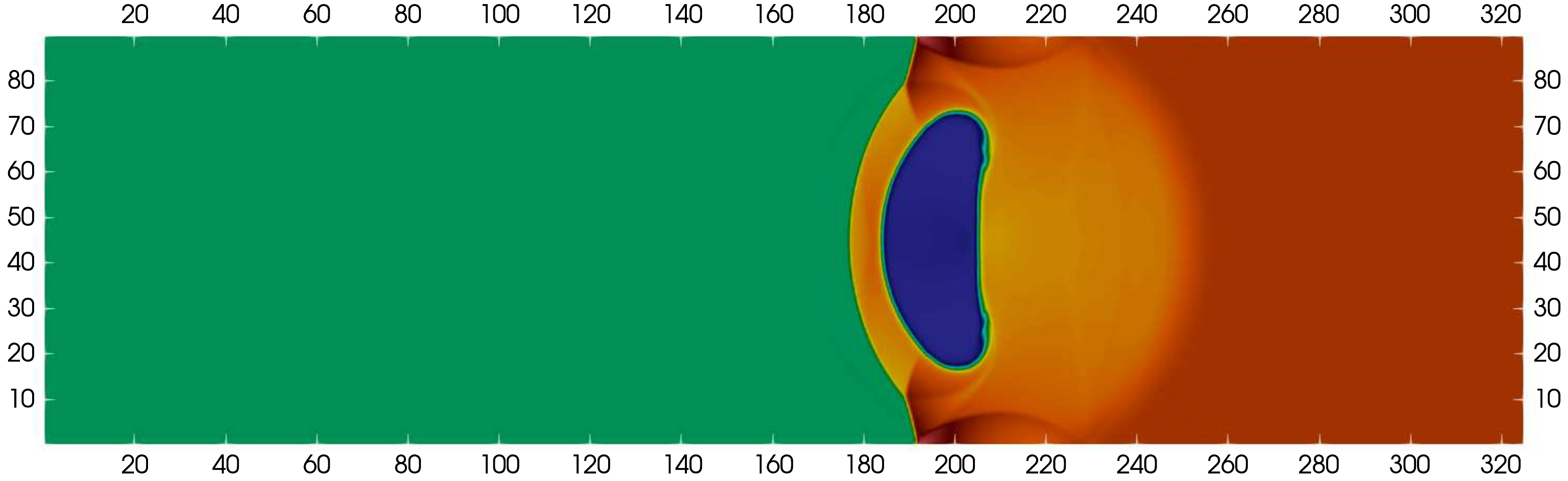}
			\caption{At time $t=180$.}
		\end{subfigure}
		\begin{subfigure}{0.49\textwidth}
			\includegraphics[width=\linewidth]{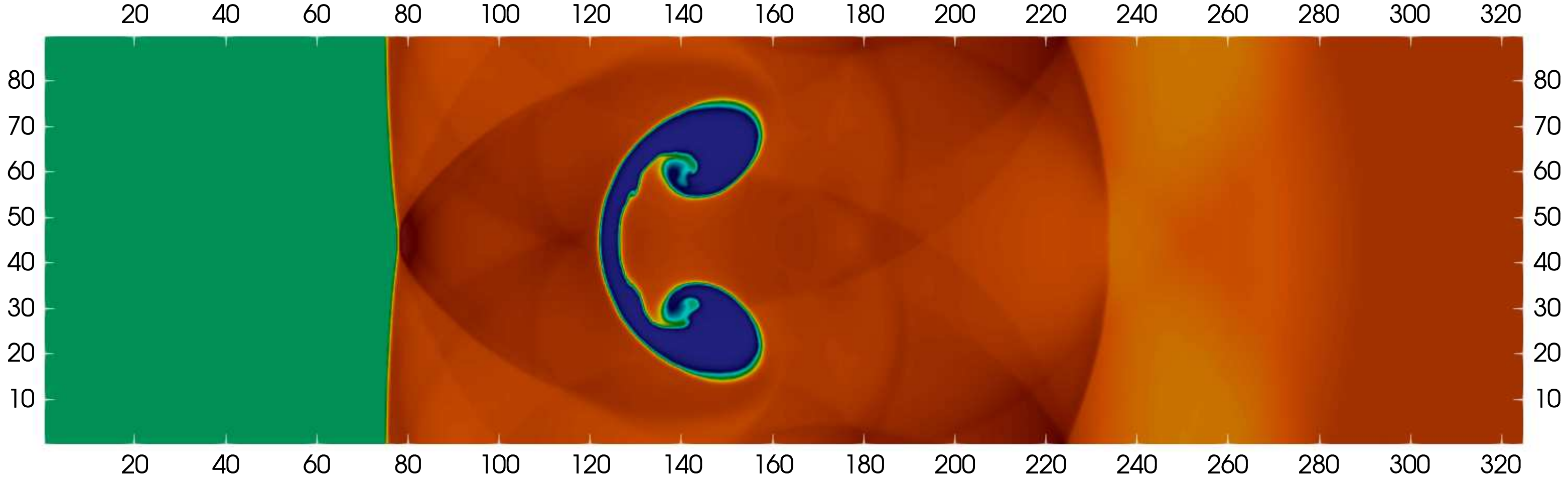}
			\caption{At time $t=450$.}
		\end{subfigure}
		\caption{2-D bubble shock interaction: Plot of density ($\rho$) with $650\times 180$ cells and $N=4$ and using RC-EOS for case~I.}
		\label{fig:wu2sb1.eos3}
	\end{figure}
	\begin{figure}[]
		\centering
		\begin{subfigure}{0.85\textwidth}
			\centering
			\includegraphics[width=0.7\linewidth, height = 0.8cm]{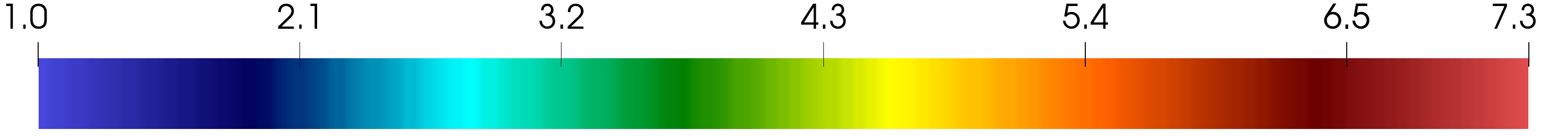}
		\end{subfigure}
		\begin{subfigure}{0.49\textwidth}
			\includegraphics[width=\linewidth]{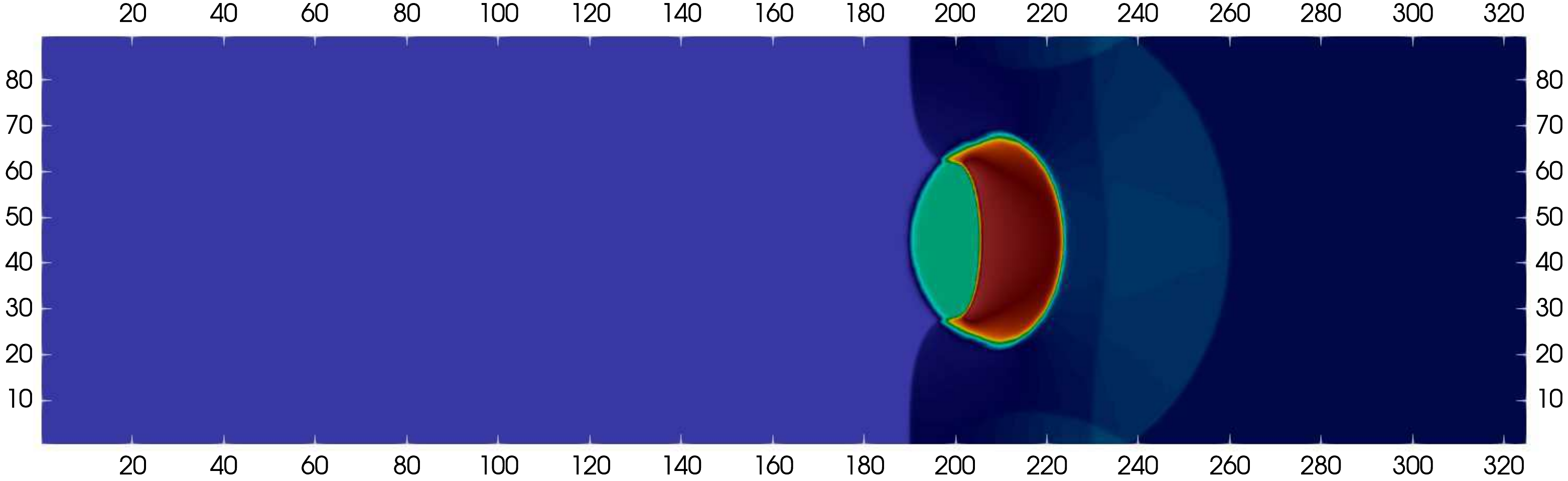}
			\caption{At time $t=180$.}
		\end{subfigure}
		\begin{subfigure}{0.49\textwidth}
			\includegraphics[width=\linewidth]{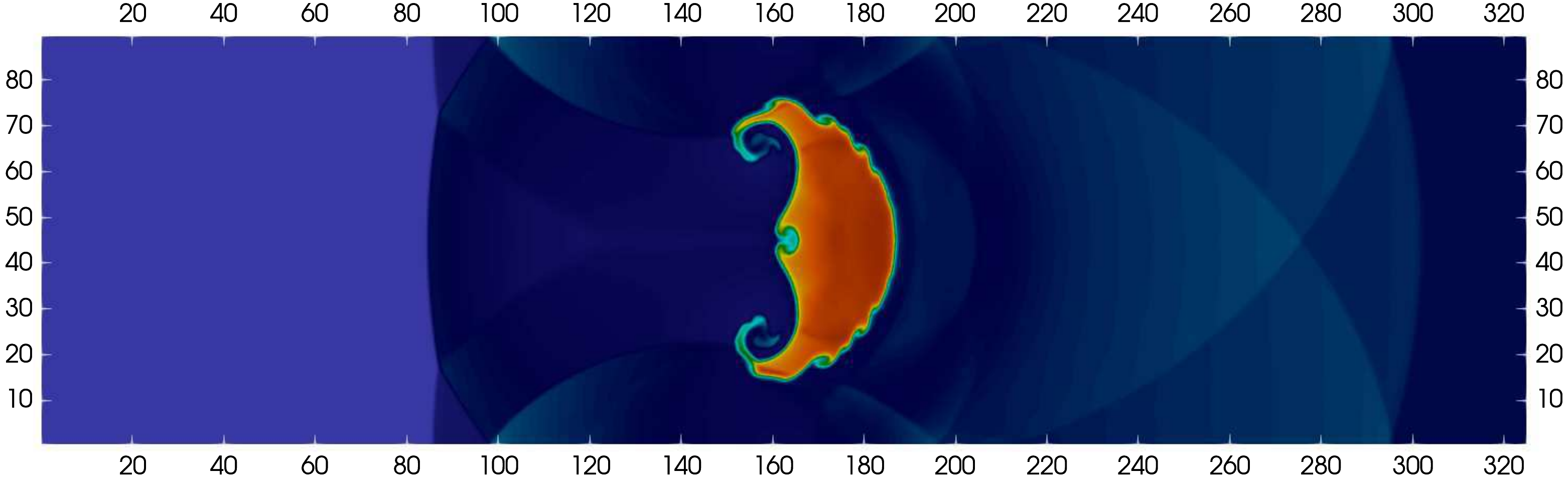}
			\caption{At time $t=450$.}
		\end{subfigure}
		\caption{2-D bubble shock interaction: Plot of density ($\rho$) with $650\times 180$ cells and $N=4$ and using ID-EOS with $\gamma = \frac{5}{3}$ for case~II.}
		\label{fig:wu2sb2.eos1_53}
	\end{figure}
	\begin{figure}[]
		\centering
		\begin{subfigure}{0.85\textwidth}
			\centering
			\includegraphics[width=0.7\linewidth, height = 0.8cm]{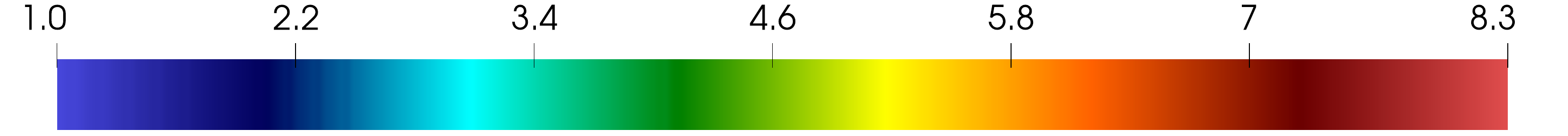}
		\end{subfigure}
		\begin{subfigure}{0.49\textwidth}
			\includegraphics[width=\linewidth]{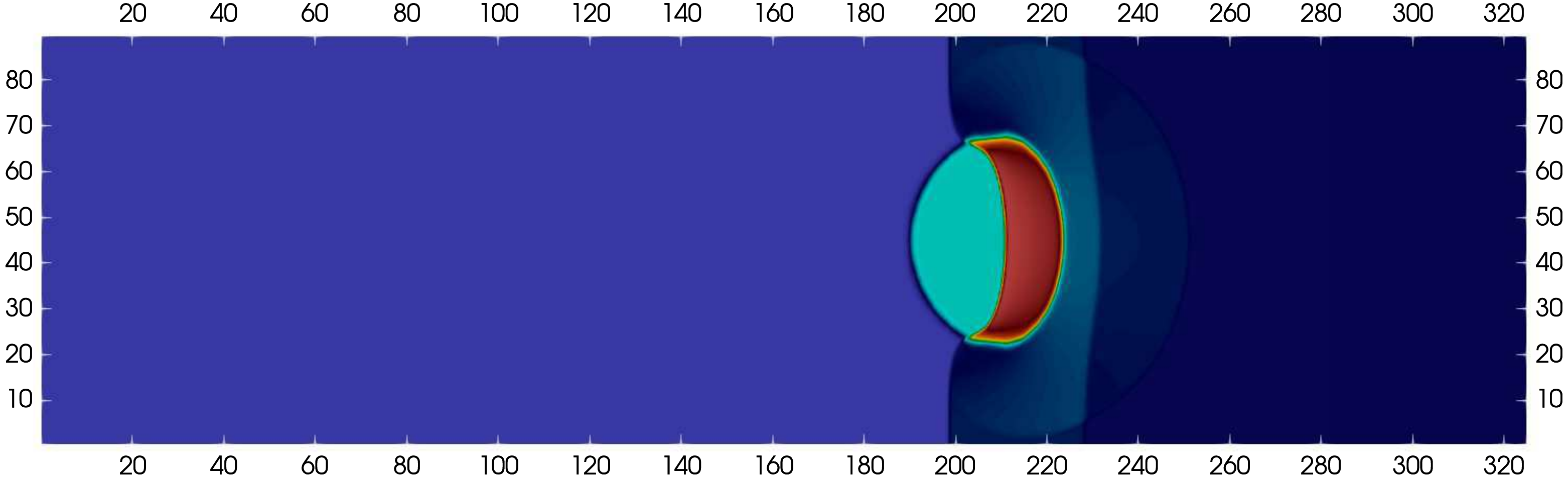}
			\caption{At time $t=180$.}
		\end{subfigure}
		\begin{subfigure}{0.49\textwidth}
			\includegraphics[width=\linewidth]{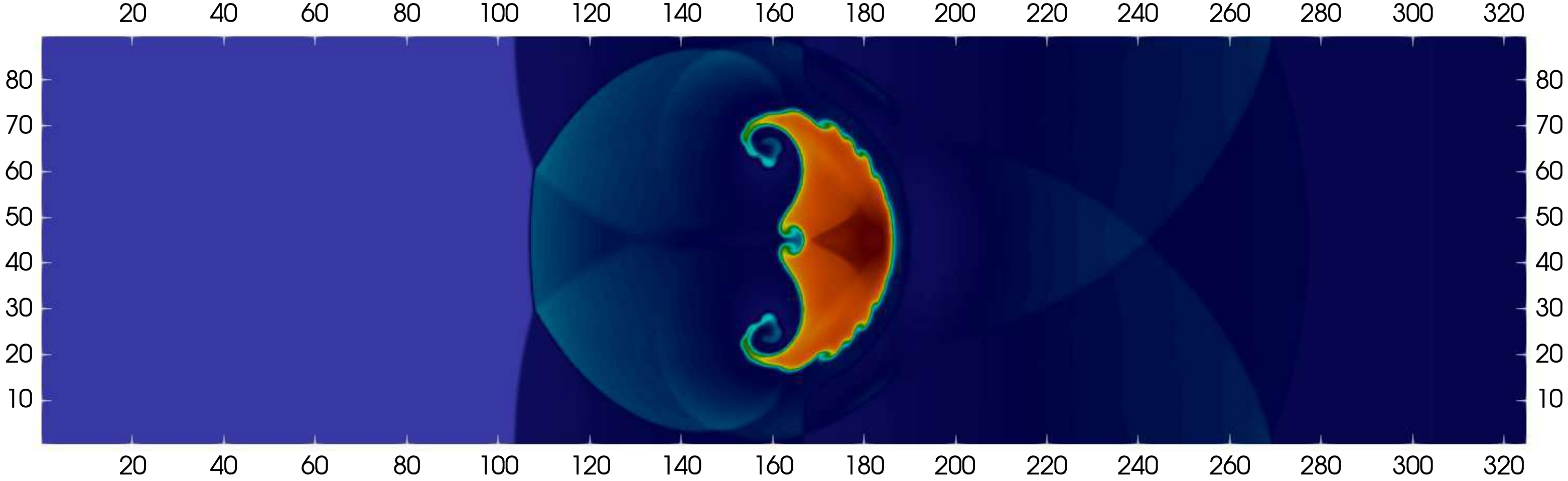}
			\caption{At time $t=450$.}
		\end{subfigure}
		\caption{2-D bubble shock interaction: Plot of density ($\rho$) with $650\times 180$ cells and $N=4$ and using ID-EOS with $\gamma = \frac{4}{3}$ for case~II.}
		\label{fig:wu2sb2.eos1_43}
	\end{figure}
	\begin{figure}[]
		\centering
		\begin{subfigure}{0.85\textwidth}
			\centering
			\includegraphics[width=0.7\linewidth, height = 0.8cm]{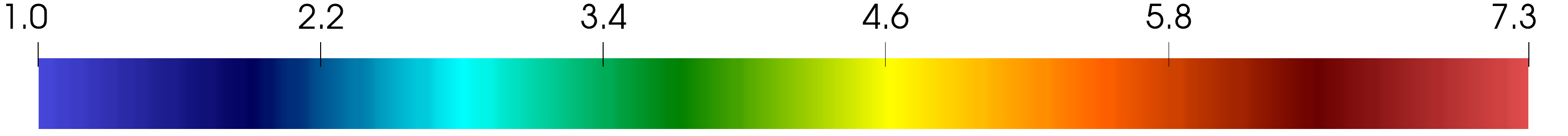}
		\end{subfigure}
		\begin{subfigure}{0.49\textwidth}
			\includegraphics[width=\linewidth]{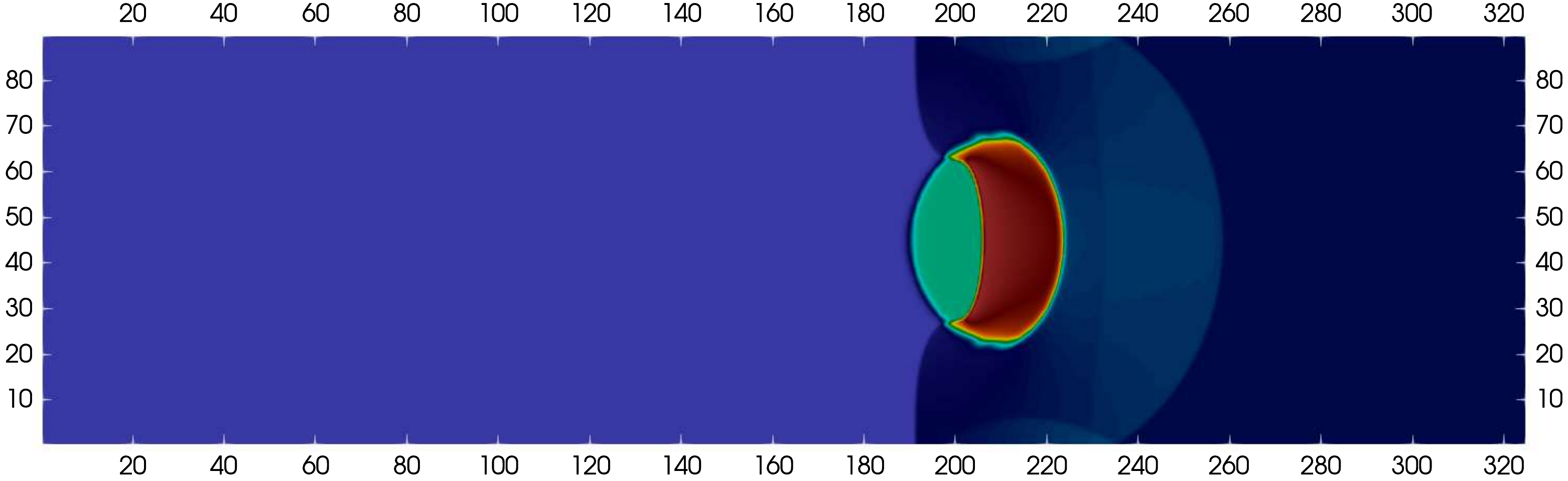}
			\caption{At time $t=180$.}
		\end{subfigure}
		\begin{subfigure}{0.49\textwidth}
			\includegraphics[width=\linewidth]{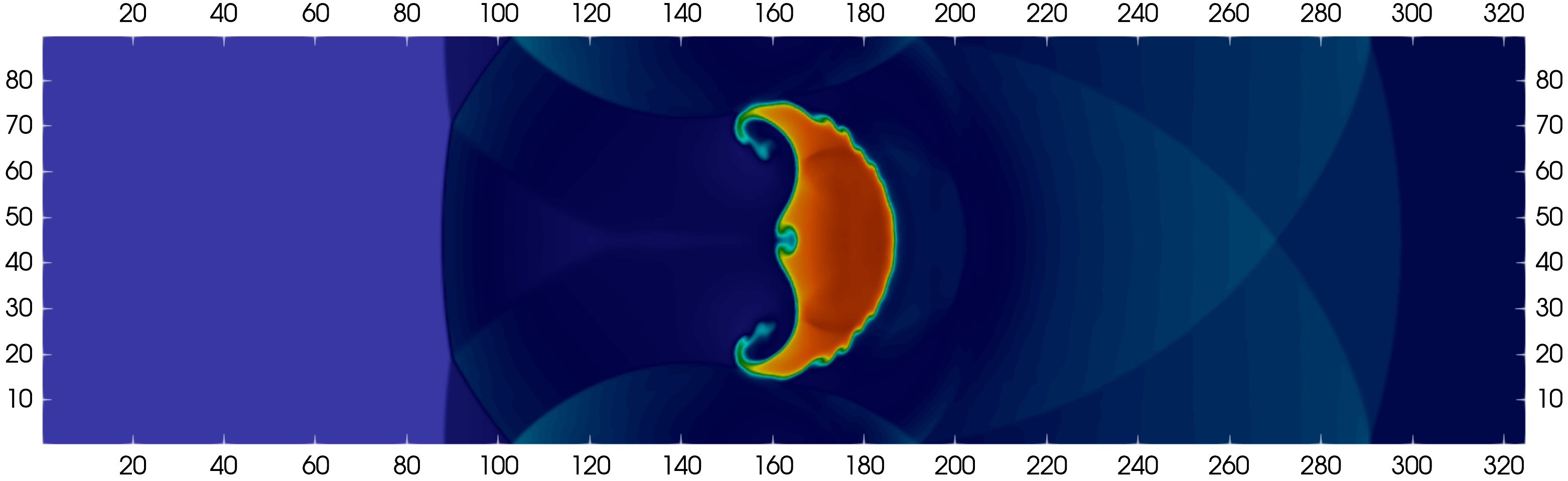}
			\caption{At time $t=450$.}
		\end{subfigure}
		\caption{2-D bubble shock interaction: Plot of density ($\rho$) with $650\times 180$ cells and $N=4$ and using TM-EOS for case~II.}
		\label{fig:wu2sb2.eos2}
	\end{figure}
	\begin{figure}[]
		\centering
		\begin{subfigure}{0.85\textwidth}
			\centering
			\includegraphics[width=0.7\linewidth, height = 0.8cm]{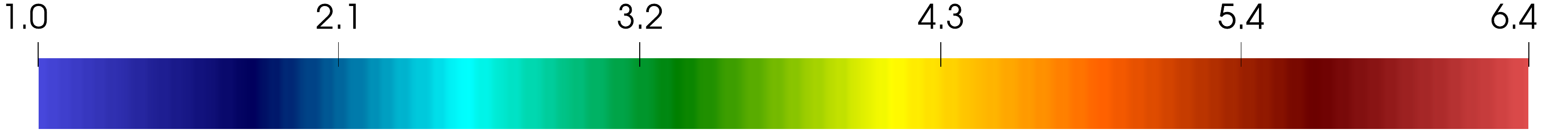}
		\end{subfigure}
		\begin{subfigure}{0.49\textwidth}
			\includegraphics[width=\linewidth]{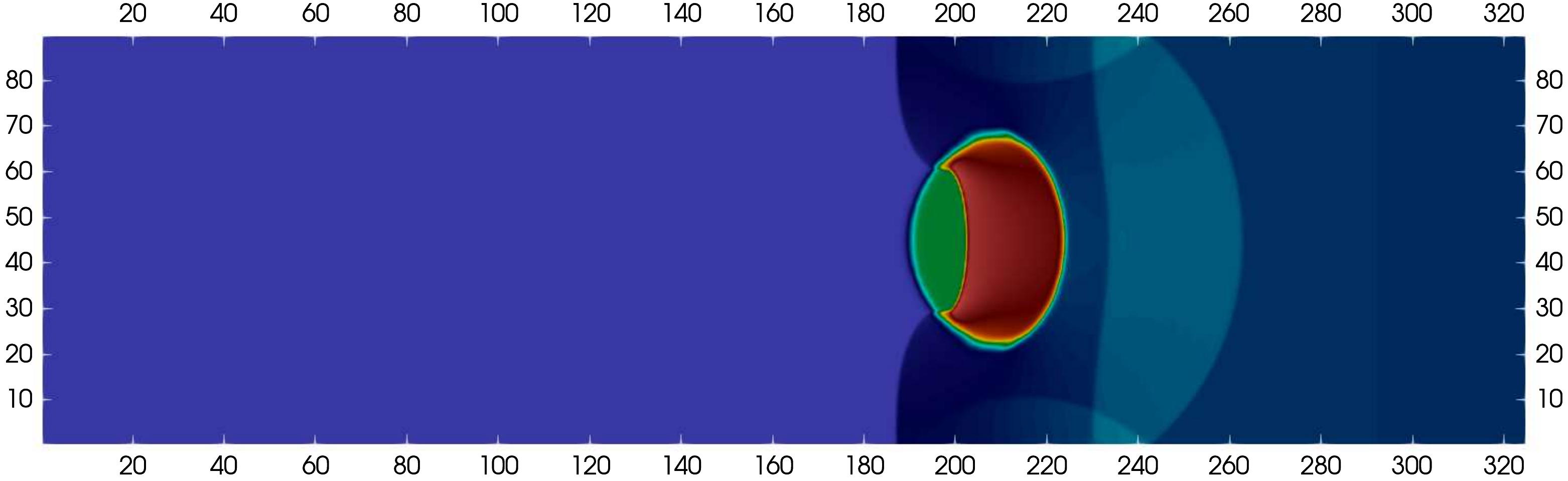}
			\caption{At time $t=180$.}
		\end{subfigure}
		\begin{subfigure}{0.49\textwidth}
			\includegraphics[width=\linewidth]{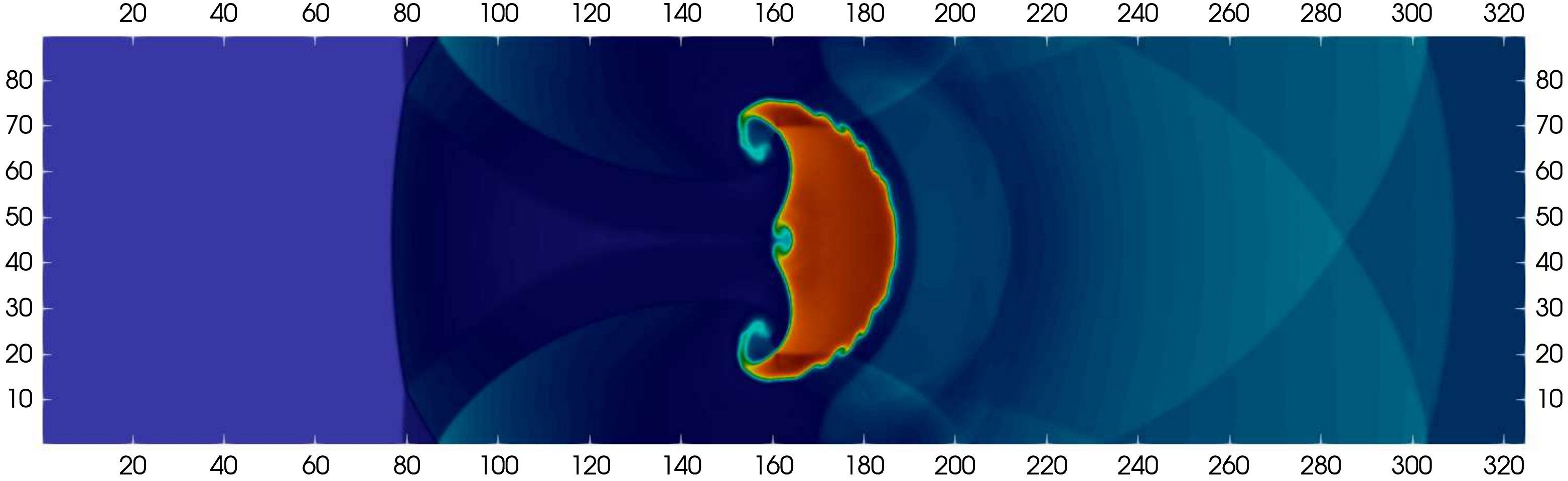}
			\caption{At time $t=450$.}
		\end{subfigure}
		\caption{2-D bubble shock interaction: Plot of density ($\rho$) with $650\times 180$ cells and $N=4$ and using IP-EOS for case~II.}
		\label{fig:wu2sb2.eos4}
	\end{figure}
	\begin{figure}[]
		\centering
		\begin{subfigure}{0.85\textwidth}
			\centering
			\includegraphics[width=0.7\linewidth, height = 0.8cm]{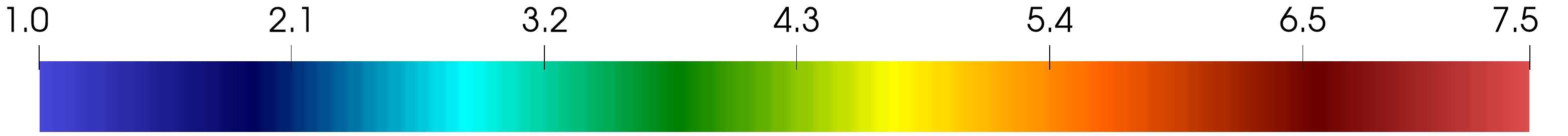}
		\end{subfigure}
		\begin{subfigure}{0.49\textwidth}
			\includegraphics[width=\linewidth]{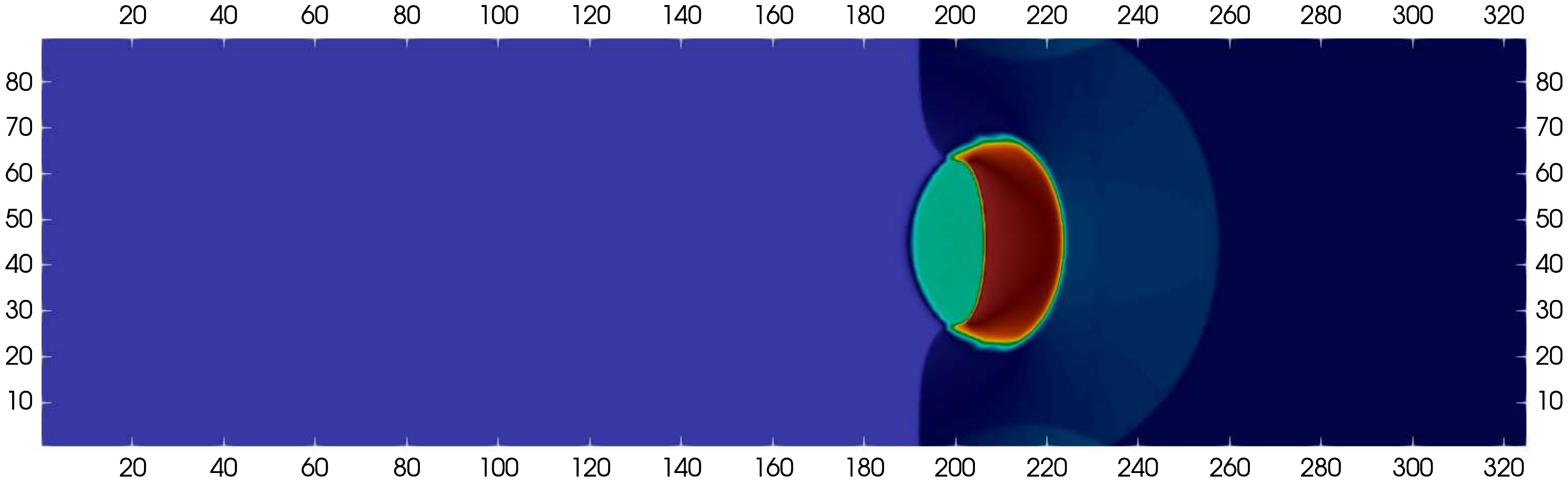}
			\caption{At time $t=180$.}
		\end{subfigure}
		\begin{subfigure}{0.49\textwidth}
			\includegraphics[width=\linewidth]{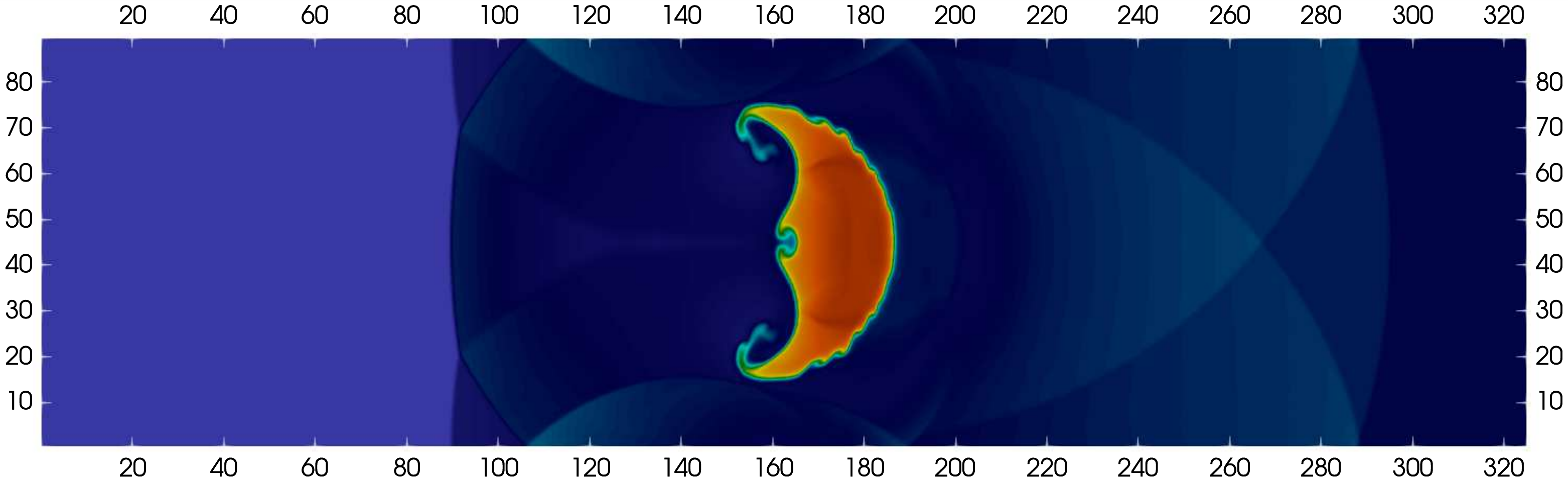}
			\caption{At time $t=450$.}
		\end{subfigure}
		\caption{2-D bubble shock interaction: Plot of density ($\rho$) with $650\times 180$ cells and $N=4$ and using RC-EOS for case~II.}
		\label{fig:wu2sb2.eos3}
	\end{figure}
	
	\subsubsection{2-D double Mach reflection}
	The double Mach reflection is a standard benchmark problem used in the literature of non-relativistic hydrodynamic codes since it was introduced in~\cite{woodward1984numerical}. In~\cite{zhang2006ram}, the authors have extended it to the ideal relativistic fluids with adiabatic index $\gamma=1.4$. Later, it was used in several works~\cite{he2012adaptive,zhao2013runge,cao2025robust} to test high-resolution shock-capturing methods for solving RHD equations with the ID-EOS. Here, we use this test case first with the ID-EOS to compare with available literature, and later for the cases of TM-EOS, IP-EOS, and RC-EOS. 
	
	At initial time, an oblique shock wave, moving with velocity $v_s=0.4984$ from left to right is placed at $(x,y)=\left(\frac{1}{6}, 0\right)$ at an angle of $60^\circ$ with the horizontal direction. At time $t$, the position of the shock front is described by~\cite{zhao2013runge,cao2025robust},
	\[
	S(x,t) = \sqrt{3}\left(x-\frac{1}{6}\right) - 2 v_s t.
	\]
	The primitive variables on the left and right of the shock wave are given by,
	\begin{align*}
		(\rho, v_1, v_2, p)|_{\raisebox{-4pt}{L}} &= (8.564, 0.4247 \sin{60^\circ}, -0.4247 \cos{60^\circ}, 0.3808),\\ 
		(\rho, v_1, v_2, p)|_{\raisebox{-4pt}{R}} &= (1.4, 0.0, 0.0, 0.0025),
	\end{align*}
	respectively. The computational domain is taken to be $[0,4]\times[0,1]$, and the boundary conditions are set as the constant post and pre-shock states at left and right boundaries, respectively. The bottom boundary has post-shock state when $x\leq\frac{1}{6}$ and a reflective boundary condition otherwise. For the upper boundary, we set post and pre-shock states when $x<x_s$ and $x>x_s$ respectively, with $x_s$ determined by solving $S(x,t) = 1$. 
	
	We run the simulations till $t=4$, taking $960\times 240$ cells with $N=4$ and present the results in Figure~\ref{fig:DM_cont} for different equations of state. We observe from the Figure~\ref{fig:DM.eos1} that the result obtained using ID-EOS with $\gamma = 1.4$ is similar to the results in the literature where no Kelvin-Helmholtz instability arises~\cite{zhang2006ram,he2012adaptive,zhao2013runge,cao2025robust}. However, the results with the other equations of state seem to develop the Kelvin-Helmholtz instabilities in the solution.
	
	\begin{figure}[]
		\centering
		\begin{subfigure}{0.49\textwidth}
			\includegraphics[width=\linewidth]{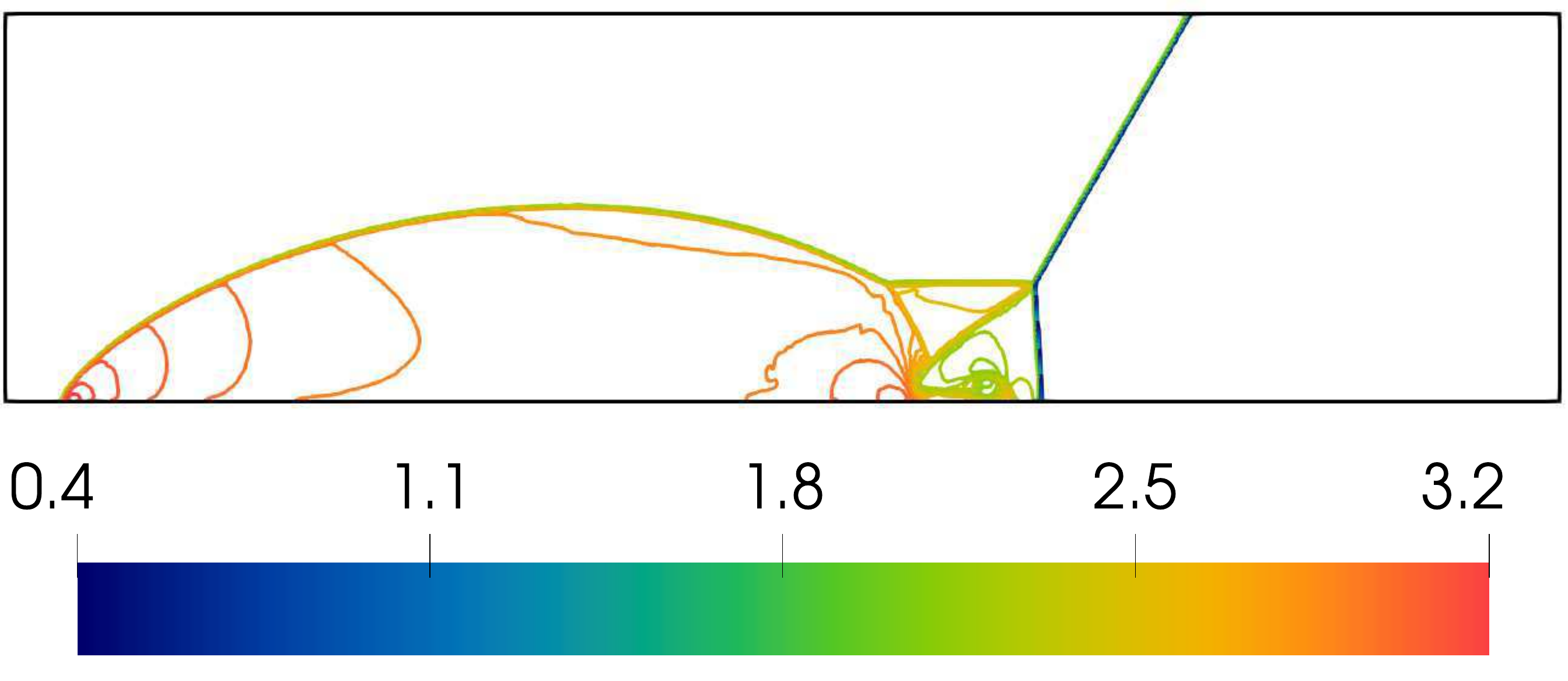}
			\caption{ID-EOS with $\gamma = 1.4$: 50 contours in $[0.4, 3.2]$.}\label{fig:DM.eos1}
		\end{subfigure}
		\begin{subfigure}{0.49\textwidth}
			\includegraphics[width=\linewidth]{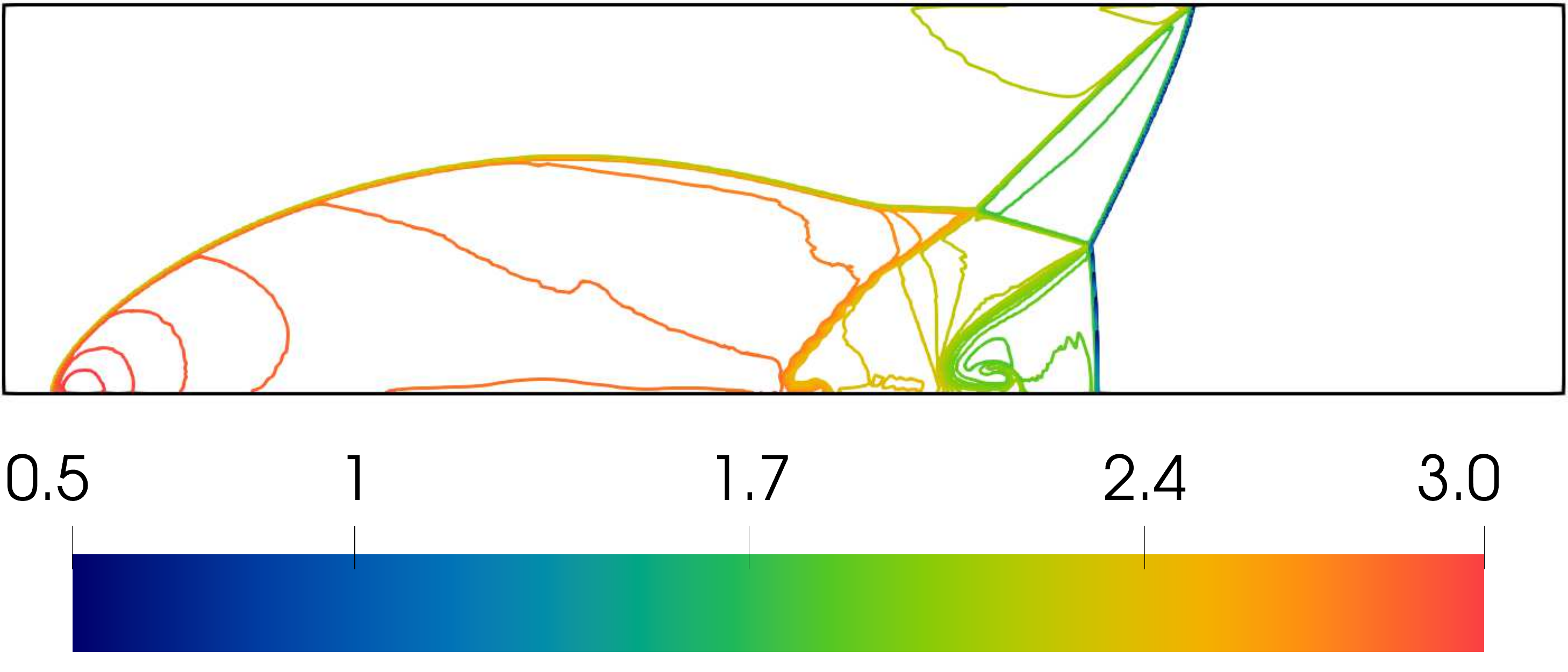}
			\caption{TM-EOS: 50 contours in $[0.5, 3.0]$.}
		\end{subfigure}
		\begin{subfigure}{0.49\textwidth}
			\includegraphics[width=\linewidth]{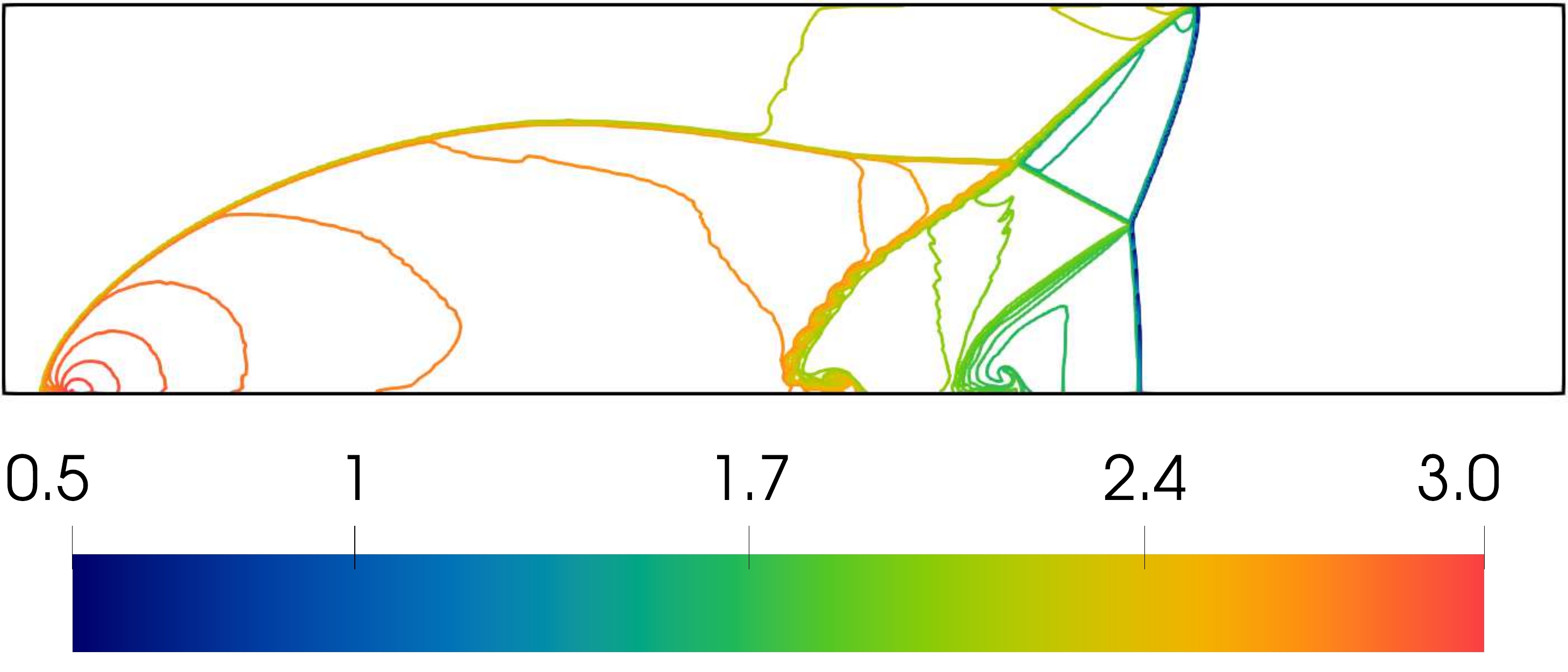}
			\caption{IP-EOS: 50 contours in $[0.5, 3.0]$.}
		\end{subfigure}
		\begin{subfigure}{0.49\textwidth}
			\includegraphics[width=\linewidth]{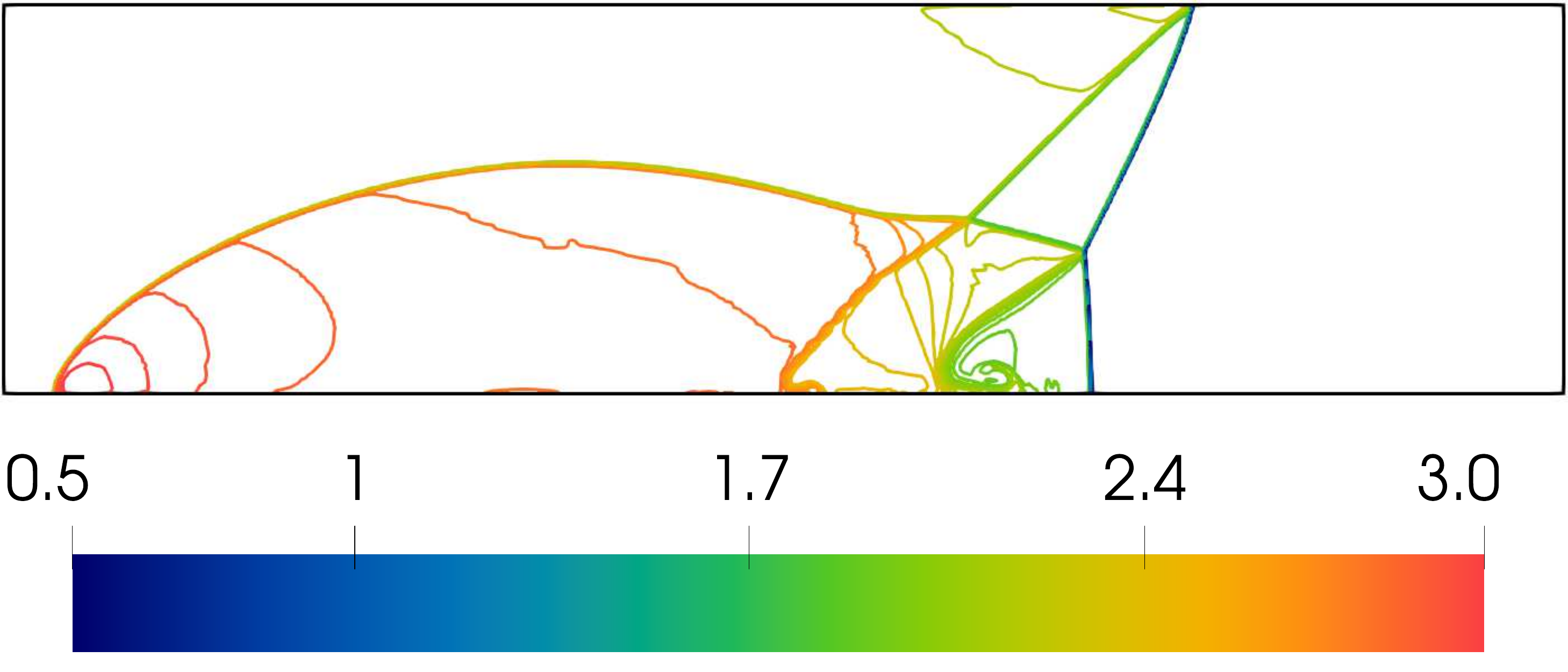}
			\caption{RC-EOS: 50 contours in $[0.5, 3.0]$.}
		\end{subfigure}
		\caption{2-D double Mach reflection: Plot of $\ln \rho$ with $960\times 240$ cells and $N=4$.}
		\label{fig:DM_cont}
	\end{figure}
	
	\subsubsection{2-D Kelvin-Helmholtz instability}
	We take a test from~\cite{beckwith2011second,radice2012thc,zanotti2015high}, known as the Kelvin-Helmholtz instability test, which is a benchmark problem for the RHD codes. The computational domain is taken to be $[-1.0, 1.0]\times [-0.5, 0.5] $ with periodic boundaries. The initial state of the fluid in the left half of the domain $(x<0)$ is given by,
	\begin{align}
		\rho &= 0.505 - 0.495 \tanh \left(\frac{x+0.5}{a}\right),\\
		v_1 &= -\eta_0 v_s \sin(2\pi y) \exp\left(\frac{-(x+0.5)^2}{\sigma} \right),\\
		v_2 &= -v_s \tanh\left(\frac{x+0.5}{a}\right),
	\end{align}
	and in the right-half of the domain $(x>0)$ is given by,
	\begin{align}
		\rho &= 0.505 + 0.495 \tanh \left(\frac{x-0.5}{a}\right),\\
		v_1 &= \eta_0 v_s \sin(2\pi y) \exp\left(\frac{-(x-0.5)^2}{\sigma} \right),\\
		v_2 &= v_s \tanh\left(\frac{x-0.5}{a}\right),
	\end{align}
	with unit pressure in the whole domain. Here, $v_s = 0.5$ and the characteristic size is taken to be $a=0.01$. The velocity in the x-direction $v_1$ is taken with a small perturbation having amplitude and length $\eta_0=0.1$ and $\sigma = 0.1$ respectively, which triggers the small instabilities in the solution. These small instabilities are very hard to capture with a diffusive scheme, and we observe that our scheme can capture these small-scale instabilities for all the equations of state. Here we have presented the result using ID-EOS with $\gamma = \frac{4}{3}$, TM-EOS, IP-EOS, and RC-EOS in Figure~\ref{fig:KH.eos3} at time $t=3$. We have taken $640\times 320$ cells with $N=4$ for the simulations with the indicator parameter $\alpha_{\max}'=0.25$ (refer to Section~5 in~\cite{my_paper} for more details about this parameter).
	
	\begin{figure}[]
		\centering
		\begin{subfigure}{0.49\textwidth}
			\includegraphics[width=\linewidth]{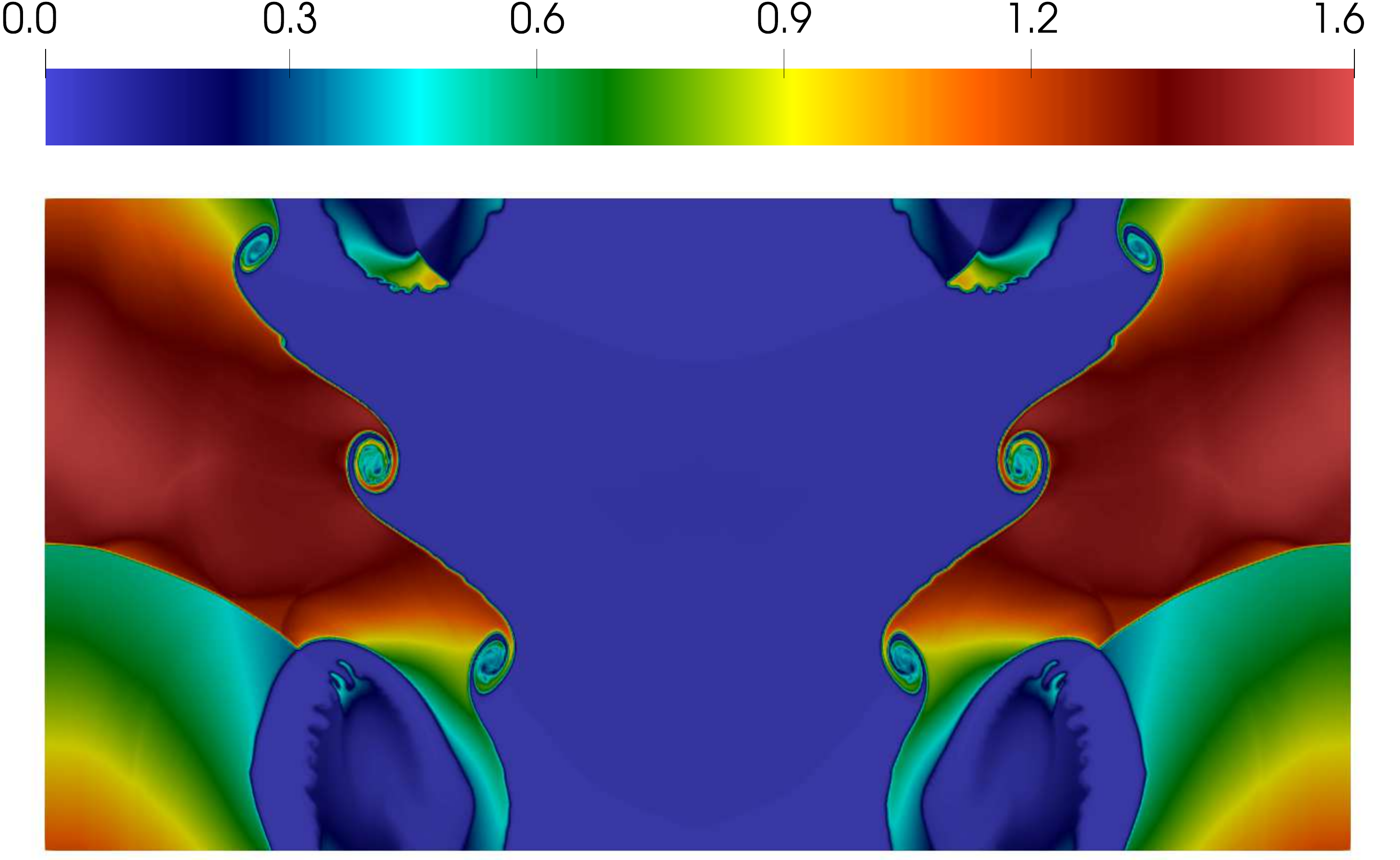}
			\caption{ID-EOS with $\gamma = \frac{4}{3}$.}
		\end{subfigure}
		\begin{subfigure}{0.49\textwidth}
			\includegraphics[width=\linewidth]{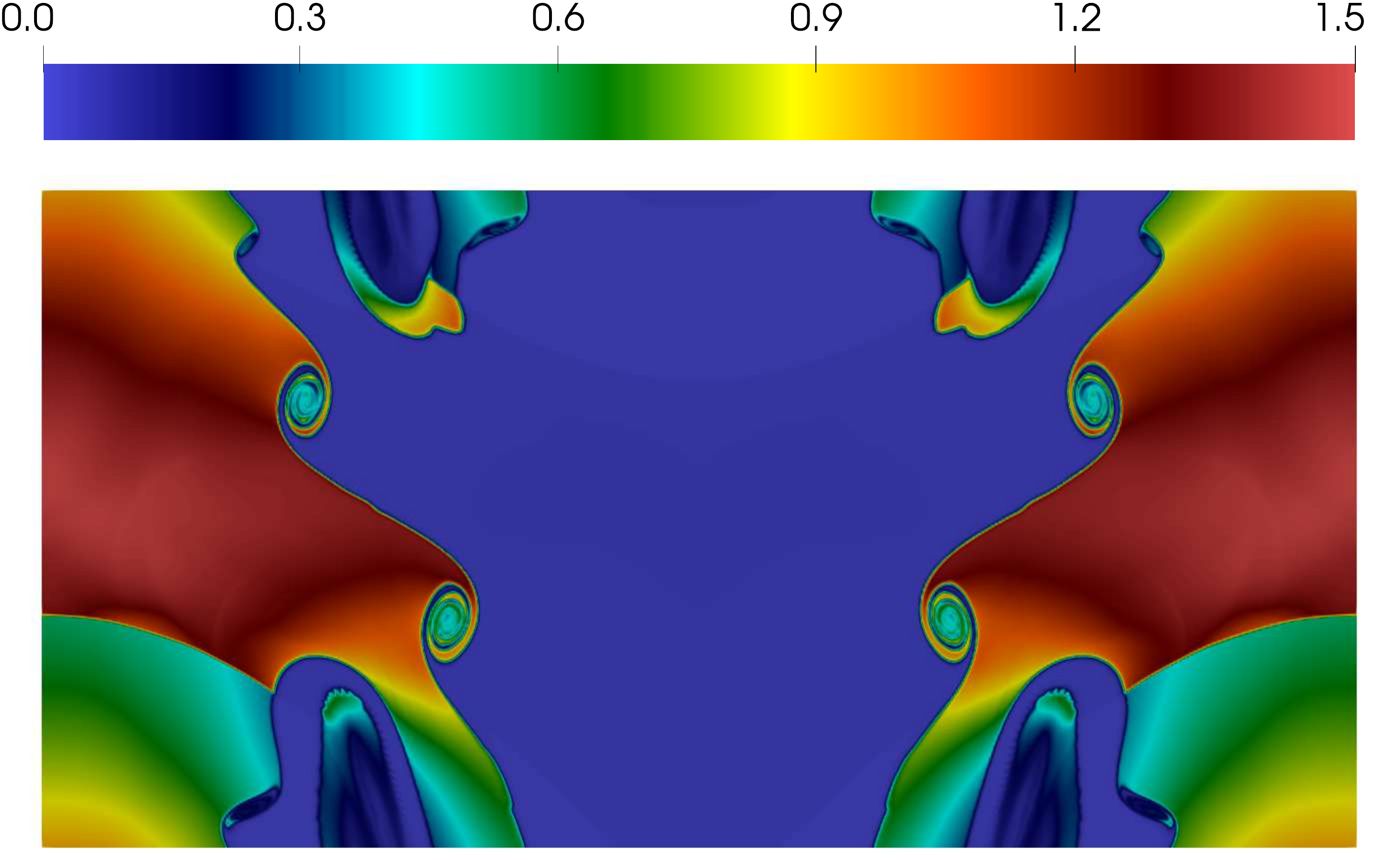}
			\caption{TM-EOS.}
		\end{subfigure}
		\begin{subfigure}{0.49\textwidth}
			\includegraphics[width=\linewidth]{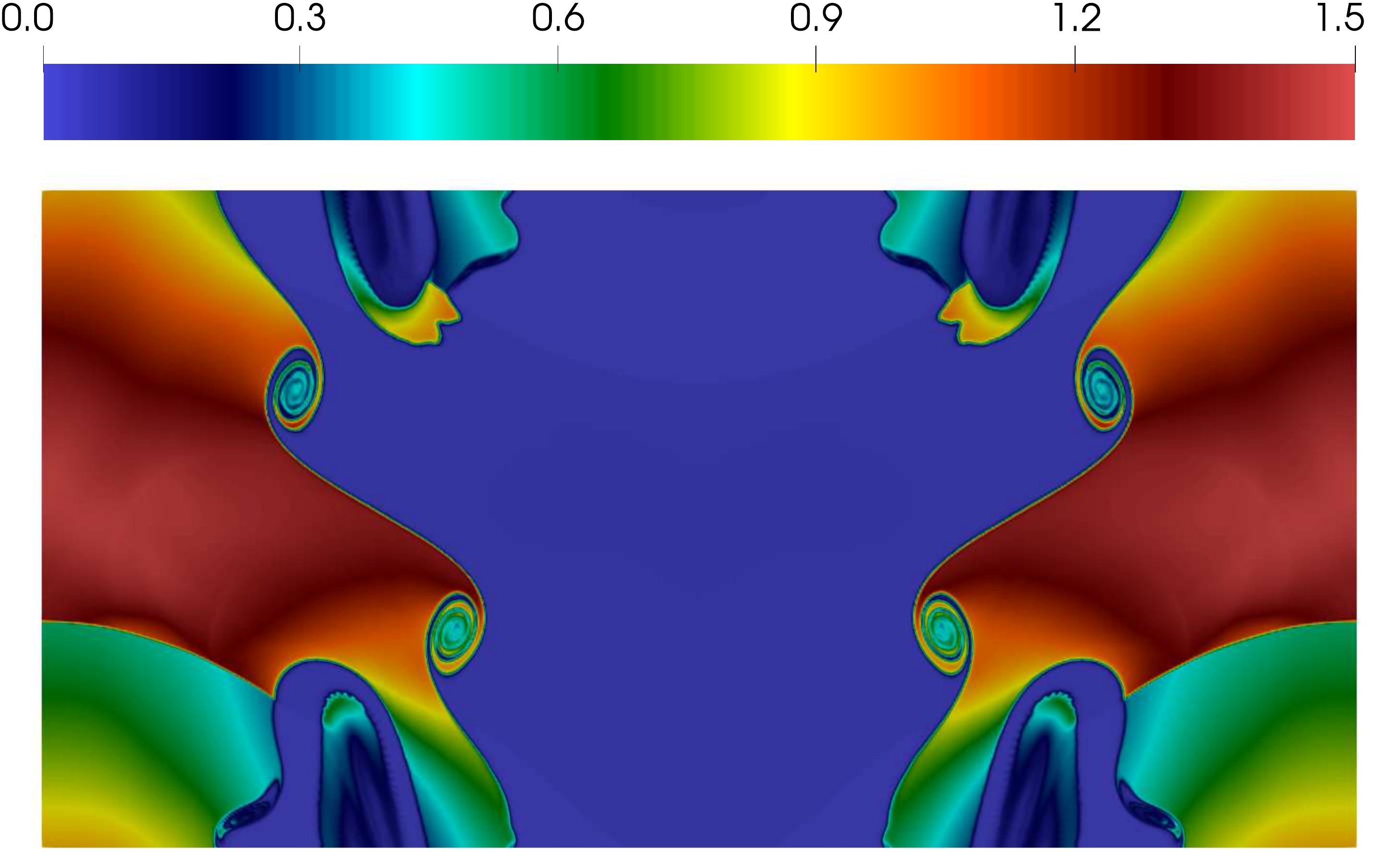}
			\caption{IP-EOS.}
		\end{subfigure}
		\begin{subfigure}{0.49\textwidth}
			\includegraphics[width=\linewidth]{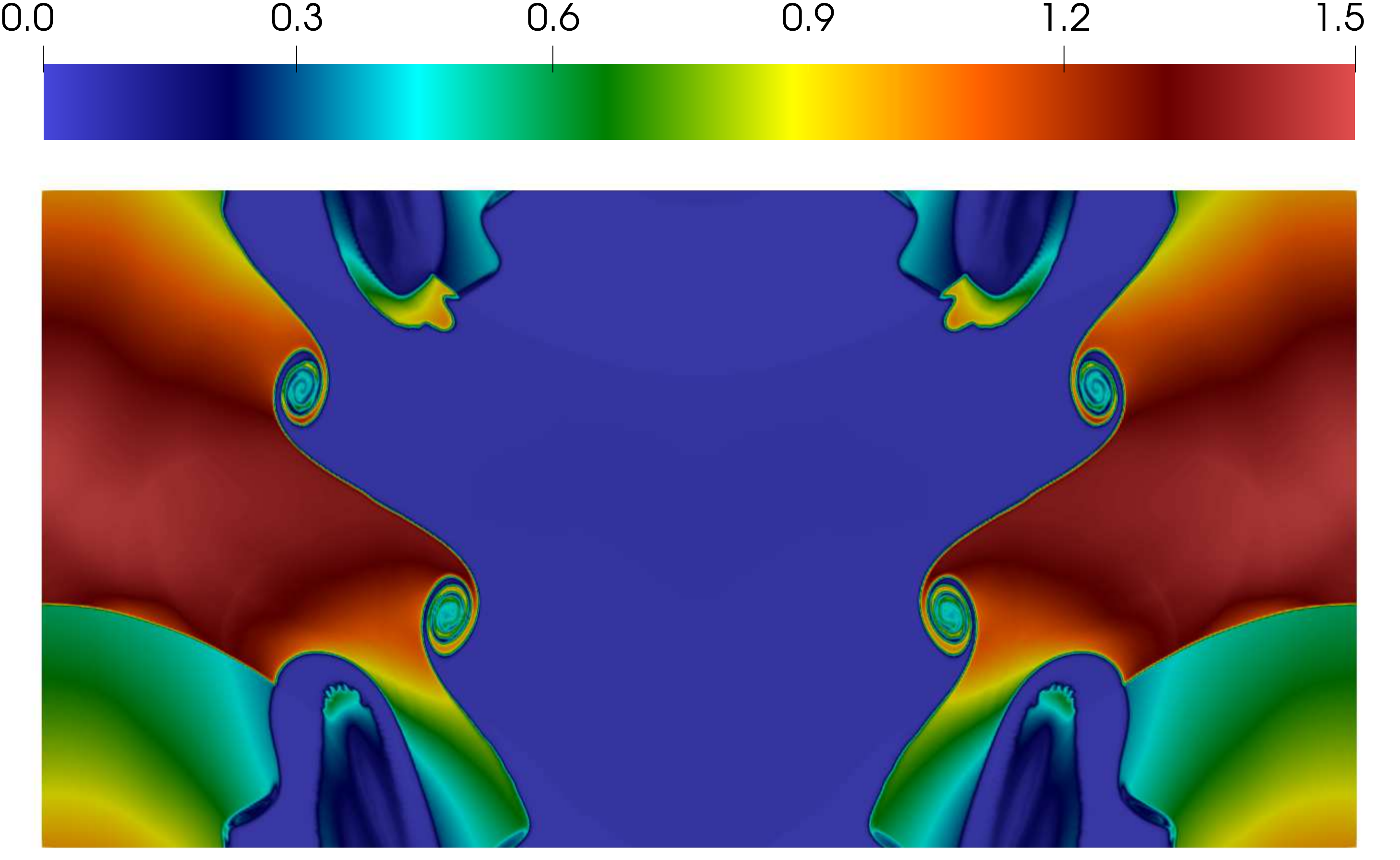}
			\caption{RC-EOS.}
		\end{subfigure}
		\caption{2-D Kelvin-Helmholtz instability: Plot of density ($\rho$) with $640\times 320$ cells and $N=4$.}
		\label{fig:KH.eos3}
	\end{figure}
	
	\section{Summary and conclusions}\label{sec: summary}
	Recently, in~\cite{my_paper}, the authors have designed the Lax-Wendroff flux reconstruction method~\cite{BABBAR2022111423} for the RHD equations with ideal equation of state (ID-EOS). The ID-EOS is derived from non-relativistic thermodynamics, resulting in a poor choice for relativistic cases, making the study of more general equations of state an active area of research. In this work, we have designed a high-order LWFR scheme for RHD equations with several equations of state. Following~\cite{wu2016physical}, we have changed the characterization of the admissible region
	to have concave constraints, which is an essential requirement for the scheme to be implemented. This alternative characterization of the admissible set has constraints that are directly computable from the conservative variables, resulting in an efficient limiting procedure identical to~\cite{my_paper}. However, the conversion from conservative to primitive variables is still needed at various other steps of the scheme. The conversion procedure for the RC-EOS, introduced in~\cite{ryu2006equation}, which is available in the same paper, does not match expectations for some cases, and hence we have proposed a new way of conversion by finding the pressure, which needs a non-linear equation to be solved by the Newton-Raphson method. The second major challenge arose when we needed to find the flux form of quantities that are not inside the admissible region in calculating the time average flux. This is because the flux needs the primitive form of the quantities, and the conservative to primitive conversion needs the conservative variables to be in the admissible region. Hence, we have scaled the required quantities to the admissible region. Again, for the admissibility of the solution, following~\cite{babbar2024admissibility,my_paper} we have blended the high-order method with a first-order finite volume method after proving its admissibility for the case of all the equations of state, used in this work. Finally, we have shown the numerical results with different test cases from the literature using our scheme. From the accuracy tests, we can observe that the additional scaling of the flux arguments does not hurt the accuracy of the scheme, and the numerical order is consistent with the analytical order of the scheme. The variety of test cases presented here shows the robustness of the scheme for the cases having high Lorentz factor, low density or pressure, rarefaction, and strong shock waves or other discontinuities.
	
	\section*{Declarations}
	\textbf{Competing interest:} The authors declare that they have no known competing financial interests or personal relationships that could have appeared to influence the work reported in this paper.
	
	\section*{Acknowledgements}
	The work of Sujoy Basak is supported by the Prime Minister's Research Fellowship, with PMRF ID 1403239.  The work of Arpit Babbar is supported by the Mainz Institute of Multiscale Modeling (M$^\text{3}$ODEL) and the Alexander von Humboldt Foundation. The work of Praveen Chandrashekar is supported by the Department of Atomic Energy, Government of India, under project no.~12-R\&D-TFR-5.01-0520. 
	
	\section*{Data availability}
	The code and data supporting the findings of this work will be made publicly available at~\cite{RHDTenkai}, after the paper gets published.

	\section*{Additional data}
	The animations showing temporal evolutions for some of the simulations can be viewed at
	
	{
		\centering
		\href{https://www.youtube.com/playlist?list=PLrZ1LUocyVaTLnzX45R9QpmqbxX35NhMs}{https://www.youtube.com/playlist?list=PLrZ1LUocyVaTLnzX45R9QpmqbxX35NhMs}
	}
	
	\appendix
	\section{Constraints preserving nature of first-order finite volume scheme}\label{sec: appendix}
	Here, we prove the admissibility constraints preserving nature of the first-order finite volume scheme for the case of TM-EOS, IP-EOS, and RC-EOS. The case of ID-EOS is proved in~\cite{my_paper}.
	The first-order finite volume method at the $i^\text{th}$ solution point can be written as,
	\begin{equation}\label{eq:app low order scheme}
		\mb{u}^{n+1}_i = \mb{u}^n_i - \frac{\Delta t}{w_i \Delta x}[\numflux{\mb{f}}(\mb{u}_i,\mb{u}_{i+1}) - \numflux{\mb{f}}(\mb{u}_{i-1},\mb{u}_i)].
	\end{equation}
	Here, the numerical flux $\numflux{\mb{f}}$ is the Rusanov flux~\eqref{eq:rusanov flux}. The equation~\eqref{eq:app low order scheme} can be expressed as,
	\begin{equation}\label{eq:conv.comb}
		\mb{u}^{n+1}_i = A \mb{u}^{n}_i + B_1 \mb{u}_{\Lambda_{i-\frac{1}{2}}}^{n+} + B_2 \mb{u}_{\Lambda_{i+\frac{1}{2}}}^{n-},
	\end{equation}
	where the coefficients are given by,
	\begin{align}\label{eq:coefficients.conv.comb}
		A = \left[1-\frac{\Delta t}{2 w_i \Delta x}(\Lambda_{i-\frac{1}{2}} + \Lambda_{i+\frac{1}{2}}) \right], \qquad
		B_1 = \frac{\Delta t \Lambda_{i-\frac{1}{2}}}{2 w_i \Delta x}, \qquad B_2 =  \frac{\Delta t \Lambda_{i+\frac{1}{2}}}{2 w_i \Delta x},
	\end{align}
	with
	\begin{equation}
		\Lambda_{i\pm\frac{1}{2}} = \max\big\{r\big(f'(\mb{u}_{i\pm 1}^n)\big), r\big(f'(\mb{u}_i^n)\big)\big\},
	\end{equation}
	and 
	\begin{equation}\label{eq:directional_evolutions_first}
		\mb{u}_{\Lambda_{i-\frac{1}{2}}}^{n+} = \mb{u}^n_{i-1} + \frac{1}{\Lambda_{i-\frac{1}{2}}}\mb{f}(\mb{u}^n_{i-1}), \qquad \mb{u}_{\Lambda_{i+\frac{1}{2}}}^{n-} = \mb{u}^n_{i+1} - \frac{1}{\Lambda_{i+\frac{1}{2}}}\mb{f}(\mb{u}^n_{i+1}).
	\end{equation}
	Now following~\cite{my_paper,wu2017design}, we will first prove the admissibility of the quantities,
	\begin{align}\label{eq:directional_evolutions_second}
		\mb{u}_{\Lambda_{i-1}}^{n+} = \mb{u}^n_{i-1} + \frac{1}{\Lambda_{i-1}}\mb{f}(\mb{u}^n_{i-1}), \qquad \mb{u}_{\Lambda_{i+1}}^{n-} = \mb{u}^n_{i+1} - \frac{1}{\Lambda_{i+1}}\mb{f}(\mb{u}^n_{i+1}).
	\end{align}
	From now on, we will again suppress the temporal and spatial indices in~\eqref{eq:directional_evolutions_second} for the sake of notational simplicity and denote the admissibility constraints~\eqref{eq: ad_region_2} for $\mb{u}_{\Lambda}^\pm$ as,
	\begin{align}\label{eq:ad.constraint.app}
		D^{\pm}_\Lambda \qquad \text{and} \qquad q^{\pm}_\Lambda = E^{\pm}_\Lambda - \sqrt{(D^{\pm}_\Lambda)^2 + |(m_1)^{\pm}_\Lambda|^2},
	\end{align}
	where $D^{\pm}_\Lambda, (m_1)^{\pm}_\Lambda$, and $E^{\pm}_\Lambda$ are the relativistic density, momentum, and energy components of $\mb{u}^{\pm}_\Lambda$.
	The proof of positivity for the first constraint $D^{\pm}_\Lambda$ can be seen directly as,
	\[
	D_\Lambda^{\pm} = D \pm \frac{1}{\Lambda} D v_1  = D\left[1 \pm \frac{v_1}{\Lambda}\right]
	> 0,    \quad \text{since}\ \Lambda>|v_1|>0,\ D>0.
	\]
	For the second constraint $q^{\pm}_\Lambda$, we first show the non-negativity of $E_\Lambda^{\pm}$, and then proving $(D_\Lambda^{\pm})^2 + |(m_1)_\Lambda^{\pm}|^2 - (E_\Lambda^{\pm})^2 < 0$ would be sufficient to show $q^{\pm}_\Lambda > 0$.
	\begin{align}
		E_\Lambda^{\pm} &= E \pm \frac{1}{\Lambda} m_1\\
		&\geq E - \frac{1}{\Lambda} |m_1|, \quad \text{since } \Lambda > 0 \nonumber\\
		&= \rho h \Gamma^2 \left(1 - \frac{1}{\Lambda}|v_1| \right) - p. \label{eq:E.first}
	\end{align}
	Now from~\cite{bhoriya2020entropy,ryu2006equation,my_paper} we have the following expressions for the eigenvalues of the flux Jacobian of the RHD equations,
	\begin{align}
		\lambda_1 = \frac{v_1 - c_s}{1-c_s v_1}, \qquad \lambda_2 = v_1, \qquad \lambda_3 = \frac{v_1 + c_s}{1 + c_s v_1},
	\end{align}
	where $0<c_s<1$ is the speed of sound. Hence, the spectral radius $\Lambda$ is given by,
	\begin{equation}\label{eq:spectral.radius}
		\Lambda = \frac{|v_1| + c_s}{1+c_s |v_1|}.
	\end{equation}
	Since $c_s<1$, it is easy to see that $\Lambda < 1$. Putting this expression of $\Lambda$ into equation~\eqref{eq:E.first} we have,
	\begin{align*}
		E_\Lambda^{\pm} &\geq \rho h \Gamma^2 \left(1 - \frac{1+c_s |v_1|}{|v_1| + c_s}|v_1| \right) - p\\
		& = \rho h \Gamma^2 \left(\frac{c_s(1-|v_1|^2)}{|v_1|+c_s} \right) - p\\
		& = \frac{\rho h c_s}{|v_1| + c_s} - p\\
		& > \frac{\rho h c_s}{1 + c_s} - p, \quad \text{since}\ |v_1|<1\ \text{and}\ \rho, h, c_s > 0\\
		& > \frac{\rho h c_s^2}{1+c_s^2} - p, \quad \text{since}\ 0<c_s<1\ \text{and}\ \rho, h > 0.
	\end{align*}
	Now for TM-EOS, using equation~\eqref{eq: TM_eos} and replacing $c_s$ using~\eqref{eq:sound.TMEOS} we have,
	\begin{align*}
		E_\Lambda^{\pm} &> \left(\frac{5p + \sqrt{9p^2 +4 \rho^2}}{2}\right) \left(\frac{5p \sqrt{9p^2+4\rho^2}+9p^2}{17 p \sqrt{9p^2 +4 \rho^2}+45p^2+6\rho^2}\right) - p\\
		&= \frac{8 p \rho^2}{34p \sqrt{9p^2 + 4\rho^2}+90p^2 + 12 \rho^2}\\
		&>0,\quad \text{since}\ p>0.
	\end{align*}
	Again for IP-EOS, using equations~\eqref{eq: IP_eos} and~\eqref{eq:sound.IPEOS} we get,
	\begin{align*}
		E_\Lambda^{\pm} &> \left(2p + \sqrt{4p^2+\rho^2}\right) \left(\frac{2p \sqrt{\rho^2+4p^2}}{4p^2 + \rho^2 +6p\sqrt{\rho^2+4p^2}}\right) - p\\
		&= \frac{p\sqrt{4p^2+ \rho^2} \left(\sqrt{4p^2+\rho^2}-2p\right)}{4p^2 + \rho^2 + 6p \sqrt{4p^2 + \rho^2}}\\
		&>0,\quad \text{since $p, \rho>0$.}
	\end{align*}
	For RC-EOS, using equations~\eqref{eq: RC_eos} and~\eqref{eq:sound.RCEOS},
	\begin{align*}
		E_\Lambda^{\pm} &> \left(\frac{12p^2+8p\rho +2\rho^2}{3p+2\rho}\right)\left(\frac{p(3p + 2\rho)(18p^2+24p\rho +5\rho^2)}{216p^4 + 432 p^3 \rho + 270 p^2 \rho^2 +70 p \rho^3 + 6\rho^4}\right) - p\\
		&= \frac{18 p^3 \rho^2 + 18 p^2 \rho^3 +4 p \rho^4}{216p^4 + 432 p^3 \rho + 270 p^2 \rho^2 +70 p \rho^3 + 6\rho^4}\\
		&>0,\quad \text{since $p,\rho>0$.}
	\end{align*}
	Hence, we have $E_\Lambda^{\pm} > 0$ for all the cases considered here. Now,
	\begin{align*}
		(D_\Lambda^{\pm})^2 + &|(m_1)_\Lambda^{\pm}|^2 - (E_\Lambda^{\pm})^2\\ 
		&=\left(D\pm \frac{1}{\Lambda}D v_1\right)^2 + \left(m_1 \pm \frac{1}{\Lambda}(m_1v_1 +p)\right)^2 - \left(E\pm \frac{1}{\Lambda}m_1\right)^2\\
		&= \left(1\pm \frac{v_1}{\Lambda}\right)^2 \Gamma^2 \left(\rho^2 + p^2 - (\rho h - p)^2\right) + p^2 \left(\frac{1}{\Lambda^2}-1\right)\\
		&\leq \left(1- \frac{|v_1|}{\Lambda}\right)^2 \Gamma^2 \left(\rho^2 + p^2 - (\rho h - p)^2\right) + p^2 \left(\frac{1}{\Lambda^2}-1\right),\\
		&\mspace{340mu}\text{since}\ \rho^2 + p^2 \leq (\rho h - p)^2,\ \text{by}~\eqref{eq: weaker_taub}\\
		&=\left(\frac{1}{\Lambda^2}-1\right)\left[\left(\frac{\Lambda^2}{1 - \Lambda^2}\right)  \left(1- \frac{|v_1|}{\Lambda}\right)^2 \Gamma^2 \left(\rho^2 + p^2 - (\rho h - p)^2\right) + p^2)\right]\\
		&=\left(\frac{1}{\Lambda^2}-1\right)\left[\left(\frac{c_s^2}{1-c_s^2}\right) \left(\rho^2 + p^2 - (\rho h - p)^2\right) + p^2)\right], \\
		& \mspace{290mu} \text{after some direct calculations using \eqref{eq:spectral.radius}}\\
		&= \left(\frac{1}{\Lambda^2}-1\right) \left(\frac{1}{1-c_s^2}\right) \left[c_s^2\left(\rho^2 - (\rho h - p)^2\right) + p^2 \right].
	\end{align*}
	Hence, to show $(D_\Lambda^{\pm})^2 + |(m_1)_\Lambda^{\pm}|^2 - (E_\Lambda^{\pm})^2 < 0$, it is sufficient to prove 
	\[ c_s^2\left(\rho^2 - (\rho h - p)^2\right) + p^2 < 0\] 
	as $\Lambda<1$ and $c_s < 1$.
	
	Now for TM-EOS, using~\eqref{eq:sound.TMEOS} and~\eqref{eq: TM_eos} we get,
	\begin{align*}
		c_s^2\big(\rho^2& - (\rho h - p)^2\big) + p^2 \\
		&= \left[\frac{5p \sqrt{9p^2 +4\rho^2} + 9p^2}{12p\sqrt{9p^2 +4\rho^2} + 36p^2 +6\rho^2}\right] \left[\rho^2 - \left( \frac{3p + \sqrt{9p^2+4\rho^2}}{2}\right)^2 \right] + p^2\\
		&= - \frac{96p^3 \sqrt{9p^2+4\rho^2}+288p^4 +96p^2\rho^2}{48p\sqrt{9p^2 + 4\rho^2} + 144 p^2 + 24 \rho^2}\\
		&<0,\qquad \text{since $p>0$.}
	\end{align*}
	Again for IP-EOS using~\eqref{eq:sound.IPEOS} and~\eqref{eq: IP_eos} we get,
	\begin{align*}
		c_s^2\big(\rho^2 &- (\rho h - p)^2\big) + p^2 \\
		&= \left[\frac{2p \sqrt{4p^2 + \rho^2}}{4p \sqrt{4p^2 + \rho^2} + 4p^2 + \rho^2}\right] \left[\rho^2 - \left(p+\sqrt{4p^2+\rho^2}\right)^2 \right] + p^2\\
		&= -\frac{6p^3\sqrt{4p^2 + \rho^2}+12 p^4 +3 p^2\rho^2}{4p \sqrt{4p^2 + \rho^2} + 4p^2 + \rho^2}\\
		&<0,\qquad \text{since $p>0$.}
	\end{align*}
	For RC-EOS, using~\eqref{eq:sound.RCEOS} and~\eqref{eq: RC_eos},
	\begin{align*}
		c_s^2\big(\rho^2 &- (\rho h - p)^2\big) + p^2 \\
		&= \left[\frac{p (3p +2\rho) (18p^2 + 24 p\rho + 5\rho^2)}{3 (6p^2 + 4p\rho + \rho^2) (9p^2 + 12 p\rho + 2\rho^2)}\right] \left[\rho^2 - \left( \frac{9p^2+6p\rho+2\rho^2}{3p+2\rho}\right)^2 \right] + p^2\\
		&= -\frac{324p^7+864p^6\rho +954 p^5 \rho^2 + 558 p^4 \rho^3 + 155 p^3\rho^4 +16 p^2 \rho^5}{162 p^5 + 432 p^4 \rho +423 p^3 \rho^2 + 198 p^2 \rho^3 + 46 p \rho^4 + 4 \rho^5}\\
		&<0,\qquad \text{since $p,\rho>0$.}
	\end{align*}
	Hence we have $D_\Lambda^\pm, q_\Lambda^\pm > 0$ implying
	\begin{align*}
		\mb{u}_{\Lambda_{i-1}}^{n+}, \mb{u}_{\Lambda_{i+1}}^{n-} \in \Uad'.
	\end{align*}
	Now using Lemma A.1 of~\cite{my_paper} we have,
	\begin{align*}
		\mb{u}_{\Lambda_{i-\frac{1}{2}}}^{n+}, \mb{u}_{\Lambda_{i+\frac{1}{2}}}^{n-} \in \Uad',
	\end{align*}
	since $\Lambda_{i\pm\frac{1}{2}} > \Lambda_{i\pm1}$. Now, as the quadrature weights and spectral radius are positive, we have $B_1, B_2$ in~\eqref{eq:coefficients.conv.comb} are positive. Again with a CFL-type restriction,
	\begin{align}
		\Delta t\left[ \frac{\Lambda_{i-\frac{1}{2}} + \Lambda_{i+\frac{1}{2}}}{2 w_i \Delta x}\right] < 1,
	\end{align}
	we have positivity of the coefficient $A$ in~\eqref{eq:coefficients.conv.comb}. Hence, the convex combination~\eqref{eq:conv.comb} is admissible, since the admissible set $\Uad'$ is a convex set (see Lemma 2.2 of~\cite{wu2015high}).
	
	\section{Proof of $S'(\Pi)>0$ for $\Pi \geq E$}\label{sec: appendix_s'pi_proof}
	From~\eqref{eq:s'pi}, we have,
	\[
	S'(\Pi) = \left(2 - \frac{T(h)}{h}-T'(h)\right)\Pi - E  = R(h) \Pi - E
	\]
	where,
	\[
	h = \frac{\sqrt{\Pi^2 - |\mb{m}|^2}}{D}, \qquad R(h) = 2 - \frac{T(h)}{h}-T'(h).
	\]
	Now, for $\Pi \geq E$ and $D,\mb{m} \in \Uad'$ we have $h>1$, and
	\begin{align*}
		R'(h) &= - \frac{1}{h^2}\left(h T'(h) + h^2 T''(h) - T(h)\right)\\
		&= - \frac{1}{3 h^2}\left[1 + \frac{288h -27h^3 -216 h^2 -128}{\left(\sqrt{(3h+8)^2 - 96}\right)^3} \right]\\
		& < 0 ,\quad \text{for}\ h>1.
	\end{align*}
	So, $R(h)$ is a decreasing function in the interval $(1, \infty)$ with $R(1)=8/5$, and
	\[
	\lim_{h\to\infty} R(h) = \frac{3}{2}.
	\]
	Hence, $R(h) > 1$ in the interval $(1, \infty)$, giving $S'(\Pi) > 0$ for $\Pi \geq E$.
	
	\newpage
	\section{Julia code for conversion to primitive variables in the case of RC-EOS}\label{sec: appendix_cons2prim}
	\begin{algorithm}[!ht]
		\includegraphics[width=\linewidth]{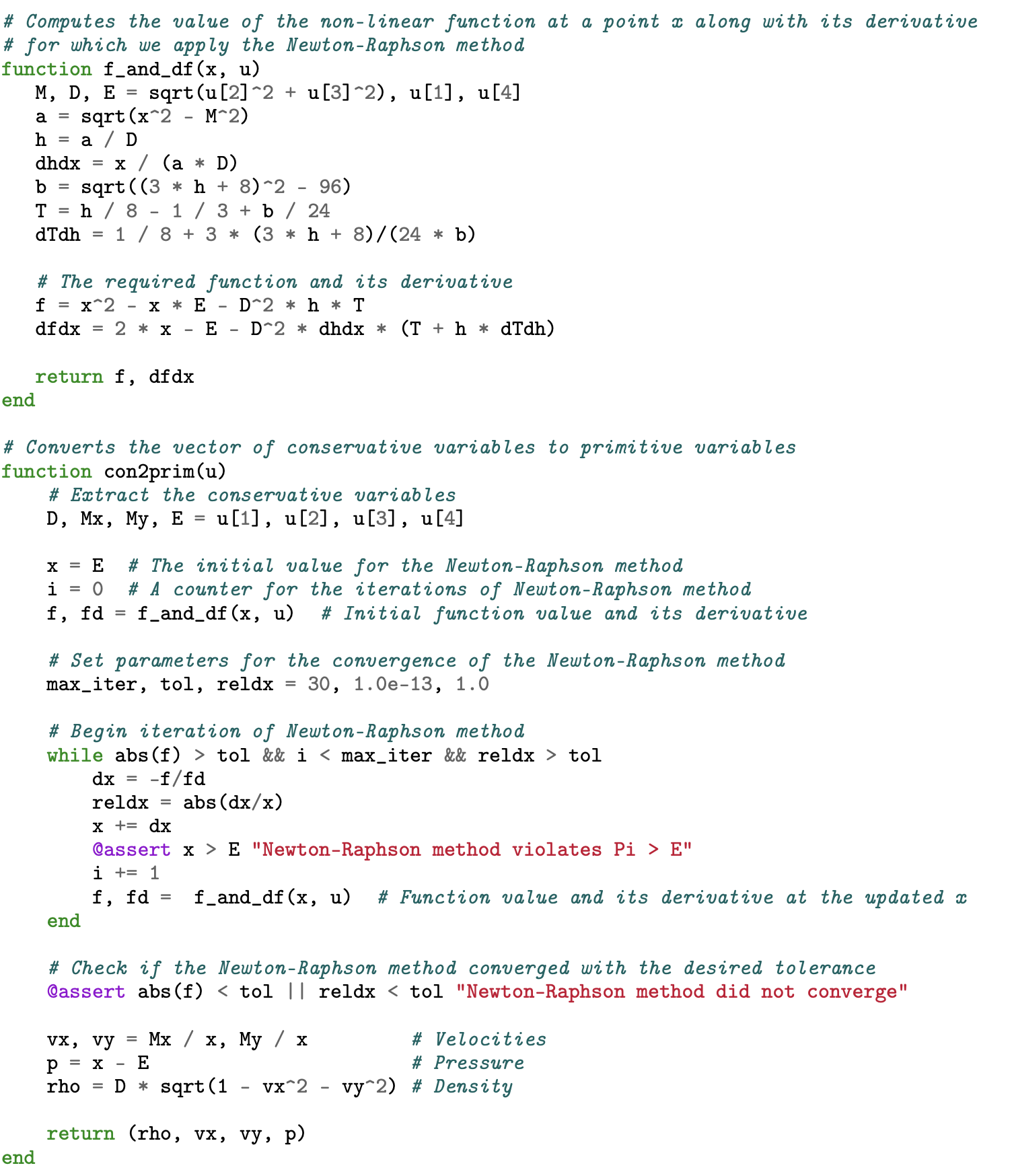}
		\caption{Conservative to primitive conversion for RC-EOS} \label{alg:cons2prim}
	\end{algorithm}
	
	\section{Comparison of conservative to primitive conversion for ideal equation of state}\label{sec: appendix_cons2prim_IDEOS}
	Here, we will compare the methods of conversion from conservative to primitive variables in one dimension from Section~2.3 in~\cite{cai2024provably} and Section~3 in~\cite{schneider1993new}. The second method was also used in~\cite{bhoriya2020entropy}.
	
	From~\cite{cai2024provably}, the pressure $p$ corresponding to the conservative variables $(D,m_1,E)$ satisfies
	\[
	\Phi(p):=\frac{p}{\gamma - 1} - E + \frac{|m_1|^2}{E + p} + D \sqrt{1 - \frac{|m_1|^2}{(E + p)^2}} = 0.
	\] 
	Again from~\cite{schneider1993new,bhoriya2020entropy}, the absolute velocity $|v_1|$ corresponding to the conservative variables $(D,m_1,E)$ satisfies
	\[
	\Xi(v_1):= |v_1|^4 + a_3 |v_1|^3 + a_2 |v_1|^2 + a_1 |v_1| + a_0 = 0,
	\]
	with 
	\begin{align*}
		a_3 &= - \frac{2 \gamma (\gamma - 1) m_1 E}{(\gamma - 1)^2 (m_1^2 + D^2)}, \ \ a_2 = \frac{(\gamma^2 E^2 + 2(\gamma - 1) m_1^2 - (\gamma - 1)^2 D^2)}{(\gamma - 1)^2 (m_1^2 + D^2)},\\
		a_1 &= \frac{-2\gamma m_1 E}{(\gamma - 1)^2 (m_1^2 + D^2)}\ \ a_0 = \frac{m_1^2}{(\gamma - 1)^2 (m_1^2 + D^2)}.
	\end{align*}
	In fact, these are the equations that are solved using some iterative root-finding methods for the conversions in the corresponding papers.
	
	Here, we will explicitly mention three examples to compare the methods in terms of accuracy. We have taken $\gamma = \frac{5}{3}$ for all the examples. We have again taking equal tolerance in the corresponding Newton-Raphson methods from~\cite{cai2024provably} and~\cite{schneider1993new} for the following comparisons.
	
	\vspace{1em}
	\noindent\textbf{Example 1:}\\
	We take the conservative variables,
	\[
	(D, m_1, E) = (0.001, 25.0, 25.001).
	\]
	The converted primitive variables with the method from~\cite{cai2024provably} are,
	\[
	(\new{\rho}, \new{v_1}, \new{p}) = (1.9913276960883976e-5, 0.999801711041084, 0.003958207130631426),
	\]
	and the converted primitive variables with the method from~\cite{schneider1993new,bhoriya2020entropy} are,
	\[
	(\old{\rho}, \old{v_1}, \old{p}) = (1.9913287827370157e-5, 0.9998017108246536, 0.003958210730555717).
	\]
	Now,
	\[
	\Phi(\new{p}) = 2.16817e-13,\quad \Phi(\old{p})= 4.52245e-8,
	\]      
	and
	\[
	\Xi(\new{v_1}) = -4.44089e-16,\quad \Xi(\old{v_1}) = -1.11022e-15.
	\]
	Hence, the new method from~\cite{cai2024provably} is more accurate.
	
	\vspace{1em}
	\noindent\textbf{Example 2:}\\
	We take the conservative variables,
	\[
	(D, m_1, E) = (0.26215012530349685, 42.10522585617847, 42.10705317285818).
	\]
	The converted primitive variables with the method from~\cite{cai2024provably} are,
	\[
	(\new{\rho}, \new{v_1}, \new{p}) = (0.003097928215833704, 0.999930172301406, 0.001112999656126819),
	\]
	and the converted primitive variables with the old method from~\cite{schneider1993new,bhoriya2020entropy} are,
	\[
	(\old{\rho}, \old{v_1}, \old{p}) = (0.003097928355847064, 0.999930172295094, 0.0011129997399655805).
	\]
	Now,
	\[
	\Phi(\new{p}) = -2.01838e-13, \quad \Phi(\old{p})= 3.62513e-9,
	\]      
	and
	\[
	\Xi(\new{v_1}) = 3.33066e-16,\quad \Xi(\old{v_1}) = 3.33066e-16.
	\]
	Here, even though $\Xi(\new{v_1}) \approx \Xi(\old{v_1})$, we can see a significant difference from the values of the $\Phi$ function, and we can say that the new method from~\cite{cai2024provably} is more accurate.
	
	\vspace{1em}
	\noindent\textbf{Example 3:}\\
	Here, we take the conservative variables,
	\[
	(D, m_1, E) = (0.1, 50.0, 50.01).
	\]
	The converted primitive variables with the method from~\cite{cai2024provably} are,
	\[
	(\new{\rho}, \new{v_1}, \new{p}) = (0.004084552892614892, 0.9991654731658531, 0.03176119254315)
	\]
	and the converted primitive variables with the old method from~\cite{schneider1993new,bhoriya2020entropy} are,
	\[
	(\old{\rho}, \old{v_1}, \old{p}) = (0.004084552899268397, 0.9991654731631332, 0.03176119262938212).
	\]
	Now,
	\[
	\Phi(\new{p}) = -3.557432126655158e-133,\quad \Phi(\old{p})= 2.3753539135640267e-9,
	\]      
	and
	\[
	\Xi(\new{v_1}) = 7.771561172376096e-16,\quad \Xi(\old{v_1}) = -8.881784197001252e-16.
	\]
	Here we see the values of the $\Xi$ functions are equal up to the machine precision, but the values of the $\Phi$ function have a significant difference, making the new method from~\cite{cai2024provably} more accurate.
	
	\bibliographystyle{abbrv}
	\bibliography{reference}
\end{document}